\documentclass[12pt,openany,centertags]{amsbook}
\usepackage{amsmath} 
\usepackage{amssymb}
\usepackage{mathtools}
\mathtoolsset{centercolon}
\usepackage{enumerate}
\usepackage{here}
\usepackage{braket}
\usepackage{textcase}
\usepackage{centernot}
\usepackage[matrix,arrow]{xy}
\usepackage{xspace}
\usepackage{tikz}
\usepackage{pgf}
\usepackage{calc}
\usepackage[mathscr]{eucal}
\usepackage[normalem]{ulem}

\usepackage{hyperref}
\usepackage[figure]{hypcap}

\newtheorem{theorem}{Theorem}[chapter]
\newtheorem*{theorem*}{Theorem}
\newtheorem*{theoremI}{Theorem I}
\newtheorem*{theoremII}{Theorem II}
\newtheorem*{theoremIII}{Theorem III}
\newtheorem*{theoremIV}{Theorem IV}
\newtheorem*{theorem1.1}{Theorem 1.1}
\newtheorem*{theorem1.2}{Theorem 1.2}
\newtheorem*{theorem1.3}{Theorem 1.3}
\newtheorem*{theorem7.1}{Theorem 7.1}
\newtheorem*{theorem8.1}{Theorem 8.1}
\newtheorem*{assumptionI}{Assumption I}
\newtheorem*{assumptionII}{Assumption II}

\newtheorem{lemma}{Lemma}[chapter]
\newtheorem{auxlemma}{Aux-Lemma}[chapter]
\newtheorem{proposition}{Proposition}[chapter]
\newtheorem{corollary}{Corollary}[chapter] 
\newtheorem*{corollary*}{Corollary}

\theoremstyle{definition}
\newtheorem{definition}{Definition}
\newtheorem*{definition*}{Definition}

\theoremstyle{remark} 
\newtheorem{conclusion}{Conclusion}[chapter]
\newtheorem{remark}{Remark}[chapter]
\newtheorem*{remark*}{Remark}

\newcommand{\Red}{black}
\newcommand{\Black}{red}

\newcommand{\mP}{\mathbb{P}}
\newcommand{\A}{\mathbb{A}}
\newcommand{\C}{\mathbb{C}}
\newcommand{\G}{\mathbb{G}}
\newcommand{\N}{\mathbb{N}}
\renewcommand{\P}{\mathbb{P}}
\newcommand{\Q}{\mathbb{Q}}
\newcommand{\R}{\mathbb{R}}
\newcommand{\Z}{\mathbb{Z}}
\newcommand{\FF}{\mathbf{F}}
\newcommand{\HH}{\mathbf{H}}
\newcommand{\HHCM}{\mathbf{H}^{\mathrm{CM}}}
\newcommand{\HHcm}{\mathbf{H}^{\mathrm{cm}}}
\newcommand{\GG}{\mathbf{G}}
\newcommand{\CC}{\mathbf{C}}

\newcommand{\cA}{\mathcal{A}}
\newcommand{\cB}{\mathcal{B}}
\newcommand{\cC}{\mathcal{C}}
\newcommand{\cD}{\mathcal{D}}
\newcommand{\cE}{\mathcal{E}}
\newcommand{\cF}{\mathcal{F}}
\newcommand{\sF}{\mathscr{F}}
\newcommand{\cG}{\mathcal{G}}
\newcommand{\sG}{\mathscr{G}}
\newcommand{\cH}{\mathcal{H}}
\newcommand{\cI}{\mathcal{I}}
\newcommand{\cJ}{\mathcal{J}}
\newcommand{\sJ}{\mathscr{J}}
\newcommand{\cK}{\mathcal{K}}
\newcommand{\cL}{\mathcal{L}}
\newcommand{\cM}{\mathcal{M}}
\newcommand{\cN}{\mathcal{N}}
\newcommand{\cO}{\mathcal{O}}
\newcommand{\cp}{\mathfrak{p}}
\newcommand{\cq}{\mathfrak{q}}
\newcommand{\cP}{\mathcal{P}}
\newcommand{\cR}{\mathcal{R}}
\newcommand{\sR}{\mathscr{R}}
\newcommand{\cS}{\mathcal{S}}
\newcommand{\cY}{\mathcal{Y}}

\newcommand{\cZ}{\mathcal{Z}}
\newcommand{\fp}{\mathfrak{p}}
\newcommand{\fq}{\mathfrak{q}}
\newcommand{\fm}{\mathfrak{m}}
\newcommand{\fX}{\mathfrak{X}}

\renewcommand{\phi}{\varphi}
\renewcommand{\Im}{\mathrm{Im}}
\newcommand{\oring}[1]{\stackrel{\circ}{#1}}
\newcommand{\dcup}{\,\dot{\cup}\,}
\newcommand{\id}{\mathrm{id}}
\newcommand{\CM}{\operatorname{CM}}
\newcommand{\Grass}{\operatorname{Grass}}
\newcommand{\Hilb}{\operatorname{Hilb}}
\newcommand{\Quot}{\operatorname{Quot}}

\newcommand{\Char}{\mathrm{char}}
\newcommand{\red}{\mathrm{red}}
\newcommand{\reg}{\mathrm{reg}}
\newcommand{\Ann}{\operatorname{Ann}}
\newcommand{\Ass}{\operatorname{Ass}}
\newcommand{\Aut}{\operatorname{Aut}}
\newcommand{\Div}{\operatorname{Div}}
\newcommand{\Ext}{\operatorname{Ext}}
\newcommand{\cExt}{\operatorname{\mathscr{E}xt}}
\newcommand{\Hom}{\operatorname{Hom}}
\newcommand{\cHom}{\operatorname{\mathscr{H}om}}
\newcommand{\Flag}{\operatorname{Flag}}
\newcommand{\GL}{\operatorname{GL}}
\newcommand{\HP}{\operatorname{HP}}
\newcommand{\Inv}{\operatorname{Inv}}
\newcommand{\Ker}{\operatorname{Ker}}
\newcommand{\NNT}{\operatorname{NNT}}
\newcommand{\NS}{\operatorname{NS}}
\newcommand{\PGL}{\operatorname{PGL}}
\newcommand{\Pic}{\operatorname{Pic}}
\newcommand{\Proj}{\operatorname{Proj}}

\newcommand{\Spec}{\operatorname{Spec}}
\newcommand{\Symm}{\operatorname{Symm}}
\newcommand{\Tor}{\operatorname{Tor}}
\newcommand{\depth}{\operatorname{depth}}
\newcommand{\length}{\operatorname{length}}
\newcommand{\colength}{\mathrm{colength}}

\newcommand{\supp}{\mathrm{supp}}
\newcommand{\ord}{\mathrm{ord}}
\newcommand{\lambdadeg}{\lambda\mbox{-}\deg}
\newcommand{\alphadeg}{\alpha\mbox{-}\deg}

\newcommand{\grad}{\mathrm{grad}}

\newcommand{\proj}{\mathrm{proj}}
\newcommand{\inn}{\mathit{in}}

\newcommand{\bigsum}{\sum}

\numberwithin{figure}{chapter}
\newcommand{\An}{$A(0)$\xspace}
\newcommand{\Aln}{$A(\lambda_0)$\xspace}

\numberwithin{equation}{chapter}
\makeindex
\begin{document}

\baselineskip=16pt
\frontmatter
\thispagestyle{empty}

  \begin{center}
    \vspace{4.5cm} { \fontsize{20}{20} \textbf{
        Some geometric properties of\\[3mm] Hilbert schemes of space curves}
       }\\
      \vspace{2cm} {\Large
        Gerd Gotzmann\\[5mm] \large
   } 
\centerline{20-6-2014}
\vspace{1.5cm}
      { \textsc{Abstract}}\\[5mm]
      \begin{minipage}{0.9\linewidth}\small
 Let $H$ be the Hilbert scheme of curves in complex
  projective $3$-space, with degree $d\geq 3$ and genus $g \leq (d-2)^2/4$. A
  complete, explicit description of the cone of curves and the ample cone of
  $H$ is given. From this, partial results on the group $\Aut(H)$ are
  deduced.
\end{minipage}
  \end{center}
\vspace{0.5cm}
\begin{center} {\Large \bf Introduction}
\end{center}
\vspace{0.5cm} 
 The ground field is $k=\C$ and we always assume $d\geq 3$, $g\leq
  g(d):=(d-2)^2/4$, $\HH = \Hilb^P(\P^3_k)$, $P(n) = dn-g+1$. \\

\textsc{Chapter 1}. We prove:
\begin{theorem1.1}
 Rational equivalence = numerical equivalence.
\end{theorem1.1}

  The complementary Hilbert polynomial $Q(n) = \binom{n+3}{3} -P(n)$ has the
  form $Q(n) = \binom{n-1+3}{3}+\binom{n-a+2}{2} + (n-b+1)$, where $a=d+1$
  and $g= (a^2-3a+4)/2 -b$.

\begin{theorem1.2}
  The cone of (effective) curves is freely generated by (the equivalence classes of) the
  following curves:
  \begin{align*}
    C_0 & = \Set{ (x^2, xy, xz, y^a,y^{a-1}z^{b-a+1},
      xt^{b-2}+\alpha y^{a-1}z^{b-a}) | \alpha \in k }^- \\
    C_1 & = \Set{ (x,y^a,y^{a-1}z^{b-a}(\alpha z+t))| \alpha\in k }^- \\
    C_2 & = \Set{(x,y^{a-1}(\alpha y+z),y^{a-2}z^{b-a+1}(\alpha y+z)) |
      \alpha\in k}^- .
  \end{align*}
\end{theorem1.2}

Let $\CC \subset \HH\times \P^3$ be the universal curve with Hilbert
polynomial $P$ over $\HH$. Let $\cF$ be the structure sheaf of $\CC$, let
$\pi$ be the projection from $\HH \times \P^3$ onto $\HH$ and $\cF_n:=\pi_*\cF(n)$. Then
$\cF_n$ is locally free of rank $P(n)$ on $\HH$ for all $n\geq d-2$ and 
$\cM_n:=\dot{\bigwedge}\cF_n$ is called tautological line bundle on $\HH$,
where the dot denotes the exterior power of highest degree. Put
$\rho:=(b-a)(b-a+1)/2$, $\cL_2:= \cM_{n-1}\otimes
\cM_n^{-2}\otimes\cM_{n+1}$ if $n\geq d-1$ is any integer,
$\cL_1:=\cM_{d-2}^{-1}\otimes \cM_{d-1}$ and $\cL_0:=\cM_{b-1} \otimes
\cL_{2}^{-\rho}$.

\begin{theorem1.3}
  The ample cone of $\HH$ is freely generated by (the classes of) $\cL_0$,
  $\cL_1$, $\cL_2$.
\end{theorem1.3}
\index{cone}
In a simple direct way it is proved that $\cM_{n-1}^{-1}\otimes \cM_n$ is
globally generated, if $n\geq d-1$, especially $\cL_1$ is globally
generated.  As for $\cL_2$, one has to use the method of Fogarty (see~
\cite[Section 3]{F1}) to show that $\cL_2$ is globally generated.  In spite
of every effort, I could not decide, if $\cL_0$ is globally generated or
not, even in the case of $H_{4,1}$. So the only example I know is the case
$H_{3,0}$, where $\cL_0=\cM_3$ is globally generated.

\textsc{Chapter 2}. We will determine those curves, which lie on some of
the subcones of the cone of curves. One cannot expect to obtain complete
results, but at least one can show that curves, which are rationally equivalent
to multiples of $[C_1]$ or $[C_2]$, lie on a special subscheme $H_m$
respectively $\cG$ of $\HH$, which we will have to use later on (Corollary
2.1 and Proposition 2.2).

\textsc{Chapter 3}. This is an attempt to understand Fogarty's general
construction of certain morphisms $\omega_t^P(m)$ from $\Hilb^P(\P^N_k)$ to projective
spaces, at least in the case $N=3$, $t=1$, $P(n)=dn-g+1$. For this reason it
is shown in a direct way that the fibres of the morphism $f_n$, which is
defined by the globally generated line bundle $\cM_{n-1}^{-1}\otimes\cM_n$,
have the same description as the fibres of $\omega_1^P(m)$. The difference
is that in \cite[Theorem 10.4, p. 84]{F1} one has to choose $m\gg 0$,
whereas now one only has to suppose $n\geq d$. Moreover, it is shown (by
means of the method Fogarty used in the proof of \cite[Proposition 2.2]{F2})
that each two closed points in a fibre of $f_n$ can be connected by a curve
rationally equivalent to a multiple of $[C_0]$, a result, which one has to 
use in Chapter 6.

We now try to approach the group $\Aut(\HH)$ of $k$-automorphisms of $\HH$,
which we will do in several steps:

\textsc{Chapter 4}. We show that $\Aut(\HH)$ trivially acts on the first
Chow groups $A_1(\HH)$ and $A_1(\CC)$, if $d\geq 5$ is supposed. (Probably
this is true if $d\geq 3$, but I cannot prove it.) Moreover, it is shown that
the subschemes $H_m$ and $\cG$ of $\HH$, which are mentioned above (and are
constructed in Appendix~\ref{cha:C}), are invariant under $\Aut(\HH)$.

\textsc{Chapter 5}.
If $\phi\in \Aut(\HH)$ and $d\geq 6$ is supposed, we show that there is a
$\gamma \in G:=\PGL(3;k)$ such that $\phi|H_m = \gamma|H_m$.  Replacing
$\phi$ by $\phi\circ \gamma^{-1}$, one obtains a so called \emph{normed}
automorphism of $\HH$. The set of all such normed automorphism is a subgroup
$N$ of $\Aut(\HH)$, which is normalized by $G$.
Moreover,  we prove that $\phi|\cG=\mathrm{id}$ for all $\phi \in N$. 
(Here one has to use $\Aut(\Hilb^d(\P^2)) =\PGL(2;k)$, which is proved in
Appendix~\ref{cha:D} under the assumption $d\geq 6$.)

\textsc{Chapter 6}. A very nice result would be to show that
$\phi(\xi)=\xi$ for all $\xi \in \HH(k)$ and all $\phi \in N$, but with
regard to the methods used here, this seems to be impossible.  If however
the ideal corresponding to $\xi$ has a special shape, similar or weaker
results hold true and are used in
\textsc{Chapter 7}. Here the result is:

\begin{theorem7.1}
  Suppose $d\geq 6$. If $h$ is the Hilbert--Chow morphism, then
  $h(\phi(\xi))=h(\xi)$ for all $\xi \in \HH(k)$ and all $\phi\in N$.
\end{theorem7.1}

\textsc{Chapter 8}. Recall from Chapter 3 that the tautological morphism
$f$ is defined by the globally generated line bundle $\cL_1\otimes\cL_2$ (or
$\cM_{n-1}^{-1}\otimes\cM_n$, if $n\geq d$).

\begin{theorem8.1}
  Suppose $d\geq 6$. Then $f(\phi(\xi)) = f(\xi)$ for all $\xi \in \HH(k)$
  and all $\phi\in N$.
\end{theorem8.1}
\begin{remark*}
In more concrete terms one can express this as follows. Let $\xi \in \HH(k)$
correspond to the ideal $\cI = \cJ \bigcap \cR$, where $\cJ$ is the
$\CM$-part \index{CM-part} and $\cR$ is the punctual part
\index{punctual part} which are defined in the following way: The curve defined by the ideal $\cJ$ in
$\P^3$ has no embedded or isolated points (we call such an ideal a $\CM$-ideal \index{CM - ideal} in order to avoid the correct but awkward notation "locally Cohen - Macaulay") and $\cR = \bigcap Q_i$,
where the $Q_i$ are primary to ideals $P_i$, which corresponds to
different closed points in $\P^3$.

Then $\phi(\xi)$ corresponds to the ideal $\cJ \cap \cR'$, where $\cR'=
\bigcap Q'_i$, the $Q'_i$ are $P_i$-primary and $\length(\cJ/\cJ\cap Q_i) =
\length(\cJ/\cJ\cap Q'_i)$ for all $i$.
\end{remark*}

\begin{corollary*}
 Assume as before that $d\geq 6 $ and $g\leq g(d) : = (d-2)^2/4 $. Let $\HHCM$, respectively $\HHcm$, denote the open subscheme of $\HH$, whose closed points correspond to curves without embedded or isolated points, respectively to curves without embedded points. Then the restriction of a $k$-automorphism  of $\HH$ to $\HHCM $, respectively to $\HHcm$, is induced by a linear transformation of $\P^3_k$.
\end{corollary*}

\noindent
\textit{Acknowledgment}: This article could not have been written in \LaTeX\ without the help of Christian Gorzel.

\tableofcontents
\mainmatter

\chapter[The cone of curves and the ample cone]
{The cone of curves and the ample cone of a\\  Hilbert scheme of space curves}
\label{cha:1}

\section{Notations and summary of earlier results}
\label{sec:1.1}

The ground field is $k=\C$ and $\HH = H_{d,g} = \Hilb^P(\P^3_k)$ is the
Hilbert scheme, which parametrizes the curves of degree $d$ and genus $g$ in
$\P^3_k$, i.e.~the closed subschemes of $\P^3_k$ with Hilbert polynomial
$P(n) = dn-g+1$. We also write $\HH_Q$ instead of $H_{d,g}$ in order to
express this Hilbert scheme likewise parametrizes the ideals $\cI\subset
\cO_{\P^3}$ with Hilbert polynomial $Q(n) = \binom{n+3}{3} - P(n)$.
According to F.~S.~Macaulay $\HH_Q$ is not empty if and only if $Q(n)$ has
the form $Q(n) = \binom{n-1+3}{3}+\binom{n-a+2}{2}$ or $Q(n) =
\binom{n-1+3}{3}+\binom{n-a+2}{2} + \binom{n-b+1}{1}$, where $a$ is an
integer $\geq 1$, respectively $a$ and $b$ are integers (Macaulay
coefficients) such that $2\leq a \leq b$. Between the degree and the genus
on the one hand and the Macaulay coefficients on the other hand, one has the
following relations:
\begin{alignat*}{3}
  d&=a,\quad &  g&=(d-1)(d-2)/2, \quad & \text{if } Q(n) &= \tbinom{n-1+3}{3}+\tbinom{n-a+2}{2}\\[-1ex]
  \intertext{and}
  d& =a-1,\quad & g&=(a^2-3a+4)/2 -b, \quad & \text{if }  Q(n) &= \tbinom{n-1+3}{3}+\tbinom{n-a+2}{2} + \tbinom{n-b+1}{1},
\end{alignat*}
respectively. One sees that the first case occurs if and only if one is
dealing with plane curves. Therefore in the following we always suppose
$d\geq 3$ and $g<(d-1)(d-2)/2$.

In the following we furthermore assume that $g\leq g(d):=(d-2)^2/4$, because
in this case we have ``rational equivalence = numerical equivalence'' (see
below Theorem~\ref{thm:1.1}). But as, according to a theorem of Castelnuovo 
(see~\cite[Thm. 6.4, p. 351]{H}), $d\geq 3$ and $g\leq (d-2)^2/4$ is a
necessary condition for $H_{d,g}$ to contain a point, which corresponds to a
smooth curve, this does not seem to be an artificial assumption.

If $\CC$ is the universal curve over $\HH$, one has the diagram 
\[
 \xymatrix{\CC \ar@{^{(}->}[rr] \ar[dr]_f & & \HH\times_k \P^3 \ar[dl]^\pi 
                        \ar[dr]^\kappa &  \\
                                         & \HH &  & \P^3}   
\]
where $\pi$ and $\kappa$ are the projections and $f$ is a surjective flat
morphism, such that for all $\xi \in \HH$ the fiber $f^{-1}(\xi)$ is a
closed curve $C_\xi \subset \P^3 \otimes k(\xi)$ with Hilbert polynomial
$P$. If $\cI$ is the universal ideal sheaf on $X:=\HH\times \P^3_k$, which
defines $\CC$, then $\cF:=\cO_X/\cI$ is the structure sheaf of $\CC$. If one
puts $\cI(n):=\cI \otimes \kappa^*\cO_{\P^3}(n)$ and 
$\cF(n):=\cF \otimes \kappa^*\cO_{\P^3}(n)$ and if $n\geq b-1$, then one
has exact sequences
\[
0 \longrightarrow \pi_*\cI(n) \longrightarrow \cO_\HH\otimes S_n
\longrightarrow \pi_*\cF(n) \longrightarrow 0,
\]
where $S=k[x,y,z,t]$ and $\pi_*\cI(n)$ and $\pi_*\cF(n)$ are locally free on
$\HH$ of rank $Q(n)$ respectively $P(n)$.
$\cM_n:=\dot{\bigwedge}\pi_*\cF(n)$ is called tautological line bundle.
\index{tautological line bundle} This is valid for all $n\geq b-1$, because
each ideal in $\cO_{\P^3}$ with Hilbert polynomial $Q$ is $b$--regular
\cite[Lemma 2.9]{G78}. As we will show in Section~\ref{sec:1.5}, $\cM_n$ is
a line bundle for all $n\geq d-2$.

$A_1(-)$ denotes the first Chow group with coefficients in $\Q$ and
$\NS(-)=\Pic(-)/\Pic^0(-)$ is the N\'eron--Severi group. If one assumes that
$a\geq 4$ and $b\geq (a^2-1)/4$, i.e.~$d\geq 3$ and $g\leq g(d)$, then one
has the following results.

\begin{theoremI}
\label{thm:I}
  $A_1(\HH)$ is freely generated by the rational equivalence classes of the
  following curves:
\begin{align*}
  C_0 & = \Set{ (x^2,xy,xz,y^a,y^{a-1}z^{b-a+1} ,
    xt^{b-2}+\alpha y^{a-1}z^{b-a}) | \alpha \in k }^- \\
  D  &  = \Set{ (x^2,xy,y^{a-1},z^{b-2a+4}(y^{a-2}+\alpha xz^{a-3}))
   |\alpha\in k}^-\\
  C_2 & = \Set{(x,y^{a-1}(\alpha y+z),y^{a-2}z^{b-a+1}(\alpha y+z))
   |\alpha\in k}^- .
\end{align*}
\end{theoremI}

\begin{theoremII}
\label{thm:II}
 The classes of $\cM_n$, $\cM_{n+1}$, $\cM_{n+2}$ freely generate $\NS(\HH)$
 over $\Z$, for all $n\geq b-1= d(d-1)/2-g$.
\end{theoremII}

In order to formulate corresponding results for $\CC$, one defines curves on
$\CC$ by putting 
\[
  C_0^*:=C_0\times\{p_\infty\}, \quad D^*:=D\times\{p_\infty\}, \quad 
C_2^*:=C_2\times\{p_\infty\}, \quad L^*:=\{ \omega\} \times L, 
\]
where $p_\infty =(0:0:0:1)\in \P^3$, $L = V(x,y) \subset \P^3$ and $\omega$
is the closed point of $\HH$ corresponding to the lexicographic ideal with
Hilbert polynomial $Q$.

\begin{theoremIII}
\label{thm:III}
  The rational equivalence classes of $C_0^*$, $D^*$, $C_2^*$ and $L^*$ form
  a basis of $A_1(\CC)$.
\end{theoremIII}

\begin{theoremIV}
\label{thm:IV}
  $\NS(\CC)$ is freely generated by the classes of $\pi^*\cM_n$,
  $\pi^*\cM_{n+1}$, $\pi^*\cM_{n+2}$ and $\kappa^*\cO_{\P^3}(1)$ for
  all $n\geq b-1$.
\end{theoremIV}

If $d\geq 5$ is uneven or $d\geq 8$ is even, these are the results of
\cite[S\"atze 2--5]{T4}, and the remaining cases are treated in
\cite[p. 119--127]{T5}. At this point I would like to mention that the
largest part of \cite{T5} serves to prove that $g\leq (d-2)^2/4$ is a
necessary condition (see \cite[last line of page 2 and Chap.~17,
p.~129]{T5}).

\section{Rational and numerical equivalence}
\label{sec:1.2}

Let be
\[
  Z\in A_1^\tau(\HH) = \Set{ Z\in A_1(\HH) | (\cL\cdot Z) = 0 \text{ for all
    } \cL \in \Pic(\HH) }.
\]
According to Theorem I one can write
\[
  Z = q_0[C_0] + q_1[D] + q_2[C_2], \quad  q_i \in \Q.
\]
Using the formulas of \cite[pp.134 - 135]{T2} it follows that 
\[
   (\cM_n\cdot Z) = q_0 + q_1 (n-b+a-1) + q_2\left[ \tbinom{n-a+2}{2} +
     (n-b+1)\right] =0,
\]
hence $q_0 = q_1 = q_2 =0$.

As the restriction of $\pi$ to $D^*$ induces an isomorphism of $D^*$ onto
$D$, one has
\[
  (\pi^*\cM_n\cdot D^*) = (\cM_n\cdot D).
\]
In the same way one obtains
  $(\kappa^*\cO_{\P^3}(n)\cdot L^*) = (\cO_{\P^3}(n)\cdot L) = n$
and
$(\kappa^*\cO_{\P^3}(n)\cdot D^*) = 0$
and finally 
  $(\pi^*\cM_n\cdot L^*)  =0$.
Clearly one has corresponding results for the intersection numbers with
$C^*_2$, etc. According to Theorem~III
one can write $Z\in
A_1^\tau(\CC)$ as
\[
  Z = q_0[C^*_0] + q_1[D^*] + q_2[C^*_2] + q_3[L^*].
\]
If one forms the intersection numbers with $\pi^*\cM_n$ and
$\kappa^*\cO_{\P^3}(n)$, then one gets $q_0 = \dots = q_3 = 0$ and hence

\begin{theorem}
  \label{thm:1.1}
  Rational equivalence and numerical equivalence in $A_1(\HH)$ (resp.\
  $A_1(\CC))$ agree. \qed
\end{theorem}

As an application we replace $D$ in the above basis of $A_1(\HH)$ by the
cycle
\[
    C_1 = \Set{ (x,y^a,y^{a-1}z^{b-a}(\alpha z+t))| \alpha\in k }^-
\]
(see \cite[p. 91]{T1}). Writing
\[
  [C_1] = q_0[C_0] + q_1[D] + q_2[C_2]
\]
and forming the intersection number with $\cM_n$ gives $q_2 =0$ and
$(n-b+1) = q_0 +q_1(n-b+a-1)$ (cf. \cite[p. 134]{T2}). As this is equivalent
to $q_1=1$, $q_0=2-a$ one obtains
\begin{equation}
  \label{eq:1.1}
  [D] = (a-2)[C_0] + [C_1]\,.
\end{equation}
The same argumentation gives
\begin{equation}
  \label{eq:1.2}
  [D^*] = (a-2)[C^*_0] + [C^*_1]\,.
\end{equation}

$A^+_1(\HH)$ denotes the cone of curves on $\HH$, i.e.~the set of
$1$--cycles $\sum q_i[C_i]$, where $q_i\geq 0$ are rational numbers and the
$C_i$ are closed, reduced, irreducible curves on $\HH$. The cone
$A^+_1(\CC)$ is analogously defined. The determination of both cones will be
given in the next section and Theorem~\ref{thm:1.1} is the main tool.

\textbf{Convention:} If not otherwise stated, in the remaining sections of
Chapter 1 we write $P=k[x,y,z,t]$ and $S=k[x,y,z]$.

\pagebreak

\section{The cone of curves on $\HH$ and on $\CC$}
\label{sec:1.3}

\subsection{Combinatorial cycles}

\label{sec:1.3.1}

\subsubsection{Combinatorial cycles of type $1$}
\label{sec:1.3.1.1}

Let $\cJ\in \HH(k)$ be an ideal of type $1$ (cf.~\cite[p. 7]{T1}). We have
to take up the notations of \cite[1.4.6]{T1} and \cite[Anhang 2,
1.2]{T3} (cf. Appendix~\ref{cha:H}).

We consider a decomposition of $H^0(\cJ(b))$ into $y$-layers
(cf.~\cite[p. 51]{T3}) and we want to show that such a $y$-layer has
convex shape. To do so, we have to show that the following case cannot
occur:
\begin{figure}[H]
\centering
\begin{center}
  \begin{minipage}{10.5cm*\real{0.7}} 
    \tikzstyle{help lines}=[gray,very thin]
    \begin{tikzpicture}[scale=0.7]
      \draw[style=help lines] grid (10,8);
      \draw[very thick,->] (0,0) -- (0,8); 
      \draw[very thick,->] (0,0) -- (10,0);
     { \pgftransformxshift{0.5cm} 
       \draw (1,6) node[above=2pt] {$w$}; 
       \draw (2,5) node[above=2pt] {$u$}; 
       \draw (6,1) node[above=2pt] {$v$}; 
       \draw (-1,7) node[above=2pt] {$t$}; 
       \draw (10,0) node[above=1pt] {$z$}; 
     }
      \draw[\Red,ultra thick] (2,7) -- (2,5); 
      \draw[\Red,ultra thick,dotted] (2,5) -- (3,5) -- (3,4)  -- (4,4)  --
      (4,3)  -- (5,3) -- (5,2);
      \draw[\Red,ultra thick] (5,2) -- (7,2);
          \end{tikzpicture}
  \end{minipage}
\end{center}
\caption{}
 \label{fig:1.1}
\end{figure}

Otherwise we would have $tu = zw$ and $z^\delta u= t^\delta v$, with
$\delta:=\ord_t u - \ord_t v = \ord_z v -\ord_z u$, hence $tu \in
H^0(\cJ(b+1))$ and $z^\delta u \in H^0(\cJ(b+\delta))$. Suppose $u= t^ru'$,
$r\geq 0$, $u'$ a monomial without $t$. As we just have obtained $tu \in
H^0(\cJ(b+1))$, from the $G_1$-invariance of $H^0(\cJ(b+1))$ it follows
that: $(\alpha y+t)^{r+1}u'\in H^0(\cJ(b+1))$ for all $\alpha \in k
\Rightarrow [t^{r+1} + (r+1)\alpha y t^r + \cdots ]u' \in
H^0(\cJ(b+1))$.
As $\Char(k)=0$, because the Vandermonde--determinant is not equal to zero,
$yt^ru' = yu \in H^0(\cJ(b+1))$. In the same way it follows that $(\alpha x
+t)^{r+1} u' \in H^0(\cJ(b+1))$ and hence $xt^ru' = xu\in
H^0(\cJ(b+1))$. But $(x,y,z,t)^\delta u \subset H^0(\cJ(b+\delta))$ implies
$u\in H^0(\cJ(b))$, contradiction. From this we deduce that each $y$-layer
of $H^0(\cJ(b))$ has the shape as in Figure~\ref{fig:1.2}.

\begin{figure}
\begin{center}
  \begin{minipage}{16cm*\real{0.7}} 
    \tikzstyle{help lines}=[gray,very thin]
    \begin{tikzpicture}[scale=0.7]
      \draw[style=help lines] grid (16,14);
      \draw[very thick,->] (0,0) -- (16,0); 
      \draw[very thick,->] (0,0) -- (0,14);
      { \pgftransformxshift{0.5cm} 
        \draw (-1,0) node[above=2pt] {$0$}; 
       \draw (-1,13) node[above=2pt] {$t$}; 
       \draw (15,0) node[above=2pt] {$z$}; 
      }
      \draw[\Red,ultra thick,dashed] (1,13) -- (3.5,11.33); 
      \draw[\Red,ultra thick,dotted] (4,11) -- (6,11) -- (6,10); 
      \draw[\Red,ultra thick] (6,10) -- (8,10) -- (8,9) -- (9,9) -- (9,8)
         -- (10,8) -- (10,7) -- (11,7) -- (11,5); 
      \draw[\Red,ultra thick,dotted] (11,5) -- (12,5) -- (12,4)
        -- (13,4) -- (13,2) -- (14,2) -- (14,0);
    \end{tikzpicture}
  \end{minipage}
\end{center}
\caption{}
 \label{fig:1.2}
\end{figure}

This shows that
\[
  E(\cJ(b)) = \text{ set of monomial in } H^0(\cJ(b))
\]
is a disjoint union of the monomials in so called ``tracks'', i.e.~sets of
the form $B(u)=u\cdot k[z,t]_{b-r}$, where $u\in P_r$ is a monomial. We
write $u = v\cdot t^s$, $v\in S$ a monomial without $t$. Then 
$0 \leq s \leq r \leq b$. Put
\[
   C:=\Set{ \psi^1_\alpha(\cJ) | \alpha \in k}^-\,.
\]
Then the contribution delivered by the track $B(u)$ to the intersection
number $(\cM_b\cdot C)$ is equal to the highest power with which $\alpha$
appears in $\bigwedge\limits^{b-r+1} v(\alpha z+t)^s\cdot k[z,t]_{b-r}$,
i.e.~is equal to $s(b-r+1)$. But only tracks with $\grad_t u >0$ yield a
contribution to $(\cM_b\cdot C)$. Let $B(u)$ be such a track. Multiplication
of $B(u)$ with $\langle z,t \rangle$ gives the track $B'(u)=u\cdot
k[z,t]_{b+1-r} \subset H^0(\cJ(b+1))$. On the other hand one has $xu = tv$
and $yu=tw$ where $v$ and $w$ are monomials in $E(\cJ(b))$, as follows from
the invariance of $E(\cJ(b))$ under the maps $m \mapsto (y/t)\cdot m$ and $m
\mapsto (x/t)\cdot m$, if the monomial $m\in E(\cJ(b))$ is divisible by $t$.
It follows that $x\cdot B(u)$ (resp.~$y\cdot B(u)$) is contained in the
track $B'(v)$ (resp.~$B'(w)$), which arises from the track $B(v)$
(resp.~$B(w)$) in the same way as has been described before. By going over
from $E(\cJ(b))$ to $E(\cJ(b+1))$, each track $B(u)$ thus is transformed
into a track $B'(u)$ of $E(\cJ(b+1))$. On the other hand, as $P_1
H^0(\cJ(b)) = H^0(\cJ(b+1))$, each track $B(v) \subset H^0(\cJ(b+1))$ has
the form $B'(u)$, where $B(u) \subset H^0(\cJ(b))$ is a track. Thus one has
the situation described in Fig.~\ref{fig:1.3}. (As $\cJ$ is not contained in
$x\cO_{\P^3}(-1)$, from the $G_1$-invariance of $H^0(\cJ(b))$ it follows
that $y^b \in H^0(\cJ(b))$.) The same is true, if one goes over from
$E(\cJ(b+1))$ to $E(\cJ(b+2))$, etc. All in all one can say that, going over
from $E(\cJ(b))$ to $E(\cJ(n))$, each track $B(u) = u\cdot k[z,t]_{b-r}$
with $\grad_t u> 0$ is transformed into the track $B'(u) = u\cdot
k[z,t]_{n-r}$ and that each track $B(v) \subset E(\cJ(n))$ with $\grad_t v>
0$ is obtained in this way. $B'(u)$ then delivers the contribution
$s(n-r+1)$ to
\[
\alphadeg \bigwedge\limits^{Q(n)}
\psi^1_\alpha(H^0(\cJ(n))) = (\cM_n\cdot C)\,.
\]
It follows that 
\[
  (\cM_n\cdot C) = \sum s_i(n-r_i+1) = \sum s_i(n-b+1) + \sum s_i(b-r_i) = q_1
  (n-b+1) + q_2\,
\]
where $q_1:=\sum s_i \in \N$ and $q_2:=\sum s_i(b-r_i) \in \N$, i.e.~one has 
\[
  (\cM_n\cdot C) = q_1 (\cM_n\cdot C_1) + q_0 (\cM_n\cdot C_0)\,.
\]
If all $s_i=0$, then $H^0(\cJ(b))$ would be invariant under $\Delta =
U(4;k)$. By Theorem~\ref{thm:1.1} we get:
\begin{conclusion}
  \label{concl:1.1}
  If $C$ is a combinatorial cycle of type $1$, then $[C] = q_1 [C_1] +
  q_0[C_0]$
where $q_0, q_1 \in \N$ and $q_1 >0$. $\qed$
\end{conclusion}

\begin{figure}
   \begin{center}
     \begin{tikzpicture}
      \draw [->] (0,0,13) -- (0,0,14.5) node [below left] {$y$}; 
      \draw [->] (8,0,0) -- (9.5,0,0) node [right] {$z$}; 
      \draw [->] (0,7,0) -- (0,8,0) node [above] {$t$};
\draw
(8,0,0) -- (8,2,0) -- (7,2,0) -- (7,3,0) -- (6,3,0) -- (6,4,0) -- (5,4,0) -- (5,5,0) -- (3,5,0) -- (3,6,0) -- (2,6,0) -- (2,7,0) -- (0,7,0)
        (8,0,1) -- (8,2,1) -- (7,2,1) -- (7,3,1) -- (6,3,1) -- (6,4,1) -- (5,4,1) -- (5,5,1) -- (3,5,1) -- (3,6,1) -- (2,6,1) -- (2,7,1) -- (0,7,1)
       (8,1,1) -- (7,1,1) -- (7,2,1)
       (8,2,0) -- (8,2,1)
        (8,0,3) -- (8,1,3)
        (8,0,2) -- (8,1,2) -- (7,1,2) -- (7,3,2) -- (6,3,2) -- (6,4,2) -- (5,4,2) -- (5,5,2) -- (3,5,2) -- (3,6,2) -- (2,6,2) -- (2,7,2) -- (0,7,2)
        (7,2,2) -- (6,2,2) -- (6,3,2) -- (5,3,2) -- (5,4,2) -- (4,4,2) -- (4,5,2)
        (3,5,2) -- (2,5,2) -- (2,6,2) -- (1,6,2) -- (1,7,2)
       (7,0,3) -- (7,2,3) -- (5,2,3) -- (5,4,3) -- (3,4,3) -- (3,5,3)
        (6,1,3) -- (6,3,3) -- (4,3,3) -- (4,5,3) -- (2,5,3) -- (2,6,3) -- (0,6,3)
        (1,7,3) -- (0,7,3)
       (7,0,4) -- (7,1,4) -- (6,1,4) -- (6,2,4) -- (5,2,4) -- (5,3,4) -- (4,3,4) -- (4,4,4) -- (3,4,4) -- (3,5,4) -- (1,5,4) -- (1,6,4)
        (2,5,4) -- (2,6,4) -- (0,6,4)
        (7,0,5) -- (7,1,5) 
        (6,0,5) -- (6,2,5) -- (5,2,5) -- (5,3,5) -- (4,3,5) -- (4,4,5) -- (2,4,5) -- (2,5,5) -- (1,5,5) -- (1,6,5) -- (0,6,5)
       (5,0,8) -- (5,1,8) -- (4,1,8) -- (4,2,8) -- (3,2,8) -- (3,3,8) -- (2,3,8) -- (2,4,8) -- (0,4,8)
        (6,0,6) -- (6,2,6) -- (4,2,6) -- (4,4,6) -- (1,4,6) -- (1,5,6)
        (5,1,6) -- (5,3,6) -- (3,3,6) -- (3,4,6)
       (2,4,6) -- (2,5,6) -- (0,5,6)
 (1,6,6) -- (0,6,6)
 (6,0,7) --  (6,1,7)
 (1,5,7) -- (0,5,7)
        (4,1,7) -- (4,3,7) -- (2,3,7) -- (2,4,7)
        (5,0,7) -- (5,2,7) -- (3,2,7) -- (3,4,7) -- (0,4,7)
 (5,0,9) --  (5,1,9)
        (4,0,9) -- (4,2,9) -- (2,2,9) -- (2,4,9) -- (0,4,9)
  (4,0,10) -- (4,1,10)  (1,4,10) -- (0,4,10)
        (3,1,9) -- (3,3,9) -- (1,3,9) -- (1,4,9)
        (3,4,5) -- (3,5,5) -- (2,5,5)
        (3,0,10) -- (3,2,10) -- (1,2,10) -- (1,3,10)
      (2,1,10) -- (2,3,10) -- (0,3,10)
        (3,0,11) -- (3,1,11)
         (1,1,11) -- (1,2,11)
        (2,0,11) -- (2,2,11) -- (0,2,11)
         (1,2,12) -- (0,2,12)
         (1,3,9) -- (1,3,10)
        (1,3,11) -- (0,3,11)
         (1,2,10) -- (1,2,11)
         (2,3,7) -- (2,3,10)
         (3,3,6) -- (3,3,9)       
        (8,0,0) -- (8,0,3) -- (7,0,3) -- (7,0,5) -- (6,0,5) -- (6,0,7) -- (5,0,7) -- (5,0,9) -- (4,0,9) -- (4,0,10) -- (3,0,10) -- (3,0,11) -- (2,0,11) -- (2,0,12) -- (1,0,12) -- (1,0,13) -- (0,0,13)
        (8,1,0) -- (8,1,3) -- (7,1,3) -- (7,1,5) -- (6,1,5) -- (6,1,7) -- (5,1,7) -- (5,1,9) -- (4,1,9) -- (4,1,10) -- (3,1,10) -- (3,1,11) -- (2,1,11) -- (2,1,12) -- (0,1,12)
 (1,0,12) -- (1,1,12)  (2,0,12) -- (2,1,12)
         (7,1,1) -- (7,1,3) -- (6,1,3) -- (6,1,5)
        (6,1,6) -- (5,1,6) -- (5,1,7) -- (4,1,7) -- (4,1,9) -- (3,1,9) -- (3,1,10) -- (2,1,10) -- (2,1,11) -- (1,1,11) -- (1,1,12)
     (7,2,0) -- (7,2,3)
        (7,3,0) -- (7,3,2)
       (6,4,0) -- (6,4,2)
      (6,2,2) -- (6,2,6)
       (6,3,0) -- (6,3,3)
       (5,5,0) -- (5,5,2)
        (5,4,0) -- (5,4,3)
        (4,5,0) -- (4,5,3)
          (3,5,0) -- (3,5,5)
      (3,6,0) -- (3,6,2)
        (2,6,0) -- (2,6,4)
   (2,7,0) -- (2,7,2)
    (2,4,5) -- (2,4,9)
        (0,7,0) -- (0,7,3) -- (0,6,3) -- (0,6,6) -- (0,5,6) -- (0,5,7) -- (0,4,7) -- (0,4,10) -- (0,3,10) -- (0,3,11) -- (0,2,11) -- (0,2,12) -- (0,1,12) -- (0,1,13) -- (0,0,13)
        (1,7,0) -- (1,7,3) -- (1,6,3) -- (1,6,6) -- (1,5,6) -- (1,5,7) -- (1,4,7) -- (1,4,10) -- (1,3,10) -- (1,3,11) -- (1,2,11) -- (1,2,12) -- (1,1,12) -- (1,1,13) -- (1,0,13)
  (1,1,13) -- (0,1,13)
   (1,6,2) -- (1,6,3)
   (1,5,4) -- (1,5,6)
   (3,4,3) -- (3,4,7)
   (2,5,2) -- (2,5,6)
   (1,4,6) -- (1,4,7)
  (5,3,2) -- (5,3,6)
  (4,4,2) -- (4,4,6)
   (4,3,3) -- (4,3,7)
   (5,2,3) -- (5,2,7)
   (4,2,6) -- (4,2,9)
  (3,2,7) -- (3,2,10)
  (2,2,9) -- (2,2,11)
;
     \end{tikzpicture}
   \end{center}
\caption{}
 \label{fig:1.3}
 \end{figure}
 
\subsubsection{Combinatorial cycles of type $2$}
\label{sec:1.3.1.2}

We first want to recall the computation of degree in \cite[1.3, p. 12
]{T1}: $\G_a$ operates on $R=k[y,z]$ by $\psi_\alpha:y\mapsto y, z\mapsto
\alpha y+z$. Let $V\subset y^r z^s R_n$ be a $(\mu+1)$-dimensional subspace,
which is generated by monomials.

Then $\bigwedge\limits^{\mu+1} \psi_\alpha(V)$ contains a non-zero term
without $\alpha$, respectively with a power of $\alpha$, which is $\geq
(\mu+1)\cdot s$. In order to see this, we write 
\[
   V = y^r z^s \langle y^{n-a_0}z^{a_0},\dots,y^{n-a_\mu}z^{a_\mu} \rangle \,,
\]
where $0\leq a_0 < \dots < a_\mu\leq n$. Then $\bigwedge\limits^{\mu+1}
\psi_\alpha(V)$ contains the term 
\[
  y^rz^sy^{n-a_0}z^{a_0} \wedge \dots \wedge y^rz^sy^{n-a_\mu}z^{a_\mu}
\]
without $\alpha$. The highest power of $\alpha$ in $\bigwedge\limits^{\mu+1}
\psi_\alpha(V)$ is equal to
\[
   (s+a_0) + \cdots + (s+a_\mu)-(1+ \cdots + \mu)
\] 
(see~\cite[p. 13/14]{T1}). As $0\leq a_0 < \dots < a_\mu\leq n$, this sum is
$\geq s(\mu+1)$.

Let be $Q(n) = \tbinom{n-1+3}{3}+\tbinom{n-a+2}{2}+ \tbinom{n-b+1}{1}$ where
{\begin{equation}
\label{eq:1.3}
  f_0(n):=1, \quad f_1(n):= (n-b+1), \quad f_2(n):= \tbinom{n-a+2}{2} + (n-b+1)\,.
\end{equation}
Let $\cJ\in H_Q(k)$ have the type $2$, i.e.$\cJ$ is invariant under
$T(4;k)\cdot G_2$ (cf.~\cite[p. III]{T1} and Appendix~\ref{cha:H}). \\
Let $\cJ \leftrightarrow \xi \in H_Q(k)$ and $C:=\Set{ \psi^2_\alpha(\xi) |
  \alpha \in k}^-$.
\begin{auxlemma}
  \label{auxlem:0}
 If $g(n):=\alphadeg \bigwedge\limits^{Q(n)}
 \psi^2_\alpha(H^0(\cJ(n)))$, then there are $n_i \in \N$ such that for all
 $n\gg 0$
\[
   g(n) = n_2 f_2(n) + n_1 f_1(n) + n_0 f_0(n)\,.
\]
\end{auxlemma}
\begin{proof}
  By induction on the colength of  
$\cJ':= \cJ + t\cO_{\P^3}(-1)/t\cO_{\P^3}(-1) \subset \cO_{\P^2}$.
If this colength is equal to zero, then $a=1$ and $\cJ \subset \cO_{\P^3}$ has the
colength $b$. Then $S_n\subset H^0(\cJ(n))$ if $n\geq b$ and $\alphadeg
\bigwedge\limits^{Q(n)} \psi^2_\alpha(H^0(\cJ(n)))$ is independent of $n\geq
b$, hence $g(n) = cf_0(n)$. Thus we can assume, without restriction of
generality, that $a>1$. If $\cJ'$ is invariant under $U(3;k)$, then the same 
argumentation shows that $g(n)$ is constant. Hence we can assume without
restriction that $\cJ'$ is not $U(3;k)$-invariant and we have the situation
described in~\cite[2.2]{T1}.

\paragraph{$1^\circ$}
\label{sec:1.3.1.2.1}
We consider the outer shell of the pyramid $E(\cJ(n))$, $n\geq b$. We
imagine the outer shell completed (in Figure~\ref{fig:1.4} by the dotted
monomials, but this figure is very special, because in general it is not
true that ``old pyramid $\cup$ new monomials'' is a pyramid in the usual
sense, see Appendix~\ref{cha:H}). The completed outer shell has the form
$u\cdot k[y,z](-\alpha)$, where $u=y^r z^s$, and $\alpha:=r+s = \reg(\cJ')$
(see~\cite[p. 55]{T1}).

Now one has $1\leq \alpha < a$ as $\reg(\cJ') \leq a-1$, and on the other
hand from $\alpha=0$, because of $\alpha=r+s$, it would follow that
$\reg(\cJ')=0$, hence $\cJ'=\cO_{\P^2}$.

It follows that in the complete outer shell one has $n-\alpha+1$ monomials
of degree $n$, if $n\geq\alpha$. We define the number $\beta$ by the
condition that the outer shell is complete in degree $\beta$ but not in
degree $\beta-1$. Hence $\beta\leq\reg(\cJ) \leq b$.

\paragraph{$2^\circ$}
\label{sec:1.3.1.2.2}
Let $r_i$, $\alpha\leq i\leq\beta-1$, be the number of monomials in
$E(\cJ(n))$ of degree $i$ in the outer shell. Then $r_i\leq i-\alpha+1$. As
the outer complete shell is equal to $uk[y,z](-\alpha)$, where $u=y^rz^s$,
$r+s=\alpha$, the $i$-th layer of the outer shell of the pyramid
$E(\cJ(n))$ contributes to the $\alphadeg$ of $H^0(\cJ(n))$ the term
$s(i-\alpha+1)$, if $i\geq \beta$, respectively  a term $\geq s\cdot r_i$,
if $\alpha \leq i \leq \beta-1$ (see above). Hence the outer
shell of the pyramid $E(\cJ(n))$ contributes to the $\alphadeg$ of
$H^0(\cJ(n))$ the term 
\begin{align*}
  r(n)&:= s\cdot \sum_{i=\beta}^n (i-\alpha+1) + s\cdot \sum_{i=1}^{\beta-1}
  r_i + \delta \\
      & = s\left[\tbinom{n-\alpha+2}{2}+ \tbinom{\beta-\alpha+1}{2} +
        \rho\right] + \delta\;,
\end{align*}
where $\rho:=\sum_{i=\alpha}^{\beta-1} r_i$ and $\delta\in \N$. 
The numbers $s,\rho,\delta$ are independent of $n\geq b$. \\

\textbf{N.B.}
$\delta=0$ iff in the $i$-th layer of the outer shell of $E(\cJ(n))$ there
are only monomials with the smallest possible $z$-degree; especially there
are no holes in the $i$-th layer, for all $\alpha\leq i \leq \beta-1$.

\paragraph{$3^\circ$}
\label{sec:1.3.1.2.3}
 We put $f(n):= \tbinom{n-\alpha+2}{2}-\gamma$,
 $\gamma:=\tbinom{\beta-\alpha+1}{2} -\rho$ and want to find $q_i\in \N$
 such that
 \begin{equation}
   \label{eq:1.4}
    f(n) = q_2 f_2(n) + q_1 f_1(n) + q_0f_0(n)\,
 \end{equation}
(see ~\eqref{eq:1.3}). First one sees that $q_2=1$ and we get the equivalent equations
\begin{alignat*}{2}
 & &   \tbinom{n-\alpha+2}{2}- \gamma & = \tbinom{n-a+2}{2} + (n-b+1) +.
   q_1(n-b+1) + q_0 \\
\iff {} & & (n-\alpha+2)(n-\alpha+1) -2\gamma & = (n-a+2)(n-a+1)
+2(n-b+1)\\
& & & \qquad \quad +2q_1(n-b+1)+2q_0 \\
 \iff {} & & n^2 -(2\alpha-3)n+(\alpha-1)(\alpha-2) -2\gamma & = n^2
 -(2a-3)n+(a-1)(a-2)  \\
&&  & \qquad \quad 
+(2q_1+2)(n-b+1)+2q_0 \\
\iff {} &&   -(2\alpha-3) & = 2q_1 +2 -(2a-3) \quad \text{ and } \\
    &&  (\alpha-1)(\alpha-2) -2\gamma & = (2q_1+2)(-b+1) +(a-1)(a-2) +2q_0 \,.
\end{alignat*}
This is equivalent to
\begin{equation}
  \label{eq:1.4}
  q_1 = a-\alpha -1
\end{equation}
and 
\begin{equation}
  \label{eq:1.5}
(\alpha-1)(\alpha-2) -2\gamma = 2(a-\alpha)(-b+1)+(a-1)(a-2) +2q_0
\end{equation}
\begin{align*}
  \iff 2q_0 & = \alpha^2-3\alpha+2-(a^2-3a+2)-2\gamma+2(a-\alpha)(b-1) \\
            & = \alpha^2-3\alpha-a^2+3a-2\gamma+2(a-\alpha)(b-1) \\
            & = -[a^2-\alpha^2-3(a-\alpha)]-2\gamma+2(a-\alpha)(b-1) \\
            & = -(a-\alpha)[a+\alpha-3]-2\gamma+2(a-\alpha)(b-1) \\
\Rightarrow\quad   q_0 & = (a-\alpha)[b-1 -\tfrac{1}{2}(a+\alpha-3)]-\gamma \\
              q_0 & = (a-\alpha)\left[\tfrac{1}{2}(b-a+1)+
    \tfrac{1}{2}(b-\alpha)\right] -\gamma
\end{align*}
where 
\[
\gamma= \tbinom{\beta-\alpha+1}{2}-\rho =
\sum_{i=\alpha}^{\beta-1}(i-\alpha+1-r_i)
\]
is the number of monomials, which are missing in the layer of degree $i$ in the
outer shell of $E(\cJ(b))$, $\alpha\leq i\leq\beta-1$. Hence $\gamma\geq0$
and on the other hand $\gamma\leq$ colength of $\cJ$ in $\cI^*$, where
$\cI^*$ is the ideal in $\cO_{\P^3}$, which is generated by $\cI:=\cJ'\subset
\cO_{\P^2}$. As this colength is $\leq b-a+1$, one has
\[
  0\leq\gamma\leq b-a+1\,.
\]
As $\alpha< a$, one has $\tfrac{1}{2}(b-\alpha) \geq \tfrac{1}{2}(b-a+1)$
    and we get
    \begin{equation}
    \label{eq:1.6}
      q_0\geq (a-\alpha)(b-a+1)-\gamma\geq (a-\alpha-1)(b-a+1)-\gamma\geq 0\,.
  \end{equation}
Now from equation~\eqref{eq:1.5}  it follows that $q_0\in \Z$ and hence
$q_0\in \N$. As $q_1 = a-\alpha-1$ and $\alpha < a$ (see above) we have
solved~\eqref{eq:1.3}  with natural numbers $q_0,q_1,q_2$, and hence it
follows that there are $q'_i\in \N$ such that 
\begin{equation}
  \label{eq:1.7}
  r(n) = s\cdot f(n) + \delta = q'_2 f_2(n) +q'_1 f_1(n) +q'_0 f_0(n)\,.
\end{equation}

\paragraph{$4^\circ$}
\label{sec:1.3.1.2.4}
If one takes away from $E(\cJ(n))$ the outer shell, one gets a pyramid
$E(\cK(n))$, where $\cK = x \tilde{\cJ}(-1)$. Here $\tilde{\cJ}\subset
\cO_{\P^3}$ is an ideal of type $2$ with Hilbert polynomial $\tilde{Q}$ such
that $\tilde{\cJ}'\subset \cO_{\P^2}$ has a smaller colength than
$\cJ'\subset \cO_{\P^2}$ (see~\cite[Fig. 2.5, p. 55]{T1}). One sees that
\[
\tilde{Q}(n-1) + \sum_{i=\beta}^n (i-\alpha+1) + \sum_{i=\alpha}^{\beta-1}
r_i = Q(n)\,.
\]
We write
\[
\tilde{Q}(n) = \tbinom{n-1+3}{3}+\tbinom{n-\tilde{a}+2}{2}
+\tbinom{n-\tilde{b}+1}{1}, 
\]
where $\tilde{b} = \tilde{a}-1$ is possible, in which case 
$ \tilde{Q}(n) = \tbinom{n-1+3}{3}+\tbinom{n-(\tilde{a}-1)+2}{2}$. Also
$\tilde{a} =1$ is possible, i.e.~$\tilde{Q}(n) = \tbinom{n+3}{3}$. In any
case, one has
\[
\tilde{Q}(n-1)+ \tbinom{n-\alpha+2}{2}-\tbinom{\beta-\alpha+1}{2} + \rho =
\tilde{Q}(n-1)+ \tbinom{n-\alpha+2}{2} -\gamma,
\]
where $\rho$ and $\gamma$ have been introduced in $2^\circ$
respectively in $3^\circ$. We get 
\begin{align*}
 & \tbinom{n-1-1+3}{3}+\tbinom{n-1-\tilde{a}+2}{2} +
  \tbinom{n-1-\tilde{b}+1}{1} +\tbinom{n-\alpha+2}{2} -\gamma \\
 = {} &\tbinom{n-1+3}{3}+\tbinom{n-a+2}{2} +
  \tbinom{n-b+1}{1}\,.
\end{align*}
Hence
\[
\tbinom{n-\tilde{a}+1}{2} + (n-\tilde{b}) +\tbinom{n-\alpha+2}{2} -\gamma
= \tbinom{n-1+2}{2}  +\tbinom{n-a+2}{2}+ (n-b+1)\,.
\]
\begin{alignat*}{2}
  \Rightarrow & \quad  &
  & (n-\tilde{a}+1)(n-\tilde{a})+2(n-\tilde{b})+(n-\alpha+2)(n-\alpha+1)-2\gamma \\
  & & = {} &(n+1)n+(n-a+2)(n-a+1)+2(n-b+1) \\
   \Rightarrow & \quad & & n^2 -2\tilde{a}n + \tilde{a}^2 +n -\tilde{a} +2n
  -2\tilde{b}+n^2 -2\alpha n +\alpha^2 +3n -3\alpha+2 -2\gamma \\
   & & = {} & n^2+n+n^2-2an+a^2+3n-3a+2+2n-2b+2\,.
\end{alignat*}
Comparing the coefficients of $n$ gives
\begin{equation}
  \label{eq:1.8}
  a = \tilde{a} + \alpha\;.
\end{equation}
Moreover,
 it follows
\[
\tilde{a}^2-\tilde{a}-2\tilde{b} +\alpha^2-3\alpha+2 -2\gamma =
  a^2-3a-2b+4 \qquad\]
\begin{alignat*}{2}
    \Rightarrow & \quad & 2b-2\tilde{b}  &= a^2 -3a+4
    -\tilde{a}^2+\tilde{a}-\alpha^2+3\alpha -2
    +2\gamma \\
      & &  & = 2\tilde{a}\alpha-2\tilde{a}+2+2\gamma \\
    \Rightarrow & \quad & b-\tilde{b} & = \tilde{a}(\alpha-1) +1 + \gamma\,.
\end{alignat*}
As $\gamma\geq 0$ and $\alpha \geq 1$ (see p.9, line 10 resp. p.7, line 5 from bottom), we get
\begin{equation}
  \label{eq:1.9}
  b > \tilde{b}\,.
\end{equation}

Now by the induction-assumption one can write 
\[
  \tilde{g}(n):=\alphadeg
\bigwedge\limits^{\tilde{Q}(n)} \psi^2_\alpha(H^0(\tilde{\cJ}(n))) 
\]
in the form $\tilde{g}(n) = \tilde{q}_2 \tilde{f}_2(n) + \tilde{q}_1
\tilde{f}_1(n) + \tilde{q}_0 \tilde{f}_0(n)$, where $\tilde{f}_2(n) =
\tbinom{n-\tilde{a}+2}{2} +(n-\tilde{b}+1)$, $\tilde{f}_1(n) =
(n-\tilde{b}+1)$, $\tilde{f}_0(n) =1$, and $\tilde{q}_0$, $\tilde{q}_1$ and
$\tilde{q}_2$ are natural numbers. 

Now by direct computation one gets 
$\tilde{f}_2(n-1) =  f_2(n) + c_1 f_1(n) + c_0f_0(n)$, where
$c_1:=a-\tilde{a}-1$ and
$c_0:=(\alpha-1)(b-a)+\tfrac{1}{2}\alpha(\alpha-1) + b -\tilde{b}-1$ are
natural numbers, because of $\alpha\geq 1$, $a\leq b$ and~\eqref{eq:1.8}
and~\eqref{eq:1.9}. We write $\tilde{f}_1(n-1) = (n-1-\tilde{b} +1) = f_1(n)
+c_2$, where $c_2:=b-\tilde{b}-1\in \N$ because
of~\eqref{eq:1.9}.

Using~\eqref{eq:1.7} we get
\begin{equation}
  \label{eq:1.10}
  \begin{aligned}
    g(n) & = r(n) + \tilde{g}(n-1) \\
         &  = q'_2 f_2(n) + q'_1 f_1(n) + q'_0 f_0(n) \\
         & \quad + \tilde{q}_2[ f_2(n) + c_1 f_1(n) + c_0f_0(n)] + 
             \tilde{q}_1(f_1(n) + c_2) + \tilde{q}_0 f_0(n)\,.
  \end{aligned}
\end{equation}
\qedhere
\end{proof}
Hence the Aux-Lemma~\ref{auxlem:0} follows and the same argumentation as
in~\ref{sec:1.3.1.1} gives:
\begin{conclusion}
  \label{concl:1.2}
If $C$ is a combinatorial cycles of type $2$, there are natural numbers
$q_i$ such that 
$[C] = q_2[C_2] +  q_1[C_1] + q_0[C_0]$. \hfill $\qed$
\end{conclusion}
\begin{corollary}
  \label{cor:1.1}
If $C$ is a combinatorial cycles of type $2$, then $(\cM_n\cdot C)$, as
function of $n$, is either a quadratic function or a constant function.
\end{corollary}
\begin{proof}
  If $(\cM_n\cdot C)$ is not a quadratic function, then the above formula
  for $g(n)$ shows that $r(n)$ is not a quadratic function and
  $\tilde{q}_2=0$. It follows that $s=0$, hence $r(n) = \mathrm{constant}$
  and by using an induction argument again, one can assume that
  $\tilde{q}_1=0$. But then $g(n)$ is a constant function.
\end{proof}
\begin{corollary}
  \label{cor:1.2}
If $C$ is a combinatorial cycles of type $2$ such that $(\cM_n\cdot C) =
n_1f_1(n) + n_2f_2(n)$ with natural numbers $n_1$ and $n_2$, then $n_1=0$.
\end{corollary}
\begin{proof}
 From the formula~\eqref{eq:1.10} we deduce
 \begin{equation}
   \label{eq:1.11}
   q'_0 + \tilde{q}_2\cdot c_0 + \tilde{q}_1\cdot c_2 + \tilde{q}_0  =0\,.
 \end{equation}
\textsc{Case 1}: $\tilde{q}_2=0$. By Corollary~\ref{cor:1.1} it follows that
 $\tilde{g}(n) = \tilde{q}_2 f_2(n) + \tilde{q}_1 f_1(n) +\tilde{q}_0
 f_0(n)$ is a constant function. Now from equation~\eqref{eq:1.11} follows
 $\tilde{q}_0=0$, hence $\tilde{g}(n)=0$ and~\eqref{eq:1.7} implies
 \begin{equation}
   \label{eq:1.12}
   g(n) = r(n) = s[q_2f_2(n)+q_1f_1(n) + q_0f_0(n)] + \delta\,.
 \end{equation}
 \textsc{Subcase 1}: $s=0$, hence $g(n)=0$. \\
 \textsc{Subcase 2}: $s\neq 0$, hence $q_0=0$. Then~\eqref{eq:1.6}
 implies $a=\alpha+1$, and then equation~\eqref{eq:1.4} gives $q_1=0$ and
 hence $n_1 = s\cdot q_1 =0$. \\
\noindent
\textsc{Case 2}: $\tilde{q}_2\neq 0$.  From~\eqref{eq:1.11} it follows that
$c_0=0$, hence, because of~\eqref{eq:1.9}, $\alpha=1$ follows.
Besides~\eqref{eq:1.11} gives $q'_0=sq_0+\delta=0$. As
$\alpha=r+s$ (see above), we get 2 cases. \\
\textsc{Subcase 1}: $s=0$, hence $r=1$ and $u=y$. As $E(\cJ'(n))$ has a
convex shape, it follows that $\cJ'$ is $B(3;k)$ invariant
(see~\cite[Section 2.2]{T1} and e.g.~Fig.~\ref{fig:1.5}). But then
$(\cM_n\cdot
C)$ is a constant.\\
\textsc{Subcase 2}: $s=1$. Hence $r=0$ and $u=z$. Besides $q_0=0$, hence
$a=\alpha+1=2$ by equation~\eqref{eq:1.6}. As $\cJ$ is invariant under
$G_2$, it follows that $\cJ'$ is invariant under
\[
 G'_2:=\Set{ \left(\begin{smallmatrix} 1 & * & * \\ 0 & 1 & 0 \\ 0 & 0 & 1
     \end{smallmatrix} \right)
 } < U(3;k)\,,
\]
hence $\cJ' =(x,z)$. It follows that 
$\alphadeg \dot\bigwedge \psi^2_\alpha(H^0(\cJ'(n)))=n $, if $n\geq 0$.
But this number is equal to 
\[
(\cM_n\cdot C) -(\cM_{n-1}\cdot C) = n_1 + n_2 \left[\tbinom{n-a+1}{1} +1
\right] = n_1 +n_2(-a+2)+ n_2\cdot n\,,
\]
and because of $a=2$ it follows that $n_1=0$.
\end{proof}
\begin{corollary}
  \label{cor:1.3}
  Let $C:=\Set{ \psi^2_\alpha(\xi) | \alpha \in k}^-$ be a combinatorial
  cycle of type $2$, and let $\xi \leftrightarrow \cJ$.  If $[C] = n_0[C_0]
  + n_2[C_2]$ and $n_2\neq 0$, then the ideal $\cJ' \leftrightarrow r(\xi)$
  has maximal Hilbert function.
\end{corollary}
\begin{proof}
  If $s=0$, then $u=y^r$ and the convex shape of $E(\cJ')$ shows that $\cJ'$
  is $B(3;k)$-invariant (see~\cite[Section 2.2]{T1}). But then $(\cM_n\cdot
  C)$ is a constant function, contrary to the assumption. Hence we have
  $s>0$ and from~\eqref{eq:1.10},~\eqref{eq:1.7} and~\eqref{eq:1.3} follows
  that $q_1=0$. Then~\eqref{eq:1.4} gives $\alpha=\reg(\cJ') = a-1$. It
  follows that $\cJ'$ has maximal regularity, hence maximal Hilbert function
  (see Appendix~\ref{cha:C}).
\end{proof}

\begin{figure}
  \begin{center}
    \begin{tikzpicture}
      \draw [->] (0,0,9) -- (0,0,10.5) node [below left] {$y$}; 
      \draw [->] (10,0,0) -- (11.5,0,0) node [right] {$z$}; 
      \draw [->] (0,8,0) -- (0,9,0) node [above] {$t$};

      \draw (0,0,9) -- (0,1,9) -- (0,1,8) -- (0,2,8) -- (0,2,7) -- (0,3,7)
      -- (0,3,6) -- (0,4,6) -- (0,4,5) -- (0,4,5) -- (0,5,5) -- (0,5,4) --
      (0,6,4) -- (0,6,3) -- (0,7,3) -- (0,7,2) -- (0,8,2) -- (0,8,0);
 
      \draw (1,0,9) -- (1,1,9) -- (1,1,8) -- (1,2,8) -- (1,2,7) -- (1,3,7)
      -- (1,3,6) -- (1,4,6) -- (1,4,5) -- (1,4,5) -- (1,5,5) -- (1,5,4) --
      (1,6,4) -- (1,6,3) -- (1,7,3) -- (1,7,2) -- (1,8,2) -- (1,8,0);
 
      \draw (2,0,9) -- (2,1,9) -- (2,1,8) -- (2,2,8) -- (2,2,7) -- (2,3,7)
      -- (2,3,6) -- (2,4,6) -- (2,4,5) -- (2,4,5) -- (2,5,5) -- (2,5,4) --
      (2,6,4) -- (2,6,3) -- (2,7,3) -- (2,7,2) -- (2,8,2) -- (2,8,0);

      \draw (3,0,9) -- (3,1,9) -- (3,1,8) -- (3,2,8) -- (3,2,7) -- (3,3,7)
      -- (3,3,6) -- (3,4,6) -- (3,4,4) -- (3,5,4) -- (3,5,3) -- (3,6,3) --
      (3,6,2) -- (3,7,2) -- (3,7,0);

      \draw (3,6,2) -- (3,5,2) -- (4,5,2) -- (4,4,2) -- (5,4,2) -- (5,3,2)
      -- (6,3,2) -- (6,2,2) -- (7,2,2) -- (7,1,2);

      \draw (2,6,3) -- (2,5,3) -- (2,5,4) -- (2,4,4) -- (2,4,5);

      \draw (3,5,3) -- (4,5,3) -- (4,4,3) -- (5,4,3) -- (5,3,3) -- (6,3,3)
      -- (6,2,3) -- (7,2,3) -- (7,1,3) -- (9,1,3) -- (9,0,3)

      (0,6,4) -- (2,6,4) (2,5,3) -- (3,5,3) (8,1,3) -- (8,1,2) (8,2,1) --
      (8,2,0) (8,3,1) -- (8,3,0) (7,3,2) --(7,3,0);

      \draw (4,0,8) -- (4,1,8) -- (4,1,7) -- (4,2,7) -- (4,2,6);

      \draw (0,8,0) -- (2,8,0) -- (2,7,0) -- (4,7,0) -- (4,6,0) -- (5,6,0)
      -- (5,5,0) -- (6,5,0) -- (6,4,0) -- (7,4,0) -- (7,3,0) -- (8,3,0) --
      (8,2,0) -- (9,2,0) -- (9,1,0) -- (10,1,0) -- (10,0,0);

      \draw (0,8,1) -- (2,8,1) -- (2,7,1) -- (4,7,1) -- (4,6,1) -- (5,6,1)
      -- (5,5,1) -- (6,5,1) -- (6,4,1) -- (7,4,1) -- (7,3,1) -- (8,3,1) --
      (8,2,1) -- (9,2,1) -- (9,1,1) -- (10,1,1) -- (10,0,1);

      \draw (3,7,1) -- (3,6,1) -- (4,6,1) -- (4,5,1) -- (5,5,1) -- (5,4,1)
      -- (6,4,1) -- (6,3,1) -- (7,3,1) -- (7,2,1) -- (8,2,1) -- (8,1,1) --
      (9,1,1) -- (9,0,1);

      \draw (0,8,2) -- (2,8,2) -- (2,7,2) -- (3,7,2) -- (3,6,2) -- (4,6,2)
      -- (4,5,2) -- (5,5,2) -- (5,4,2) -- (6,4,2) -- (6,3,2) -- (7,3,2) --
      (7,2,2) -- (8,2,2) -- (8,1,2) -- (9,1,2) -- (9,0,2);

      \draw (0,0,9) -- (3,0,9) -- (3,0,8) -- (5,0,8) -- (5,0,7) -- (6,0,7)
      -- (6,0,6) -- (7,0,6) -- (7,0,5) -- (8,0,5) -- (8,0,3) -- (9,0,3) --
      (9,0,1) -- (10,0,1) -- (10,0,0);

      \draw (3,0,8)  -- (3,1,8)  (3,1,7) -- (3,2,7) -- (3,2,6) -- (3,3,6)
      -- (3,3,5) -- (3,4,5); 

 
      \draw (0,1,9) -- (3,1,9) -- (3,1,8) -- (5,1,8) -- (5,1,7) -- (6,1,7)
      -- (6,1,6) -- (7,1,6) -- (7,1,5) -- (8,1,5) -- (8,1,3) -- (9,1,3) --
      (9,1,1) -- (10,1,1) -- (10,1,0);

      \draw (0,1,8) -- (3,1,8) -- (3,1,7) -- (5,1,7) -- (5,1,6) -- (6,1,6)
      -- (6,1,5) -- (7,1,5) -- (7,1,2) -- (8,1,2) -- (8,1,1)
      (9,1,1) -- (9,1,0);

      \draw (0,2,8) -- (3,2,8) -- (3,2,7) -- (5,2,7) -- (5,2,6) -- (6,2,6)
      -- (6,2,5) -- (7,2,5) -- (7,2,2) -- (8,2,2) -- (8,2,1)
      (9,2,1) -- (9,2,0);

      \draw (0,7,2) -- (2,7,2) -- (2,7,0) (0,7,3) -- (2,7,3) (2,6,2) --
      (3,6,2) (2,6,2) -- (2,7,2) (2,6,2) -- (2,6,3) (0,6,3) -- (3,6,3)
      (3,6,2) -- (3,6,1) (4,7,1) -- (4,7,0) (4,6,2) -- (4,6,0)

      (5,6,1) -- (5,6,0) (6,5,1) -- (6,5,0) (5,5,2) -- (5,5,0) (0,5,5) --
      (2,5,5) (0,5,4) -- (4,5,4) -- (4,5,1) (0,4,5) -- (4,4,5) (0,4,6) --
      (3,4,6) (2,4,4) -- (4,4,4) (0,3,6) -- (4,3,6) (0,3,7) -- (3,3,7)
      (0,2,7) -- (3,2,7) (5,0,8) -- (5,1,8) (5,0,7) -- (5,2,7) (6,4,2) --
      (6,4,0) (7,4,1) -- (7,4,0) (6,0,7) -- (6,1,7) (6,0,6) -- (6,2,6)
      (7,0,6) -- (7,1,6) (7,2,2) -- (7,2,1) (6,2,4) -- (6,2,2) (6,3,4) --
      (6,3,1) (5,3,4) -- (5,3,2) (5,4,4) -- (5,4,1) (5,2,6) -- (5,1,6)
      (4,4,5) -- (4,4,2) (7,1,5) -- (7,0,5) (8,1,5) -- (8,0,5) (8,1,3) --
      (8,0,3) (7,1,4) -- (8,1,4) -- (8,0,4) (4,5,4) -- (4,4,4) (3,3,5) --
      (4,3,5) (3,5,3) -- (3,5,2) (3,2,6) -- (5,2,6) -- (5,2,5) -- (6,2,5) ;

      \draw(4,4,5) -- (5,4,5) -- (5,3,5) -- (6,3,5) -- (6,2,5) -- (7,2,5) --
      (7,1,5) (4,4,4) -- (5,4,4) -- (5,3,4) -- (6,3,4) -- (6,2,4) -- (7,2,4)
      -- (7,1,4) (5,3,6) -- (5,3,4) (4,2,6) -- (4,3,6) -- (4,3,5) -- (4,4,5)
      (4,3,5) -- (5,3,5) -- (5,2,5) (4,3,6) -- (5,3,6) -- (5,2,6) (6,2,5)--
      (6,2,4) (6,2,5)-- (6,1,5) (6,3,5)-- (6,3,4) (5,4,5)-- (5,4,4) (7,2,5)
      -- (7,2,4);

      \draw[very thick,dashed] (4,4,5) -- (5,4,5) -- (5,3,5) -- (6,3,5) --
      (6,2,5) -- (7,2,5) -- (7,1,5) (4,4,4) -- (5,4,4) -- (5,3,4) -- (6,3,4)
      -- (6,2,4) -- (7,2,4) -- (7,1,4) (5,3,6) -- (5,3,4) (4,2,6) -- (4,3,6)
      -- (4,3,5) -- (4,4,5) (4,3,5) -- (5,3,5) -- (5,2,5) (4,3,6) -- (5,3,6)
      -- (5,2,6) (6,2,5)-- (6,2,4) (6,2,5)-- (6,1,5) (6,3,5)-- (6,3,4)
      (5,4,5)-- (5,4,4) (7,2,5) -- (7,2,4);

    \end{tikzpicture}
  \end{center}
\caption{}
 \label{fig:1.4}
\end{figure}

\begin{figure}
\begin{center}
  \begin{minipage}{11cm*\real{0.7}}
    \tikzstyle{help lines}=[gray,very thin]
    \begin{tikzpicture}[scale=0.7]
      \draw[style=help lines] grid (11,11);
      \draw[very thick,->] (0,0) -- (0,11); 
      \draw[very thick,->] (0,0) -- (11,0);
      { \pgftransformxshift{0.5cm} 
          \pgftransformyshift{-0.55cm} 
         \foreach \x in {0,1,2,3,4,5,6,7,8,9} 
           \draw[anchor=base] (\x,0) node {$\x$}; 
       \foreach \y in {0,1,2,3,4,5,6,7,8,9} 
          \draw[anchor=base] (-1,1+\y) node {$\y$}; 
      }
      { \pgftransformxshift{0.5cm} 
        \draw (-1,10) node[above=2pt] {$z$}; 
        \draw (10,0) node[above=2pt] {$y$}; 
      }
      \draw[\Red,ultra thick] (0,7) -- (1,7) -- (1,6) -- (2,6) -- (2,5) --
       (4,5) -- (4,4) -- (5,4) -- (5,3) -- (6,3) --(6,0);
      \draw[\Red,ultra thick,dotted] (4,5)-- (5,5) -- (5,4);
      \filldraw[fill=gray,opacity=0.5] 
        (4,5)-- (5,5) -- (5,4) -- (4,4) -- cycle;
       \draw[\Black, ultra thick] (0,10) -- (10,0);
    \end{tikzpicture}
  \end{minipage}
\end{center}
\caption{}
 \label{fig:1.5}
\end{figure}

\subsubsection{Combinatorial cycles of type $3$}
\label{sec:1.3.1.3}

Let $\cJ \in H(k)$ be an ideal of type $3$. Then 
\[
   H^0(\cJ(n)) = \bigoplus_{i=0}^n t^{n-i} U_i, 
\]
$U_i\subset S_i$ monomial subspace such that $S_1 U_i \subset U_{i+1}$ ($t$
is a $\NNT$ on $P/\bigoplus\limits_{n\geq 0} H^0(\cJ(n))$.
Put $\cI:=\cJ':=\cJ + t\cO_{\P^3}(-1)/t\cO_{\P^3}(-1)$.
Then $b\geq
\reg(\cJ) \geq \reg(\cI)$ and hence for all $n\geq b$,
\[
  0 \longrightarrow H^0(\cJ(n-1)) \stackrel{t}{\longrightarrow}
  H^0(\cJ(n)) \longrightarrow H^0(\cI(n)) \longrightarrow 0
\]
is an exact sequence. It follows that $U_n = H^0(\cI(n))$, $n\geq b$. Now
$\cI$ is invariant under $T(3;k)$ and $\Set{\left(\begin{smallmatrix} 1 & 0
      & * \\ 0 & 1 & * \\ 0 & 0 & * \end{smallmatrix}\right) }< U(3;k)$,
hence $z$ is a $\NNT$ on $S/\bigoplus\limits_{n\geq 0} H^0(\cI(n))$. \\
If $R = k[x,y]$ then, for all $n\geq b$, the sequence
\[
  0 \longrightarrow H^0(\cI(n-1)) \stackrel{z}{\longrightarrow}
  H^0(\cI(n)) \longrightarrow R_n \longrightarrow 0
\]
is exact. As $H^0(\cI(n))$ is monomial, one gets $R_n \subset H^0(\cI(n))$,
$n\geq b$. As $\cI$ is $b$-regular, one has $S_1 H^0(\cI(n)) =
H^0(\cI(n+1))$, $n\geq b$.

\begin{figure}
  \centering
  \begin{center}
    \begin{minipage}{23cm*\real{0.7}}
      \tikzstyle{help lines}=[gray,very thin]
      \begin{tikzpicture}[scale=0.7]
        \draw[style=help lines] grid (22,13);
        \draw[very thick,->] (0,0) -- (22,0); \draw[very thick,->] (0,0) --
        (0,13); { \pgftransformxshift{0.5cm} \draw (-1,12) node[above=2pt]
          {$z$}; \draw (21,0) node[above=2pt] {$y$}; \draw (19.5,10)
          node[above=1pt]{$H^0(\cJ(b+1))$}; \draw (19,11)
          node[above=1pt]{$H^0(\cJ(b))$}; } \draw[\Red,ultra thick,dotted]
        (16,10.5) -- (18,10.5); \draw[\Red,ultra thick] (16,11.5) --
        (18,11.5);
        \draw[\Red,ultra thick] (0,3) -- (1,3) -- (1,6) -- (3,6) -- (3,5) --
        (4,5) -- (4,4) -- (5,4) -- (5,3) -- (6,3) -- (6,2) -- (8,2) -- (8,8)
        -- (9,8) -- (9,7) -- (10,7) -- (10,6) -- (11,6) -- (11,5) -- (12,5)
        -- (12,6) -- (13,6) -- (13,5) -- (14,5) -- (14,4) -- (15,4) --
        (15,3) -- (16,3) -- (16,2) -- (17,2) -- (17,1) -- (18,1) -- (18,0);
        \draw[\Red,ultra thick,dotted] (0,4) -- (1,4); \draw[\Red,ultra
        thick,dotted] (1,6) -- (1,7) -- (3,7) -- (3,6) -- (4,6) -- (4,5) --
        (5,5) -- (5,4) -- (6,4) -- (6,3) -- (8,3); \draw[\Red,ultra
        thick,dotted] (8,8) -- (8,9) -- (9,9) -- (9,8) -- (10,8) -- (10,7)
        -- (11,7) -- (11,6) -- (12,6) -- (12,7) -- (13,7) -- (13,6) --
        (14,6) -- (14,5) -- (15,5) -- (15,4) -- (16,4) -- (16,3) -- (17,3)
        -- (17,2) -- (18,2) -- (18,1) -- (19,1) -- (19,0);
      \end{tikzpicture}
    \end{minipage}
  \end{center} 
  \caption{}
 \label{fig:1.6}
\end{figure}

We write $H^0(\cI(b))= \bigoplus_{i=0}^bz^{b-i}V_i$, $V_i\subset R_i$
monomial, and one has $V_b=R_b$. As $H^0(\cI(b-1)) =
\bigoplus_{i=0}^bz^{b-1-i}V_i$, it follows that $R_1V_i \subset V_{i+1}$,
$0\leq i \leq b-1$. Let $0\leq c \leq b$ be the natural number such that
$V_i=R_i$, if $c\leq i\leq b$ and $V_i\subsetneq R_i$, if $i<c$. We choose a
natural number $m < c$ and we write $V_m = \langle x^{m-c_0}y^{c_0},\dots,
x^{m-c_r}y^{c_r} \rangle$ where $0 \leq c_0 < \dots < c_r \leq m$ are
natural numbers. We let $\G_a$ act by $ \psi^3_\alpha: x \mapsto x, \quad y
\mapsto \alpha x+y, \quad z \mapsto z, \quad t \mapsto t$.  Then
$\bigwedge\limits^{r+1}\psi_\alpha^3(V_m)$ has the $\alpha$-degree
$D_m:=(c_0 + \cdots + c_r) -(1+2+\cdots +r)$ (see~\cite[p.~13/14]{T1}).
It follows that $\bigwedge\limits^{Q'(b)} \psi^3_\alpha(H^0(\cI(b)))$ has the
$\alpha$-degree $D:=\sum_{m<c}D_m$.  By considering Fig.~\ref{fig:1.6} one
sees that
\[
\alphadeg \bigwedge\limits^{Q'(n)}
\psi^3_\alpha(H^0(\cI(n))) = D \quad\text{ for all } n\geq b\,.
\]
Now from 
\[
H^0(\cJ(n)) = \bigoplus_{i=0}^n t^{n-i}U_i \quad \text{and } 
U_i = H^0(\cI(n)), \text{ if } i\geq b,
\]
it follows that
\begin{align*}
  & \phantom{=} \alphadeg \bigl(\bigwedge\limits^{Q(n)} \psi^3_\alpha(H^0(\cJ(n)))\bigr) \\
  & = \alphadeg \bigl(\bigwedge\limits^{Q(b-1)}
  \psi^3_\alpha(H^0(\cJ(b-1))\bigr) + \sum^n_{i=b} D \\
  & = q_1(n-b+1) +q_0,
\end{align*}
where $q_1 = D$ and $q_0 = \alphadeg \bigwedge\limits^{Q(b-1)}
  \psi^3_\alpha(H^0(\cJ(n)))$.
The same argumentation as in ~\ref{sec:1.3.1.1} gives 
\begin{conclusion}
  \label{concl:1.3} 
  If $C$ is a combinatorial cycle of type $3$, there are natural numbers
  $q_i$ such that $[C] = q_1[C_1] + q_0[C_0]$. \hfill $\qed$
\end{conclusion}

\begin{remark}
  \label{rem:1.1}
\[
 q_1 =0 \iff D=0 \iff \cI = \cJ' \text{ invariant under } \psi^3_\alpha \iff
\cI \text{ is } B(3;k) \text{-invariant}.
\]
\hfill $\qed$
\end{remark}

From Conclusions~\ref{concl:1.1}--\ref{concl:1.3} follows:
\begin{proposition}
  \label{prop:1.1}
  If $C$ is a combinatorial cycle on $\HH$, then there are $q_i\in \N$ such
  that $[C] = q_2[C_2] +q_1[C_1] + q_0[C_0]$. \hfill $\qed$
\end{proposition}
\subsection{Algebraic cycles}
\label{sec:1.3.2}

\index{algebraic cycle}
According to~\cite[Korollar 1, p. 8]{T1} $A_1(\HH)$ is generated over $\Z$
by the combinatorial cycles and the image of
$A_1(\HH^\Delta)\stackrel{\mathrm{can.}}{\longrightarrow} A_1(\HH)$,
$\Delta:=U(4;k)$. $A_1(\HH^\Delta)$ is generated over $\Z$ by the so called 
algebraic cycles $C\subset \HH^\Delta$, where $C= \overline{\G_m\cdot \xi}$, $\xi
\leftrightarrow \cJ \subset \cO_{\P^3}$ an ideal with Hilbert polynomial $Q$,
which is invariant under $\Delta\cdot T(\rho)$. Here $\G_m$ operates by
$\sigma(\lambda): x\mapsto x$, $y\mapsto y$, $z\mapsto z$, $t\mapsto \lambda
t$, if $\rho_3 \neq 0$, respectively by $\sigma(\lambda): x\mapsto x$,
$y\mapsto y$, $z\mapsto \lambda z$, $t\mapsto t$, if $\rho_3 =0$
(see~\cite[Abschnitt 1.1 and 1.2]{T1}). If $\rho_3\neq 0$, then $[C] =
q\cdot [C_0]$, $q \in \N$ (\cite[Prop. 1, p. 25 and Prop. 3, p.
58]{T2}; ~\cite[Prop. 1, p. 6]{T4}). \\
The case $\rho_3 =0$ remains. In this case $\cJ$ is invariant under
$t\mapsto \lambda t$, $\lambda\in k^*$, and we can write: $H^0(\cJ(n)) =
\bigoplus_{i=0}^n t^{n-i}V_i$, $V_i \subset S_i$ invariant under $U(3;k)$
and $T(\rho)$, $S_1 V_i \subset V_{i+1}$ for all $i\geq 0$
and $V_i = H^0(\cI(i))$, if $i\geq b-1$. Here $\cI :=\cJ' = \cJ +
t\cO_{\P^3}(-1)/t \cO_{\P^3}(-1)$ is invariant under $U(3;k)\cdot T(\rho)$,
which follows from $b\geq \reg(\cJ) \geq \reg(\cI)$.

Now $\rho_0+\rho_1 +\rho_2=0$ (cf.~\cite[Bemerkung 1, p. 2]{T1} and
Appendix~\ref{cha:H}) and we show that $\rho_2\neq 0$. Otherwise $V_i =
\bigoplus_{j=0}^i z^{i-j}U_j$,
$U_j\subset k[x,y]_j$ invariant under $U(2;k)$. As $\Char(k)=0$, it follows
that $U_j$ is monomial for all $j$, hence $\cJ$ is monomial, too. But then
$C$ would not be a curve. Thus one has $C = \Set{ \sigma(\lambda)\xi |
  \lambda\in k^* }^-$, where $\sigma(\lambda): x\mapsto x$, $y\mapsto y$,
$z\mapsto \lambda z$, $t\mapsto t$.  Now
\begin{align*}
  (\cM_n \cdot C) & = \lambda\mbox{-}\mathrm{red}\deg \dot\bigwedge H^0(\cJ(n)) \\
  & := \frac{1}{\ell}
  \lambdadeg\bigl(\stackrel{Q(n)}{\bigwedge}\sigma(\lambda)H^0(\cJ(n))\bigr) \\
  & = \frac{1}{\ell} \sum_{i=0}^n
  \lambdadeg\bigl(\stackrel{\phi(n)}{\bigwedge}\sigma(\lambda)V_i\bigr)
\end{align*}
with $\phi(i) = \dim V_i$. If $V\subset P_n$ is any subspace with dimension
$m$, the reduced $\lambda$-degree of $\stackrel{m}{\bigwedge}V$ is defined
as in the proof of~\cite[Hilfssatz 5, pp. 8]{T2}
as $\sum(e_i-d_i)/\ell$, where $\ell$ is the order of the inertia group of
$V$ in $\G_m$ (loc.~cit, p. 9 lines 11 and 12). Now the above sum is equal
to 
\[
 c + \frac{1}{\ell} \sum_{i=b}^n
\lambdadeg\bigl(\stackrel{\phi(n)}{\bigwedge} H^0(\cI_\lambda(i))\bigr)
\]
where 
$c := \tfrac{1}{\ell} \sum_{i=0}^{b-1}
\lambdadeg(\sigma(\lambda)V_i)$  and $\cI_\lambda:=\sigma(\lambda)\cI$, hence 
$H^0(\cI_\lambda(i))= \sigma(\lambda)(H^0(\cI(i)))$. As $U(3;k)$ is
normalized by $\sigma(\lambda)$, $\cI_\lambda$ is invariant under
$U(3;k)$. Hence the sequence
\[
  0 \longrightarrow \cI_\lambda(-1) \stackrel{\cdot z}{\longrightarrow}
  \cI_\lambda \longrightarrow \cO_{\P^1} \longrightarrow 0
\]
is exact, and it follows that the sequence
\[
  0 \longrightarrow H^0(\cI_\lambda(i-1)) \stackrel{\cdot z}{\longrightarrow}
  H^0(\cI_\lambda(i)) \longrightarrow R_i \longrightarrow 0
\]
is exact for all $i\geq b$. Now $R_i \subset H^0(\cI(i))$ for all
$i\geq b$ (see Appendix~\ref{cha:D}, Lemma 1), 
hence  $R_i \subset H^0(\cI_\lambda(i))$. But then 
\[
 \stackrel{\phi(i)}{\bigwedge} H^0(\cI_\lambda(i)) \simeq
 \stackrel{\phi(i-1)}{\bigwedge} H^0(\cI_\lambda(i-1)) \otimes 
 \stackrel{i+1}{\bigwedge} R_i
\]
and hence for all $i\geq b-1$
\[
\frac{1}{\ell}\cdot  
\lambdadeg\bigl(\stackrel{\phi(i)}{\bigwedge} H^0(\cI_\lambda(i))\bigr)
= \frac{1}{\ell} 
\lambdadeg\bigl(\stackrel{\phi(b-1)}{\bigwedge} H^0(\cI_\lambda(b-1))\bigr)=: \gamma
\]
is independent of $i$.
As the inertia group $T_i$ of $H^0(\cI(i))$ in $\G_m$ contains the inertia
group of $\cJ$ in $\G_m$, the number $\ell$ divides $\#T_i$, and this number
divides the $\lambda$-degree of
$\stackrel{\phi(i)}{\bigwedge}H^0(\cI_\lambda(i))$ (cf.~\cite[eq. (2), p.
9]{T2}).  Hence $\gamma$ is a natural number. The same argumentation also
shows that $c\in \N$. But then $(\cM\cdot C) = c + \gamma(n-b+1)$ and the
same argumentation as in~\ref{sec:1.3.1} gives
\begin{proposition}
\label{prop:1.2}
If $C$ is an algebraic cycle on $\HH$, then there are $q_i\in \N$ such that $[C] =
 q_1[C_1] + q_0[C_0]$. \hfill $\qed$
\end{proposition}

\begin{remark}
  \label{rem:1.2}
  As in the proof of Proposition~\ref{prop:1.1} we consider the case $\rho_3
  =0$. Then 
  \begin{align*}
    q_1 = 0 & \iff \lambdadeg(\dot \bigwedge H^0(\cI_\lambda(i)) = 0 \text{
      for all }
    i\geq b \\
    & \iff \cI \text{ invariant under } \sigma(\lambda): x\mapsto x,
    y\mapsto y, z\mapsto \lambda z, t\mapsto t \\
    & \iff \cI \text{ is } T(3;k) \text{-invariant} \\
    & \iff \cI \text{ is } B(3;k) \text{-invariant}. 
  \end{align*}
\hfill $\qed$
\end{remark}
\subsection{Computation of $A^+_1(\HH)$ and $A^+_1(\CC)$ }
\label{sec:1.3.3}

We now prove the result mentioned in the Introduction.
\begin{theorem}
 \label{thm:1.2}
  Let be $\HH = \Hilb^P(\P^3)$, $P(n) = dn-g+1$, $d\geq 3$ and $g\leq
  g(d):=(d-2)^2/4$.

  If $Z \in A_1(\HH)$ (resp.~$Z\in A_1(\CC)$) is an effective $1$-cycle
  with integer coefficients, then there are uniquely determined natural
  numbers $q_i$ such that $Z= q_0[C_0] + q_1[C_1] + q_2[C_2]$
  (resp.~$Z= q_0[C^*_0] + q_1[C^*_1] + q_2[C^*_2] + q_3[L^*]$). \\
Hence the
  cone $A^+_1(\HH)$ (resp.~$A^+_1(\CC)$) of effective $1$-cycles on $\HH$
  (resp.~on $\CC$) is freely generated by (the classes of) $C_0, C_1, C_2$
  (resp.~$C^*_0, C^*_1, C^*_2, L^*)$.
\end{theorem}
\begin{proof}
  Let $C\subset \HH$ be a closed curve. By applying~\cite[Lemma 1, p.
  6]{T1} several times, one constructs a cycle $C_* = \sum n_j C_j \in
  Z_1(\HH)$ such that $[C] = [C_*]$, $n_j\in \N$, and the irreducible
  components $C_j$ are $B(4;k)$-invariant. From~\cite[Proposition 0, p.
  3]{T1} it follows that one of the following cases can occur:

\textsc{Case 1:} $C_j$ is a combinatorial cycle of type $i \in \Set{1,2,3}$,
i.e.~$C_j= \overline{\G_a\cdot \xi}$, where $\xi\in \HH(k)$ corresponds to
an ideal $\cJ$ of type $i$ with Hilbert polynomial $Q(n) = \tbinom{n+3}{3}
-P(n)$. 

\textsc{Case 2:} $C_j$ is an algebraic cycle, 
i.e.~$C_j= \overline{\G_m\cdot \xi}$, $\xi\in \HH(k)$ invariant under
$\Delta\cdot T(\rho)$ (see~\cite[Section 1.1 and 1.2]{T1}). As to the cone 
$A^+_1(\HH)$, the assertion follows from Propositions~\ref{prop:1.1}
and~\ref{prop:1.2}.

Now let $\cC\subset \CC$ a closed curve. From Theorem III
in Section~\ref{sec:1.1} and equation~\eqref{eq:1.2} in Section~\ref{sec:1.2}
it follows that 
\[
  [\cC] =  q_0[C^*_0] + q_1[C^*_1] + q_2[C^*_2] + q_3[L^*]
\]
with $q_i\in \Q$. If $\pi:\CC \to \HH$ is the projection, the restrictions
$\pi|C_i:C^*_i \to C_i$ are isomorphisms, $0\leq i \leq 2$, 
it follows that 
\[
\pi_*[\cC] =  q_0[C_0] + q_1[C_1] + q_2[C_2],
\]
hence, by what just has been shown, $q_i\in\N$, $0\leq i \leq 2$.  If
$\kappa$ is the projection $\CC\to \P^3$, the 
 of $\kappa_*$
gives 
\[
 \deg(\kappa|\cC) \cdot[\kappa(\cC)] = q_3\cdot[L]\,.
\]
As $A_1(\P^3) = \Z[L]$ one has $q_3\in \N$.
\end{proof}

\section{The ample cone of $\HH$ and of $\CC$}
\label{sec:1.4}

If $Q(n) = \tbinom{n-1+3}{3}+\tbinom{n-a+2}{2} + \tbinom{n-b+1}{1}$ is the
given Hilbert polynomial we always make the assumptions $a\geq 4$ and $b\geq
(a^2-1)/4$ (equivalently $d\geq 3$ and $g\leq (d-2)^2/4$). We put $r=b-a$
and $\rho=r(r+1)/2$ and define three line bundles on $\HH$:
\begin{align*}
  \cL_0 &:= \cM^{1-\rho}_{b-1} \otimes \cM^{2\rho}_{b} \otimes
  \cM^{-\rho}_{b+1} \\ 
  \cL_1 &:= \cM^{-r-3}_{b-1} \otimes \cM^{2r+5}_{b} \otimes
  \cM^{-r-2}_{b+1} \\ 
  \cL_2 &:= \cM_{b-1} \otimes \cM_{b}^{-2} \otimes \cM_{b+1}\,.
\end{align*}
Now the formulas of \cite[pp. 134-135]{T2}
\begin{equation}
  \label{eq:1.13}
  (\cM_n\cdot E) = 1, \quad (\cM_n\cdot C_1) = (n-b+1), \quad (\cM_n\cdot
  C_2) = \tbinom{n-a+2}{2} + (n-b+1)
\end{equation}
are derived under the assumption $n\geq b$, because then $\cM_n$ defines an
embedding of $\HH$ in a suitable projective space. But the expressions for
$H^0(\cJ(n))$ in (loc.~cit.) shows that the formulas are also true if
$n=b-1$, because $H^0(\cJ(n))$ is a subbundle of $P_n\otimes \cO_C$ of
degree $Q(n)$ for all $n\geq b-1$, if $C\subset \HH$ is any curve.

\textbf{Intersection numbers of $\cL_0$}
\begin{align*}
  (\cL_0\cdot E)   & = (1-\rho) + 2\rho-\rho =1 \\
  (\cL_0\cdot C_1) & = 0 + 2\rho-2\rho = 0 \\
  (\cL_0\cdot C_2) & = (1-\rho)\tbinom{b-a+1}{2}
                     + 2\rho\left[\tbinom{b-a+2}{2}+1\right]
                     -\rho\left[\tbinom{b-a+3}{2}+2\right]\\
                 & =  \rho + \rho\left[\tbinom{b-a+2}{2}-
                   \tbinom{b-a+1}{2}\right] - \rho\left[\tbinom{b-a+3}{2}-
                   \tbinom{b-a+2}{2}\right] \\
                 & = \rho + \rho(b-a+1) -\rho(b-a+2) =0
\end{align*}

\textbf{Intersection numbers of $\cL_1$}
\begin{align*}
  (\cL_1\cdot E)   & = -(r+3)+(2r+5)-(r+2) = 0 \\
  (\cL_1\cdot C_1) & = 0 + (2r+5) -(r+2)\cdot 2 = 1 \\
  (\cL_1\cdot C_2) & = -(r+3)\tbinom{b-a+1}{2} 
                     + (2r+5) \left[\tbinom{b-a+2}{2}+1\right] 
                     -(r+2)\left[\tbinom{b-a+3}{2}+2\right]\\
                   & = -(r+3)\tbinom{b-a+1}{2} + (r+3)\tbinom{b-a+2}{2}
                       + (r+2)\tbinom{b-a+2}{2} - (r+2)\tbinom{b-a+3}{2} + 1
                       \\
                   & = (r+3)(b-a+1) -(r+2)(b-a+2)+1 \\
                   & = (b-a)+(r+3)-2(r+2)+1 =0
\end{align*}

\textbf{Intersection numbers of $\cL_2$}
\begin{align*}
  (\cL_2\cdot E)   & = 0 \\
  (\cL_2\cdot C_1) & = 0 -2 +2 = 0 \\
  (\cL_2\cdot C_2) & = \tbinom{b-a+1}{2}  - 2\left[\tbinom{b-a+2}{2}+1\right]
                     + \left[\tbinom{b-a+3}{2}+2\right] \\
                   & =  -(b-a+1) + (b-a+2) = 1\,.
\end{align*}

\begin{conclusion}
 \label{concl:1.4}
  From now on we write $C_0$ instead of $E$ and hence have the formula
  \begin{equation}
    \label{eq:1.14}
    (\cL_i\cdot C_j) = \delta_{ij}\,.  \qedhere 
  \end{equation}
 \qed
\end{conclusion}
                 
If $\cL$ is any line bundle on $\HH$ and $\nu_i = (\cL \cdot C_i)$ we put
$\cM:=\cL_0^{\nu_0} \otimes \cL_1^{\nu_1} \otimes \cL_2^{\nu_2}$ and 
$\cN:=\cL \otimes \cM^{-1}$. From Theorem~\ref{thm:1.2} follows $(\cN\cdot
Z)=0$ for all $Z\in A^+_1(\HH)$, hence $\cN\in \Pic^\tau(\HH)$. Now $\HH$ is
simply connected and the argumentation in \cite[Section 4.2]{T2} shows
that $\Pic^\tau(\HH) = \Pic^0(\HH)$. 

\begin{theorem}
  \label{thm:1.3}
  \begin{enumerate}[(i)]
  \item $\Pic(\HH)/\Pic^0(\HH)$ is freely generated by $\cL_0$, $\cL_1$,
    $\cL_2$ and a line bundle $\cL$ on $\HH$ is ample iff it has the form
    $\cL=\cL_0^{\nu_0} \otimes \cL_1^{\nu_1} \otimes \cL_2^{\nu_2}\otimes
    \cN$, where the $\nu_i$ are positive natural numbers and $\cN\in
    \Pic^0(\HH)$.
  \item A line bundle $\cL$ on the universal curve $\CC$ is ample iff it has
    the form $\cL=\pi^*\cL_0^{\nu_0} \otimes \pi^*\cL_1^{\nu_1} \otimes
    \pi^*\cL_2^{\nu_2}\otimes \kappa^*\cO_{\P^3}(\nu_3) \otimes \cN$ where
    $\nu_i$ are positive natural numbers, $\pi$ and $\kappa$ are the
    projection of $\CC$ to $\HH$, respectively to $\P^3$, and $\cN\in
    \Pic^0(\CC)$.
  \end{enumerate}
\end{theorem}
\begin{proof}
  \begin{enumerate}[(i)]
  \item follows from the foregoing computations and the theorem of Kleiman.
  \item Using the same notations as in Theorem~\ref{thm:1.2} one sees that the
    restriction of $\pi$ to $C^*_i$ and of $\kappa$ to $L^*$ gives an
    isomorphism $C^*_i\to C_i$ respectively $L^* \to L$. Hence one has
\[
(\pi^* \cL\cdot C^*_j) = \delta_{ij}, \quad (\pi^* \cL\cdot L)=0 \quad
\text{ and } \quad (\kappa^* \cO_{\P^3}(1) \cdot L^*) = 1\,.
\]
 The argumentation in the
    proof of \cite[Satz 3, p. 40]{T3} shows that
    $H^1_{\mathrm{sing}}(\CC,\Z/n)=0$ and hence $\Pic^\tau(\CC) =
      \Pic^0(\CC)$ (see the argumentation in \cite[Section 4.2]{T2}).
  \qedhere
  \end{enumerate}
\end{proof}
\begin{remark}
  We will show (cf.~Lemma~\ref{lem:1.1} and ~\ref{lem:1.5}) that $\cM_n$ is a
  line bundle on $\HH$ for all $n\geq a-3$ and that the
  formulas~\eqref{eq:1.13} are true for all $n\geq a-3$, too.
\end{remark}

\section{Some globally generated line bundles}
\label{sec:1.5}

\subsection{Regularity of sheaves}
\label{sec:1.5.1}

\begin{auxlemma}
  \label{auxlem:1}
 Let be $Y/k$ a noetherian scheme, $X = \P^r\times_kY$ and $\cF$ a coherent
 $\cO_X$-module such that $H^i(X,\cF(n-i))=(0)$ for all $i>0$ and all
 $n\geq m$ (i.e.~$\cF$ is $m$-regular).
 \begin{enumerate}[(a)]
 \item There are Zariski-many linear forms $\ell\in S:=k[X_0,\dots,X_r]$
   such that
   \begin{equation}
     \label{eq:1.15}
     0 \longrightarrow  \cF(-1) \stackrel{\ell}{\longrightarrow}  \cF \longrightarrow \cF' \longrightarrow  0
   \end{equation}
is exact, where $\cF':= \cF/\ell\cF(-1)$.
\item If~\eqref{eq:1.15} is exact for any linear form $\ell\in S$, then one
  has 
  \begin{enumerate}[(i)]
  \item $H^i(\cF'(n-i))=(0)$ for all $i>0$ and $n\geq m$.
  \item If $\pi:X\to Y$ is the projection and $\cF_n:=\pi_*\cF(n)$,
    $\cF'_n:=\pi_*\cF'(n)$, then 
    \begin{equation}
      \label{eq:1.16}
      0 \longrightarrow \cF_{n-1} \longrightarrow  \cF_n \longrightarrow \cF'_n \longrightarrow 0
    \end{equation}
is exact and $\cF'_n\simeq \cF_n/\ell\cF_{n-1}$ for all $n\geq m$.
  \end{enumerate}
 \end{enumerate}
\end{auxlemma}
\begin{proof}
  As $Y$ can be covered by finitely many open affine subsets, without
  restriction one can assume $Y = \Spec A$, $\cF=\tilde{M}$, $M$ a graded
  $S\otimes A$-module of finite type. Let $P_1,\dots,P_s$ be the finitely
  many associated primes of $M$, which are different from $S_+ \otimes A$.
  If $\ell \in S_1 - \cup P_i$, then~\eqref{eq:1.15} is exact. By means of
  the exact sequence
\[
    H^i(\cF(n-i)) \longrightarrow H^i(\cF'(n-i))\longrightarrow H^{i+1}(\cF(n-i-1))
\]
the assertion (i) in part (b) follows.
Because of $H^1(\cF(n-1))=0$, if $n\geq m$, the  exactness
of~\eqref{eq:1.16} follows. From the diagram
\[
\xymatrix{
            &  \cF_{n-1} \ar[r]\ar@{=}[d] &  \cF_n \ar[r]\ar@{=}[d]  &
            \cF_n/\ell \cF_{n-1}\ar[r]\ar[d]^{\mathrm{can.}} &  0 \\
   0 \ar[r] &  \cF_{n-1} \ar[r] &   \cF_n \ar[r] &  \cF'_n \ar[r] &  0
}
\]
which is commutative and has exact rows, it follows that
$\cF_n/\ell\cF_{n-1} \simeq \cF'_n$ if $n\geq m$.
\end{proof}
\textbf{The case of curves}

\subsubsection{}
\label{sec:1.5.1.1}
Let $Q(n) = \tbinom{n-1+3}{3}+\tbinom{n-a+2}{2} + \tbinom{n-b+1}{1}$ be the
Hilbert polynomial, $K/k$ an extension field, $X=\P^3_K$, $\cI \subset
\cO_X$ an ideal with Hilbert polynomial $Q$, $\cF = \cO_X/\cI$ the structure
sheaf of the curve defined by $\cI$, $P=k[x,y,z,t]$. If $\ell\in P_1$ is
sufficiently general, the sequences
\begin{equation}
  \label{eq:1.17}
      0 \longrightarrow  \cF(n-1) \stackrel{\ell}{\longrightarrow}  \cF(n) 
\longrightarrow \cF'(n) \longrightarrow  0  
\end{equation}
\begin{equation}
  \label{eq:1.18}
  0 \longrightarrow  \cI(n-1) \stackrel{\ell}{\longrightarrow}  \cI(n) \longrightarrow \cI'(n) \longrightarrow  0  
\end{equation}
are exact, where $\cI':=\cI +\ell\cO_X(-1)/\ell\cO_X(-1)$ is an ideal on
$\Proj(P/\ell P(-1))\simeq \P^2_K$ with Hilbert polynomial $Q'(n) =
\tbinom{n-1+2}{2}+\tbinom{n-(a-1)+1}{1}$. Hence $\cI'$ is $(a-1)$-regular,
i.e.~$H^i(\cI'(n))=0$ for all $n\geq a-1-i$ and $i\geq 1$ \cite[Lemma
2.9]{G78}. From~\eqref{eq:1.18} one gets the exact sequence
\[
H^{i-1}(\cI'(n-1)) \longrightarrow H^i(\cI(n-1)) \longrightarrow H^i(\cI(n))
\longrightarrow H^i(\cI'(n))
\]
where the first term and the last term vanish, if $n\geq (a-1)-(i-1)$ and
$i-1\geq 1$. It follows that $H^i(\cI(n-1)) \simeq H^i(\cI(n))$, if $n\geq
a-i$ and $i\geq 2$. As $H^i(\cI(n))=(0)$, if $i\geq 1$ and $n\gg 0$
\begin{equation}
  \label{eq:1.19}
  H^i(\cI(n))=(0), \quad \text{ if } n\geq a-i-1 \text{ and } i\geq 2 
\end{equation}
follows. On the other hand one has the exact sequence
\[
 H^1(\cO_X(n)) \longrightarrow H^1(\cF(n)) \longrightarrow H^2(\cI(n)) 
\]
where the first term vanishes if $n\geq 0$ and the last term vanishes if
$n\geq a -3$. As $\dim C =1$, one has $ H^i(\cF(n)) =(0)$ if $i\geq 2$,
hence
\begin{equation}
  \label{eq:1.20}
  \reg(\cF) \leq a-2\,.
\end{equation}

\subsubsection{}
\label{sec:1.5.1.2}
 Let $Y/k$ be a noetherian scheme, $\mathfrak{X} = Y\times_k \P^3_k$,
 $\cI\subset \cO_\fX$ an ideal such that $\cF = \cO_\fX/\cI$ is flat over
 $Y$ with Hilbert polynomial $P(n) = \tbinom{n+3}{3}-Q(n)$ in each fibre.
 According to~\eqref{eq:1.20} $\cF\otimes k(y)$ is $(a-2)$-regular and hence
 $H^1(\cF(n)\otimes k(y)) =(0)$ for all $n\geq a-3$, $y\in Y$.
 From~\cite[Cor. 1, p. 51]{M2} it follows that $\pi_*(\cF(n))\otimes k(y)
 \stackrel{\sim}{\to} H^0(\cF(n)\otimes k(y))$ for all $n\geq a-3$, $y\in
 Y$, where $\pi:\fX \to Y$ is the projection. 
 From~\cite[Cor.~2, p.~52]{M2} it follows that
 \begin{equation}
   \label{eq:1.21}
  \pi_* \cF(n) \text{ is locally free of rank } P(n) \text{ on } Y, \text{
    for all } n\geq a-3\,.
 \end{equation}
 The assertion concerning the rank follows from $H^i(\cF(n)\otimes k(y))
 =(0)$, for all $i\geq 2$, $y\in Y$.  From~\cite[Cor.~1, p.~51]{M2}
 and~\eqref{eq:1.20} it follows that
\[
  R^1\pi_*(\cF(n))\otimes k(y) \simeq H^1(\cF(n)\otimes k(y)) =(0)
\]
if $n\geq a-3$ and $y\in Y$. By Nakayama,  this implies $R^1\pi_*\cF(n)
=(0)$ for all $n\geq a-3$. The same argumentation shows that $R^i\pi_*\cF(n)
=(0)$ for all $i \geq 2$ and all $n$.  All in all, we get
\begin{lemma}
  \label{lem:1.1} 
 $\cF$ is $(a-2)$-regular, $\cF_n:=\pi_*\cF(n)$ is locally free on $Y$ of
 rank $P(n)$ and hence $\cM_n:=\bigwedge\limits^{P(n)} \pi_*\cF(n)$
is a line bundle on $Y$ for all $n\geq a-3$. \hfill $\qed$
\end{lemma}

\subsection{$\cM^{-1}_n\otimes \cM_n$ globally generated}
\label{sec:1.5.2}
We consider the curve case of the last section: $Y/k$ noetherian scheme,
$P=k[x,y,z,t]$, $S=k[x,y,z]$ and $\cC\in \HH(Y)$, i.e.~a diagram
\[
 \xymatrix{
 \cC \ar@{^{(}->}[rr] \ar[dr] &  & Y \times \P^3 = \fX \ar[dl]^p \\
                              & Y &  
 }
\]
where $\cC$ is a flat curve over $Y$ with Hilbert polynomial $P(n) =
\tbinom{n+3}{3}-Q(n)$, defined by the ideal $\cI \subset \cO_\fX$, $\cF =
\cO_\fX / \cI$ the structure sheaf of $\cC$, $\cI_n:=p_* \cI(n)$, $\cF_n:=p_*
\cF(n)$, $\cP:= \cO_Y \otimes_k P$, $\cS:=\cO_Y \otimes_k S$. 

We consider linear forms $\ell \in P_1$ of the form $\ell = \alpha x +\beta
y +\gamma z +t$, $\alpha,\beta,\gamma \in k$. For each such $\ell$
\[
   U(\ell):=\Set{ y \in Y | \cF(-1) \otimes k(y)  
\stackrel{\cdot \ell}{\longrightarrow} \cF \otimes k(y) \text{ is injective}
}
\]
is an open subset of $Y$ (possibly empty). The openness follows
from~\cite[Lemma 1 and 2]{G88}, for example.  If $\cF(-1) \stackrel{\cdot
  \ell}{\longrightarrow} \cF$ is injective and $\ell \in P_1$, then we write
$\ell \in \NNT(\cF)$.
\begin{auxlemma}
  \label{auxlem:2}
For each $y\in Y(k)$ there are Zariski-many $\ell = \alpha x +\beta
y +\gamma z +t \in \NNT(\cF)$ such that $y\in U(\ell)$.
\end{auxlemma}
\begin{proof}
  Let $U_i= \Spec A_i$, $1\leq i \leq m$, be an open, affine covering of
  $Y$.  Put $\cI:= \bigoplus_{n\geq 0} p_* \cI(n)$, $I^i:=\cI\otimes A_i$,
  $I:=\cI\otimes k(y)$. Let $\cP^i_j$ be the associated prime ideals of
  $I^i$, which are different from $P_+ \otimes A_i$, $1\leq j \leq r(i)$;
  let $\cP^0_j$ be the associated prime ideals of $I^0:=I$, which are
  different from $P_+ \otimes k(y)$, $1\leq j \leq r(0)$.

 For the moment, we fix the index $i$. Then
\[
 V^i_j:=\Set{ (\alpha, \beta, \gamma) \in k^3 | \ell := \alpha x +\beta
y +\gamma z +t \in \cP^i_j \cap P_1 }
\]
is a closed subset of $\A^3_k$. It is different from $\A^3_k$, because
otherwise $\cP^i_j=P_+ \otimes A_i$, respectively $\cP^0_j=P_+ \otimes k(y)$
would follow. Hence $W^i_j:=\A^3_k - V^i_j$ is open and non-empty in
$\A^3_k$, hence the same is true for 
\[
    \bigcap_{j=1}^{r(0)} W^0_j \bigcap_i \bigcap_{j=1}^{r(i)} W_j^i\,.
\]
If one puts 
\[
  L_i(y) := 
 \Set{ \ell = \alpha x +\beta y +\gamma z +t | y  \in  U(\ell) \text{ and } 
  \cF(-1) \otimes \cO_{U_i}  
\stackrel{\ell}{\longrightarrow} \cF \otimes \cO_{U_i}  \text{ is injective}
},
\]
then it follows that $L_i(y) \subset \A^3_k $ is non-empty and Zariski-open,
hence the same is true for $L(y):= \bigcap\limits_1^m L_i(y)$.
\end{proof}

\begin{auxlemma}
  \label{auxlem:3}
  There are finitely many $\ell_i = \alpha_i x +\beta_i y +\gamma_i z +t \in
  P_1$, $1\leq i\leq m$, such that $ Y= \bigcup\limits_1^m U(\ell_i)$, and
  for each $\ell \in \{\ell_1,\dots,\ell_m\}$ one has:
 \begin{enumerate}[(a)]
 \item\mbox{}\par
{\centering \centerline{$0 \longrightarrow \cF_{n-1} \stackrel{\ell}{\longrightarrow} \cF_n
\longrightarrow \cF'_n \longrightarrow 0$}}

 is exact on $Y$ and  $\cF'_n:=p_*(\cF(n) /\ell\cF(n-1))$ is isomorphic to
 $\cF_n/\ell \cF_{n-1}$ for all $n\geq a-2$.
\item If $U:=U(\ell)$, then $\cF'_n\otimes \cO_U$ is locally free on $U$ of
  rank $(a-1)=d$ for all $n$.
 \end{enumerate}
\end{auxlemma}
\begin{proof}
  To each $y\in Y(k)$ choose a linear form $\ell_y\in L(y)$. Then
  $U(\ell_y)$ is an open neighborhood of $y$ and finitely many such
  neighborhoods cover $Y$. The assertion (a) is true for each $\ell\in
  \NNT(\cF)$ according to Aux-Lemma~\ref{auxlem:1}b, because $\cF$ is
  $(a-2)$-regular (cf.~Lemma~\ref{lem:1.1}).

As to assertion (b), let $y \in U(\ell)$. Then
\[
0 \longrightarrow H^0(\cF(n-1) \otimes k(y))
\stackrel{\ell}{\longrightarrow} H^0(\cF(n) \otimes k(y)) \longrightarrow
H^0(\cF'(n) \otimes k(y)) \longrightarrow 0
\]
is exact for $n\geq a-3$, as $\cF \otimes k(y)$ is
  $(a-2)$-regular (cf.~\eqref{eq:1.20}). As has been noted in
  Section~\ref{sec:1.5.1.2}, it follows that $\cF_n \otimes k(y)
 \simeq H^0(\cF(n) \otimes k(y))$, if $n\geq a-3$. It follows that for $n\geq
 a-2$
\[
0 \longrightarrow \cF_{n-1}\otimes k(y) \stackrel{\ell}{\longrightarrow}
\cF_n \otimes k(y) \longrightarrow (\cF_n /\ell\cF_{n-1}) \otimes
k(y)\longrightarrow 0
\]
is exact. The local flatness criterion then shows that $\cF'_n \otimes
\cO_{Y,y}$ is flat over $\cO_{Y,y}$ of rank $P'(n)$, if $n\geq a-2$. It
follows that $\cF' \otimes \cO_{Y,y}$ is flat over $\cO_{Y,y}$.  As $\cF'$
is $0$-regular, it follows that $\cF'_n \otimes \cO_U$ is locally free.
\end{proof}

\textbf{N.B.} Whereas $\cF'_n$ is locally free over $U(\ell)$ of rank
$d=(a-1)$ for all $n$, this is true for $\cF_n /\ell\cF_{n-1}$ only if
$n\geq a-2$, as $\cF'_n\simeq \cF_n /\ell\cF_{n-1}$ is, in general, only
true for $n\geq a-2$. \\

Now let be $\ell = \alpha x +\beta y +\gamma z +t \in \NNT(\cF)$
such that $U:=U(\ell)\neq\emptyset$ (e.g., one can choose $\ell \in
\{\ell_1,\dots,\ell_m\}$ as in Aux-Lemma~\ref{auxlem:3}).  
We consider the diagram
\[
\xymatrix{
            &                      &                     &
 \cS_n   \ar[dl]_\kappa \ar[d]^\pi &  \\
   0 \ar[r] &  \cF_{n-1} \ar[r]^\phi &  \cF_n \ar[r]^\psi &  \cF'_n \ar[r] &  0
\;.} 
\]
Here $n\geq a-2$, $\phi$ is the multiplication with $\ell$, $\psi$ is the
canonical map and $\pi$ is the composed map
\[
 \cS_n = \cO_Y \otimes_k S_n \stackrel{\mathrm{can.}}{\longrightarrow}
\cP_n/\ell\cP_{n-1} \stackrel{\mathrm{can.}}{\longrightarrow}
\cF_n/\ell\cF_{n-1} \stackrel{\sim}{\longrightarrow} \cF'_n\,.
\]
$\kappa$ is the composed map $\cS_n \hookrightarrow \cP_n
\stackrel{\mathrm{can.}}{\longrightarrow} \cF_n$. \\
(\textbf{N.B.} $\cP_n \longrightarrow \cF_n$ is not necessarily surjective,
if $n\leq b-1$.) Now $\cF' = \cF/\ell \cF(-1) \simeq \cS/\cI'$, where $\cI'
\simeq \cI + \ell \cP(-1)/\ell \cP(-1)$ is an ideal in $\cS$. $\cF'$ is flat
over $U$ with rank $d$, hence $\cI'$ is flat over $U$ with Hilbert
polynomial $Q'(n) = \tbinom{n-1+2}{2}+\tbinom{n-(a-1)+1}{1}$.Thus
$R^1\pi_*(\cI'(n))\otimes \cO_U = (0)$, if $n\geq a-2$, and it follows that:
\begin{equation}
  \label{eq:1.22}
  \cS_n \rightarrow \cF'_n \quad \text{ is surjective on } U \text{ for } n\geq a-2\,.
\end{equation}

We consider for $y\in U$ a sufficiently small open, affine neighborhood
$V\subset U$ and we argue as follows: $z \in \cF_n(V) \Rightarrow \psi(z) =
\pi(y) = \psi \kappa(y)$ for an element $y\in \cS_n \Rightarrow z-\kappa(y)
\in \Ker(\psi|V) = \Im(\phi|V) \Rightarrow z-\kappa(y) = \phi(x)$ with $x\in
\cF_{n-1}(V)$. It follows that
\begin{equation}
  \label{eq:1.23}
  \begin{aligned}
   \cF_{n-1} \oplus  \cS_n & \rightarrow \cF'_n  \qquad \quad \text{with}\\
   (x,y)& \mapsto \phi(x)+\kappa(y) 
  \end{aligned}
\end{equation}
is a globally defined homomorphism, which is surjective on $U$. 

Put $p= P(n-1)$, $d=P'(n)$. One defines $\phi_\ell$ by means of the diagram
\[
\xymatrix{
 \bigwedge\limits^{p+d} (\cF_{n-1} \oplus \cS_n) \ar[rr] \ar@{-}[d]^\simeq
 &   & \bigwedge\limits^{p+d} \cF_{n} \\
\bigoplus\limits_{\stackrel{i+j}{=p+d}} \bigwedge\limits^{i} \cF_{n-1} \otimes
\bigwedge\limits^{j} \cS_n &  & \\
\bigwedge\limits^{p} \cF_{n-1} \otimes
\bigwedge\limits^{d} \cS_n \ar[u] \ar[uurr]_{\phi_\ell} & & 
}
\]
$\phi_\ell$ is a homomorphism of $\cO_Y$-modules and the horizontal map is
surjective on $U=U(\ell)$. Then Aux-Lemma~\ref{auxlem:3} says that
\[
0 \longrightarrow \cF_{n-1}\otimes \cO_U \stackrel{\ell}{\longrightarrow}
\cF_n \otimes \cO_U \longrightarrow \cF'_n \otimes \cO_U\longrightarrow 0
\]
is exact and $\cF'_n\otimes \cO_U$ is locally free over $U$ of rank $d$, if
$n\geq a-2$. Let $y\in U$ and let $V =\Spec A$ be a sufficiently small open
neighborhood of $y$ in $U$. Put $F_n = \cF_n\otimes A$, $F'_n =
\cF'_n\otimes A$. Then one has a commutative diagram, with lower row an
exact sequence of free $A$-modules, if $n\geq a-2$:
\[
\xymatrix{
            &                      &                     &
 S_n \otimes A  \ar[dl]_\kappa \ar@{->>}[d]^\pi &  \\
   0 \ar[r] &  F_{n-1} \ar[r]^\ell &  F_n \ar[r]^\psi &  F'_n \ar[r] &  0\;.
}
\]
$F'_n\otimes k(y)$ has a basis over $k(y)$, which consists of monomials 
$m_1,\dots,m_d \in S_n$. Then they form a basis of the free
$A_y$-module $F'_n\otimes A_y$. It follows that there is an element $f\in A$
such that $y\in D(f)$ and the images of $m_1,\dots,m_d$ generate the free 
$A_f$-module $F'_n\otimes A_f$, and hence form a basis of this
module. Replacing $V$ by $D(f)$ one can achieve that $m_1,\dots,m_d$ are
monomials of $S_n$, such that $\pi(m_1),\dots,\pi(m_d)$ is a basis of the free
$A$-module $F'_n$. 

We now describe the homomorphism $\phi_\ell$. Let $L$ and $M$ be free 
$A$-modules of rank $p$, resp.~$d$, and $\alpha:L\to N$, $\beta:M\to N$ 
homomorphisms of $A$-modules.\\
Define $\gamma: L\oplus M \to N$ by $(x,y) \mapsto
\alpha(x)+\beta(y)$. Then 
\[
  \wedge \gamma:\bigwedge^{p+d} L\oplus M \rightarrow \bigwedge^{p+d} N
\]
operates by
\[
   \gamma(z_1\wedge \cdots \wedge z_{p+d}) = \gamma(z_1) \wedge \cdots
   \wedge \gamma(z_{p+d})\,.
\] 
Hence one has a homomorphism 
\[
\bigwedge^p L \otimes \bigwedge^d M \rightarrowtail
\bigoplus_{\substack{i+j}{=p+d}} \bigwedge^i L \otimes \bigwedge^j M
\stackrel{\sim}{\rightarrow} 
\bigwedge^{p+d} L\oplus M \rightarrow \bigwedge^{p+d} N
\]
which can be described by
\begin{multline*}
  x_1 \wedge \cdots \wedge x_p \otimes y_1 \wedge \cdots \wedge y_d \mapsto
  x_1 \wedge \cdots \wedge y_d \\
 \mapsto \gamma(x_1) \wedge \cdots \wedge
  \gamma(y_d) = \alpha(x_1) \wedge \cdots \wedge \alpha(x_p) \wedge
  \beta(y_1) \wedge \cdots \wedge \beta(y_d)\,.
\end{multline*}
Applying this to $\phi_\ell$ gives:
\[
\phi_\ell(x_1 \wedge \cdots \wedge y_d ) =  \ell x_1 \wedge \cdots
\wedge\ell  x_p \wedge \kappa(y_1) \wedge \cdots \wedge \kappa(y_d)\,.
\]

We now choose $V =\Spec A$ so small that $F_{n-1}$ has an $A$-basis
$n_1,\dots,n_p$ and we define $s:F'_n\to F_n$ by $s(\pi(m_i)):=\kappa(m_i)$,
which is possible as $\pi(m_i), 1\leq i\leq d$, is a basis of $F'_n$. It
follows that $\psi \circ s \circ \pi(m_i) = \psi \circ \kappa(m_i)= \pi(m_i)$, 
which means that $s$ is a section of $\psi$ over $A$. Hence $\{ \ell
n_1,\dots,\ell n_p, \kappa(m_1),\dots, \kappa(m_d)\}$ is an $A$-basis of
$F_n$ and 
\[
\phi_\ell(n_1 \wedge \cdots \wedge n_p \wedge m_1 \wedge \cdots \wedge m_d)
 =  \ell n_1 \wedge \cdots \wedge\ell  n_p \wedge \kappa(m_1) \wedge \cdots \wedge \kappa(m_d)\,.
\]
As this element is a basis of $\bigwedge^{p+d} F_n$, it follows that
$\phi_\ell$ is surjective in a neighborhood of $y\in U(\ell)$, hence is
surjective on $U(\ell)$. 

Now let $\ell_i$, $1\leq i \leq m$, be as in the Aux-Lemma~\ref{auxlem:3}.
We define a homomorphism
\begin{align*}
  \sigma:\bigoplus_1^m  \bigwedge^p \cF_{n-1} \otimes \bigwedge^d \cS_n &
\longrightarrow
  \bigwedge^{p+d} \cF_n \quad \text{by} \\
(x_1,\dots,x_m) & \mapsto \sum_1^m\phi_{\ell_i} (x_i), \quad \text{where }
 x_i \in \bigwedge^p \cF_{n-1} \otimes \bigwedge^d \cS_n\,.
\end{align*}
As $\phi_{\ell_i}$ is surjective on $U(\ell_i)$ and the $U(\ell_i)$ cover
$Y$, $\sigma$ is surjective. Now
\[
\bigoplus_1^m  \bigwedge^p \cF_{n-1} \otimes \bigwedge^d \cS_n \simeq 
\bigwedge^p \cF_{n-1} \otimes_Y \cE\,,
\]
where $\cE:=\bigoplus_1^m  \bigwedge^d \cS_n \simeq E \otimes_k \cO_Y$ and 
$E:= \bigoplus_1^m  \bigwedge^d S_n$. One obtains a surjective 
homomorphism of $\cO_Y$-modules 
\[
\bigwedge^p \cF_{n-1} \otimes_Y E \otimes \cO_Y \rightarrow 
\bigwedge^{p+d} \cF_n\,.
\]
If it is tensored with $\bigl(\bigwedge^p \cF_{n-1}\bigr)^{-1}$, one gets a
surjective homomorphism of $\cO_Y$-modules
\begin{equation}
  \label{eq:1.24}
  E \otimes \cO_Y \longrightarrow \bigl(\bigwedge^p
  \cF_{n-1}\bigr)^{-1}\otimes_Y \bigl(\bigwedge^{p+d} \cF_n\bigr)\,.
\end{equation}

\begin{lemma}
  \label{lem:1.2}
  $\cM^{-1}_{n-1} \otimes \cM_n$ is a globally generated line bundle on
  $Y$ if $n\geq a-2$. \hfill $\qed$
\end{lemma}

\subsection{Some proporties of determinants (after Fogarty and Mumford)}
\label{sec:1.5.3}

\subsubsection{}
\label{sec:1.5.3.1}

Let $A$ be a noetherian ring, $M$ an $A$-module of finite type.  $M$ has
finite $\Tor$-dimension, if there is a finite projective resolution of
$M$. With somewhat different terminology, this is denoted as $\proj \dim_A
M < \infty$. It is known that 
\[
  \proj \dim_A M = \min \Set{ n\in N | \Ext^{n+1}(M,N) =(0) \text{ for all }
    A\text{-modules } N }.
\]
\begin{remark} \label{rem:1.4}
  If $0\to M'\to M \to M''\to 0$ is an exact sequence of $A$-modules of
  finite type and any two of the modules have finite projective dimension,
  then the third module has finite projective dimension, too. This follows
  from the exact sequence
\[
  \cdots \to \Ext^i_A(M'',N) \to \Ext^i_A(M,N) \to \Ext^i_A(M',N) \to \cdots
\]
\end{remark}

\begin{remark}(see~\cite[p. 66]{F1})
 \label{rem:1.5}
If 
  $0 \to \cF_1 \stackrel{\phi_1}{\to} \cF_2 \stackrel{\phi_2}{\to} \cF_3
  \stackrel{\phi_3}{\to} \cF_4 \to 0$
is an exact sequence of coherent $\cO_X$-modules with finite projective
dimension, then
\[
 \Inv(\cF_1)\otimes \Inv(\cF_2)^{-1}\otimes \Inv(\cF_3)\otimes
 \Inv(\cF_4)^{-1} \stackrel{\sim}{\longrightarrow} \cO_X\,.
\]
To prove this, one splits the exact sequence into the exact sequences
\[
0 \longrightarrow \cF_1 \longrightarrow \cF_2 \longrightarrow \Im(\phi_1)
\longrightarrow 0
\]
and
\[
  0 \longrightarrow \Ker(\phi_3) \longrightarrow \cF_3 \longrightarrow \cF_4 \longrightarrow 0\,.
\]
According to Remark~\ref{rem:1.4}, all modules occurring in the exact
sequences have finite projective dimension and then formula (i) in \cite[p.
67]{F1} gives the above assertion.
\end{remark}

\begin{lemma} \label{lem:1.3}
If 
\[
  0 \to \cF_1 \to \cF_2 \to \cF_3 \to \cG \to 0
\]
is an exact sequence of coherent $\cO_X$-modules, $\cF_i$ locally free on
$X$ of rank $r_i$ and $\cG_x=(0)$ for all $x\in \Ass(\cO_X)$, then there is an
injective homomorphism
\[
 \sigma : \cO_X \longrightarrow \stackrel{r_1}{\bigwedge}\cF_1 \otimes 
            \bigl(\stackrel{r_2}{\bigwedge}\cF_2 \bigr)^{-1} \otimes
            \stackrel{r_3}{\bigwedge}\cF_3=:\cL
\]
which has the following property: If $y\in X$ is any point such that
$\cG_y=(0)$, then the section $s=\sigma(1)$ generates the fiber $\cL_y$.
\end{lemma}
\begin{proof}
  Let $x\in X$ be an associated point, i.e.~a point $x\in X$ such that
  $\depth(x)= \depth(\cO_{X,x}) =0$. By assumption $\cG_x=(0)$, and hence
$\Hom_{\cO_{X,x}}(k(x),\cG_x) =(0)$ and $x$ is not an associated point of
$\cG$. Using the terminology of Fogarty, $\cG$ is a torsion module and hence
defines a canonical injective homomorphism $\sigma : \cO_X \to \Inv(\cG)$
(see \cite[Theorem 2.2 and property 1, p. 69]{F1}). The same argumentation
as in Remark~\ref{rem:1.5} shows that  $\proj \dim \cG < \infty$ and, at the
same time, the properties (i) and (ii) in~\cite[p. 67]{F1} show that
$\Inv(\cG) \stackrel{\sim}{\longrightarrow} \cL$.

Let $s:=\sigma(1)\in \Gamma(X,\cL)$. Then $s_x\neq 0$ for all $x\in X$. Let
$U\subset X$ be an open and affine subset and $\phi:\cL|U\simeq\cO_U$ an
isomorphism of $\cO_U$-modules. Then $f:=\phi(s)$ is a non-zero divisor of
$\cO_U$. If $U_i$ is a covering of $X$ by open affine subsets and
$\phi_i:\cL|U_i\to U_i$ are isomorphisms, then the $f_i:=\phi_i(s)$ define an
effective Cartier divisor. It is denoted by $\Div(\cG)$ and one has that
\[
\supp(\Div(\cG)):= \Set{ x\in X | (f_i)_x \text{ is not a unit in } \cO_{X,x}}
\]
is contained in $\supp(\cG)$ (\cite[Sec. 5.3]{M1}).

Now one has the following simple fact:

\emph{
Let $A$ be a local ring, $\phi:L\simeq A$ an isomorphism of $A$-modules and
$s\in L$. Then $\phi(s)$ is not a unit in $A$ iff $s\cdot A \subsetneq L$.}

From this it follows that 
\[
  \supp(\Div(\cG))= \Set{ x\in X | s_x \text{ does not generate the fiber }
    \Div(\cG)_x }.
\]
As mentioned above, this set is contained in $\supp(\cG)$ and because of
$\Inv(\cG) \simeq \cL$ the assertion of the lemma follows.
\end{proof}

\subsection{Utilization of determinants}
\label{sec:1.5.4}

 We assume the case of curves as in Section~\ref{sec:1.5.1}. From
 Aux-Lemma~\ref{auxlem:1} and Lemma~\ref{lem:1.1} it follows that 
if $\ell, h
 \in P_1$ are sufficiently general linear forms one has a commutative diagram
 \begin{equation}
   \label{eq:ast} \tag{$*$}
   \begin{aligned}
     \xymatrix{
       & 0 \ar[d]  & 0 \ar[d]   & 0 \ar[d]   & \\
       0 \ar[r] & \cF_{n-1} \ar[r]^{\ell} \ar[d]^{h} & \cF_{n} \ar[d]^{h}
       \ar[r] & \cF_{n}/\ell\cF_{n-1} \ar[r] \ar[d]^{h} & 0 \\
       0 \ar[r] & \cF_n \ar[r]^{\ell} \ar[d] & \cF_{n+1} \ar[d]
       \ar[r] & \cF_{n+1}/\ell\cF_{n} \ar[r] \ar[d] & 0 \\
       0 \ar[r] & \cF_{n}/\cF_{n-1} \ar[r]^{\ell} \ar[d] &
       \cF_{n+1}/h\cF_{n} \ar[d] \ar[r]
       & \cF_{n+1}/\ell\cF_{n}+ h\cF_{n} \ar[r] \ar[d] & 0 \\
       & 0 & 0 & 0 & }
   \end{aligned}
 \end{equation}
with exact rows and columns, for all $n\geq a-2$. Besides this, we consider
finitely many given points $y_1,\dots,y_r\in Y$. At first, there are
Zariski-many $\ell\in P_1$ such that
\begin{equation}
  \label{eq:sharp}
  0 \longrightarrow  \cF_{n-1} \stackrel{\ell}{\longrightarrow}  
  \cF_{n} \longrightarrow \cF_{n}/\ell\cF_{n-1} \longrightarrow  0
\tag{$\#$}
\end{equation}
is an exact sequence and~\eqref{eq:sharp}$\otimes k(y_i)$ is an exact sequence, too, for
all $1\leq i \leq r$ and $n\geq a-2$. This goes as in the proof of
Aux-Lemma~\ref{auxlem:2}, because one can avoid not only the associated
prime ideals of $\cI\otimes A_i$ but also the associated prime ideals of
$\cI\otimes k(y_i)$. In any case, there are Zariski-many $\ell\in \NNT(\cF)$
such that $y_i\in U(\ell)$, $1\leq i\leq r$. From Aux-Lemma~\ref{auxlem:1}
we conclude that $\cF'_n\simeq \cF_n/ \ell \cF_{n-1}$, if $n\geq a-2$. From
Aux-Lemma~\ref{auxlem:3}, respectively from its proof, it follows that
$(\cF_n/ \ell \cF_{n-1})\otimes \cO_U$ is flat on $U = U(\ell)$, if $n\geq
a-2$.

Now fix such a linear form $\ell$. Then $\cF' \otimes \cO_U$ is flat over
$U$ with Hilbert polynomial $P'(n) =d$. The same argumentation as before,
with $\cF' \otimes \cO_U$ instead of $\cF$, shows that there are
Zariski-many $h\in \NNT(\cF' \otimes \cO_U)$ such that $y_i \in U(h)$. Hence
there are Zariski-many $h\in \NNT(\cF) \cap \NNT(\cF' \otimes \cO_U)$ such
that $y_i \in U(h)$, where $h$ operates on $\cF$ and on $\cF' \otimes
\cO_U$ by multiplication. From Aux-Lemma~\ref{auxlem:1} it follows that 
$\pi_* ((\cF/h\cF(-1))(n)) \simeq \cF_n/h\cF_{n-1}$ over $Y$, and 
\[
   \pi_* ((\cF'/h\cF'(-1))(n+1)) \simeq \cF'_n/h\cF'_{n-1} \simeq
   (\cF_{n+1}/\ell\cF_n)/h(\cF_{n}/\ell\cF_{n-1}) \simeq
   \cF_{n+1}/\ell\cF_n +h\cF_{n}
\]
over $U$, if $n\geq a-2$. It follows that $\cF_{n+1}/\ell\cF_n +h\cF_{n}$ is
flat over $V= U\cap U(h)$, if $n\geq a-2$. All in all we get:
\begin{auxlemma}
 \label{auxlem:4}
 Let be $y_i$, $1\leq i\leq r$, finitely many points in $Y$,
 not necessarily closed. Then there are Zariski-many linear forms $\ell,
 h\in P_1$ such that \eqref{eq:ast} is a commutative diagram with exact rows
 and columns, for all $n\geq a-2$ and the same is true for the diagrams~\eqref{eq:ast}$\otimes k(y_i)$, $1\leq i\leq r$. \hfill $\qed$
\end{auxlemma}

For simplification, we put $K=k(y_i)$. Then $\dim (\cF_n /\ell
\cF_{n-1})\otimes K = P(n) -P(n-1)=d$ for all $n\geq a-2$ (see the proof of
Aux-Lemma~\ref{auxlem:3}, e.g.), hence $\dim (\cF_{n+1}/\ell\cF_n +h\cF_{n})
\otimes K = 0$. It follows that $(\cF_{n+1}/\ell\cF_n +h\cF_{n}) \otimes
\cO_{Y,y_i}=(0)$ for all $n\geq a-2$ and $1\leq i\leq r$.  If one especially
chooses the associated points of $Y$ among the $y_i$'s, it follows that 
\begin{equation}
  \label{eq:1.25}
  \cG_{n+1}:=\cF_{n+1}/\ell\cF_n +h\cF_{n} \text{ is a torsion module, if }
 n\geq a-2\,.
\end{equation}

Now from the diagram~\eqref{eq:ast} one gets an exact sequence
\[
0 \longrightarrow \cF_{n-1} \stackrel{\alpha}{\longrightarrow} \cF_{n}\oplus \cF_{n}
\stackrel{\beta}{\longrightarrow}
\cF_{n+1}\stackrel{\gamma}{\longrightarrow} \cG_{n+1} \longrightarrow 0
\]
where $\alpha(f):=(\ell f,-hf)$, $\beta(f,g):=hf+\ell g$ and
$\gamma(f):=\bar{f}$. It remains exact, if it is tensored with $k(y_i)$. Now
if $y\in Y$ is any point, which we count  among the $y_i$'s, then $\cG_y
=(0)$ follows, i.e.~$y\centernot\in \supp(\cG)$. But then
Lemma~\ref{lem:1.3} says that there is an injective morphism $\sigma:
\cO_Y\to  \cM_{n-1} \otimes \cM^{-2}_n \otimes \cM_{n+1}$ such that
$s=\sigma(1)$ generates the fiber of this line bundle at the point $y$,
hence generates the line bundle itself in an open neighborhood of $y$. We
have proven:
\begin{lemma}
  \label{lem:1.4} 
 If $n\geq a-2$ then the line bundle $\cM_{n-1} \otimes \cM^{-2}_n \otimes
 \cM_{n+1}$ is globally generated. \mbox{} \hfill $\qed$
\end{lemma}

\subsection{Computation of intersection numbers}
\label{sec:1.5.5}

\subsubsection{}
\label{sec:1.5.5.1}
 We consider the universal curve $\CC$ over $\HH=\HH_Q$, i.e., one has the
 same situation  as in Section~\ref{sec:1.5.1.2} with $Y=\HH$.

 Let $C\stackrel{i}{\hookrightarrow} \HH$ be a closed curve, which is
 contained in the open set $U(t)$ (cf.~Section~\ref{sec:1.5.2}). Then one has
 a cartesian diagram
\[
 \xymatrix{X:=C \times \P^3\ar@{^{(}->}[r]^-j \ar[d]_p & \HH\times_k \P^3
   \ar[d]^\pi \\ 
    C \ar@{^{(}->}[r]^i         & \HH }   
\]
with $j=i\times \mathrm{id}$.  If $\cF$ denotes the structure sheaf of
$\CC$, then for $\sF:=j^*(\cF)$ the same statements as in~\ref{sec:1.5.1.2}
for $\sF$ are true. As 
\[
H^1(\P^3\otimes k(y), \cF(n)\otimes k(y)) =(0)\quad \text{ if } n\geq a-3,
\]
one has $p_*\sF(n) \stackrel{\sim}{\longrightarrow} i^*\pi_* \cF(n)$ for all
$n \geq a-3$ (see the formula~\eqref{eq:1.20} and~\cite[Cor. 1, p. 51]{M2}).
The intersection number $(\cM_n\cdot C)$ is equal to the power of $\nu$ in
the polynomial $\chi(M_n^{\otimes \nu})$, where
$M_n:=\bigwedge\limits^{P(n)} p_*(\sF(n))$, if $n \geq a-3$,
i.e.~$(\cM_n\cdot C)$ can be computed as the intersection number $(M_n \cdot
C)$.  As the sequence
\[
      0 \longrightarrow  \sF(n-1) \stackrel{t}{\longrightarrow}  \sF(n) 
\longrightarrow \sF'(n) \longrightarrow  0  
\]
is exact for all $n\in \N$ and as $R^1 p_*\sF(n-1)=(0)$ for all $n \geq
a-2$, one has the exact sequence 
\[
      0 \longrightarrow  p_*\sF(n-1) \longrightarrow  p_*\sF(n) 
\longrightarrow p_*\sF'(n) \longrightarrow  0  
\]
for $n \geq a-2$.  It follows that $M_n \simeq M_{n-1}\otimes M'_n$ where
$M'_n :=\bigwedge\limits^d p_*\sF(n-1)$ is line bundle as $\sF'_n$ is
locally free on $C$ (cf.~Aux-Lemma~\ref{auxlem:3}).  It follows that
\begin{equation}
  \label{eq:1.26}
  (M_{n-1}\cdot C) = (M_n\cdot C) - (M'_n\cdot C), \quad n\geq a-2  
\end{equation}
which also can be written as 
\begin{equation}
  \label{eq:1.27}
  (\cM_{n-1}\cdot C) = (\cM_n\cdot C) - (\cM'_n\cdot C), \quad n\geq a-2\,.
\end{equation}
Now $\sF'=\cO_{\P^2\times C} / \cI'$, where $\cI'$ is an ideal such that
$\cI'\otimes k(y)$ has the Hilbert polynomial $Q'(n) =
\tbinom{n-1+2}{2}+\tbinom{n-(a-1)+1}{1}$ for all $y\in C$. Hence $\cI'$ is
$(a-1)$-regular and $R^1 p_*\cI'(n)=(0)$ for all $n \geq
a-2$.
It follows that
\[
      0 \longrightarrow  p_*\cI'(n) \longrightarrow  S_n\otimes \cO_C 
\longrightarrow p_*\sF'(n) \longrightarrow  0
\]
is exact and $\cI'_n:= p_*\cI'(n) \subset S_n\otimes \cO_C$ is a subbundle of
rank $Q'(n)$, if $n \geq a-2$. It follows that 
\[
  \bigwedge\limits^{Q'(n)} \cI'_n \otimes M'_n \simeq \cO_C\,.
\]
$L'_n:=\bigwedge\limits^{Q'(n)} \cI'_n$ is a line bundle on $C$ and because
of~\eqref{eq:1.27} it follows that
\begin{equation}
  \label{eq:1.28}
  (\cM_{n-1}\cdot C) = (\cM_n\cdot C) + (L'_n\cdot C), \quad \text{ if } 
  n\geq a-2\,.  
\end{equation}

\subsubsection{}
\label{sec:1.5.5.2}

We now consider the curves $C \in \Set{C_2, D, E=C_0 }$, which all are
contained in $U(t)$. Then $C\simeq \Spec k[\alpha]\cup \{\infty\}$.\\

\textsc{Case 1:}
\[
C = C_2 = \Set{(x, y^{a-1}(\alpha y+z), y^{a-2}z^{b-a+1}(\alpha y+z)) |
  \alpha\in k}^- 
\]
In order to see that $\cI'_n\otimes k[\alpha] = (x, y^{a-2}(\alpha y+z))_n$ if
$n\geq a-2$, it suffices to note that the vector space on the right hand side
is isomorphic to $(x,y^{a-2}z)_n$, which vector space has the dimension
$Q'(n)$. 
It follows that 
\[
  \cI'_n\otimes k[\alpha] = xS_{n-1}\oplus y^{a-2}(\alpha y+z)\cdot
  k[y,z]_{n-a+1} \quad \text{ for all } n \geq a-2.
\]
If $n >a-2$, the map $\alpha\mapsto \cI'\otimes
k[\alpha]$ is injective and hence $(L'_n\cdot C_2) = -(n-a+2)$ follows in
this case (see~\cite[Bemerkung 3, p. 11]{T1}). If $n=a-2$, the argumentation
is as follows:
\[
      0 \longrightarrow \cI'_{a-2} \longrightarrow  S_{a-2} \otimes 
 \cO_{\P^1} \longrightarrow \cF'_{a-2} \longrightarrow  0  
\]
is exact on $\P^1$, as $\cI'$ is $(a-1)$-regular. Now $h^0(\cI'(a-2)\otimes
k(y)) = Q'(a-2) = \tbinom{a-1}{2}$ for all $y\in C$, because of the
$(a-1)$-regularity. It follows that $\cI'_{a-2}\otimes k(y) = x S_{a-3}
\otimes k(y)$ for all $y\in C$,
hence $\cI'_{a-2} = x S_{a-3}\otimes_k \cO_C$. But then from the exact
sequence it follows that $\cF'_{a-2}$ is also a constant sheaf on $C$, hence
\begin{equation}
  \label{eq:1.29}
   (\cM'_n \cdot C_2) = (n-a+2)\quad \text{ for all } n \geq a-2.   
\end{equation}

\textsc{Case 2:}
\[
C =   D  = \Set{ (x^2, xy, y^{a-1}, z^{b-2a+4}(y^{a-2}+\alpha xz^{a-3}))
   |\alpha\in k}^-
\]
One sees that 
\[
\cI'\otimes k[\alpha] = (x^2, xy, y^{a-1}, y^{a-2}+\alpha xz^{a-3}).
\]
To see that for all $n \geq a-2$ one has
\[
\cI_n'\otimes k[\alpha] = x^2S_{n-2} \oplus xy k[y,z]_{n-2} \oplus y^{a-1}
k[y,z]_{n-a+1} \oplus (y^{a-2}+\alpha xz^{a-3}) k[z]_{n-a+2}\,   
\]
we compute the dimension of the vector space on the right
as 
\[
   \tbinom{n-2+2}{2}+ (n-2+1) + (n-a+2) + 1 = Q'(n)
\]
for all $n \geq a-2$. As $\cI'$ is $(a-1)$-regular, the equality follows.

As $\alpha\mapsto \cI'_n\otimes
k[\alpha]$ is injective for all $n \geq a-2$ from~\cite[loc. cit.]{T1} it
follows that
\begin{equation}
  \label{eq:1.30}
   (\cM'_n \cdot D) = 1 \quad \text{for all }n \geq a-2\,. 
\end{equation}

\textsc{Case 3:}
\[
C = E = \Set{ (x^2, xy, xz, y^a, y^{a-1}z^{b-a+1},
      xt^{b-2}+\alpha y^{a-1}z^{b-a}) | \alpha \in k }^-
\]
One sees that $\cI'_n \otimes k[\alpha]$ is a constant sheaf and that 
\begin{equation}
  \label{eq:1.31}
  (\cM'_n\cdot E) = 0, \quad \text{ for all  } n\geq a-2\,.
\end{equation}

\textbf{Intersection numbers of $E$}\\
From~\eqref{eq:1.28} and~\eqref{eq:1.31} we get 
\[
    (\cM_n\cdot E) = 1, \quad \text{ if } n\geq a-3\,.
\]

\textbf{Intersection numbers of $C_1$}  \\
In Section~\ref{sec:1.2} we had obtained that $[D] = (a-2)[E] + [C_1]$. It
follows that $(\cM_n\cdot C_1) = (\cM_n\cdot D_2) - (a-2)$ if $n\geq a-3$,
hence
\begin{align*}
  (\cM_{n-1}\cdot C_1) &= (\cM_{n-1}\cdot D) - (a-2),\quad n\geq a-2 \\
                       &= (\cM_n\cdot D) + (L'_n\cdot D) -
                          (a-2)\quad \text{ (cf.~\eqref{eq:1.28})} \\
                       &= (\cM_n\cdot D) -1 -
                           (a-2)\quad \text{ (cf.~\eqref{eq:1.30})} \\
                       &= (\cM_n\cdot C_1) -1\quad n\geq a-2\,.
\end{align*}

As $(\cM_n\cdot C_1) = (n-b+1)$ for $n\geq b$ (see~\cite[p. 134]{T2}), we
get
\[
  (\cM_n\cdot C_1) = (n-b+1) \quad \text{ for all } n\geq a-3\,.
\]

\textbf{Intersection numbers of $C_2$} \\
From~\eqref{eq:1.27} and~\eqref{eq:1.29} we obtain 
\begin{align*}
  (\cM_{b-1}\cdot C_2) &= (\cM_{b}\cdot C_2) - (b-a+2) \\
  (\cM_{b-2}\cdot C_2) &= (\cM_{b-1}\cdot C_2) - (b-1-a+2) \\
 \cdots\cdots\cdots\cdots & \cdots\cdots\cdots\cdots\cdots\cdots\cdots\cdots\cdots\cdots\cdots\\
  (\cM_{a-3}\cdot C_2) &= (\cM_{a-2}\cdot C_2) - (a-2-a+2)
\end{align*}
Summing up gives
\begin{align*}
  (\cM_{a-3}\cdot C_2) & = (\cM_{b}\cdot C_2) - \sum_{i=1}^{b-a+2} i \\
                       & = \tbinom{b-a+2}{2} +(b-b+1) - \tbinom{b-a+3}{2} \\
                       & = -\left[\tbinom{b+1-a+2}{2}-\tbinom{b-a+2}{2}\right]
                       +1 \\
                       & = -(b+1-a+1) +1 = (a-b-1)\,.
\end{align*}

In the same way, using~\eqref{eq:1.27} and~\eqref{eq:1.29}:
\begin{align*}
  (\cM_{n-1}\cdot C_2) &= (\cM_{n}\cdot C_2) - (n-a+2) \\
  (\cM_{n-2}\cdot C_2) &= (\cM_{n-1}\cdot C_2) - (n-1-a+2) \\
 \cdots\cdots\cdots\cdots & \cdots\cdots\cdots\cdots\cdots\cdots\cdots\cdots\cdots\cdots\cdots\\
  (\cM_{a-3}\cdot C_2) &= (\cM_{a-2}\cdot C_2) - (a-2-a+2)
\end{align*}

Summation gives
\begin{align*}
  (\cM_{a-3}\cdot C_2) & = (\cM_{n}\cdot C_2) - \sum_{i=1}^{n-a+2} i, \quad
  \text{hence} \\
 (\cM_n\cdot C_2)      & = \tbinom{n-a+3}{2} +(a-b+1) \\
                       & = \tbinom{n+1-a+2}{2} +(a-b+1) \\
                       & = \tbinom{n-a+2}{2} + \tbinom{n+1-a+1}{1}+(a-b+1)
                       \\
                       & = \tbinom{n-a+2}{2} + (n-b+1)\,.
\end{align*}

All in all we have
\begin{lemma}
 \label{lem:1.5}
  For all $n\geq a-3$ one has 
\[
(\cM_n\cdot E) = 1, \quad (\cM_n\cdot C_1) = (n-b+1) \text{ and }
(\cM_n\cdot C_2) = \tbinom{n-a+2}{2} + (n-b+1)\,. 
\]
\qed
\end{lemma}

\subsection{Globally generated line bundles on $\HH$}
\label{sec:1.5.6}

By Lemma~\ref{lem:1.2} the line bundle $\cL:=\cM^{-1}_{a-3} \otimes \cM_{a-2}$
is g.g.~(globally generated). From Lemma~\ref{lem:1.5} we deduce that
$(\cL\cdot E)=0$, $(\cL\cdot C_1)=1$ and $(\cL\cdot C_2)=0$. It follows that
$\cL \equiv \cL_1$ in $\NS(\HH)$, where $\cL_1$ is the line bundle introduced
in Section~\ref{sec:1.4}.

By Lemma~\ref{lem:1.4} the line bundle $\cL =  \cM_{n-1} \otimes \cM^{-2}_n
\otimes \cM_{n+1}$ is g.g.\
for all $n\geq a-3$. Using the
formulas of Lemma~\ref{lem:1.5} we get 
\[
(\cL\cdot E)=0,\quad (\cL\cdot C_1)=0\quad \text{ and } \quad (\cL\cdot
C_2)=1\,,
\]
from which we deduce that $\cL \equiv \cL_2$ in $\NS(\HH)$, where $\cL_2$ 
is defined in Section~\ref{sec:1.4}.

Finally we compute 
\[
(\cM_{b-1}\cdot E) = 1,\quad (\cM_{b-1}\cdot C_1) = 0, \quad \text{ and }
\quad (\cM_{b-1}\cdot C_2) = \tbinom{b-a+1}{2}\,,
\]
 and we deduce from this that Theorem~\ref{thm:1.3} implies Theorem~\ref{thm:1.2}. Moreover, we get 
\begin{proposition}
  \label{prop:1.3}
 The (residue classes in $\NS(\HH)$ of the) line bundles $\cL_1, \cL_2$ and 
$\cL_0 \otimes \cL^\rho_2 $, $\rho:=\tbinom{b-a+1}{2}$ are globally
generated. \hfill $\qed$
\end{proposition}

Now according to Theorem~\ref{thm:1.3}(i) each line bundle $\cL$ on $\HH$
can be written (modulo $\Pic^0(\HH)$) in the form 
\[
\cL_0^{\nu_0} \otimes \cL_1^{\nu_1} \otimes \cL_2^{\nu_2} = 
\cL_0^{\nu_0} \otimes \cL_2^{\nu_0\rho}\otimes \cL_1^{\nu_1} \otimes
\cL_2^{\nu_2-\nu_0\rho}
\] with $\nu_i \in \Z$.

\begin{corollary}
  \label{cor:1.4}
 If all $\nu_i \in \N$ and $\nu_2\geq\nu_0 \rho$, then (the residue class in
 $\NS(\HH)$ of) $\cL_0^{\nu_0} \otimes \cL_1^{\nu_1} \otimes \cL_2^{\nu_2}$
 is globally generated. \hfill $\qed$
\end{corollary}

\begin{corollary}
  \label{cor:1.5}
 If $\HH = H_{3,0}$, then $\cL_0, \cL_1$ and $\cL_2$ are globally generated
 and $\Pic(\HH)$ is freely generated by $\cL_0, \cL_1, \cL_2$.
\end{corollary}
\begin{proof}
  If $d=3$, $g=0$ one has $Q(n) = \tbinom{n-1+3}{3} + \tbinom{n-4+2}{2} +
  \tbinom{n-4+1}{1}$ hence $\rho=0$ and $\cL_0, \cL_1, \cL_2$ are  g.g.\
by the Proposition. Moreover, in this case $H^1(\HH,\cO_\HH)
  =(0)$, hence $\Pic(\HH) = \NS(\HH) \cong \Z^3$ (see~\cite[last line on
  p. 137]{T2} and Theorem II
in Section~\ref{sec:1.1}).
\end{proof}


\chapter{Subcones of the cone of curves}
\label{cha:2}

The aim is the description of those curves, which lie on subcones of
$A^+_1(\HH)$. This is not possible for all subcones and ``description'' is
to be understood in a weak sense, only. In this chapter $P=k[x,y,z,t]$ and
$S=k[x,y,z]$.

\section{Limits of $1$--cycles}
\label{sec:2.1}

\subsection{Limits of points}
\label{sec:2.1.1}
 
We first introduce some notations: $\sigma(\lambda)$ (resp.~$\tau(\lambda)$)
denotes the $\G_m$-operation $x\mapsto x$, $y\mapsto y$, $z\mapsto z$,
$t\mapsto \lambda t$ (resp.~$x\mapsto x$, $y\mapsto y$, $z\mapsto \lambda z$,
$t\mapsto t$). $\delta^i_\alpha$, $1 \leq i \leq 6$, 
denotes the $\G_a$-operations, which are defined by the following matrices
\[
\delta^1_\alpha \leftrightarrow 
\begin{pmatrix}
        1&0&0&\alpha \\ 0&1&0&0\\ 0&0&1&0\\ 0&0&0&1
\end{pmatrix},
\quad
\delta^2_\alpha \leftrightarrow 
\begin{pmatrix}
        1&0&0&0\\ 0&1&0&\alpha \\ 0&0&1&0\\ 0&0&0&1
\end{pmatrix},
\quad
\delta^3_\alpha \leftrightarrow 
\begin{pmatrix}
        1&0&\alpha &0\\ 0&1&0&0\\ 0&0&1&0\\ 0&0&0&1
\end{pmatrix},
\]
\[
\delta^4_\alpha \leftrightarrow 
\begin{pmatrix}
        1&\alpha &0&0\\0&1&0&0\\0&0&1&0\\0&0&0&1
\end{pmatrix},
\quad
\delta^5_\alpha \leftrightarrow 
\begin{pmatrix}
        1&0&0&0\\0&1&\alpha &0\\0&0&1&0\\0&0&0&1
\end{pmatrix}, 
\quad
\delta^6_\alpha \leftrightarrow 
\begin{pmatrix}
        1&0&0&0\\0&1&0&0\\0&0&1&\alpha \\0&0&0&1
\end{pmatrix}.
\]
Let be $\xi\in U(t) \subset \HH$, and $\xi \leftrightarrow \cJ \in
\HH_\Phi(k)$, i.e. the Hilbert function of $\cJ$ is $\Phi$. Then
\[
\xi_0 = \lim_{\lambda \to 0} \sigma(\lambda) \xi \leftrightarrow \cJ_0 \in
G_\Phi(k)
 \]
(see Appendix~\ref{cha:G} for definition and notation). One can write
 $H^0(\cJ_0(n)) = \bigoplus^n_0 t^{n-i} U_i$, $U_i \subset S_i$ subspaces
 with $S_1 U_i \subset U_{i+1}$ for all $i$.  As $\xi\in U(t)$ one has
 $\sigma(\lambda) \xi \in U(t)$. Let $r:U(t) \rightarrow
 H^d:=\Hilb^d(\P^2_k)$ be the restriction morphism with respect to the
 variable $t$. Define $\xi':=r(\xi) \leftrightarrow \cI \in H^d(k)$.  Then
 $\xi' = r(\sigma(\lambda)\xi)$ for all $\lambda \in k^*$, hence $r(\xi_0) =
 \xi'$ and $U_n = H^0(\cI(n))$ for all $n\geq \reg(\cJ)$. By applying a
 suitable linear transformation $g \in \GL(4;k)$, which leaves $t$
 invariant, one can achieve that $r(g(\xi)) \in U(z) \subset H^d$. In the
 statement of the lemma (see below) it will become clear that one can assume
 without restriction $\cI \in U(z)$. Let $\cI^* \subset \cO_{\P^3}$ be the
 ideal, which is generated by $\cI$, i.e.~$H^0(\P^3,\cI^*(n)) = \bigoplus^n_0
 t^{n-i}H^0(\P^2,\cI(i))$. Because of $U_i \subset H^0(\cI(i))$, one has
 $\cJ_0 \subset \cI^*$. Put $\cI_0:= \lim_{\lambda \to 0}
 \tau(\lambda)\cI$. Then $\cI_0$ has the same Hilbert function $\phi$ as
 $\cI = \cJ'$ (see Appendix~\ref{cha:G}).

Let $\cL:=\lim_{\lambda \to 0}\tau(\lambda)\cJ_0$. Then $H^0(\cL(n)) = 
\bigoplus^n_0 t^{n-i}L_i$, where $L_i:= \lim_{\lambda \to
  0}\tau(\lambda)U_i$, and $\cL \subset \cI^*_0$ has finite colength.

$\cJ_0$ and $\cL$ are invariant under the subgroup $\Gamma = \Set{
  \left(\begin{smallmatrix} 1&0&0& *\\ 0&1&0&*\\ 0&0&1&*\\ 0&0&0&1
 \end{smallmatrix}\right)}$ of $U(4;k)$. As $U_i$ and $L_i$ have the same
dimension, $\cL\in G_\Phi(k)$ follows. As $\cI_0$ is invariant under the 
subgroup $\gamma = \Set{
  \left(\begin{smallmatrix} 1&0& *\\ 0&1&*\\ 0&0&1
 \end{smallmatrix}\right)}$ of $U(3;k)$, one has $\cI_0 \in G_\phi(k)$. (As
to the notation, see Appendix~\ref{cha:G}.)

Let $c$ be the colength of $\cL \subset \cI^*_0$ and $\fX =
\Quot^c(\cI^*_0)$. Define $\cL^1:=\lim\limits_{\alpha \to \infty}
\delta^3_\alpha(\cL)$ and $\cL^2:=\lim\limits_{\alpha \to \infty}
\delta^5_\alpha(\cL^1)$. As $\cI_0$ is fixed by $\gamma$, $\cL^1$ and
$\cL^2$ are in $\fX(k)$. Noted in a somewhat more explicit way, one has 
\[
H^0(\cL^1(n)):= \bigoplus^n_0 t^{n-i} L^1_i\,, \qquad \text{where} \quad
L^1_i:=\lim_{\alpha \to \infty} \delta^3_\alpha(L_i)\,.
\]
As $\dim L^1_i = \dim L_i$, one has $\cL^1 \in G_\Phi(k)$. In the same way 
\[
H^0(\cL^2(n)):= \bigoplus^n_0 t^{n-i} L^2_i\,, \qquad 
L^2_i:=\lim_{\alpha \to \infty} \delta^5_\alpha(L^1_i)\,.
\]
As $\delta^3_\alpha$ and $\delta^5_\alpha$ commute, $\cL^2$ is invariant
under $\Gamma$ and $\gamma$. 

Put $\cN:=\lim\limits_{\alpha \to \infty} \delta^4_\alpha(\cI_0)$. As $\cI_0
\in U(z)$ (loc.~cit.) we can write $H^0(\cI_0(n))= \bigoplus^n_0 z^{n-i}
V_i$, where $V_i \subset R_i$, $R = k[x,y]$, are subspaces such that $R_1 V_i
\subset V_{i+1}$ for all $i$.  It follows that $H^0(\cN(n))= \bigoplus^n_0
z^{n-i} W_i$, where $W_i:= \lim\limits_{\alpha \to \infty}
\delta^4_\alpha(V_i)$.  As $W_i \subset R_i$ is invariant under $U(2;k)$ and
$\Char(k) =0$, it follows that $W_i$ is $B(2;k)$-invariant, especially is
generated by monomials. As $\cI_0$ is fixed by $\gamma$ and $\gamma$ is
normalized by $\delta^4_\alpha$, the ideal $\delta^4_\alpha(\cI_0)$ is fixed
by $\gamma$ for all $\alpha\in k$, hence $\cN$ is fixed by $\gamma$. It follows
that $\cN$ is fixed by $B(3;k)$ and hence $\cN^*$ is fixed by $B(4;k)$. As
$\cL^2$ is fixed by $\gamma$, the ideal $\delta^4_\alpha(\cL^2)$ is fixed by
$\gamma$ for all $\alpha \in k$, hence $\cK := \lim\limits_{\alpha \to \infty}
\delta^4_\alpha(\cL^2)$ is fixed by $\gamma$ and $\delta^4_\alpha$, hence
fixed by $U(3;k)$. By construction, $\cK$ is invariant under
$\sigma(\lambda)$ and lies in $U(t)$, hence $\cK$ is $\Gamma$-invariant,
hence $\cK$ is $U(4;k)$-invariant. Moreover, by construction, the Hilbert
functions of $\cL$, $\cL^1$, $\cL^2$ and $\cK$ still are equal to $\Phi$.

Besides the $\G_m$-operations $\sigma$ and $\tau$ introduced above, we
consider the $\G_m$-operations 
\[
\sigma_1: x\mapsto \lambda x, y\mapsto y, z\mapsto z, t\mapsto t
\quad \text{and} \quad 
\sigma_2: x\mapsto x, y\mapsto \lambda y, z\mapsto z, t\mapsto t.
\]
We form
\[
\cK^1 = \lim_{\lambda \to 0} \sigma_1(\lambda)\cK, \quad \cK^2 =
\lim_{\lambda \to 0} \sigma_2(\lambda)\cK^1 \quad \text{and} \quad  \cM = \lim_{\lambda
  \to 0} \tau(\lambda)\cK^2\,.
\] 
By construction, $\cL$, $\cL^1$, $\cL^2$, $\cK$ are invariant under
$\sigma(\lambda)$ and $\cM$ is invariant under $T(4;k)$. As $\cK$ is
$U(4;k)$-invariant and $U(4;k)$ is normalized by $T(4;k)$, $\cK^1$, $\cK^2$,
$\cM$ are $U(4;k)$-invariant, hence $\cM$ is $B(4;k)$-invariant. Now by
construction, $\cK \subset \cN^*$, hence $\cK^1$, $\cK^2$,
$\cM$ are contained in $\cN^*$, too. And again by construction, all the
Hilbert functions of $\cL$, $\cL^1$, $\cL^2$, $\cK$, $\cK^1$,
$\cK^2$, $\cM$ are equal to the Hilbert function $\Phi$ of $\cJ$. Hence $\cM
\subset \cN^*$ again has the colength $c$ and $r(\cM) = \cN$. Finally the
Hilbert function $\phi$ of $\cJ' = \cI$ is equal to the Hilbert function of
$\cN$, as $\dim W_i = \dim V_i$.

\subsection{Limits of integral curves}
\label{sec:2.1.3n}

We first recall the construction in the proof of~\cite[Lemma 1, p. 6]{T1}.

Let $X/k$ be a projective scheme. If $\psi: \G_a \to \Aut_k(X)$ is a
homomorphism and $A$ is a $k$-algebra, then we denote the image of $\alpha
\in A$ in $\Aut_k(X \otimes A)$ by $\psi_\alpha$.

Let $C \subset X$ be an integral (i.e.~a closed, irreducible, reduced) curve
and $p$ the Hilbert polynomial of $C$ with regard to any closed embedding of
$X$ into a projective space. Then $\alpha \mapsto \psi_\alpha(C)$ defines a
morphism $\G_a \to \cH:=\Hilb^p(X)$, which has a unique extension $f:\P^1
\to \cH$. This gives a family $\cC/\P^1$ and a cartesian diagram
\[
      \xymatrix{
        \cC \ar[r] \ar[d]_\pi & X \times_k \cH \ar[d]^{p_2} \\
        \P^1 \ar[r]^f & \cH }
\]
such that $\pi$ is flat and surjective. As $C \otimes k[t]$ is irreducible
and reduced, the generic fibre $\psi_t(C)$, where $t \in K$ and $K$ denotes
the quotient field of $k[t]$, is reduced and irreducible. From general
properties of flat morphisms it follows that $\cC \subset X \times_k \P^1$
is reduced and irreducible, too.  One has $\cC_\alpha:=\pi^{-1}(\alpha) =
\psi_\alpha(C)$, if $\alpha \neq \infty$; $\cC_\infty:=\pi^{-1}(\infty) =:
C_\infty =: \lim_{\alpha\to \infty}\psi_\alpha(C)$. As $\cC$ has the
dimension two, it follows that $[C] = [C_\infty]$ in $A_1(X)$.

If $x \in C(k)$, then $(C,x) \in I:=\mathrm{Incidence}(\cH \times_k X)$ is
an integral curve in $X \times_k X$ with Hilbert polynomial $p$. $\G_a$
operates by $(C,x) \mapsto (\psi_\alpha(C),\psi_\alpha(x))$ on $I$. The
limit curve $\lim_{\alpha\to \infty} \psi_\alpha(C,x)$ is contained in $I$
and is equal to $(C_\infty,x_\infty)$, where $x_\infty:= \lim_{\alpha\to
  \infty}\psi_\alpha(x)$ is formed in $\Hilb^1(X) =X$ (see~\cite[Bemerkung
  3, p. 127]{T2}).

If one has a homomorphism $\sigma:\G_m \to \Aut_k(X)$, an analogous
construction gives two limit curves $C_{0/\infty} = \lim_{\lambda\to
  0/\infty}\sigma(\lambda) C \subset X$, such that $[C] = [C_0] =
[C_\infty]$.

\begin{conclusion}
  \label{concl:2.1n}
  The incidence ``curve, point'' is preserved, if one performs the limit
  curves under the $\G_a$- or $\G_m$-action. \hfill $\qed$
\end{conclusion}

\subsection{Limits of connected cycles}
\label{sec:2.1.4n}

If $C \in Z_1(X)$, then one can write $C = \sum m_i C_i$, where $C_i \subset
X$ are the prime components, i.e.~different integral curves in $X$, and $m_i
\in \Q$. We say that $C$ is a connected $1$-cycle, if $|C|:= \bigcup C_i$ is a
connected curve in $X$. If one puts $C_\infty := \sum m_i (C_i)_\infty$,
then $[C] = [C_\infty]$, and if $x$ is a closed point in $C_i \cap C_j$,
then Conlcusion~\ref{concl:2.1n} shows that $x_\infty \in (C_i)_\infty \cap
(C_j)_\infty$. Now the limits of the prime components are not necessarily
integral curves, but they still are connected curves by ``the principle
of connectedness'' (cf.~\cite[Ex. III 11.4]{H}).

\begin{conclusion}
  \label{concl:2.2n}
  The limit of a connected cycle under the $\G_a$- or $\G_m$-action is a
  connected cycle. \hfill $\qed$
\end{conclusion}

\subsection{}
\label{sec:2.1.2}

 Let $C\subset \HH$ be a connected curve and $\xi\in C(k)$ a point, which
 fulfills the above assumptions, i.e.~$\xi \in U(t)$ and $r(\xi)\in
 U(z)$.
We form the limit under the operations $\sigma$, $\tau$,
$\delta^3_\alpha$, $\delta^5_\alpha$, $\delta^4_\alpha$, $\sigma_1$,
$\sigma_2$, $\tau$, one after the other in this order, and arrive at the
limit $D$, 
where $D$ is a connected $1$-cycle in $\HH$, $[C] =[D]$, $\xi_* \in D(k)$ is
$B(4;k)$-invariant and the Hilbert functions of the ideals belonging to
$\xi$ and $\xi_*$ are equal.  Then one carries out with $D$ the
usual construction, i.e.\ one forms the limits under $\delta^i_\alpha$,
$1\leq i \leq 6$, and $\sigma$, $\tau$, $\sigma_1$, $\sigma_2$, and finally
arrives at a $B(4;k)$-invariant connected cycle $C_*$ such that $[C] =
[C_*]$ and $\xi_* \in C_*$.

\subsection{}
\label{sec:2.1.3}

\begin{lemma}
  \label{lem:2.1}
Let $C\subset \HH$ be a connected $1$-cycle and $\xi \in C(k)$. Then
there is a $B(4;k)$-invariant  connected $1$-cycle  $C_*\subset \HH$ and
a $B(4;k)$-invariant point $\xi_* \in C_*(k)$ such that:

\begin{enumerate}[(i)]
\item $[C] =[C_*]$
\item If $h(\xi)$ and $h(\xi_*)$ are the Hilbert functions of the ideals,
  which correspond to the points $\xi$ and $\xi_*$, then $h(\xi)
  =h(\xi_*)$.
\item $\xi$ and $\xi_*$ can be connected by a sequence of rational curves in
  $\HH_\Phi$, where $\Phi=h(\xi) = h(\xi_*)$.
\item There are Zariski-many $g\in \GL(4;k)$ such that the pair
  $(g(C),g(\xi))$ fulfills the statements (i)-(iii) and in addition one has
  $g(\xi) \in U(t)$ and $r(g(\xi))\in U(z)$. In this case the Hilbert
  functions of the ideals belonging to $r(\xi)$ and $r(\xi_*)$ are equal,
  too.
\end{enumerate}
\end{lemma}
\begin{proof}
  If $g\in \GL(4;k)$, then $[C]=[g(C)]$, $h(\xi) = h(g(\xi))$ and $\xi$ and $g(\xi)$ 
can be connected by a sequence of rational curves in
  $\HH_\Phi$. Then, for a general choice of $g$, one has $g(\xi)\in U(t)$ and
  $r(g(\xi))\in U(z)$. Then the assertions follow from what has been proved
  in~\ref{sec:2.1.1} --~\ref{sec:2.1.2}.
\end{proof}

We begin to determine the geometric properties of the subschemes $H_m$ and
$\cG$ of $\HH$ (see Appendix~\ref{cha:C} for definitions).

\section{Cycles without $C_0$-component}
\label{sec:2.2}

\subsection{Algebraic cycles}
\label{sec:2.2.1}

If $C$ is an algebraic cycle then $C=\overline{\G_m\cdot \xi}$, where
  $\xi\in \HH^\Delta(k)$. By Proposition~\ref{prop:1.2}, one can write 
$[C]= q_0[C_0] +q_1[C_1]$ with $q_i\in \N$.

\begin{auxlemma}
\label{auxlem:2.1}
    If $q_0=0$, then $C\subset H_m$.
  \end{auxlemma}
  \begin{proof}
    From the argumentation in Section~\ref{sec:1.3.2} it follows that
    without restriction $\rho_3=0$,
\[
(\cM_n \cdot C) = c + \gamma(n-b+1)\,,\quad  c = 
\frac{1}{\ell}
  \lambdadeg\bigl(\dot\bigwedge \sigma(\lambda)H^0(\cJ(b-1))\bigr) 
\]
and $\gamma\in \N$, $\gamma > 0$, where $\G_m$ operates by $\sigma(\lambda):
x\mapsto x$, $y\mapsto y$, $z\mapsto \lambda z$, $t\mapsto t$. By
assumption, $c=0$, hence $H^0(\cJ(b-1))$ is invariant under $\G_m$. If
$\reg(\cJ) <b$, then $\cJ$ would be $\G_m$-invariant and $C$ would not be a
curve. Hence $\reg(\cJ) =b$ and by the Corollary in Appendix~\ref{cha:C}
follows that $\xi \in H_m(k)$, hence $C \subset H_m$.
  \end{proof}

\subsection{Combinatorial cycles}
\label{sec:2.2.2}
\par
$1^\circ$ 
Let $C$ be a combinatorial cycle of type $1$.  Then
Conclusion~\ref{concl:1.1} in Section~\ref{sec:1.3.1} gives $[C]=q_1[C_1]+
q_0[C_0]$. 

\begin{auxlemma}
    \label{auxlem:2.2} 
    If $q_0=0$, then $C\subset H_m$.
  \end{auxlemma}
  \begin{proof}
    One has $C = \Set{ \psi^1_\alpha(\xi) | \alpha \in k}^-$, $\xi
    \leftrightarrow \cJ \subset{ \cO_{\P^3}}$ with Hilbert polynomial $Q$
    and of type $1$. Using the notations of~\ref{sec:1.3.1.1} we had got 
\[
  (\cM_n \cdot C) = \sum s_i (b-r_i) + \sum s_i (n-b+1)
\]
with $s_i >0$. From the assumption it follows that $b-r_i = 0$ for all
$i$. Hence there is at least one element $u \in H^0(\cJ(b))$ such that
$u\centernot\in P_1 \cdot H^0(\cJ(b-1))$ ( see Figure~\ref{fig:1.3}). It follows that $\reg(\cJ) =b$ is
maximal and hence $\xi \in H_m(k)$ ( see the Corollary in Appendix~\ref{cha:C}).
  \end{proof}

$2^\circ$ 
Now let $C$ be a combinatorial cycle of type $2$.  Then one can write
\[
[C]= q_0[C_0] +q_1[C_1]+ q_2[C_2]\,,
\]
where each $q_i$ is a natural number (Prop.~\ref{prop:1.1}).

  \begin{auxlemma}
    \label{auxlem:2.3} 
    If $q_0=0$, then $q_1=0$, too.
  \end{auxlemma}
  \begin{proof}
    This follows from Corollary~\ref{cor:1.2}.
  \end{proof}

$3^\circ$ 
Let $C$ be a combinatorial cycle of type $3$. 
Then one can write $[C]= q_0[C_0] +q_1[C_1]$, $q_i\in \N$ (Conclusion~\ref{concl:1.3}). 

  \begin{auxlemma}
    \label{auxlem:2.4} 
  If $q_0=0$, then $C\subset H_m$.
  \end{auxlemma}
  \begin{proof}
    Write $C = \Set{ \psi^3_\alpha(\xi) | \alpha \in k}^-$, $\xi
    \leftrightarrow \cJ$ of type $3$. In the proof of
    Conclusion~\ref{concl:1.3} we got $(\cM_n \cdot C) = c + D(n-b+1)$,
    where $D\in \N$ and $c = \alphadeg \dot\bigwedge
    \psi^3_\alpha (H^0(\cJ(b-1)))$.  By assumption $c=0$, hence
    $H^0(\cJ(b-1))$ is $B(4;k)$-invariant.  If $\reg(\cJ) <b$, then $\cJ$
    would be $B(4;k)$-invariant, too.  Hence $\reg(\cJ) =b$ and $\xi \in
    H_m(k)$.
  \end{proof}

  \begin{proposition}
\label{prop:2.1} 
    Let $C\subset \HH$ be a closed connected curve such that $[C]=
    q_1[C_1] + q_2[C_2]$, and $q_1\neq 0$. Then $C\subset H_m$.
\end{proposition}
\begin{proof}
  Clearly one can assume that $C$ is reduced and irreducible.  Applying a
  suitable linear transformation, one can further assume without restriction
  that $C \cap U(t) \neq \emptyset$. Suppose that $C \centernot\subset H_m$.
  If $U:=\HH - H_m$, then $C \cap U(t) \cap U$ is not empty and we take a
  closed point $\xi$ of this set.  By Lemma~\ref{lem:2.1} there is a
  connected curve $C_* \sim \sum^s_1 n_i D_i \sim C$ and a point $\xi_* \in
  C_*(k)$ such that each $D_i$ is a $B(4;k)$-invariant $1$-prime cycle,
  $n_i\in \N$ and $h(\xi) = h(\xi_*)$.  As $[D_i]$ has no $C_0$-component,
  from the Auxiliary Lemmas~\ref{auxlem:2.1}--\ref{auxlem:2.4} it follows
  that either $D_i\subset H_m$ or $D_i$ is a combinatorial cycle of type $2$
  and in this case $[D_i] = n [C_2]$.  Suppose that $D_i \subset H_m$ if $1
  \leq i \leq r$ and $D_i \centernot\subset H_m$ if $r+1 \leq i \leq s$. As
  $q_1 \neq 0$ by assumption, one has $r \geq 1$.  If $r+1 \leq i \leq s$
  and $D_i \centernot\subset H_m$, then $D_i \subset \HH - H_m$ as all
  points in $D_i$ have the same Hilbert function, because $D_i$ is a
  combinatorial cycle of type~$2$. It follows that $C_*$ is disjoint union
  of $D_1 \cup \cdots \cup D_r$ and $D_{r+1} \cup \cdots \cup D_s$. As $C_*$
  is connected, $D_i \subset H_m$ for all $i$ follows.  But then $h(\xi) =
  h(\xi_*)$ is maximal, hence $\xi \in H_m$ by the corollary in
  Appendix~\ref{cha:C}.
\end{proof}

\begin{corollary}
  \label{cor:2.1}
  If $C\subset \HH$ is a connected curve such that $[C]= q_1[C_1]$, then
  $C\subset H_m$. \hfill $\qed$
\end{corollary}

\section{Cycles without $C_0$ and $C_1$ component}
\label{sec:2.3}

\subsection{}
\label{sec:2.3.1}

\begin{auxlemma}
      \label{auxlem:2.5} 
 Suppose that $C$ is combinatorial cycle of type $2$ such that $[C] = n \cdot
 [C_2]$, $n$ a natural number. Then $C \subset \cG$. 
\end{auxlemma}
\begin{proof}
  We have to take up the notation and the argumentation as in
  Section~\ref{sec:1.3.1.2}, especially the proof of Corollary~\ref{cor:1.2}.
  As now $a\geq 4$ is supposed, only Case~1 can occur and one has
  $\tilde{g}(n) =0$. It follows that $\tilde{\cJ}$ is $B(4;k)$-invariant and
  $g(n) =r(n)$. Then~\eqref{eq:1.12} shows that $s>0$, as $g(n)$ is not
  constant. It follows that $q_1=0$, hence~\eqref{eq:1.4} gives $a=
  \alpha+1$. Now $\reg(\cJ') = \alpha$ (see~\cite[p. 55]{T1}). As
  $\cJ'\subset \cO_{\P^2}$ has the colength $a-1$, it follows that $\cJ'$
  has the maximal possible Hilbert function (cf.~Appendix~\ref{cha:C}). From
  the $G_2$-invariance of $\cJ$ we deduce that $\cJ' = (x,y^rz^s)$, hence
  $(\tilde{\cJ})' = \cO_{\P^2}$.

  Suppose that $\tilde{\cJ}\subsetneq \cO_{\P^3}$. Then we take away in the
  topmost layer of the outer shell of the pyramid $E(\cJ)$ the monomial $M$,
  which has maximal $z$-degree. Then we add a monomial $m \in E(x
  \cO_{\P^3}(-1))- E(x\tilde{\cJ}(-1))$ such that $E(x\tilde{\cJ}(-1)) \cup
  m$ generates a $B(4;k)$-invariant ideal $x\tilde{\cJ}_1(-1)$. Then $E(\cJ)
  - M \cup m$ generates an ideal $\cK$ of type~$2$ and 
\[
\alphadeg\bigl(\dot\bigwedge \psi^2_\alpha H^0(\cK(n))\bigr) = 
\alphadeg\bigl(\dot\bigwedge \psi^2_\alpha H^0(\cJ(n))\bigr) -\nu\,, 
\]
where $\nu$ is a positive natural number (e.g.~see~\cite[p.~13, last
line]{T1}).  If $D$ is the combinatorial cycles of type~$2$, which is defined
by $\cK$, then $(\cM_n \cdot D) = (\cM_n \cdot C) - \nu$, hence $[D] =
-\nu[C_0] + n[C_2]$ where $n\in \N$, contrary to Theorem~\ref{thm:1.2}.  It
follows that $\tilde{\cJ} = \cO_{\P^3}$.

From this we deduce that $\cJ = (x,\cL)$, where $\cL$ is an ideal on
$\Proj(k[y,z,t])=:\P$ with Hilbert polynomial
$\tbinom{n-a+2}{2}+\tbinom{n-b+1}{1}$.  It follows that $\cJ \in \cG (k)$,
whence $C \subset \cG$ (cf.~Appendix~\ref{cha:C}).
\end{proof}

\subsection{}
\label{sec:2.3.2}

Let $C\subset \HH$ be an irreducible curve such that $[C]= n\cdot [C_2]$,
$n\in \N$, $n>0$. Assume $C \cap U(t) \neq \emptyset$. Take $\xi\in C \cap
U(t)$. Then from Lemma~\ref{lem:2.1} we deduce that 
\[
   [C_*] = \sum n_i [D_i] = [C]\,,
\]
where the $D_i$ are the $B(4;k)$-invariant prime components of $C_*$.
Moreover, there is a $B(4;k)$-invariant point $\xi_* \in C_*$ such that
$h(\xi) = h(\xi_*)$. No $D_i$ can have a $C_0$- or $C_1$-component, hence
by~\ref{sec:2.3.1} we get $D_i \subset \cG$, hence $\xi_*\in \cG$. If $\cJ
\leftrightarrow \xi$, from $h(\xi) = h(\xi_*)$ it follows that
$h^0(\cJ(1))=1$, hence $\cJ$ contains a linear form $\ell \in P_1$ from
which we deduce that $\xi \in \cG(k)$. As $C \cap U(t)$ is dense in $C$ and
$\cG \subset \HH$ is closed, we get $C \subset \cG$.  If $C \cap U(t)=
\emptyset$, we replace $C$ by $g(C)$ such that $g(C) \cap U(t) \neq
\emptyset$, $g \in \GL(4;k)$.  Then $g(C) \subset \cG$, hence $C \subset
\cG$.
We get
\begin{lemma}
 \label{lem:2.2}
 If $C \subset \HH$ is connected and $[C]\in \Z\cdot[C_2]$, then $C\subset
 \cG$. \hfill $\qed$
\end{lemma}

\subsection{}
\label{sec:2.3.3}

Suppose $C\subset \HH$ is an irreducible curve such that $[C]= s [C_2]$,
$s\in \N$, $s>0$.  Then $C \subset \cG$ by Lemma~\ref{lem:2.2}. Now we
consider the morphism
\[
  p: \cG \longrightarrow \Hilb^c(\P^3), \quad c = b-a +1\,,
\]
defined by $(\ell, f\cdot \cK) \longmapsto (\ell, \cK)$ (see
Appendix~\ref{cha:C}). 
Now 
\[
   C_2 = \Set{ (x,y^{a-2}(\alpha y+z)(y,z^{b-a+1})) | \alpha \in k }^-
   \subset \cG
\]
and $p(C_2) = 1\ \mathrm{point}$. It follows that 
\[
  p_*([C]) = \deg (p| C)\bigl(\cO_{\P^3}(1) \cdot p(C)\bigr) = (0)\,,
\]
hence $p(C) =1\ \mathrm{point}$, too. 

Besides one has the morphism $\kappa: \cG \rightarrow \Hilb_q(\P^3)$, $q(n)
= \tbinom{n-1+3}{3}+\tbinom{n-d+2}{2}$, $d=a-1$, defined by $(\ell, f\cK)
\mapsto (\ell, f)$. As 
\[
 \kappa(C_2) = \Set{(x,y^{a-2}(\alpha y + z)) | \alpha \in k }^- \,,
\]
one sees that $\kappa| C_2 : C_2 \rightarrow \kappa(C_2)$ is an isomorphism.
If $\xi \in C(k)$ is any point, then we have shown that $\xi \leftrightarrow
(\ell, f(\xi) \cdot \cK)$, where $f(\xi) \in [P/\ell P(-1)]_d$ depends on
$\xi$, whereas $\ell$ and $\cK$ are independent of $\xi$.  But then it
follows that the restriction $\kappa |C$ is injective, too. It follows that
$[\kappa(C)] = s [\kappa(C_2)]$.

Choose any $g \in \GL(4;k)$ such that $g(\ell) = x$. Then $[\kappa(g(C))] = s
[\kappa(C_2)]$. Now $\kappa(C_2)$ is a curve in $\Proj
k[y,z,t]=:\P$ 
of degree $1$, hence $\kappa(g(C)) \subset \mathbf{\P}$
is a curve of degree $s$, and it follows that $\kappa(C) \subset
\Proj(P/\ell P(-1))$ is a curve of degree $s$, too.
We have proved one direction of the following:

\begin{proposition}
  \label{prop:2.2}
  If $C\subset \HH$ is an irreducible curve such that $[C]= s [C_2]$, then
  there is a linear form $\ell \in P_1$, an ideal $\cK \subset
  \cO_{\mathbf{\P}}$, $\P:= \Proj(P/\ell P(-1))$ of colength $c = b-a+1$,
  and a curve $\cC \subset \mathbf{\P}$ of degree $s$, such that
\[
  C = \Set{(\ell, f\cdot \cK) | \langle f \rangle \in \cC }\,.
\]
If conversely $C$ is defined in this way, then $[C] = s[C_2]$. 
\end{proposition}
\begin{proof}
  Without restriction $\ell = x$. Let $n\geq b$ be a fixed natural
  number. We consider the embedding 
\[
    j: \HH \rightarrow W:= \Grass_{Q(n)}(P_n)
\]
defined by $\cI \mapsto H^0(\cI(n))$, and the Pl\"ucker-embedding 
\[
   p: W \rightarrow \P = \P^N
\]
defined by $E \mapsto \bigwedge^{Q(n)} E$, where $E$ is a subbundle of rank
$Q(n)$.

Put $S = P/x P(-1)$. One has a closed embedding $i: V:= \P(S_d) \rightarrow
\HH$ defined by $\langle f \rangle \mapsto (x, f \cdot \cK)$. Then the
closed embedding
$j \circ i : V  \rightarrow W$
is defined by
\[
      \langle f \rangle  \mapsto x P_{n-1} \oplus f \cdot K_{n-d}\,,
\]
where $K_n:= H^0(\P^2,\cK(n))$, $\P^2 = \Proj(S)$. 

Let $\cI$ be the universal ideal sheaf on $\P^3 \times \HH$ and 
$\cL_n := \bigwedge^{Q(n)} (p_2)_* \cI(n)$.

Now $\cO_W(-1)= p^* (\cO_\P(-1))$ and $\cL_n = j^* \cO_W(-1)$. Let 
\[
\cF \stackrel{\sim}{\longrightarrow} \cO_V(-1) \subset S_1 \otimes \cO_V
\]
be the universal rank-$1$ subbundle. Then 
\[
  \cL_n | i(\cC) = \bigl(\bigwedge^{Q(n)}(x P_{n-1} \oplus i_*(\cF)
  \otimes_k K_{n-d} \bigr)| i(\cC) = 
\bigl(\bigwedge^N x P_{n-1} \otimes \bigwedge^{h(n)} i_*(\cF) \otimes_k 
K_{n-d} \bigr)| i(\cC)\,,
\]
where $N = \tbinom{n-1+3}{3}$ and $h(n) = \tbinom{n-a+2}{2} +
\tbinom{n-b+1}{1}= \dim K_{n-d}$. It follows that 
\[
  \cL_n | i(\cC) \simeq i_* \cF^{\otimes h(n)}| i(\cC)
\]
and, because of $i^*i_* \cF = \cF$, one obtains 
\[
 \bigl(i_* \cF^{\otimes h(n)} \cdot i(\cC) \bigr) = 
 \bigl(i^* i_* \cF^{\otimes h(n)} \cdot \cC \bigr) = h(n) (\cF \cdot \cC)\,.
\]
As $h(n) = \tbinom{n-a+2}{2} +\tbinom{n-b+1}{1} = (\cM_n\cdot C_2)$ it follows
that 
\[
 (\cM_n\cdot C) = - (\cL_n \cdot i(\cC)) = - (\cM_n\cdot C_2) \cdot (\cF \cdot
 \cC) 
  = (\cM_n\cdot C_2) \cdot ( \cO_V(1) \cdot \cC) = (\cM_n\cdot C_2)\cdot s\,.
\]
But then $[C] = s \cdot [C_2]$ follows.
\end{proof}

\section{Cycles without  $C_1$-component}
\label{sec:2.4}

\subsection{Notations}
\label{sec:2.4.1}
 
 In the following $C\subset \HH$ is an irreducible curve such that $C \sim
 q_0C_0 + q_2 C_2$ and $q_2\neq 0$. If $\xi\in C(k)$, there is an open,
 non-empty set $U\subset \GL(4,k)$ such that $g(\xi)$ is Borel-normed for
 all $g\in U$, i.e.~$g(\xi)\in W_H(b)$, where $b\in \HH(k)$ corresponds to a
 Borel-ideal $\cB$ (as for the notation, cf.~Appendix~\ref{cha:H}). If $\xi$
 and $C$ are replaced by $g(\xi)$ and $g(C)$, then $\xi$ and $g(\xi)$ have
 the same Hilbert function, which is denoted by $\Phi$, and $[C]=[g(C)]$. To
 simplify the notation, we write $X=\P^3_k$, $B=B(4;k)$. 

 From~\cite[Lemma 4]{G88} it follows that $b\leftrightarrow \cB$ has the
 Hilbert function $\Phi$, too. (According to Appendix~\ref{cha:H}, the
 Hilbert function $h(\xi)$ of a point $\xi\in \HH$ is defined as the Hilbert
 function of the ideal $\cJ \leftrightarrow \xi$.) Replacing $U$ by a
 smaller open subset $V$, we can assume that $g(\xi) \in U(t)$ and
 $[g(\xi)]' = r(g(\xi))\in U(z)$, and we write again $U$ instead of $V$. In
 order to simplify the notation we write again $\xi$ and $C$ instead of
 $g(\xi)$ and $g(C)$. If $f\in P_n$, and $f'\in S_n$ is the image under the
 canonical map $P \longrightarrow P/t P(-1)$, then $[\inn(f)]'= \inn(f')$,
 where $\inn(0)=0$ and the order of the monomials in $S$ is the induced 
order of the monomials in $P$ (cf.\ Appendix~\ref{cha:H}). It follows that
$\xi'\in W_H(b')$, where now $H = \Hilb^d(\P^2)$ and the prime as usual
denotes the image under the restriction map $r:U(t) \longrightarrow
\Hilb^d(\P^2)$, which is induced by $P \rightarrow S$. Hence $\xi'$ and $b'$
have the same Hilbert function, which is denoted by $\phi$.
\subsection{}
\label{sec:2.4.2}

\begin{auxlemma}
  \label{auxlem:2.6}
  The point $b'$ corresponds to $(x,y^d)$, hence $[g(\xi)]'$ has maximal
  Hilbert function for all $g\in U$.
\end{auxlemma}
\begin{proof}
  Let $C_*$ and $\xi_*$ be as in Lemma~\ref{lem:2.1}. $C_*$ is a connected
  union of curves of the form
\begin{figure}[H]
\centering
\begin{center}
  \begin{minipage}{15cm*\real{0.7}} 
    \begin{tikzpicture}[scale=0.7]
   \filldraw (4,3) circle (1mm);
   \draw (4,3) node[below=2pt] {$b_1$}; 
   \filldraw (11,3) circle (1mm);
   \draw[below] (11,3) node {$b_2$};

   \draw[\Red,thick] (4,3) to [out=110, in = 180] (4,6);
   \draw[\Red,thick] (4,6) to [out=0, in = 80] (4,3);
   \draw (5,6) node[right] {$D_1$};

   \draw[\Red,thick] (4,3) to [out=160, in = 210] (1.5,4);
   \draw[\Red,thick] (1.5,4) to [out=30, in = 130] (4,3);
   \draw (2,5) node[right] {$D_2$}; 
   \draw[dotted] (2,3) arc (220:300:2cm);

   \draw[\Red,thick] (4,3) to [bend left]  (11,3);
   \draw[] (7.5,4.5) node {$A$}; 

   \draw[\Red,thick] (4,3) to [bend right=10]  (11,3);

   \draw[\Red,thick] (11,3) to [out=100, in = 190] (12,5.5);
   \draw[\Red,thick] (12,5.5) to [out=10, in = 70] (11,3);
   \draw (13,6) node[] {$E_1$}; 

   \draw[\Red,thick] (11,3) to [out=20, in = 100] (13,4);
   \draw[\Red,thick] (13,4) to [out=-80, in = -10] (11,3);
   \draw (13,2.5) node[right] {$E_2$}; 

   \draw[dotted] (9,3) arc (220:300:2cm);
    \end{tikzpicture}
  \end{minipage}
\end{center}
\caption{}
 \label{fig:2.3}
 \end{figure}
where $D_i$, $E_i$ are combinatorial cycles, $A$ is an algebraic cycle, and
$b_1$ and $b_2$ are $B$-fixed points. From Conclusion~\ref{concl:1.1}
follows that only combinatorial cycles of type~$2$ or $3$ can occur. If
$D\in \Set{ D_i }$ is such a cycle of type $3$,
i.e.~$D=\Set{\psi^3_\alpha(\eta) }^-$, then by~\ref{sec:1.3.1}
Remark~\ref{rem:1.1} it follows that $\eta'$ is $B(3;k)$-invariant, hence
$\eta'= b_1'$. If $A=\Set{ \sigma(\lambda)\zeta}^-$ two cases can occur:
$\sigma(\lambda)$ operates by $x \mapsto x, \; y \mapsto y,
\; z \mapsto z, \; t \mapsto \lambda t$, if $\rho_3 \neq 0$ or
$\sigma(\lambda)$ operates by $x \mapsto x, \quad y \mapsto y,
\quad z \mapsto \lambda z, \quad t \mapsto t$, if $\rho_3 = 0$
(see~\ref{sec:1.3.2}). If $\rho_3 \neq 0$, then $\zeta'$ is invariant under
$T(3;k)$ and under $U(3;k)$, as $\zeta$ is invariant under $T(\rho)$ and
$U(4;k)$. If $\rho_3=0$, then by~\ref{sec:1.3.2} Remark~\ref{rem:1.2}
follows that $\zeta'$ is $B(3;k)$-invariant. In both cases we get $r(A)$ is
the point $\zeta'$, hence $\zeta' = b'_1 = b'_2$.  If $D\in \Set{ D_i }$,
$D=\Set{ \psi^2_\alpha(\eta) }^-$ is of type $2$, then again two cases are
possible: If $(\cM_n\cdot D)$ is constant, then $\eta'$ is invariant under
$\psi^2_\alpha$, hence $\eta'$ is $B(3;k)$-invariant and $\eta'=b'_1$. If
$(\cM_n\cdot D)$ is not constant, by Corollary~\ref{cor:1.3} in
Section~\ref{sec:1.3} $\eta'$ has maximal Hilbert function. As $b'_1 =
\lim_{\alpha\to \infty} \psi^2_\alpha (\eta')$, $b'_1$ has maximal Hilbert
function, too. Now $C_*$ is contained in the fixed point scheme $\HH^\Gamma
\subset U(t)$, hence each point of $r(C_*)$ has maximal Hilbert function and
from Lemma~\ref{lem:2.1} the statement of the Aux-Lemma~\ref{auxlem:2.6}
follows.
\end{proof}
\subsection{}
\label{sec:2.4.3}
 We take up the initial situation of~\ref{sec:2.4.2}.
 Then 
\[
g(\xi)' \leftrightarrow g(\cJ) + t\cO_X(-1)/t\cO_X(-1)
\]
 has maximal
 Hilbert function hence equals an ideal $(h,f) \subset P/tP(-1)$, where
 $h\in P_1/t\cdot k$, $f \in P_d/tP_{d-1}$ not divisible by $h$.

 If we apply $u=g^{-1}$ to this ideal, we obtain $ \ell = u(t)$, $\cJ\in
 U(\ell)$ and
\[
\cJ + \ell\cO_X(-1)/\ell\cO_X(-1) = \bigl(u(h),u(f)\bigr)
\subset P/\ell P(-1)
\]
has maximal Hilbert function.

\begin{conclusion}
   \label{concl:2.1}
   Let $Z\subset X = \P^3$ be the curve, which belongs to $\xi\in C(k)$. If $L =
   V(\ell)\subset X$ is a general hyperplane, then the ideal in $\cO_L$,
   which defines the subscheme $L\cap Z \subset L$, has maximal Hilbert
   function and all points of $Z \cap L$ lie on the same line. \hfill $\qed$
\end{conclusion}
 
We take $Z$ as in Conclusion~\ref{concl:2.1} and have $Z_\red = \bigcup Z_i$
the decomposition in irreducible components.  Assume $\dim Z_1 =1$ and $Z_1$
non-degenerate. Then from~\cite[Prop. 18.10]{Har} it follows that $Z_1 \cap
L$ is non-degenerate, contradiction. It follows that all $Z_i$ such that
$\dim Z_i =1$ are degenerate, i.e.~$Z_i\subset V(\ell_i)$, where $\ell_i \in
P_1$, hence $Z_i \cap L \subset V(\ell,\ell_i)$. Assume $\dim Z_1 = \dim Z_2
= 1$ and $\ell_1 \neq \ell_2$. Then 
\[
Z_1 \cap Z_2 \cap L \subset
V(\ell,\ell_1) \cap V(\ell,\ell_2) = \emptyset,
\]
if $\ell$ is general, contrary to Conclusion~\ref{concl:2.1}.

 \begin{conclusion}
   \label{concl:2.2}
  Let $Z$ be as in Conclusion~\ref{concl:2.1}. Each $1$-dimensional
  irreducible component $Z_i$ of $Z_\red$ has the form $V(\cP_i)$, where
  $\cP_i = (\ell,F_i) \subset \cO_X$ is a prime ideal, $\ell \in P_1 - (0)$
  is independent of $i$, and $F_i \in P/\ell P(-1)$ is an irreducible form
  of degree $d_i$. If one replaces $\xi \leftrightarrow Z$ by $g(\xi)
  \leftrightarrow g(Z)$, $g$ a suitable linear transformation of $X$, one
  can achieve that $\ell = z$ (independent of $i$). \hfill $\qed$
 \end{conclusion}

\subsection{Assumption $\mathbf{I}$}
\label{sec:2.4.4}

Let $Q\subset \cO_X$ be primary to the prime ideal $\cP = (z,F)$, $F\in
k[x,y,t]$ irreducible of degree $d$, and $e$ is the multiplicity of $Q$. If
$\ell$ is a linear form in $P$, the images under the canonical map 
$P \rightarrow P/\ell P(-1)=:S$ are denoted by ${}'$.

\begin{assumptionI}
\index{assumption! I}
  If $\ell$ is a general linear form, the image $Q'$ of $Q$ in $S$ contains
the variable $z$, i.e.~$z\in Q':=Q+\ell P(-1)/\ell P(-1)$.
\end{assumptionI}

There is a filtration $Q = Q_0 \subset \cdots \subset Q_r = P$ such that
$Q_i/Q_{i-1} \simeq (P/\cP_i) (-\ell_i)$ modulo $\equiv$.
Here $\equiv$
denotes ``equality of the components of sufficient high degree''. 
$Q_i$ is a graded $P$-module, $\cP_i \subset P$ a graded prime ideal,
$Q\subset \cP_i$ for all $i$ (\cite[Prop. 7.4, p. 50]{H}), and either $\cP_i
= \cP$ or $\cP_i$ is a maximal ideal in $\Proj(P)$. As the multiplicity is
$e$, one has $\cP_i= \cP$ for $e$ indices $i$. Let $\ell$ be general for
$Q$, e.g.\ $\ell = \alpha x + \beta y + \gamma z +t$. Applying the
transformation $x \mapsto x, \; y \mapsto y,
\; z \mapsto z, \; t \mapsto t -(\alpha x + \beta y + \gamma
z)$, one can assume without restriction $\ell = t$.  We consider the images
under the map $P \longrightarrow P/t(P(-1) = S = k[x,y,z]$ and we obtain a
filtration $Q' = Q'_0 \subset \cdots \subset Q'_r = S$ and surjective morphisms
(modulo $\equiv$) 
\[
(S/\cp_i)(-\ell_i) \longrightarrow Q'_i/Q'_{i-1},\quad
\cp_i:=\cP'_i\,.
\]
If $\cP_i = (\ell_1, \ell_2, \ell_3)$, $\ell_i \in P_1$ are linearly
independent, then $\cp_i = S_+$ and $Q'_i/Q'_{i-1} \equiv (0)$. Changing the
numeration, one can achieve that $Q' = Q'_0 \subset \cdots \subset Q'_s= S$,
$(S/\cp)(-\ell_i) \longrightarrow Q'_i/Q'_{i-1}$ surjective mod $\equiv$,
$\cp = \cP'$, hence $s\leq e$.  The Hilbert polynomial of $(P/\cP)$ has
the form $dn+\mathrm{const.}$, hence $\chi(P/Q) = de\cdot n +c$.  It follows
that
\[
  de = \chi(S/Q') = \sum^s_1 \chi(Q'_i/Q'_{i-1}) \leq  \sum^s_1 \chi(S/\cp)\,.
\]
Now $S/\cp = S/(z,f) = R/fR(-d)$, where $f = F'$ and $R= k[x,y]$. Hence
$\chi(S/\cp) =d$, $s=e$, $\chi(Q'_i/Q'_{i-1}) = \chi(S/\cp)$ and
$(S/\cp)(-\ell_i) \longrightarrow Q'_i/Q'_{i-1}$ is an isomorphism modulo
$\equiv$, $1 \leq i \leq e$.  As $z \in Q'$, one also has $z\in Q'_i$ and
putting $\cq = Q'/zS(-1)$ and $\cq_i = Q_i'/zS(-1)$ we get a filtration $\cq
= \cq_0 \subset \cdots \subset \cq_e = R$ and isomorphisms modulo $\equiv$
\[
  [R/fR(-d)](-\ell_i) \longrightarrow \cq_i/\cq_{i-1}\,,
\]
which are defined by multiplication with a form $g_i\in R_{\ell_i}$. We want
to show that then $\cq_i = f^{e-i}R(i-e)$ for all $0 \leq i \leq e$. This is
done by descending induction. The assertion is true if $i=e$, and we  assume
$\cq_i
\equiv f^{e-i} R(e-i)$. 
Changing the notation, we then have to show:

 \begin{auxlemma}
\label{auxlem:2.7}
   If $g \in R_\ell$, $I \subset f^n R(-dn)$ is a graded ideal and 
\[
  [R/fR(-d)](-\ell) \stackrel{\cdot g}{\longrightarrow}  f^n R(n-d)/I
\]
is an isomorphism mod $\equiv$, then $I \equiv f^{n+1} R(-(n+1)d)$.
 \end{auxlemma}
 \begin{proof}
   As $f^n R(-dn)\equiv g R(-\ell) +I$,  it follows that $gR_i \subset
   (f^n)$ if $i\gg 0$, hence $g = f^n \cdot h$, $h\in R_m$, $m:=\ell
   -dn$. As $I_i \subset (f^n)$, if $i\gg 0$, $I\equiv f^n J(-nd)$, $J
   \subset R$ a graded ideal. As $R(-dn) \equiv h R(-\ell) + J(-nd)$, we
   have $R\equiv h\cdot R(-m)+J$, hence $R\equiv hR(-m)+pR(-\mu)$, where
   $p\in R_\mu$. As by assumption
\[ 
[R/fR(-d)](-\ell) \xrightarrow{\cdot f^n \cdot h} f^n
   R(-nd)/f^npR(-\mu-nd) 
\]
is an isomorphism mod $\equiv$, this is also true for
$[R/fR(-d)](-m) \stackrel{\cdot h}{\longrightarrow}  R/pR(-\mu)$,
hence $h\cdot fR(-d-m) \subset p R(-\mu)$ mod $\equiv$.  As $(h,p) \equiv
R$, $h$ and $p$ have no common divisor, hence $f = pu$, $u \in R_{d-\mu}$.
Now $pR(-\mu-m)\cdot h \subset pR(-\mu)$, and as multiplication by $h$ is
injective, $pR(-\mu-m) \subset f R(-d-m)$, hence $f$ is a divisor of $p$ and
$f=pu$, $u\in k^*$.
 \end{proof}

 \begin{conclusion}
   \label{concl:2.3}
 Suppose $\cP = (z,F)$, $F\in k[x,y,t]_d$ irreducible, hence $\cP$ is a
 prime ideal. Let $Q$ be  primary  to $\cP$ with multiplicity $e\geq 1$. If
 for a general $\ell \in P_1$ the image $Q'$ of $Q$ in $P/\ell P(-1)$
 contains the variable $z$, then $Q' \equiv (z,f^e)$ where $f$ is the image
 of $F$ in $P/\ell P(-1)$.\hfill $\qed$
 \end{conclusion}
 
\subsection{Assumption $\mathbf{II}$}
\label{sec:2.4.5}

Let now be $C \subset \HH$ an \emph{irreducible} curve without
$C_1$-component, but with non-vanishing $C_2$-component. Let be $\xi \in
C(k)$, $\cI \leftrightarrow \xi$ and $Z\subset X$ the corresponding
subscheme. Then Conclusions~\ref{concl:2.2} and~\ref{concl:2.1} give:
\begin{enumerate}[(i)]
\item The $1$-dimensional irreducible components of $Z_\red$ are defined by
  prime ideals $\cP_i = (z,F_i)$, $F_i\in P/zP(-1)$ irreducible of degree
  $d_i>0$.
\item If $\ell \in P_1$ is general, $Y:=\Proj(P/\ell(P(-1))$, then
  $\cI':=\cI + \ell \cO_X(-1)/\ell \cO_X(-1) \subset \cO_Y$ has maximal
  Hilbert function, hence $\cI' = (h,g)$, where $h\in P/\ell P(-1)$ is a
  linear form and $g\in P/(\ell,h) P(-1)$ has degree $d$.
\end{enumerate}

We deduce that $Z\cap V(\ell) = V(\cI')$ is contained in the line $V(h) \cap
V(\ell)$. But as, with exception of finitely many isolated points, $Z$ is
contained in $V(z)$, $Z\cap V(\ell)$ is contained in the line $V(z) \cap
V(\ell)$, too.

\begin{assumptionII}
 \index{assumption! II}
  If $\ell$ is general, $Z\cap V(\ell)$ consists of more than one point.
\end{assumptionII}

 \begin{conclusion}
   \label{concl:2.4}
  If this assumption is fulfilled, then for general $\ell$ the lines
  $V(h,\ell)$ and $V(z,\ell)$ are equal, hence
  \begin{equation}
    \label{eq:2.1}
    \cI' = (z,g)\,.
  \end{equation}
\hfill $\qed$
\end{conclusion}

Now let $\cI = \bigcap\limits^r_1 Q_i \cap R$ be a reduced primary
decomposition, where $Q_i$ is primary to $\cP_i = (z, F_i)$ with
multiplicity $e_i$, $F_i\in k[x,y,t]$ irreducible of degree $d_i$ and $R$
the punctual part.  Then $d= d_1 e_1 + \cdots + d_r e_r$ is the degree of
$Z$. From~\eqref{eq:2.1} it follows that $z\in Q'_i$, hence each $Q'_i$
fulfills the Assumption I. Conclusion~\ref{concl:2.3} then gives
$Q'_i = (z,f_i^{e_i})$, $f_i$ the image of $F_i$ in $P/\ell P(-1)$, and we
note that for a general choice of $\ell$, no two of the $f_i$ have a common
divisor.

Put $\sR_i =(z,F_i^{e_i})$, $1\leq i \leq r$, \quad $\sR = \bigcap^r_1 \sR_i$,
\quad $\cL = \cI + \sR$.

Then for general $\ell$ one has: 
\[
\sR' \subset \bigcap^r_1 \sR'_i = \bigcap^r_1 (z,f_i^{e_i}) = (z,f) \,,
\]
where $f = f_1^{e_1} \cdots f_r^{e_r}$. It follows that $\bigcap^r_1 Q'_i =
(z,f) \subset \cO_{\P^2}$ has the colength  $d$, $\P^2 = V(\ell)$. On the
other hand, 
\[
\cI' = \bigl(\bigcap Q_i\bigr)' \subset \bigcap Q'_i
\]
 also 
has  colength  $d$ in  $\P^2$, hence 
\[
\cI' = \bigcap^r_1 Q'_i \supset \sR'\,,
\]
and we get $\cL' = \cI' + \sR' = \cI'$.

If $F:=F^{e_1}_1 \cdots F^{e_r}_r$, then $\sR = (z,F)$ has the Hilbert
polynomial $Q^*(n) = \tbinom{n-1+3}{3}+\tbinom{n-d+2}{2}$. As $\cL' = \cI'$,
the ideal $\cL$ also has the Hilbert polynomial $Q^*$, hence $\sR=\cL
\supset \cI$ and $\cI \subset \sR$ has the colength $c=b-a+1$.

In order to formulate a preliminary result, we have to introduce some
notations:

As always $Q(n) = \tbinom{n-1+3}{3}+\tbinom{n-a+3}{2}+\tbinom{n-b+1}{1}$, 
$Q^*(n) = \tbinom{n-1+3}{3}+\tbinom{n-d+2}{2}$, $d=a-1$. Let $\FF$ be the
Flag--Hilbert scheme
\[
 \FF = \Set{ (\cI,\cJ) \in \HH_Q \times \HH_{Q^*} | \cI \subset \cJ }\,.
\]
This is a closed subscheme of $\HH_Q \times \HH_{Q^*}$. Let $\pi:\FF
   \to \HH_Q$ be the projection and put $\cZ:=(\pi(\FF))_\red$. We have
   obtained so far:

 \begin{conclusion}
   \label{concl:2.5}
  If $\xi\in C(k)$ corresponds to the subscheme $Z\subset X$, and if $Z$
  fulfills the Assumption II, then $\xi\in \cZ$. \hfill $\qed$
 \end{conclusion}

 \subsection{}
 \label{sec:2.4.6}

 We now consider the case that Assumption II is \emph{not}
 fulfilled. Let $M_1$ be the set of points $\xi\in C(k)$ such that the curve
 $\xi \leftrightarrow C_\xi\subset X$ is completely degenerate. This means
 the following: $\xi \leftrightarrow \cI=Q\cap R$, $Q$ is primary to a prime
 ideal of the form $(F_1,F_2)$, where $F_1,F_2\in P_1$ are linearly
 independent, and $R$ is the punctual part. In other words, $C_\xi$ is a
 line of multiplicity $\geq 1$ and possibly some more points, isolated or
 not. If $\xi\in M_2:=C(k) - M_1$, the $C_\xi$ looks like Fig.~\ref{fig:2.1}
 or Fig.~\ref{fig:2.2} plus some points, which are irrelevant:

\begin{figure}[H]
\centering
  \begin{minipage}[b]{7cm*\real{0.7}} 
    \begin{tikzpicture}[scale=0.7]
      \draw[\Red,thick] (0,1) -- (0,5) (7,1) -- (7,5);
      \draw[\Red,thick, dotted] (4,1) -- (4,5);
    \end{tikzpicture}
\caption{}
 \label{fig:2.1}
  \end{minipage}
\qquad \qquad
  \begin{minipage}[b]{7cm*\real{0.7}} 
    \begin{tikzpicture}[scale=0.7]
 \draw[\Red,thick] (0,1) .. controls (1,4) .. (4,3) .. controls (5,2.66) .. 
  (6,5);
    \end{tikzpicture}
\caption{}
 \label{fig:2.2}
  \end{minipage}
\end{figure}

From Conclusion~\ref{concl:2.5} then $\xi\in \cZ$ follows.

Now a family of completely degenerate curves cannot have in its closure (in
$\HH$) a curve as in Fig.~\ref{fig:2.1}, hence $M_1$ is closed in $C(k)$,
hence $M_1=C(k)$ or $M_1$ is a finite set of points. In the last case it
follows that $C\subset\cZ$, as $\cZ$ is closed. It remains the case that
$M_1= C(k)$, i.e.~all $C_\xi$ are completely degenerate.

Let $\phi$ be the smallest Hilbert function of ideals $\cI_\xi
\leftrightarrow \xi \in C(k)$. Then $C\cap \HH_\phi$ is open and non-empty in
$C$. Applying a suitable linear transformation, which leaves $z$ invariant,
we can achieve, without restriction, that $\tilde{U}(t):=C\cap \HH_\phi \neq
\emptyset$.

Let $r: U(t) \rightarrow H^d = \Hilb^d(\P^2)$ be the restriction morphism
defined by $P\to P/tP(-1)=S$. If $r(\tilde{U}(t))$ would be a single point,
all $C_\xi$ would run through a fixed point on $V(t)$, if $\xi\in
\tilde{U}(t)$.

Now take a general $\ell$ and consider the restriction morphism 
\[
   r:U(\ell)\rightarrow H^d = \Hilb^d(\Proj P/\ell P(-1))
\]
which is defined by restriction 
modulo $\ell$. If, for general $\ell$, the set $\tilde{U}(\ell):=C \cap
\HH_\phi \cap U(\ell)$ would be mapped by $r$ to a point in $H^d$, all curves
$C_\xi$ would run to a fixed point in $V(\ell)$, if $\xi \in
\tilde{U}(\ell)$. But then $C_\xi$ would be the same line with the same
multiplicity (possibly with different scheme structure) for all $\xi \in
\tilde{U}(t) \cap \tilde{U}(\ell) \neq \emptyset$, hence the Hilbert-Chow
morphism would map $C$ to a single point, which is not possible, as $C$ has
a $C_2$-component (see Chapter~\ref{cha:3}, Section~\ref{sec:3.5.3}).

By applying once more again a suitable linear transformation if necessary,
we can assume without restriction, that the closure of $r(\tilde{U}(t))$
is a curve $C' \subset H^d = \Hilb^d(\P^2)$.

We now let $\G_m$ operate by $\sigma(\lambda): x \mapsto x, \; y \mapsto
y, \; z \mapsto z, \; t \mapsto \lambda t$.  Then let $\sigma$ be
the morphism $U(t) \cap \HH_\phi \rightarrow G_\phi$, which is defined by
\[
   \xi \mapsto\xi_0 :=\lim_{\lambda\to 0} \sigma(\lambda) \xi
\]
where $\HH_\phi$ is taken as a reduced subscheme of $\HH$
(see Appendix~\ref{cha:G}). One has a commutative diagram
\[
 \xymatrix{
   U(t) \cap \HH_\phi \ar[rr]^\sigma \ar[dr]_r & & G_\phi\ar[dl]^\rho \\
                                         & H^d &  
}   
\]
where $\rho$ is the restriction morphism defined by $t$. Then
$\sigma(\tilde{U}(t)) \subset G_\phi$ is an irreducible curve, its closure
in $G_\phi$ is denoted by $D$, hence $\rho(D)=C'$.

Let $L:=\lim\limits_{\lambda\to 0} \sigma(\lambda) C$. Then $L$ contains the
irreducible curve $D$ with multiplicity $\geq 1$. Because of $[C] = [L]$,
$L$ has no $C_1$-component, too, hence 
 $[D] = \mu[C_0] + \nu [C_2]$
and
\[
 (\cM_n \cdot D) - (\cM_{n-1} \cdot D)= \nu (n-d+1)\,.
\]
Let be $\cI$ the universal ideal sheaf on $\P^3 \times \HH$ and let $\cF =
\cO_{\P^3 \times \HH}/ \cI$. Let $\sJ:=\cI |\, \P^3 \times G_\phi$ and $\sG:=
\cF |\, \P^3 \times G_\phi$.  Then
\[
  0 \longrightarrow \sG(-1)\otimes k(y)
  \stackrel{t}{\longrightarrow} \sG \otimes k(y)
\longrightarrow \sG' \otimes k(y) \longrightarrow 0
\]
is exact for all $y\in G_\phi$, where $\sG' = (\cF/t\cF(-1)) |\, \P^3 \times
G_\phi$, hence $\sG'$ is flat over $G_\phi$.  If $\cM'_n$ is the tautological
line bundle on $H^d$, then $(\cM'_n \cdot C') = \nu (n-d+1)$, hence 
$[C'] = \nu \cdot [F]$, where 
\[
 F:= \Set{ (x,y^{d-1}(\alpha y+z)) | \alpha \in k}^-
\] 
(see Appendix~\ref{cha:D}).
 As in the case of $\P^3$ (see Prop.~\ref{prop:2.2}) it follows that $C' =
 \Set{(h,f) | f \in \cC}$ where $h$ is a fixed linear form in $S$ and $\cC$
 is a curve of degree $\nu$ in $\P(R_d)$, $R = P/(\ell,h)P(-1)$.  It follows
 that $\cI'_\xi = (h,f)$, $h$ independent of $\xi \in \tilde{U}(t)$, $f\in
 R_d$ and $V(\cI'_\xi)$ is a point with multiplicity $d$ on the line $V(h)$.
 It follows that $V(t,\cI_\xi) \subset V(h,t)$ for all $\xi \in
 \tilde{U}(t)$.

 If one replaces the operation $\sigma(\lambda)$ by the $\G_m$-operation
 $\tau(\lambda)$, which is defined by the projection onto the general plane
 $V(\ell)$ from a point $P_0\centernot\in V(\ell)$ (see
 Appendix~\ref{cha:A}), then the analogous argumentation shows:

 \begin{conclusion}
   \label{conc:2.6} 
  Let $\ell$ be a general linear form and $\cI'_\xi = \cI_\xi +
  \ell\cO_X(-1)/\ell\cO_X(-1)$. Then
  \begin{enumerate}[(i)]
  \item $\cI'_\xi = (h,f)$ and $V(\cI'_\xi) = V(\ell,\cI_\xi) \subset
    V(\ell,h)$ for all $\xi \in \tilde{U}(\ell)$, where the linear form
    $h\in P/tP(-1)$ does not depend on $\xi \in \tilde{U}(\ell)$ and $f \in
    P/(\ell,h)P(-1)$ has degree $d$.
  \item $\bigcup V(\cI'_\xi)$, where $\xi$ runs through $\tilde{U}(\ell)$, 
   is an infinite set.
\hfill $\qed$
  \end{enumerate}
\end{conclusion}

Now $C_\xi \leftrightarrow \cI_\xi$ is completely degenerate, i.e.~the
unique irreducible, reduced, $1$-dimen\-sio\-nal component of $(C_\xi)_\red$
is a line $L_\xi\subset \P^3$ and, according to Conclusion~\ref{concl:2.6}, 
$V(\ell)\cap L_\xi$ is contained in a line in $V(\ell)$ for almost all 
$\xi \in \tilde{U}(\ell)$. Varying $\ell$ one sees this is possible only if
almost all $L_\xi$ are contained in one and the same plane $E$ for almost
all $\xi\in C(k)$. By applying a suitable linear transformation one can
achieve that $E=V(z)$, hence 
\[
  V(\cI_\xi) \cap  V(\ell) \subset V(z) \cap V(\ell)
\]
for general $\ell$ and almost all $\xi\in C(k)$. By
Conclusion~\ref{concl:2.6} $V(\cI_\xi) \cap  V(\ell) \subset V(\ell,h)$ for
almost all $\xi\in C(k)$ and hence the set $V(\ell,z) \cap V(\ell,h)$
contains more than 1 point.
It follows that $V(\ell, z) = V(\ell,h)$ hence $\langle \ell, z\rangle = 
\langle \ell, h\rangle$ for general $\ell$.  But then one has $h=\alpha z,
\alpha\in k$, and we get the equation
 \begin{equation}
   \label{eq:2.2}
   \cI'_\xi = (z,f) 
 \end{equation}
which is valid for general $\ell$ and for all $\xi\in U'(\ell)$, where
$U'(\ell)$ is an open non-empty subset of $\tilde{U}(\ell)$, and the form 
 $f\in P/(\ell,z)P(-1)$ of degree $d$ possibly depends on $\xi$.

We can write $\cI_\xi = Q\cap R$, where $Q$ is primary to $\cP=(z,F)$, $F\in
k[x,y,t]$ a linear form, $R$ the punctual part, both depending on $\xi\in
U'(\ell)$. Now $\ell$ is general for $\cI_\xi$ if $\xi\in
\tilde{U}(\ell)$ by definition, hence $\ell\centernot\in \cP$ and $V(R) \cap
V(\ell) =\emptyset$. With equation~\eqref{eq:2.2} it follows that 
$\cI'_\xi = Q' = (z,f)$, hence Assumption~I is fulfilled and the
same reasoning as in Section~\ref{sec:2.4.5} shows that  $\xi\in \cZ$ for
almost all $\xi\in C(k)$, hence for all $\xi\in C(k)$.

\begin{proposition}
  \label{prop:2.3}
  If $C\subset \HH$ is an irreducible curve such that $[C]= q_0[C_0] + q_2
  [C_2]$ and $q_2\neq 0$, then $C\subset \cZ$. \hfill $\qed$
\end{proposition}


\chapter{Tautological morphisms}
\label{cha:3}

The headline means morphisms \index{tautological morphism}
from $\HH$ to a projective space, which are
defined by means of the tautological line bundles $\cM_n$.  Let $f$
(respectively $f_n$) be the morphism defined by the globally generated line
bundle $\cL_1 \otimes \cL_2 = \cM^{-1}_{d-1} \otimes \cM_d$ (resp.\
$\cM^{-1}_{n-1} \otimes \cM_n$, if $n \geq d$). In this chapter $P =
k[x,y,z,t]$ and $S = k[x,y,z]$ as usual.

\section{Connection with a general hyperplane section}
\label{sec:3.1}

If $\ell = \alpha x +\beta y +\gamma z +t$, $\alpha,\beta,\gamma \in k$, is
a linear form, $U = U(\ell)$ denotes the non-empty open subset of $y \in
\HH$, such that
\[
0 \longrightarrow \cF(-1) \otimes_\HH k(y) \stackrel{\cdot
  \ell}{\longrightarrow} \cF \otimes_\HH k(y) \longrightarrow \cF'
\otimes_\HH k(y) \longrightarrow 0
\]
is exact. For abbreviation, we temporarily write $X = \HH \times_k \P^3$.
Let $\cF$ be the quotient of $\cO_X$ by the universal ideal sheaf $\cI$ on
$X$, and $\cF':=\cO_X /\cI + \ell \cO_X(-1)$. Then
\[
 0 \longrightarrow  \cF(-1) \otimes \cO_U  \stackrel{\cdot
   \ell}{\longrightarrow} \cF \otimes \cO_U 
  \longrightarrow  \cF' \otimes \cO_U   \longrightarrow 0
\]
is exact and $\cF' \otimes \cO_U$ is flat over $U$ with constant Hilbert
polynomial $d$.

Now let $A$ be a noetherian $k$-algebra and $\cI_i \leftrightarrow \xi_i\in
\HH(\Spec A)$ two points, which are mapped by $f$ to the same point of
$\P^N_k(A)$, if $f:\HH \to \P^N$ is defined by $\cL_1 \otimes \cL_2$. The
assumption is that the morphisms $\xi_i: \Spec A \to \HH$ factorize through
$U(\ell)$. An equivalent condition is that $\cI_i \otimes k(y) \in U(\ell)$,
$i=1, 2$, for all closed points $y \in \Spec A$. It follows that the ideals
$\cI'_i := \cI_i + \ell \cO_{\P^3} (-1) \otimes A / \ell \cO_{\P^3}
(-1)\otimes A$ are elements of $\Hilb^d(\P^2)(A)$, $i=1, 2$, $\P^2 \simeq
\Proj(P/\ell P(-1))$.

\begin{lemma}
   \label{lem:3.1}
  Under these assumptions one has $\cI'_1 = \cI'_2$.
\end{lemma}
\begin{proof}
  $1^\circ$ We first recall the construction of the surjective homomorphism
\[
\bigoplus^m_1 \bigwedge^d S_n \otimes \cO_\HH \longrightarrow \cM^{-1}_{n-1}
\otimes \cM_n
\]
in Chapter~\ref{cha:1}, Section~\ref{sec:1.5.2}.  In order to simplify the
notations we write $S_n$, $P_n$ etc.~instead of $S_n \otimes A$, $P_n
\otimes A$ etc.  Then one has the following diagram:
 \begin{equation}
   \label{eq:3.1}
   \xymatrix{
            &                      &                     &
 S_n   \ar[dl]_\kappa \ar[d]^\pi &  \\
   0 \ar[r] &  F_{n-1} \ar[r]^\ell &  F_n \ar[r]^\psi &  F'_n \ar[r] &  0
} 
 \end{equation}
 where $n \geq d-1$ and $\kappa$ is the composition of the canonical
 homomorphism $S_n \hookrightarrow P_n \twoheadrightarrow F_n$. $F'_n$ is
 free over $A$ with basis $\pi(m_i)$, $1 \leq i \leq d$, $m_i \in S_n$
 monomials. If $s\pi(m_i):= \kappa (m_i)$, $1 \leq i \leq d$, one has $\psi
 s \pi (m_i) = \pi(m_i)$, $1 \leq i \leq d$, hence $\psi \circ s = \id$,
 i.e.~$s$ is a section and $F_n = \ell F_{n-1} \oplus s F'_{n}$. If $\mu$
 denotes the multiplication with $\ell$, then the following diagram
 \begin{equation}
   \label{eq:3.2}
   \begin{aligned}
     \xymatrix{ F_{n-1} \oplus S_n \ar[rr]^{\mu \oplus \kappa}
       \ar[drr]_{\mathrm{id} \oplus \pi}
       &  &  F_n = \ell F_{n-1} \oplus s F'_n \\
       & & F_{n-1} \oplus F'_n \ar[u]_{\mu \oplus s} }
   \end{aligned}
 \end{equation}
is \emph{not} commutative.  But the diagram
 \begin{equation}
   \label{eq:3.3}
   \begin{aligned}
     \xymatrix{ \bigwedge\limits^p F_{n-1} \otimes \bigwedge\limits^d S_n
       \ar[rr]^{\phi_\ell}  \ar[rrd] &  & \bigwedge\limits^{p+d} F_n = \bigwedge\limits^p \ell F_{n-1} \otimes \bigwedge\limits^d s F'_n \\
       & & \bigwedge\limits^p F_{n-1} \otimes \bigwedge\limits^d F'_n
       \ar[u]_{\simeq} }
   \end{aligned}
 \end{equation}
 where $\phi_\ell (x_1 \wedge \cdots \wedge x_p \otimes y_1 \wedge \cdots
 \wedge y_d) = \ell x_1 \wedge \cdots \wedge \ell x_p \wedge \kappa(y_1)
 \wedge \cdots \wedge \kappa(y_d)$, the diagonal arrow is the homomorphism 
$\id \otimes \bigwedge^d \pi$ and the vertical arrow is the isomorphism 
$\bigwedge^p \mu \otimes \bigwedge^d s$, is commutative again.  In order to
prove this statement, we take $y\in S_n (= S_n \otimes A)$ and deduce:
 $\psi [ (s \circ \pi)(y) - \kappa (y)] = \pi(y) -\pi(y) =0 \Rightarrow 
 (s \circ \pi)(y) - \kappa(y) = \mu (z)$, 
where $z \in F_{n-1}$ depends on $y$. From this we get:
\begin{align*}
  & \phi_\ell (x_1 \wedge \cdots \wedge x_p \otimes y_1 \wedge \cdots
  \wedge y_d) \\
  = {} & \ell x_1 \wedge \cdots \wedge \ell x_p \wedge \kappa(y_1)
  \wedge \cdots \wedge \kappa(y_d) \\
  = {} & \ell x_1 \wedge \cdots \wedge \ell x_p \wedge (s \pi(y_1) + \ell z_1)
  \wedge \cdots \wedge (s \pi(y_d)+ \ell z_d) \\
  = {} & \ell x_1 \wedge \cdots \wedge \ell x_p \wedge s \pi(y_1)
  \wedge \cdots \wedge s \pi(y_d) \\
  = {} & (\bigwedge^p \mu) (x_1 \wedge \cdots \wedge x_p) \wedge (\bigwedge^d
  s) (\pi(y_1) \wedge \cdots \wedge \pi(y_d)) \\
  = {} & \bigwedge^p \mu (x_1 \wedge \cdots \wedge x_p) \otimes 
  (\bigwedge^d s \circ \bigwedge^d  \pi) (y_1 \wedge \cdots \wedge y_d)\,,
\end{align*}
hence~\eqref{eq:3.3} is commutative.
If we tensorize~\eqref{eq:3.3} with $\bigl(\bigwedge^p F_{n-1}\bigr)^{-1}$,
we get the commutative diagram:
 \begin{equation}
   \label{eq:3.4}
   \begin{aligned}
     \xymatrix{ \bigwedge\limits^d S_n \ar@{->>}[rr] \ar[drr]
       &  & \bigl(\bigwedge\limits^p  F_{n-1}\bigr)^{-1} \otimes \bigl(\bigwedge\limits^{p+d} F_n \bigr) \\
       & & \bigwedge\limits^d F'_n \ar[u]_{\sigma} }
   \end{aligned}
 \end{equation}
Here the diagonal arrow is equal to $\bigwedge^d \pi$ and the vertical
arrow is an isomorphism. (The letter $\sigma$ has not the meaning as in
Section~\ref{sec:1.5.2} but is simply used as an abbreviation.)

$2^\circ$ We continue with some general considerations: Let be $E =
k^{n+1}$, $X/k$ a scheme, $\cL$ a line bundle on $X$, which is generated by
the global sections $s_i$, $0 \leq i \leq n$. These define an epimorphism
$E\otimes \cO_X \xrightarrow{(s)} \cL$, hence an element of $\P(X)$, i.e.~a
morphism $f:X \to \P:= \P^n_k$ such that $\cL \simeq f^*(\cO_\P(1))$. 

Now let $u_i$, $i=1, 2$, be two morphisms $Y \to X$ such that $f \circ u_1 =
f \circ u_2$. This is equivalent to the condition that $u^*_i(\cL) = \cL
\otimes _X u^*_i(\cO_X)=: \cA_i$, $i = 1, 2$, give the same element in
$\P(Y)$.  According to~\cite[Prop. 4.2.3]{EGA} this means that one has a
commutative diagram:
 \begin{equation}
   \label{eq:3.5}
   \begin{aligned}
     \xymatrix{
       & E \otimes \cO_Y \ar[dl] \ar[dr] &    \\
       \cL \otimes_X u_2^*(\cO_X) & & \cL \otimes_X u_1^*(\cO_X)
       \ar[ll]_\tau }
   \end{aligned}
 \end{equation}
where the diagonal arrows are the morphisms $(s) \otimes \cO_Y$ and $\tau$ is
an isomorphism  of $\cO_Y$-modules.

$3^\circ$ We apply this to $X =\HH$, $Y =\Spec(A)$.  Then $u_i: Y \to X$ is
defined by $\cI_i \leftrightarrow \xi_i \in \HH(A)$ and $\cL =
\cM^{-1}_{d-1} \otimes \cM_d = (\cM^{-1}_{d-2} \otimes \cM_{d-1}) \otimes
(\cM_{d-2} \otimes \cM^{-2}_{d-1} \otimes \cM_d)= \cL_1 \otimes \cL_2$.
Then the diagrams~\eqref{eq:3.1},~\eqref{eq:3.4} and~\eqref{eq:3.5} give the
diagram:
 \begin{equation}
   \label{eq:3.6}
   \begin{aligned}
     \xymatrix{
       & \bigwedge\limits^d S_n \otimes \cO_Y  \ar[dl] \ar[dr] &  \\
       \bigwedge\limits^d H^0(\cF'_2(d)) \ar[d]_{\sigma_2} & &
       \bigwedge\limits^d H^0(\cF'_1(d)) \ar[d]^{\sigma_1}
       \ar@{..>}[ll]_{\tau'}  \\
       \cA_2 & & \cA_1 \ar[ll]_\tau }
   \end{aligned}
 \end{equation}
 where the diagonal arrow, respectively the vertical arrow, is the map
 $\bigwedge^d \pi$, respectively the isomorphism $\sigma$ as in
 diagram~\eqref{eq:3.4}. Then $\tau'$ can be defined as an isomorphism of
 $\cO_Y$-modules such that the upper triangle is commutative. Then from
 $2^\circ$ it follows that $\bigwedge^d H^0(\cF'_i(d))$, $i = 1, 2$ define
 the same point in $\P(\bigwedge^d S_d) (Y)$. As the Pl\"ucker-morphism
 $\Grass^d(S_d) \to \P(\bigwedge^d S_d)$ is a closed immersion, it follows
 that $H^0(\cF'_1(d)) = H^0(\cF'_2(d))$. Now 
\[
0 \longrightarrow H^0(\cI'_i(d)) \longrightarrow S_d \otimes A
\longrightarrow H^0(\cF'_i(d)) \longrightarrow 0
\]
is an exact sequence, as $\cI'_i$ is $d$-regular. It follows that
$H^0(\cI'_1(d)) = H^0(\cI'_2(d))$ and then the $d$-regularity implies
$\cI'_1 = \cI'_2$.
\end{proof}

\begin{remark}
  \label{rem:3.1}
  In the diagram~\eqref{eq:3.6} one can replace $\bigwedge^d S_d \otimes
  \cO_Y$ by $\bigwedge^d S_n \otimes \cO_Y$, if $n \geq d-1$ is any integer.
  As $H^1(\cI'_i(n)) = (0)$ if $n \geq d-1$, it follows that $H^0(\cI'_1(n))
  = H^0(\cI'_2(n))$. If $n \geq d$, from $n$-regularity one deduces $\cI'_1
  = \cI'_2$, again.  But if $n=d-1$, this is not the case, in general.
\end{remark}

\begin{corollary}
  \label{cor:3.1}
If one supposes that $\cI_i \leftrightarrow \xi_i \in \HH(k)$ are mapped to
the same point by the tautological morphism $f$, then $\cI'_1 = \cI'_2$ for
Zariski-many linear forms $\ell \in P_1$.
\end{corollary}
\begin{proof}
  Let $\cL =\cM^{-1}_{d-1} \otimes \cM_d$.  We use the same notations as in
  Chapter~\ref{cha:1}, Section~\ref{sec:1.5.2}. There a surjective morphism
  $E \otimes \cO_\HH \to \cL$ had been constructed by means of an open
  covering $\HH = \bigcup^m_1 U(\ell_i)$.  Now we add $U(\ell)$ to this
  covering, where $\ell = \alpha x +\beta y +\gamma z +t$,
  $\alpha,\beta,\gamma \in k$, is a linear form. Then one has a surjective
  morphism $D \otimes \cO_\HH \to \cL$, where $D = E \oplus \bigwedge^d
  S_d$.  If $p:D\to E$ is the obvious projection, one gets a commutative
  diagram:
\[
   \xymatrix{
   D \otimes \cO_\HH \ar[dd]_p \ar@{->>}[dr] &    \\
                     & \cL \\
   E \otimes \cO_\HH \ar@{->>}[ur]& 
} 
\]
Then the natural mapping $i:\P(E) \to \P(D)$ defined by $p$ is a closed
immersion and one has a factorization:
\[
   \xymatrix{
 \HH \ar[r]^{f(D)} \ar[dr]_{f(E)} & \P(D) \\
                                  & \P(E) \ar[u]_i
} 
\]
Therefore $\cI_1$ and $\cI_2$ are mapped to the same point by $f(E)$ iff
they are mapped to the same point by $f(D)$. As the points $\xi_1$ and
$\xi_2$ are in $U(\ell)$ for Zariski-many $\ell$, the assertion follows from
Lemma~\ref{lem:3.1}.
\end{proof}
\section{The fibers of $f$}
\label{sec:3.2}

Let be $\xi \in \HH(k)$ and $F:=f^{-1} f(\xi)$ have the \emph{reduced} scheme
structure. Let be $\xi_1 \leftrightarrow \cI_1$ and $\xi_2 \leftrightarrow
\cI_2$ in $F(k)$ and put $\cI:= \cI_1 + \cI_2$. There are Zariski-many linear
forms $\ell$, such that $\ell$ is an $\NNT$ of $\cO_{\P^3}/\cI_1$,
$\cO_{\P^3}/\cI_2$ and $\cO_{\P^3}/\cI$.  From Corollary~\ref{cor:3.1} it
follows that $\cI'_1 = \cI'_2 = \cI'$ for Zariski-many $\ell$, where ${}'$
denotes restriction modulo $\ell$. Let be $\xi_3 \leftrightarrow \cI_3$ in
$F(k)$ and $\cI:= \cI_1 + \cI_2 + \cI_3$. In the same way it follows that
$\cI'= \cI'_1 = \cI'_2 = \cI'_3$. As the ascending chain of ideals $\cI_1
\subset \cI_1 + \cI_2 \subset \cI_1 + \cI_2 + \cI_3 \subset \cdots$ becomes
stationary, one deduces that there is an ideal $\cJ \subset \cO_{\P^3}$ with
the following property: If $\cI \leftrightarrow \xi \in F(k)$, then $\cI
\subset \cJ$ and $\cI' = \cJ'$ for Zariski-many linear forms. It follows
that $\cJ/\cI$ has constant Hilbert polynomial, independent of $\xi
\leftrightarrow \cI$. Moreover, one can assume without restriction that $\cJ$
has no embedded or isolated components, hence is a CM-ideal on $\P^3$. If one puts
$X = \P^3$, 
\begin{gather*}
  \cE:= \cJ / \cI, \quad \cF:=\cO_X/\cI, \quad \cG: = \cO_X/\cJ, \\
  \cE_n:= H^0(\cE(n)), \quad \cF_n:=H^0(\cF(n)), \text{ and } \quad
  \cG_n:=H^0(\cG(n))
\end{gather*}
then one gets the exact sequence
\[
 0 \longrightarrow \cE_n \longrightarrow \cF_n \longrightarrow \cG_n
 \longrightarrow 0
\]
for all $n\geq 0$, because the support of $\cE$ has the dimension $0$. We
get:
\[
\dot \bigwedge \cF_n \stackrel{\sim}{\longrightarrow}
\dot \bigwedge \cE_n \otimes \dot \bigwedge
\cG_n\,. 
\]
If  $\cI \leftrightarrow \xi \in \HH(k)$ then $\cM_n \otimes k(\xi) = \dot
\bigwedge \cF_n$.  As the CM-part $\cJ$ of $\cI$ is constant on $F(k)$,
$\cN_n:=\dot \bigwedge \cG_n$ is constant, too. It follows that 
\[
\bigl(\cM^{-1}_{n-1} \otimes \cM_n \bigr) \otimes k(\xi) 
\stackrel{\sim}{\longrightarrow}
\bigl(\dot \bigwedge \cE_{n-1} \bigr)^{-1} \otimes 
\bigl(\dot \bigwedge \cE_n \bigr) \otimes 
\cN^{-1}_{n-1} \otimes \cN_n\,.
\]

There is a filtration of $\cE$:
\[
 (0) = \cE^0 \subset \cE^1 \subset \cE^2 \subset \cdots \subset \cE^e  = \cE
\]
such that $\cE^i/\cE^{i-1} \simeq (\cO_X/ \cP_i) (-\ell_i)$, where the
isomorphism is defined by multiplication with a form $f_i \in P$ of degree
$\ell_i$ (see~\cite[Proposition 7.4, p. 50]{H}). 

Now $\cF$ is $(a-2)$-regular (Lemma~\ref{lem:1.1}), hence $\cG$ is
$(a-2)$-regular and $\dim_k \cF_n = P(n)$ and $\dim_k \cG_n = P(n)-e$, for
all $n\geq a-3$. We conclude that $\dim_k \cE_n = e$, if $n\geq a-3$, hence
$e = \sum^e_{i=1} h^0((\cO_X/ \cP_i) (n-\ell_i))$, if $n\geq a-3$.

From the exact sequences
\[
0 \longrightarrow \cE^{i-1}_n \longrightarrow \cE^i_n \longrightarrow
(\cE^i/\cE^{i-1})_n \longrightarrow 0
\]
it follows that $\det[(\cE^i/\cE^{i-1})_n] =
  \bigl(\det[\cE^{i-1}_n]\bigr)^{-1} \otimes \det[\cE^i_n]$ and
\[
 \det \cE_n = \bigotimes_{i=1}^e \det [ \cE^i/ \cE^{i-1}]_n 
= \bigotimes_1^e \det [ H^0((\cO_X/ \cP_i) (n-\ell_i))] \cdot f_i
\]
for all $n\geq a-3$.

\emph{Additional consideration:} Let $\fp \in \Proj(P \otimes A)$ be an ideal
such that $(P \otimes A/\fp)^\sim$ is flat over $A$ with Hilbert polynomial
equal to $1$. Then $\fp$ is generated by a subbundle $L \subset P_1 \otimes
A$ of rank $3$. By shrinking $\Spec A$, if necessary, we can suppose that 
$L \subset P_1 \otimes A$ is a direct summand of of rank $3$. Applying
a suitable $A$-linear transformation of $P \otimes A$, we can suppose that
$L=\langle x, y, z \rangle \otimes_k A$. We claim that for all $n \geq 1$
one has:
\[
\bigl(\det[P_{n-1}\otimes A / \fp_{n-1}]\bigr)^{-1} \otimes
\det[P_n\otimes A / \fp_n] \stackrel{\sim}{\longrightarrow}
\bigl(\det[P_n\otimes A / \fp_n]\bigr)^{-1} \otimes
\det[P_{n+1}\otimes A / \fp_{n+1}]
\]

One sees that this is equivalent to 
\[
[P_n\otimes A / \fp_n] \otimes [P_n\otimes A / \fp_n] 
\stackrel{\sim}{\longrightarrow}
[P_{n-1}\otimes A / \fp_{n-1}] \otimes [P_{n+1}\otimes A / \fp_{n+1}]
\]
or equivalent to 
$A t^n \otimes_A A t^n  \stackrel{\sim}{\longrightarrow}
A t^{n-1} \otimes_A A t^{n+1}$, which is true for all
$n \geq 1$. From this we conclude
\begin{align*}
      & \bigl(\cE_{n-1}\bigr)^{-1} \otimes \bigl( \det\cE_n \bigr) \\
 = {} & \bigl( \bigotimes\limits^e_1 
\det [ H^0((\cO_X/ \cP_i) (n-1-\ell_i))] \cdot f_i \bigr)^{-1} 
\bigotimes\limits^e_1 
\det [ H^0((\cO_X/ \cP_i) (n-\ell_i))] \cdot f_i \\
 = {} & \bigotimes^e_1 \det [ H^0((\cO_X/ \cP_i) (n-1-\ell_i))]^{-1} 
\otimes  \det [ H^0((\cO_X/ \cP_i) (n-\ell_i))] \\
 = {} & \bigotimes^e_1 \det [ H^0((\cO_X/ \cP_i) (0))]^{-1} 
\otimes \det [ H^0((\cO_X/ \cP_i) (1))] \\
 = {} &  \bigotimes^e_1(P_1 /L_i) \qquad \text{for all } n \geq a-3,
\end{align*}
where $L_i \subset P_1$ is the $3$-dimensional vector space, which generates
$\cP_i$. The prime ideals $\cP_i$ are uniquely determined by $\cE$ as the
associated primes, and the number of times which $\cP_i$ appears is equal to
the multiplicity of $\cE_{(\cP_i)}$ as an $\cO_{(\cP_i)}$-module
(see~\cite[loc. cit.]{H}). Following~\cite[p. 82]{F1} we denote the
$0$-cycle $\sum_1^e V(\cP_i) \in \Symm^e (\P^3)$ by $\langle \cE
\rangle$.
\begin{proposition}
  \label{prop:3.1}
 Let $\xi_i \leftrightarrow \cI_i$, $i = 1, 2$ be two closed
 points in $\HH$.  We write $\cI_i = \cJ_i \cap \cR_i$, where $\cJ_i$ is the
 CM-part and $\cR_i$ is the punctual part of $\cI_i$. Let $f$ be the
 tautological morphism defined by the globally generated line bundle $\cL_1
 \otimes \cL_2$ on $\HH$. $\xi_1$ and $\xi_2$ are mapped by $f$ to the same
 point iff $\cJ_1 = \cJ_2$ and $\langle \cJ_1 / \cI_1 \rangle = \langle
 \cJ_2 / \cI_2 \rangle$.
\end{proposition}
\begin{proof}
  Suppose $f(\xi_1) = f(\xi_2)$. Then $\cJ_1 = \cJ_2$ and $\langle \cJ_1 /
  \cI_1 \rangle = \langle \cJ_2 / \cI_2 \rangle$ follow from the forgoing
  discussion. Conversely, suppose $\cI_i = \cJ \cap \cR_i$
and $\langle \cJ / \cI_1 \rangle = \langle \cJ / \cI_2 \rangle$.
 From the exact sequences
\[
 0 \longrightarrow \cE^i \longrightarrow \cF^i \longrightarrow \cG \longrightarrow 0
\]
\[
  \cE^i:= \cJ / \cI_i, \quad \cF^i:=\cO_X/\cI_i, \quad i =1, 2, \quad \cG = 
\cO_X/\cJ
\]
one deduces in the same way as before that 
\[
\bigl( \det \cE^i_{n-1} \bigr)^{-1} \otimes \bigl( \det \cE^i_n \bigr)
\stackrel{\sim}{\longrightarrow} \bigotimes^e_1 (P_1/L_j)
\]
are equal for $i=1$ and $i=2$, which then implies $\cM^{-1}_{n-1} \otimes
\cM_n \otimes k(\xi_i)$ are equal for $i=1$ and $i=2$, and all $n\geq a-2$.
If $n=a-1=d$ one deduces that $f(\xi_1) = f(\xi_2)$.
\end{proof}

As already mentioned in Remark~\ref{rem:3.1},  in the case of the morphism
defined by the globally generated line bundle $\cL_1 = \cM^{-1}_{d-2}
\otimes \cM_{d-1}$, I could not find a similar description of the fibers.

\begin{corollary}
  \label{cor:3.2}
Let $f_n$ be the morphism $\HH \to \P^{N(n)}$ defined by the globally
generated line bundle $\cM^{-1}_{n-1} \otimes \cM_n$ for $n\geq d$. Then the
fibers of $f_n$, as sets of closed points, are independent of $n \geq d$.
\end{corollary}
\begin{proof}
  Replace $d$ by $n\geq d$ in the proof of Proposition~\ref{prop:3.1}.
\end{proof}
\section{Connectedness of the fibers of $f$}
\label{sec:3.3}

From Proposition~\ref{prop:3.1} it follows, with the method invented by
Fogarty (see~\cite[Section 2]{F2}), that the fibers of $f_n$ are connected
for $n \geq d$.  For later use we need a slightly more precise statement (see
below Lemma~\ref{lem:3.2}). The proof imitates Fogarty's method (probably in
a too complicated way\dots). For the sake of simplicity we write $f$ instead of $f_n$.
\subsection{}
\label{sec:3.3.1}

Let $U$ be a unipotent group, which acts on a projective space $\P =
\P^r_k$. Let $X \subset \P$ be a closed subscheme, invariant under $U$. Let
$\ell_1, \dots, \ell_d$ be different lines in $\P$, all contained in
$X$. Then $\ell_i = V(\cP_i)$, $\cP_i\subset S:=k[x_0,\dots,x_r]$ is a prime
ideal, which is generated by a linear subspace of dimension $r-1$ of
$S_1$. Let be $\cI:=\bigcap^d_1\cP_i$ and $Z \subset X$ the closed subscheme
defined by $\cI$. Let $h$ be the Hilbert polynomial of $Z$, i.e. the Hilbert
polynomial of $\cO_\P/\cI$, and put $\cZ:=\Hilb^h(X)$. 

\begin{auxlemma}
  \label{auxlem:3.1}
Suppose that $Z:=\bigcup^d_1 \ell_i$ is a connected curve in $X$, which
connects the two $U$-invariant points $x_1$ and $x_2$ in $X(k)$. Then there
is a connected curve $C = \bigcup^e_1 L_i \subset X$, which connects $x_1$
and $x_2$, such that each $L_i \subset X$ is a pointwise $U$-invariant line
and $e\leq d$.
\end{auxlemma}
\begin{proof}
  $U$ has a composition series with quotients isomorphic to $\G_a$, hence we
  may suppose $U = \G_a$ and $U$ operates via a homomorphism $\psi_\alpha: U
  \to \Aut(X)$. Let $z \in \cZ(k)$ be the point, which belongs to $Z$.  Then
  $z_0:=\lim_{\alpha \to \infty} \psi_\alpha(z) \in \cZ(k)$ corresponds to a
  $U$-invariant subscheme $Z_0 \subset X$ with Hilbert polynomial $h$. 
 We need an additional
\begin{auxlemma}
  \label{auxlem:3.2}
 The support of $Z_0$, i.e. the underlying set of closed points, consists of
 at most $d$ lines plus finitely many closed points.
\end{auxlemma}
\begin{proof}
  If $d=1$, then $h(n) = n+1$, and as a subscheme of $\P$, $Z_0$ also has the
  Hilbert polynomial $h$, hence $Z_0$ is a line.  Suppose the
  Aux-lemma~\ref{auxlem:3.2} is proved in the case of $d-1$ lines. We put
  $Y = \bigcup^{d-1}_1 \ell_i$ and denote by $g$ the Hilbert polynomial of $Y
  \subset X$. Put $\cY:=\Hilb^g(X)$ and $\FF:=\Set{ (Y,Z) \in \cY \times \cZ
    | Y \subset Z }$. $U$ operates on $\FF$, and if $Y \leftrightarrow y \in
  \cY(k)$, then $\lim_{\alpha \to \infty} \psi_\alpha(y,z) = (y_0,z_0)$ and
  $y_0$ corresponds to a $U$-invariant subscheme $Y_0 \subset Z_0$. Now
  $h(n) = dn +a$, $g(n) = (d-1) n +b$, $a, b \in \Z$. By induction
  hypothesis, $\supp(Y_0)$ consists of $e \leq d-1$ lines plus any suitable
  points. In other words, $Y_0$ is defined by an ideal $\cJ = \bigcap^e_1
  \fq_i \cap Q_1 \dots \cap Q_s \subset \cO_\P$, where $V(Q_i)$ is a closed
  point in $X$, $\fq_i \subset S$ is a $\fp_i$-primary ideal of
  multiplicity $e_i$, $V(\fp_i) \subset X$ is a line and $\sum^e_1 e_i =
  d-1$. Because of $Y_0 \subset Z_0$, either $\supp(Y_0) = \supp(Z_0)$, or
  $\supp(Z_0)$ contains a further irreducible component, which is a point
  or a line.  Hence the Aux-lemma~\ref{auxlem:3.2} is proved.
\end{proof}

We continue the proof of Aux-lemma~\ref{auxlem:3.1}. By assumption $Z$ is
connected, hence $Z_0$ is connected, too (see~\cite[Chap. III, Ex.
11.4]{H}). Clearly $x_1, x_2 \in Z_0$, and according to a theorem of
Fogarty~\cite[Prop. 2.1, p. 515]{F2}, the fixed point scheme $Z^G_0$ is a
connected curve, which contains $x_1$ and $x_2$.  Then the proof of
(loc.~cit.) implies Aux-lemma~\ref{auxlem:3.1}.
\end{proof}

\subsection{}
\label{sec:3.3.2}

Let $\Lambda$ be local Artinian $k$-algebra with maximal ideal $\fm$,
$\Lambda/\fm \simeq k$, $\fm^n \neq (0)$, but $\fm^{n+1} =(0)$. Let $E$ be a
finitely generated $\Lambda$-module, $\dim_k E = e$.  Then
$\Grass^c(E)$ represents the functor 
\[
\GG(A):= \Set{ V \subset E \otimes_k A \text{ is a submodule such that } E
  \otimes_k A/V \text{ is flat of rank $c$ over } A }\,.
\]
One also has $ \GG(A)= \Set{ V \subset E \otimes_k A \text{ is a subbundle
    of rank } d}$, where $c+d = e$ and $A$ is a $k$-algebra.

If $m\in\fm$, then multiplication with $1+m$ is a $k$-automorphism of $E$
(because of $(1-m)(1+m+ \cdots + m^n) = 1$), hence $U:=1+\fm$ operates as a
unipotent group on $E$ and $\GG$. If one puts
\[
 X(A) := \Set{ V \in \GG(A) | V \text{ is invariant under } U}
\]
then one gets a closed subscheme $X= \Quot^c(E)$ of $\GG$ (see~\cite[Prop.2.2, p.516]{F2}). If $\GG \to \P$ is the
Pl\"ucker-embedding, then $U$ operates in an equivariant manner on $\GG$ and
$\P$, and as a subscheme of $\P$, $X$ remains invariant under $U$.

Let $v_1,\dots,v_d$ be a basis of $V \in \GG(k)$. Let $u \in E - V$. Then
\[
\Set{ v_1 \wedge \dots \wedge v_{d-1} \wedge (\lambda v_d + \mu u) |
  (\lambda : \mu) \in \P^1 }
\]
is a line in $\P$, i.e.~$\Set{ \langle v_1, \dots, v_{d-1}, \lambda v_d +
  \mu u \rangle | (\lambda : \mu) \in \P^1 }$ is a line in $\GG$. It follows
that any two points in $\GG(k)$ can be connected by a chain of lines.  From
Aux-lemma~\ref{auxlem:3.1} follows:

\begin{auxlemma}
  \label{auxlem:3.3}
 Any two points $x_1, x_2 \in X(k)$ can be connected by a chain of lines in
 $X$. \hfill $\qed$
\end{auxlemma}

\subsection{}
\label{sec:3.3.3}

\begin{auxlemma}
  \label{auxlem:3.4}
Let be $X = \P^r$ and $\cM$ a coherent $\cO_X$-module. Let $\cN \subset \cM$
be a submodule of colength $c$, such that $\supp(\cM/\cN)$ consists of a
single closed point $p$. If $P$ is the corresponding prime ideal in
$S=k[x_0,\dots, x_r]$, then $P^c \cM \subset \cN \subset \cM$. 
\end{auxlemma}
\begin{proof}
 Put $M:=\bigoplus_{n\geq 0}H^0(X,\cM(n))$, $N:=\bigoplus_{n\geq
   0}H^0(X,\cN(n))$. One has a sequence of $S_{(P)}$-modules:
\[
 N_{(P)} \subset (N + P^c M)_{(P)} \subset \cdots \subset (N + P M)_{(P)}
 \subset M_{(P)}\,.
\]
If all the inclusions are strict, then one would get a sequence of strict
inclusions $\cN \subset  \cN + P^c \cM \subset \cdots \subset \cN + P
\cM  \subset \cM$ and the colength would be $\geq c+1$. It follows that
either $N_{(P)} = ( N + P^c M)_{(P)}$ or there is an index $0\leq i \leq
c-1$ such that $(N + P^{i+1} M)_{(P)} = ( N + P^i M)_{(P)}$. It follows that
either $P^c(M/N)_{(P)} = (0)$ or $P^{i+1}(M/N)_{(P)} = P^i(M/N)_{(P)}$. By
Nakayama it follows that $P^c_{(P)}(M/N)_{(P)} = (0)$ hence $P^c_{(P)} M
\subset N_{(P)}$. Thus there is a form $f\in S- P$ such that $f\cdot P^c M
\subset N$. The associated primes of $M/N$ are contained in $\supp(M/N) =
\Set{P}$, hence multiplication with $f$ is an injective mapping $M/N \to
M/N$. It follows that $P^c M \subset N$, hence $P^c \cM \subset \cN$.
\end{proof}
\subsection{}
\label{sec:3.3.4}
 We now can give a somewhat more geometric description of the fibers of $f$.

 \begin{lemma}
   \label{lem:3.2} 
   Two points $\xi$ and $\zeta\in \HH(k)$ lie in the same fiber of $f$ iff
   they can be joined by a connected curve $C$ in the fiber, such that $[C]
   = \nu \cdot [C_0]$ for a natural number $\nu$.
 \end{lemma}
 \begin{proof}
Suppose that $\cI \leftrightarrow \xi \in \HH(k)$ and $\cJ \leftrightarrow
\zeta \in \HH(k)$ lie in the same fiber. Then by Proposition~\ref{prop:3.1}
we can write $\cI = \cN \cap Q_1 \cap \dots \cap Q_r$, $\cJ = \cN \cap R_1
\cap \dots \cap R_r$, $\cN$ is the CM-part, $Q_i$ and $R_i$ both
$\cP_i$-primary, where $\cP_i$ corresponds to a closed point of $\P^3$ and
for all $i$ $\length(\cN/\cN \cap Q_i) = \length(\cN/\cN \cap R_i) =:c_i$. 
In the exact sequence 
\[
 0 \longrightarrow \cN/\cI \longrightarrow \cO_{\P^3}/\cI \longrightarrow
 \cO_{\P^3}/\cN \longrightarrow 0
\]
one has $\cN/\cI \simeq \bigoplus^r_1 \cN/\cN \cap Q_i$. If $P$ and $p$ is
the Hilbert polynomial of $\cO_{\P^3}/\cI$ respectively of $\cO_{\P^3}/\cN$,
then $P(n) = p(n) +s$, $s:= \sum^r_1 c_i$. Ditto with $\cJ$. To simplify the
notation, put $\cP_1 = \cP$ and $c_1 = c$. One sees that from
Aux-lemma~\ref{auxlem:3.4} it follows that $\cN \cdot \cP^c \subset \cN \cap
Q_1 \subset \cN$ and therefore $\cN/\cN \cap Q_1$ and $\cN/\cN \cap R_1$
correspond to closed points in $X:=\Quot^c(E)$, where $E:=\cN /\cP^c \cN$ is
a finitely generated module over the Artinian $k$-algebra $\Lambda =
\cO_{\P^3,\cP} /\cP^c \cO_{\P^3,\cP}$. Without restriction one can assume that
$\cP = (x, y, z)$.  Putting $U = D_+(t)$, one can write $\Lambda =
\cO_{U,\cP} /\cP^c \cO_{U,\cP} = k[X, Y, Z]/\fm^c$ where $X = x/t$, $Y =
y/t$, $Z = z/t$ and $\fm = (X, Y, Z)$. By Aux-lemma~\ref{auxlem:3.3}, the
two points $\cN/\cN \cap Q_1$ and $\cN/\cN \cap R_1$ can be connected by a
curve $T \subset X$. In other words: There is a coherent $\cO_{\P^3 \times
  T}$-module $\cL$, $\cP^c \cN \otimes \cO_T \subset \cL \subset \cN \otimes
\cO_T$ such that $\cN \otimes \cO_T/ \cL$ is flat over $T$ of rank $c$ and
there are $\tau_1, \tau_2 \in T(k)$ such that $\cL \otimes k(\tau_1) = \cN
\cap Q_1$ and $\cL \otimes k(\tau_2) = \cN \cap R_1$.

If one puts $\cK:= \cN \cap \cL \cap Q_2 \cap \cdots \cap Q_r$, then
\begin{equation}
  \label{eq:3.7}
  0 \longrightarrow \cN \otimes \cO_T/ \cK \longrightarrow \cO_{\P^3 \times
    T} / \cK \longrightarrow \cO_{\P^3 \times T} / \cN \otimes \cO_T  
\longrightarrow 0
\end{equation}
is exact and 
\[
  \cN \otimes \cO_T/ \cK = ( \cN \otimes \cO_T/\cL)
     \bigoplus^r_2 (\cN/\cN \cap Q_i) \otimes \cO_T  =:\cE\,.
\]
Let $\pi$ be the projection $\P^3 \times T \to T$. 
Applying $\pi_*$ to the last sequence gives an exact sequence again, hence
an exact sequence
\[
 0 \longrightarrow \cE_n \longrightarrow \cF_n \otimes \cO_T  \longrightarrow
 \cG_n \otimes \cO_T \longrightarrow 0
\] 
where $\cE_n:= \pi_*(\cE(n))$ is locally free of rank $s$, $\cF_n$ is the
universal locally free sheaf of rank $P(n)$ on $\HH$ and $\cG_n$ is the
$k$-vector space $P_n/H^0(\P^3,\cN(n))$ of rank $p(n)$, $n$ sufficiently
large. Hence 
\[
 \cM_n \otimes \cO_T \simeq \bigwedge^s \cE_n \otimes_T (\dot\bigwedge
 \cG_n) \otimes \cO_T
\]
where $\cM_n = \dot\bigwedge \cF_n$ is a tautological line bundle on $\HH$.

If $\ell\in k[x, y, z, t]_1 - \bigcup^r_1 \cP_i$ and $\mu$ is the
multiplication with $\ell$, then
\[
 0 \longrightarrow \cE(n-1) \stackrel{\mu}{\longrightarrow} \cE(n) \longrightarrow \cE'(n) \longrightarrow 0
\]
is an exact sequence, $\cE':=\cE / \ell \cE(-1)$. Tensoring with $k(\tau)$,
$\tau \in T$, gives an exact sequence again (because of $\Ass(\cN/\cL
\otimes k(\tau)) = \{\cP_1\}$ etc.).
Applying $\pi_*$ gives exact sequences on $T$ 
\[
 0 \longrightarrow \cE_{n-1}  \longrightarrow  \cE_n \longrightarrow 0\,,
\]
as $\cE_n$ is locally free of rank $s$ on $T$, for all $n$. Hence the
intersection number $\bigl(\dot\bigwedge \cE_n\cdot T\bigr)$ is independent
of $n$ and the same is true for $(\cM_n \otimes \cO_T \cdot T)$. Now the
sequence~\eqref{eq:3.7} shows that one can take $T$ as a curve in $\HH$ and
can write 
\[
  [T] = q_0 [C_0] + q_1 [C_1] + q_2 [C_2]\,.
\]
But  then $q_1 = q_2 = 0$.  This means, one has connected the point $\xi
\leftrightarrow \cN \cap Q_1 \cap \dots \cap Q_r$ with the point 
$\xi_1 \leftrightarrow \cN \cap R_1 \cap Q_2 \cap \dots \cap Q_r$
by a curve $T \sim q C_0$. In the same way one can connect $\xi_1$ with the
point $\cN \cap R_1 \cap R_2 \cap Q_3 \cap \dots \cap Q_r$, etc.
Conversely, suppose that $\xi$ and $\zeta\in \HH(k)$ can be connected by a
curve $C \subset \HH$ such that $C \sim q_0 C_0$. Then from $(\cM^{-1}_{n-1}
\otimes \cM_n \cdot C)=0$ it follows that $f(C) = 1 \text{ point }$.    
 \end{proof}

\section{The morphism $g$ defined by $\cM_{b-1}$}
\label{sec:3.4}

Let $Y/k$ be a scheme. $\pi:\P^3 \times Y \to Y$ the projection. If $\cI \in
\HH(Y)$, $\cF:=\cO_{\P^3\times Y} / \cI$ then $\cI_n = \pi_* \cI(n) \subset
P_n \otimes \cO_Y$ is a subbundle of rank $Q(n)$ for all $n \geq b-1$, and
if $\cF_n := \pi_* \cF(n)$, the sequence
\[
0 \longrightarrow \cI_n \longrightarrow P_n \otimes \cO_Y \longrightarrow \cF_n
\longrightarrow 0
\]
is exact for $n \geq b-1$ and thus $\bigwedge^{P(n)}\cF_n$ is a globally
generated line bundle, which is nothing else but the line bundle $\cM_n
\otimes_\HH \cO_Y$ if $n \geq b-1$. (The $b$-regularity of $\cI \otimes
k(y)$ for all $y\in Y$ implies that the formation of $\pi_* \cI(n)$ and
$\pi_* \cF(n)$ commutes with base change (see~\cite{G78} and~\cite[Lecture
14]{M2}).  This gives a morphism $\gamma:\HH \to V:= \Grass^p(P_{b-1})$
defined by $\cF \mapsto H^0(\cF(b-1))$, $p:=P(b-1)$. If $q:=Q(b-1) =
\binom{b-1+3}{3} -p$, then $V$ is isomorphic in a natural way to
$W:=\Grass_q(P_{b-1})$ and $\gamma$ can be identified with the morphism $\HH
\to W$ defined by $\cI \mapsto H^0(\cI(b-1))$. Composing these maps with the
Pl\"ucker-embedding $V\to \P^n$ (or $W\to \P^n$) defined by $L \mapsto
\dot\bigwedge L$, 
$n = \left(\binom{b-1+3}{3} \atop {p}\right) - 1 
= \left(\binom{b-1+3}{3} \atop {q}\right) - 1$, 
we obtain a morphism $g:\HH \to \P^n$. 
Now suppose $\xi_i \leftrightarrow \cI_i$, $i=1, 2$ are two elements in
$\HH(Y)$ such that $\reg(\cI_i \otimes k(y)) \leq b-1$, for all $y
\in Y$ and $i=1, 2$. If $g(\xi_1) = g(\xi_2)$ then
$\gamma(\xi_1) = \gamma(\xi_2)$ and thus $\pi_*(\cI_1(b-1)) =
\pi_*(\cI_2(b-1))$. From the $(b-1)$-regularity we conclude that $\cI_1 =
\cI_2$ (see~\cite[p. 99]{M2}). 

Let $U\subset \HH$ be the open subset consisting of ideals with regularity
$\leq b-1$.  Then $H_m:= \HH - U$ has a \emph{natural} structure as a smooth
subscheme of $\HH$ (see Appendix~\ref{cha:C}). If $Q(n)= \binom{n-1+3}{3} +
\binom{n-a+2}{2} + \binom{n-b+1}{1}$ (as always), then $H_m(k)$ consists of
the ideals of the form $(\ell, f(h,g))$, $\ell \in P_1-(0)$, $f \in [P/\ell
P(-1)]_{a-1}$, $h \in P_1/\ell \cdot k -(0)$, $g \in [P/(\ell,h) \cdot
P(-1)]_{b-a+1}$.

Suppose that $(\ell_i, f_i (h_i, g_i)) \leftrightarrow \xi_i \in H_m(k)$,
$i=0, 1$ have the same image in $W$ under $\gamma$, hence $\ell_i P_{b-2} +
f_i h_i \cdot k$ are equal subspaces in $P_{b-1}$ for $i=0, 1$. It follows
that they generate the same ideal in $P$, i.e.~one has $(\ell_0, f_0 h_0) = 
(\ell_1, f_1 h_1)$.
From this it follows that we can assume $\ell_0 = \ell_1 =:\ell$ and $f_0
h_0 = f_1 h_1$ in $P/\ell P(-1)$. Two cases can occur:
\begin{enumerate}[(i)]
\item $h_1 \in h_0 \cdot k$, hence $f_1 \in f_0 \cdot k$.
\item $h_1 \centernot\in h_0 \cdot k$.  Then $h_1$ divides $f_0$ and $f_1 =
  h_0 \cdot (f_0/h_1)$.
\end{enumerate}
If $h_1,\dots, h_r$ are the essentially different linear forms in $P/\ell
P(-1)$, which divide $f_0$, then define $f_i:=h_0 \cdot (f_0/h_i)$ and put
$L_i:=[P/(\ell, h_i) P(-1)]_{b-a+1}$.  Then  $\gamma$ maps $W_i:= \Set{
  (\ell, f_i (h_i, g)) | g \in L_i } \subset H_m$ to $\gamma(\xi_0)$. As $W_i
\simeq \P(L_i) \simeq \P^{b-a}_k$, we get 

\begin{proposition}
  \label{prop:3.2}
Let $g$ be the morphism $\HH \to \P^n$ defined by $\cM_{b-1} = \cL_0 \otimes
\cL^\rho_2$. Then one has:
\begin{enumerate}[(i)]
\item $g|\HH -H_m$ is an isomorphism.
\item If $(\ell, f(h,g)) \leftrightarrow \xi \in H_m(k)$ and
  $F:=g^{-1}g(\xi)$, then $F(k)$ is a disjoint union of $r$ projective
  spaces $\P^{b-a}_k$, where $r$ is the number of essentially different
  linear forms in $P/\ell P(-1)$, which divide $f \in [P/\ell P(-1)]_d$.
\hfill $\qed$   
\end{enumerate}
\end{proposition}
\section{Connection with the results of Fogarty}
\label{sec:3.5}

In~\cite{F1} Fogarty constructed morphisms $\omega^P_t(m): \Hilb^P(\P^N_k)
\to \P^n_k$, where $m\gg 0$ is a natural number and $n$ depends on $m$, and
he gave a description of the fibers (loc.~cit.~Theorem 10.4., p.~84).
\subsection{}
\label{sec:3.5.1}
If one chooses $N=3$, $t=1$, $P(n) = dn-g +1$ in (loc.~cit.), then one sees
that the fibers of $\omega^P_1(m)$ coincide with the fibers of $f_n$, at
least as sets of points. From $f_n(C_0) = \Set{ 1 \text{ point}}$ it follows
that $\omega^P_1(m)(C_0) = \Set{1 \text{ point}}$. If then $\cL_{1,m}$ is
the line bundle belonging to $\omega^P_1(m)$ (loc.~cit.~p.~88), from
$(\cL_{1,m}\cdot C_0) = 0$ it follows that $\cL_{1,m} = \cL_1^{\nu_1}
\otimes \cL_2^{\nu_2} \otimes L$, where $L \in \Pic^0(\HH)$ and $\nu_1$ and
$\nu_2$ are natural numbers depending on $m$. But I cannot describe this
dependence more concretely.

\subsection{}
\label{sec:3.5.2}
If $N=3$, $t=2$, $P(n) = dn-g +1$, then  $\omega^P_2(m)$ is the
Hilbert-Chow morphism (loc.~cit.~p.~84). If $U = U(4;k) \subset G:=\GL(4;k)$
is the subgroup of all upper unitriangular matrices, than any integer
closed curve in $X = \P^3_k$, which is invariant under $U$, is equal to the
line $\ell = V(x,y)$, hence the fixed point scheme $\HH^U$ is mapped by 
 $\omega^P_2(m)$ to a single point. If $\cL_{2,m}$ is the line bundle
 belonging to  $\omega^P_2(m)$, then $(\cL_{2,m}\cdot C_0) = 0$ follows. Now
 the $1$-cycle $D= \Set{ (x^2,xy,y^{a-1},
   z^{b-2a+4}(y^{a-2}+\alpha xz^{a-3})) | \alpha \in k }^-$ is contained in 
$\HH^U$ and $[D] = (d-1) [C_0]+[C_1]$ (eq.~\eqref{eq:1.1} in Chapter 1). It
follows that $(\cL_{2,m}\cdot C_1) = 0$, hence $\cL_{2,m} = \cL^\nu_2
\otimes L$ where $\nu > 0$ and $L \in \Pic^0(\HH)$.  Certainly $\cL_{2,m}$
has to be equal to $\cL_2$, but I cannot prove this in a simple way.
\subsection{}
\label{sec:3.5.3}
It is for this reason that I have to use the morphism $\Phi$, which was
constructed by Mumford in~\cite[Section 5.4]{M1}.

Let be $\xi\in \HH(k)$, $\xi \leftrightarrow C$ the corresponding closed
curve in $X = \P^3_k$.  The cycle $\langle C \rangle$ of $C$ is defined as  
\[
\langle C \rangle = \sum \nu_i (C_i)_\red
\]
where the $C_i$ are the $1$-dimensional, irreducible components of $C$ and
$\nu_i$ their multiplicities.  The Hilbert-Chow morphism \index{Hilbert-Chow
  morphism} is a morphism $\Phi: \HH \to \Div^{d,d}(X\times X)$, where
$\Div^{d,d}(X\times X)$ is a projective scheme, hence a closed subscheme of
a projective space $\P$. If $\xi \in \HH(k)$, one has $\Phi(\xi) =
\text{Chow form of }\langle C \rangle$. Now Fogarty showed that
\[
\langle C \rangle \mapsto \text{Chow form of }\langle C \rangle
\]
is an injective map~\cite[proof of Lemma 10.3]{F1}. As we will make
statements on the Hilbert-Chow morphism, which only concern the fibers, we
write $h$ instead of $\Phi$, i.e.~$h(\xi) = \langle C \rangle$. As $\Phi$ is
$\PGL(3;k)$-equivariant (cf.~\cite[p. 109]{M1}), one has $h(g\xi) = g
h(\xi)$ if $g\in \GL(4;k)$. 


\chapter{The action of $\Aut(\HH)$ on the first Chow group}
\label{chap:4}

We recall the convention that $A_1(\HH)$ and $A_1(\CC)$ denote  the Chow
groups with coefficients in $\Q$, and put $S = k[x, y, z, t]$.

\section{The action of $\Aut(\HH)$ on $A_1(\HH)$}
\label{sec:4.1}

In Chapter~\ref{cha:1} it had been shown that the cone $A^+_1(\HH)$ is
freely generated by the classes of $C_0, C_1, C_2$ (cf.~Theorem~\ref{thm:1.2}).
It follows that each $\phi\in \Aut(\HH)$ permutes the set $\Set{[C_0],
  [C_1],[C_2]}$.

\textsc{Case 1}: $[\phi (C_2)] = [C_1]$.\\
Let be $(\ell, f\cK) \leftrightarrow \xi \in \cG(k)$ and $g \in S_d / \ell
S_{d-1}$ such that $f$ and $ßg$ are linearly independent modulo $\ell
S_{d-1}$. Then $D:=\Set{ \langle \alpha \bar f + \beta \bar g \rangle |
  (\alpha : \beta ) \in \P^1}$ is a curve of degree $1$ in $\P(S_d / \ell
S_{d-1})$ and $C:= \Set{ (\ell, (\alpha f + \beta g)\cK | (\alpha : \beta )
  \in \P^1} \subset \cG$ is a curve such that $[C] = [C_2]$
(cf.~Proposition~\ref{prop:2.2}), from which it follows that $[\phi (C)]
=[\phi (C_2)] = [C_1]$.  By Corollary~\ref{cor:2.1} it follows that $\phi(C)
\subset H_m$ hence $\phi(\cG) \subset H_m$. Comparing the dimensions of
$H_m$ and $\cG$ it follows that $a=b$ or $a+1=b$
(cf.~Appendix~\ref{cha:C}).

\textsc{Case 2}: $[\phi (C_1)] = [C_2]$.\\
Applying $\phi^{-1}$ one gets $[\phi^{-1}(C_2)] = [C_1]$ and as in the first
case $a=b$ or $a+1=b$ follows.

\textsc{Case 3}: $[\phi (C_0)] = [C_1]$.\\
Let be $\cI=(\ell, f) \cap \cP_1 \cap \dots \cap \cP_{b-a-1} \cap Q$, where
$ f\in (S_d / \ell S_{d-1}) - (0)$, $\cP_i \in \P^3 - V(\ell,f)$ are closed
points, different from each other, $Q$ an ideal in $S$, which is primary to a
point $\cP$ with multiplicity $2$, and $\cP \centernot\in V(\ell,f)$ and
$\cP \neq \cP_i$ for all $i$. Let $M \subset \HH(k)$ be the set of all such
ideals. Fixing $\ell, f, \cP_1, \dots, \cP_{b-a-1}$ and $\cP$, then $M$ is isomorphic to
the closed points of $V:=\Quot^2(\cO_{\P^3}/\cP^2) \simeq \P^2$.
 Take a point $\xi_0 \leftrightarrow (\ell, f)
\cap \cP_1 \cap \dots \cap \cP_{b-a-1} \cap Q_0 \in M$ and a different point
$\xi_1 \leftrightarrow (\ell, f)
\cap \cP_1 \cap \dots \cap \cP_{b-a-1} \cap Q_1$ such that $Q_0$ and $Q_1
\in V(k)$. Then $f(\xi_0) = f(\xi_1)$ if $f$ is the morphism defined by $\cL_1
\otimes \cL_2$ (see Chapter~\ref{cha:3}). By Lemma~\ref{lem:3.2}, $\xi_0$
and $\xi_1$ can be joined by a connected curve $D \subset \HH$ such that
$[D] = \nu \cdot [C_0]$. It follows that $[\phi(D)] = [\nu \phi(C_0)] =
\nu[C_1]$, and this implies $\phi(D) \subset H_m$
(Corollary~\ref{cor:2.1}), hence $\phi(M) \subset H_m$. But clearly one has
$\dim M \geq 3 + \binom{d+2}{2} -1 + 3(b-a) +1  = \binom{d+2}{2} + 3(b-a)
+3$ and $\dim H_m =  \binom{d+2}{2} + (b-a) +5$; this implies $a=b$ or
$a+1=b$.

\textsc{Case 4}: $[\phi (C_0)] = [C_2]$.\\
Using the same argumentation as in Case 3 and Lemma~\ref{lem:2.2}, it
follows that $\phi(M) \subset \cG$. As $\dim \cG = \binom{d+2}{2} + 2(b-a)
+4$, this again implies $a=b$ or $a+1=b$.

Now the general assumption was $d\geq 3$ and $g \leq g(d) =
(d-2)^2/4$. Using the formulas from~\cite[p. 92]{T1} one sees that this
amounts to $a^2 -1 \leq 4b$ and we obtain:

\begin{proposition}
 \label{prop:4.1}  
 Let $\HH = H_{d,g}$ be the Hilbert scheme, which parametrizes curves in
 $\P^3$ with degree $d \geq 3$ and genus $g \leq g(d) = (d-2)^2/4$.  If
 $(d,g) \centernot\in \Set{(3, 0),\, (3,-1),\, (4,1)}$, then $\Aut(\HH)$
 operates trivially on $A_1(\HH)$.  \hfill $\qed$
\end{proposition}

\begin{corollary}
  \label{cor:4.1} 
 If $d\geq 5$ and $g \leq g(d)$, then the subschemes $H_m$ and $\cG$ are
 invariant under $\Aut(\HH)$. 
\end{corollary}
\begin{proof}
  $1^\circ$. Let $\cI = (\ell, f(h,g)) \leftrightarrow \xi \in H_m(k)$.
  Take any $g' \in S_c$, $c = b-a+1$, such that $g$ and $g'$ are 
linearly independent  modulo $(\ell, h) S(-1)$. Put $\cI_\alpha:=(\ell, f(h,
g + \alpha \cdot g'))$. In order to compute the degree of
$C:=\Set{\cI_\alpha | \alpha \in k }^-$ one can suppose that $\ell =x$, $h =
y$ and $g, g' \in k[z, t]_c$. Then one can write:
\[
  H^0(\P^3, \cI_\alpha(n)) = x P_{n-1} \oplus f \cdot y \cdot k[y, z,
  t]_{n-a} \oplus f \cdot (g + \alpha g')  k[z,  t]_{n-b}\,.
\]
 Then  
\[
(\cM_n\cdot C) = \alphadeg\bigl(\dot\bigwedge H^0(\P^3, \cI_\alpha(n))\bigr) =
 n-b +1 \,.
\]
As \emph{numerical equivalence} $=$ \emph{rational equivalence} on $\HH$, we
have $[C] = [C_1]$.  From Proposition~\ref{prop:4.1} it follows $[\phi(C)] =
[C_1]$ and by Corollary~\ref{cor:2.1} in Chapter~\ref{cha:2} it follows that
$\phi(C) \subset H_m$.

$2^\circ$. If $\xi \in \HH(k)$, in the proof of Proposition~\ref{prop:4.1}
it was shown that there is a connected curve $C \subset \cG$ with $\xi \in
C$ and $[C] =[C_2]$. From $[\phi(C)] = [C_2]$ and Lemma~\ref{lem:2.2} in
Chapter~\ref{cha:2} it follows that $\phi(C) \subset \cG$. 
\end{proof}

\section{The action of $\Aut(\HH)$ on $A_1(H_m)$ and on $A_1(\cG)$ }
\label{sec:4.2}

\subsection{}
\label{sec:4.2.1}

Let be $\phi\in \Aut_k(H_m)$. By Appendix~\ref{cha:C}, 
Proposition~\ref{prop:C.4} the cone $A^+_1(H_m)$ is freely generated by
$[Z_i]$, $0\leq i \leq 3$, hence $\phi_*$ permutes these classes. If
$\phi_*[Z_i] = [Z_j]$ in $A_1(H_m)$, this equation is true in $A_1(\HH)$,
too. As $\phi_*$ acts trivially on $A_1(\HH)$ if $d \geq 5$ and $g \leq
g(d)$ (cf.~Proposition~\ref{prop:4.1}), it follows $[Z_i] = [Z_j]$.  Forming
the intersection numbers with $\cM_n$ shows that $i=j$, i.e.~$\phi_*$ acts
trivially on $A_1(H_m)$.
\subsection{}
\label{sec:4.2.2}

Let be $ Z = q_0 Z_0 + \cdots + q_3 Z_3$, $q_i \in \Q$, and suppose $[Z] =0$
in $A_1(\cG)$. As usual $p: \cG \to X=\P(S_1)$ is the projection $(\ell,
f\cdot \cK) \mapsto \langle \ell \rangle$, hence the restriction of $\cL_3$ to
$H_m$ agrees with the line bundle introduced in Appendix~\ref{cha:C},
Section~\ref{sec:C.7}.  Using Lemma~\ref{lem:C.1} in that section gives $q_i
=0$. As $A_1(\cG) \simeq \Z^4$ by Corollary~\ref{cor:C.2}, it follows that
$[Z_i]$ is a basis of $A_1(\cG)\otimes \Q$. If $[Z]\in A_1(\cG)$, it follows
that there are integers $n_i$ and $n\neq 0$ such that $n[Z] = \sum
n_i[Z_i]$. But then $n\phi_*[Z] = n [Z]$, hence $\phi_*[Z] = [Z]$.

\begin{proposition}
 \label{prop:4.2}
 If $d \geq 5$ and and $g \leq g(d)$, then $\Aut(\HH)$ acts
trivially on $A_1(H_m)$ and $A_1(\cG)$.
\hfill $\qed$  
\end{proposition}

\section{The action of $\Aut(\HH)$ on $A_1(\CC)$ }
\label{sec:4.3}
\subsection{}
\label{sec:4.3.1}
Each $\phi \in \Aut(\HH)$ induces an automorphism $\phi \times \id$ of
$\HH \times \P^3$ such that $(\phi \times \id)^*\CC = \CC$, hence induces an
automorphism $\psi$ of the universal curve via the cartesian diagram:
\[
 \xymatrix{\CC \ar[r]^\psi \ar[d]_f & \CC \ar@{^{(}->}[r] \ar[d]^f &
   \HH\times\P^3 \ar[dl]_\pi \ar[d]^\kappa   \\
            \HH \ar[r]^\phi & \HH & \P^3}
\]					

If $(\xi,p)$ is any element of $\CC$, then $ f(\psi(\xi, p)) = \phi(f(\xi,p))= \phi(\xi)$,
i.e. one can write $\psi(\xi,p) = (\phi(\xi), q) $, where $q $ is an element of
$\CC_{\phi(\xi)}$. In order to express that $q$ depends on $p$, $\xi$ and $\phi$, in what
follows we write $ q = \phi_{\xi}(p)$. 

\subsection{}
\label{sec:4.3.2}
As had been shown (Theorem~\ref{thm:1.2}) that $A^+_1(\CC)$ is freely
generated by the classes of $C^*_i:= C_i \times \Set{ P_0 }$ and $L^* =
\Set{ \omega } \times L$, where $P_0 = (0 : 0: 0: 1)$, $\omega \in \HH(k)$
is the point corresponding to the lexicographic ideal, and $L = V(x,y)
\simeq \P^1 \subset \P^3$, it follows that $\psi_*$ permutes the set 
$\Set{[C^*_0], [C^*_1], [C^*_2], [L^*]}$. Suppose $\psi_* [L^*] =\sum q_i
[C^*_i] + q [L^*]$. It follows that 
\[
  \pi_* \psi_* [L^*] = \sum q_i
[\pi(C^*_i)] + q\pi_* [L^*] = \sum q_i [C_i]
\]
as $\pi| C_i$ is injective and $\pi(L^*) = \Set{\omega}$. 
From the diagram above it follows that $\pi_* \psi_* [L^*] = \phi_*
\pi_*[L^*] = 0$, hence $q_i=0$, $0 \leq i \leq 2$. But then $\psi_* [L^*] = 
[L^*]$ follows and $\psi_*$ permutes the $[C^*_i]$. 
If $\psi_*[C^*_i] = [C^*_j]$, then application of $\pi_*$ and using
Proposition~\ref{prop:4.1} gives $i=j$.

\begin{proposition}
  \label{prop:4.3}
  If $d \geq 5$, $g \leq g(d) = (d-2)^2/4$, then $\Aut(\HH)$ acts
  trivially on $A_1(\CC)$. \hfill $\qed$
\end{proposition}


\chapter{Automorphisms of some special schemes}
\label{cha:5}

\section{Description of the starting situation}
\label{sec:5.1}

We write $S = k[x, y, z, t]$ or $S = k[X_0,\dots,X_3]$, $X= \P(S_1) =
\Proj(S) = \P^3$, $d=a-1$, $c = b-a+1$, where $a$ and $b$ are the Macaulay
coefficients of the Hilbert polynomial $Q(n) = \binom{n+3}{3}-P(n)$ (see
Section~\ref{sec:1.1}).

If $Y, Z, \dots$ are the schemes of Appendix~\ref{cha:C}, then
\begin{align*}
  Y(k) & =\Set{ (\ell, h) | \ell \in S_1, \; h \in S_1/\ell\cdot k } \\
  Z(k) & =\Set{ (\ell, h,g ) | \ell \in S_1,\; h \in S_1/\ell\cdot k,\; g \in S_d / \langle \ell, h\rangle S_{d-1}} \\
  \cH(k) & =\Set{ (\ell, f) | \ell \in S_1, \; f \in S_d/\ell S_{d-1}} \\
  \fX(k) & =\Set{ (\ell, \cK) | \ell \in S_1, \; \cK \in \Hilb^c(\Proj S/\ell S(-1))} \\
  H_m(k) & =\Set{ (\ell, f(h,g)) | \ell \in S_1, \; h \in S_1/\ell\cdot k,
    \; f \in S_d/\ell S_{d-1}, \; g \in S_c / \langle \ell,  h\rangle S_{c-1} } \\
  \cG(k) & = \Set{ (\ell, f \cdot \cK) | \ell \in S_1, \; f \in S_d/\ell S_{d-1}, \;
    \cK \in \Hilb^c(\Proj S/\ell S(-1))}
\end{align*}
where $\ell, h, f, g$ are all different from zero.

In Appendix~\ref{cha:C} it is shown that all these schemes are projective and
smooth. $Z$ is a closed subscheme of $\fX$, hence 
$H_m \stackrel{\sim}{\longrightarrow} \cH \times_X Z$ is a closed subscheme of
$\cG \stackrel{\sim}{\longrightarrow} \cH \times_X \fX$. One has a commutative
diagram
\begin{equation}
  \label{eq:**}
  \tag{$**$}
  \begin{aligned}
    \xymatrix{
      &  H_m \ar[dd]^{p_2} \ar[dl]_{p_1} \ar@{^{(}->}[rr] & & \cG \ar[dlll]_{p_1} \ar[dd]^{p_2} \\
      \cH \ar[ddr]_\pi  &          & &      \\
      & Z \ar[d]^-\pi \ar@{^{(}->}[rr] &  & \fX \ar[dll]^\pi \\
      & X & & }
  \end{aligned}
\end{equation}
where $\pi:Z \longrightarrow X$ factorizes in $Z
\stackrel{q}{\longrightarrow} Y \stackrel{p}{\longrightarrow} X$.
Let be $R=k[x,y,z]$, $\P^2 =\Proj(R)$. Then $U = \Set{(\ell, f) | \ell = ax
  + by +cz +t, f \in R_d \otimes A}$ is an open set in $\cH(A)$ and
$p^{-1}_1(U) = U \times_A \Hilb^c(\P^2)(A)$ respectively $p^{-1}_1(U) = U
\times_A F(A)$, where $F \subset \Hilb^c(\P^2)$ is the closed subscheme
of ideals $(h, g)$, $h\in R_1$, $g\in R_c/h R_{c-1}$. It follows that in
both cases $p_1$ defines a locally trivial fiber bundle and the other
morphisms define projective bundles.

Each $\phi \in \Aut(\HH)$ induces $k$-automorphisms of $H_m$ and $\cG$
(cf.~Corollary~\ref{cor:4.1}). The aim in this Chapter 5 is to show: \\
{\itshape There is a $\gamma\in \PGL(3,k)$, which is uniquely determined by
  $\phi$, such that $\phi|H_m$ and $\phi|\cG$ are induced by $\gamma$}
(cf.~Proposition~\ref{prop:5.3}).  

The proof uses the aformentioned properties of the different morphisms in
diagram~\eqref{eq:**}, the fact that $\phi_*$ operates as the identity on
$A_1(H_m)$ and $A_1(\cG)$ (cf. ~Proposition~\ref{prop:4.2}) and a formalism,
which is explained in the next sections.  

\section{Relative automorphisms of $\cH$}
\label{sec:5.2}

It seems rather difficult to determine the group $\Aut_k(\cH)$. But if
$\pi:\cH \to X = \P(S_1)$ is the projection, the fiber of $\pi$ over $\ell
\cdot A \in X(A)$ is $\P(S_d \otimes A / \ell S_{d-1} \otimes A)$,
i.e.~$\pi:\cH \to X$ is a projective bundle.

\begin{proposition}
   \label{prop:5.1}
  $\Aut_X(\cH) = \Set{\id}$. 
\end{proposition}
  \begin{proof}
    To simplify the notations, in this section we write $S =
    k[X_0,X_1,X_2,X_3]$. If we put $L:=x_0X_0+\cdots + x_3 X_3$, $X = \P(S_1)
    \stackrel{\sim}{\longrightarrow} \Proj k[x_0,\dots,x_3]$, then $\cL :=L
    \cdot \cO_X(-1)$ is the universal $1$-subbundle of $S_1 \otimes \cO_X$
    and we let $\cF$ be the subbundle of rank $\binom{d-1+3}{3}$ of $\cE
    :=S_d \otimes \cO_X$, which is generated by $\cL$, i.e.~we put $\cF :=
    \cL \otimes_k S_{d-1}$. As $\cL| D_+(x_i) = \bigl(\tfrac{x_0}{x_i} X_0+
    \cdots + \tfrac{x_3}{x_i} X_3\bigr)\cdot \Gamma(D_+(x_i), \cO_X)$, the
    quotient $\cG:=\cE/\cF$ is locally free over $X$ of rank
    $\binom{d+2}{2}$.  One sees that $\cH = \P(\cG)=:\P$.

 One has a commutative diagram
 \begin{equation}
   \label{eq:5.1}
   \begin{aligned}
     \xymatrix{
       \P \ar[rr]^\phi \ar[dr]_\pi & & \P\ar[dl]^\pi \\
       & X & }
   \end{aligned}
 \end{equation}
and $\phi^*$ defines an isomorphism 
 \[
 \Pic(\P) \simeq \cO_\P(1) \cdot \Z \times \cO_X(1) \cdot \Z
 \]
(see~\cite[Chapter II, Ex. 7.9]{H}).
It follows $\phi^*(\cO_\P(1)) \simeq \cO_\P(n) \otimes
\pi^*(\cO_X(m))$ with $n= \pm 1$. Now $\cO_X(1)$ is trivial over $U =
D_+(x_i)$ and $\pi^{-1}(U) \simeq \P^2_k \times U =:Y$. 
It follows
\[
  H^0(Y,\cO_Y(\nu)) \simeq H^0(Y,\phi^*(\cO_\P(\nu)|Y)) \simeq
H^0(Y,\cO_\P(\nu n)) \text{ for all } \nu \in \N\,.
\]
It follows that $n=1$ and we conclude:
\begin{alignat*}{3}
  & & \phi^*(\cO_\P(1)) & \simeq \cO_\P(1) \otimes \pi^*(\cO_X(m)) \text{ as
  }
  \cO_\P\text{-modules} \\
  \Rightarrow \quad & & \pi_*\phi^*(\cO_\P(1)) & \simeq \pi_*\cO_\P(1)
  \otimes_{\cO_X} \cO_X(m) \\
  \Rightarrow \quad & & \pi_*\phi_*\phi^*(\cO_\P(1)) & \simeq \pi_*\cO_\P(1)
  \otimes_{\cO_X} \cO_X(m) \\
  \Rightarrow \quad & & \pi_*\cO_\P(1) & \simeq \pi_*\cO_\P(1)
  \otimes_{\cO_X} \cO_X(m) \\
  \Rightarrow \quad& & \cG & \simeq \cG \otimes_{\cO_X} \cO_X(m) \text{ as }
  \cO_X\text{-modules} \\
  \Rightarrow \quad & & \cG & \simeq \cG \otimes_{\cO_X} \cO_X(\nu m) \text{
    for all } \nu \in \N
\end{alignat*}
If $m \neq 0$, then $\cG$ has a constant Hilbert polynomial and thus
$\dim(\supp(\cG)) = 0$, contradiction. It follows that $\phi$ induces an
isomorphism $\phi^*: \cG \simeq \cG$ of $\cO_X$-modules. Conversely, each
isomorphism of $\cO_X$-modules $\cG \simeq \cG$ induces an isomorphism
$\P(\cG) \simeq \P(\cG)$ over $X$~\cite[Chapter II, Ex. 7.9]{H}.

As $\cExt^1(\cE,\cF) =0$ one has an exact sequence of
$\cO_X$-modules
\[
 0 \longrightarrow \cHom_X(\cE,\cF)  \longrightarrow \cHom_X(\cE,\cE)  \longrightarrow \cHom_X(\cE,\cG)  \longrightarrow 0
\]
and one obtains an exact sequence
\begin{equation}
 \label{eq:5.2}
 0 \longrightarrow  \Gamma(X,\cHom(\cE,\cF)) \longrightarrow \Gamma(X,\cHom(\cE,\cE))  \longrightarrow \Gamma(X,\cHom(\cE,\cG))  \longrightarrow H^1(X,\cHom(\cE,\cF))\,.
\end{equation}

Now $\cF =L \cdot \cO_X(-1) \otimes_k S_{d-1}$ and putting $n = \dim_k
S_{d-1} \otimes S_d$ we obtain:
\[
 \cHom_X(\cE,\cF) = \cHom_X(\cO_X, L \cdot \cO_X(-1)) \otimes_k S_{d-1} \otimes
 S_d \simeq L \cdot \cO_X(-1) \otimes_k k^n \simeq \cO_X(-1) \otimes_k k^n
\]
which implies that the first and last term in the sequence~\eqref{eq:5.2}
are equal to $(0)$. Now $\cHom_X(\cE,\cE) \simeq \cO_X \otimes_k \Hom(S_d,
S_d)$ and thus $\Hom_k(S_d,S_d) \simeq \Hom(\cE, \cG)$.  Together with the
diagram~\eqref{eq:5.1} we deduce that $\cE
\stackrel{\mathrm{can.}}{\longrightarrow} \cG \xrightarrow[\phi^*]{\sim}
\cG$ is induced by a $k$-linear map $\psi: S_d \to S_d$, which, for all $\ell
\in S_1 -(0)$, induces a $k$-linear isomorphism $S_d/\ell S_{d-1} \to
S_d/\ell S_{d-1}$. It follows that $\psi(\ell S_{d-1}) \subset \ell S_{d-1}$
for all $\ell \in S_1$. The lemma in Appendix~\ref{cha:B} shows that $\psi$
is the multiplication by an element $\alpha \in k^*$ and thus $\phi$ is the
identical map on $\P$.
  \end{proof}

If one puts $d=1$, one obtains
\begin{corollary}
   \label{cor:5.1}
  $\Aut_X(Y) = \Set{\id}$. 
\hfill $\qed$
\end{corollary}
\section{Relative automorphisms of $Z$}
\label{sec:5.3}

If $A$ is a $k$-algebra, we defined $Z(A)$ by 
\[
  Z(A) =\Set{ (\ell, h,g ) | 
    \begin{aligned}
     & \ell \in S_1 \otimes A,\; h \in S_1 \otimes A/\ell A \text{ and } g
      \in S_d \otimes A / \langle \ell, h\rangle \cdot S_{d-1}\otimes A \\
     & \text{respectively, generate $1$-subbundles}.
    \end{aligned}
} 
\]
The aim is to show that $\Aut_Y(Z) = \Set{\id}$ and, as in the case of
$\cH$, we have to build up a more formal setting:

$S = k[X_0,\dots,X_3]$,\quad $X= \P^3 \simeq \P(S_1)$, \quad
$L_1:=x_0X_0+\cdots + x_3 X_3$, \quad $\cL_1 :=L_1 \cdot \cO_X(-1)\subset
S_1 \otimes \cO_X$ universal $1$-subbundle over $X$,\quad $\cG_1:= S_1
\otimes\cO_X/ \cL_1$ locally free over $X$ of rank $3$.

$Y:=\P(\cG_1)\simeq \Flag(1, 2, S_1) \stackrel{p}{\longrightarrow} X$ is a
projective bundle, where $p$ is defined by $(F_1, F_2) \mapsto F_1$.

Let $\cL_2 \subset \cG_1 \otimes_X \cO_Y$ be the universal $1$-subbundle. As
$\cG_1 \otimes_X \cO_Y = S_1 \otimes \cO_Y / p^*(\cL_1)$, it follows that
\[
\cG_1 \otimes\cO_Y \otimes_k S_{d-1} = S_1 \otimes_k S_{d-1} \otimes \cO_Y /
p^*(\cL_1) \otimes_k S_{d-1} \twoheadrightarrow S_d \otimes \cO_Y/ p^*(\cL_1)
\cdot S_{d-1}
\]
because, locally on $Y$, one has $p^*(\cL_1)|U =\ell_1 \cdot \cO_U$,\;
$\ell_1\in S_1 \otimes \cO_U$ not a zero-divisor. Thus $\cE:=S_d \otimes
\cO_Y/p^*(\cL_1)\cdot S_{d-1}$ is locally free over $Y$ of rank
$\binom{d+2}{2}$. As locally on $Y$ one has $\cL_2|U = \ell_2 \cdot \cO_U$
and $\ell_2 \in S_1 \otimes \cO_U/\ell_1 \cdot \cO_U$ not a zero-divisor of
$S \otimes \cO_U/\ell_1\cdot S(-1) \otimes \cO_U$, the canonical map $\cL_2
\otimes_k S_{d-1} \to \cE$ is injective and remains so if tensored with
$k(y)$, for all $y\in Y$.  It follows that the image $\cF:=\cL_2 \cdot
S_{d-1} \simeq \cL_2 \otimes_k S_{d-1}$ of this homomorphism is a subbundle,
$\cG:=\cE/\cF$ is locally free over $Y$ of rank $d+1$ and $Z = \P(\cG)$ by
construction. The canonical morphism $\cE \to \cG$ is denoted by $\kappa$.

\begin{remark}
  \label{rem:5.1}
 $H^1(Y,\cO_Y(-1)) = (0)$\,.
\end{remark}
\begin{proof}
$Y =\P(\cG_1)$, $\cG_1$ of rank $3$ $\Rightarrow R^i p_*\cO_Y(\nu)=(0)$, if
$0 < i < 2$, all $\nu$; $R^2 p_*\cO_Y(\nu)=(0)$, if $\nu > -3$
(see~\cite[III, Ex. 8.4]{H}) $\Rightarrow H^1(Y, \cO_Y(-1)) = H^1(X,
p_*\cO_Y(-1))$ (loc.\ cit.\ Ex.~8.1). As $p_*\cO_Y(-1)=(0)$
(loc.~cit.~Ex.~8.4a), the assertion follows.
\end{proof}

\begin{remark}
  \label{rem:5.2}
 $H^1(Y,p^*\cO_X(-1)) = H^1(Y,\cO_X(-1)\otimes \cO_Y) = (0)$\,.
\end{remark}
\begin{proof}
  $R^i p_*(p^*\cO_X(-1)\otimes \cO_Y) \simeq R^i p_*\cO_Y  \otimes
  \cO_X(-1)=(0)$, if $i> 0$ (loc.~cit.~Ex.~8.3, 8.4) $\Rightarrow
  H^1(Y,p^*\cO_X(-1)) = H^1(X,p_*p^*\cO_X(-1))$ (loc.~cit.~Ex.~8.1).  Now
  $p_*p^*\cO_X(-1) = p_*(p^*\cO_X(-1)\otimes_{\cO_X}\cO_Y) \simeq \cO_X(-1)
  \otimes p_* \cO_Y$ (loc.~cit.~Ex.~8.3). As $p_* \cO_Y = \Symm^0(\cG_1)=
  \cO_X$ (loc.~cit.~Ex.~8.4) one gets $H^1(Y,p^*\cO_X(-1)) =
  H^1(X,\cO_X(-1)) = (0)$.
\end{proof}

\begin{remark}
  \label{rem:5.3}
 $H^0(Y,p^*\cO_X(1)\otimes \cO_Y(-1)) = (0)$\,.
\end{remark}
\begin{proof}
  $R^i p_*(p^*\cO_X(1)\otimes \cO_Y(-1)) = R^ip_* \cO_Y(-1)\otimes \cO_X(1)
  = (0) \otimes \cO_X(1) = (0)$, if $i> 0$ (loc.~cit.) $\Rightarrow
  H^0(Y,p^*\cO_X(1)\otimes \cO_Y(-1)) \simeq H^0(X,p_*(p^*\cO_X(1)\otimes
  \cO_Y(-1))) \simeq H^0(X,\cO_X(1)\otimes p_*\cO_Y(-1)) =(0)$ as $p_*
  \cO_Y(-1) =(0)$ (loc.~cit.).
\end{proof}
\begin{remark}
  \label{rem:5.4}
 $H^1(Y, \cHom(\cE, \cF)) = (0)$\,.
\end{remark}
\begin{proof}
  $\cExt^1(\cE, \cF) = (0)$, as $\cE$ is locally free on $Y$, and there are
  two exact sequences:
\begin{equation}
  \label{eq:5.3}
 0 \longrightarrow \cHom(\cE, \cF)  \longrightarrow  \cHom(\cE, \cE) \longrightarrow  \cHom(\cE, \cG) \longrightarrow 0  
\end{equation}
\begin{equation}
  \label{eq:5.4}
 0 \longrightarrow p^*(\cL_1)\otimes S_{d-1} \longrightarrow S_d \otimes
 \cO_Y \longrightarrow  \cE \longrightarrow 0   
\end{equation}
From~\eqref{eq:5.4} we get the exact sequence
\[
 0 \longrightarrow \cHom(\cE, \cF)  \longrightarrow \cHom(S_d \otimes \cO_Y,
 \cF)  \longrightarrow  \cHom(p^*(\cL_1)\cdot S_{d-1}, \cF) \longrightarrow
\cExt^1(\cE, \cF) = 0
\]
which gives the exact sequence: 
\begin{equation}
  \label{eq:5.5}
  \begin{multlined}
    0 \longrightarrow \Gamma(Y,\cHom(\cE, \cF)) \longrightarrow
    \Gamma(Y,\cHom(S_d \otimes \cO_Y, \cF)) \longrightarrow \\
 \Gamma(Y,\cHom(p^*(\cL_1)\cdot S_{d-1}, \cF)) 
    \longrightarrow H^1(Y,\cHom(\cE, \cF)) \longrightarrow H^1(Y,\cHom(S_d
    \otimes \cO_Y, \cF))
  \end{multlined}
\end{equation}
Now from 
\[
\cHom(S_d \otimes \cO_Y, \cF) = \cHom(\cO_Y, \cF)\otimes S_d = \cF \otimes S_d
\]
it follows that $H^1(Y,\cHom(S_d \otimes \cO_Y, \cF)) \simeq H^1(Y,\cF)
\otimes S_d \simeq H^1(Y,\cL_2 \otimes S_{d-1})\otimes S_d \simeq
H^1(Y,\cL_2) \otimes S_{d-1}\otimes S_d \simeq H^1(Y,\cO_Y(-1)) \otimes
S_{d-1}\otimes S_d = (0)$ because of $\cL_2 \simeq \cO_Y(-1)$ as
$\cO_Y$-modules and Remark~\ref{rem:5.1}. We further compute:
\begin{align*}
  & \cHom(p^*(\cL_1)S_{d-1},\cF) = \cHom(p^*(\cL_1)\otimes S_{d-1},\cF) =
  \cHom(p^*(\cL_1),\cF) \otimes S_{d-1}\\ 
= {} & \cHom(p^*\cO_X(-1),\cF) \otimes
  S_{d-1} \simeq \cHom(\cO_Y, p^*\cO_X(1)\otimes \cL_2 \otimes
  S_{d-1})\otimes S_{d-1}\\ 
\simeq {} & p^*\cO_X(1)\otimes \cO_Y(-1)\otimes
  S_{d-1}\otimes S_{d-1}
\end{align*}
(see~\cite[Chap. III, Prop. 6.7 and 6.3a]{H} and use $\cL_2 \simeq
\cO_Y(-1)$).  By Remark~\ref{rem:5.3} we get $\Gamma(Y,\cHom(p^*(\cL_1) 
S_{d-1}, \cF))=(0)$ and from the exact sequence~\eqref{eq:5.5}
Remark~\ref{rem:5.4} follows.
\end{proof}
Applying $\Gamma(Y,-)$ to the sequence~\eqref{eq:5.3} gives
\begin{conclusion}
   \label{conc:5.1}
The canonical map
\[
\Gamma(Y, \cHom(\cE, \cE)) \simeq \Hom(\cE, \cE)
\stackrel{\kappa_*}{\longrightarrow} \Hom(\cE, \cG) \simeq \Gamma(Y,
\cHom(\cE, \cG))
\]
defined by $u \mapsto \kappa \circ u$ is surjective.
\hfill $\qed$
\end{conclusion}

Applying $\cHom(S_d \otimes \cO_Y, -)$ to the exact sequence~\eqref{eq:5.4}
gives the exact sequence
\[
 0 \longrightarrow \cHom(S_d \otimes \cO_Y, p^*(\cL_1)\otimes S_{d-1}) 
 \longrightarrow  \cHom(S_d \otimes \cO_Y, S_d \otimes \cO_Y) \longrightarrow 
 \cHom(S_d \otimes \cO_Y, \cE) \longrightarrow 0\,.
\]
Now one applies $\Gamma(Y,-)$ and, because of $\Gamma(Y,\cHom(\cA, \cB)) = \Hom(\cA, \cB)$, one obtains the exact sequence 
  \begin{align*}
    0 & \longrightarrow \Hom(S_d \otimes \cO_Y, p^*(\cL_1)\otimes S_{d-1})
    \longrightarrow \Hom(S_d \otimes \cO_Y, S_d \otimes \cO_Y)
    \longrightarrow
    \Hom(S_d \otimes \cO_Y, \cE)\\
    & \qquad \qquad \longrightarrow H^1(Y, \cHom(S_d \otimes \cO_Y,
    p^*(\cL_1))) \,.
  \end{align*}
As
\[
  \cHom(S_d \otimes \cO_Y, p^*(\cL_1))
  \simeq \cHom(\cO_Y,
  p^*(\cL_1)) \otimes S_d = p^*(\cL_1) \otimes S_d \simeq
  p^*(\cO_X(-1)) \otimes S_d\,,
\]
from Remark~\ref{rem:5.2} it follows that the last term in the
sequence is equal to $(0)$ and $\Hom(S_d \otimes \cO_Y, S_d \otimes \cO_Y)
\longrightarrow \Hom(S_d \otimes \cO_Y, \cE)$ is surjective. Now
\[
\Hom(S_d \otimes \cO_Y, S_d \otimes \cO_Y) \simeq \Hom(\cO_Y,\cO_Y) \otimes_k
 \Hom_k(S_d, S_d) = \Gamma(Y, \cO_Y) \otimes \Hom_k(S_d, S_d)\,.
\]
As $Y$ is a variety, $\Gamma(Y,\cO_Y) = k$  and one has
\begin{conclusion}
   \label{conc:5.2}
$\Hom_k(S_d, S_d) \twoheadrightarrow \Hom(S_d \otimes \cO_Y, \cE)$\;.
\hfill $\qed$
\end{conclusion}
As $\cExt^1(\cE,\cE) =0$, applying $\cHom(-,\cE)$ to \eqref{eq:5.4} gives the exact sequence 
\[
 0 \longrightarrow \cHom(\cE,\cE)  \longrightarrow \cHom(S_d \otimes \cO_Y,\cE)  \longrightarrow \cHom(p^*(\cL_1) \cdot  S_{d-1},\cE)  \longrightarrow 0\,.
\]
Applying $\Gamma(Y,-)$ to this sequence gives
\begin{conclusion}
   \label{conc:5.3}
$\Hom(\cE, \cE) \stackrel{\mathrm{can.}}{\rightarrowtail} \Hom(S_d \otimes
 \cO_Y, \cE)\,.
\hfill \qed$
\end{conclusion}
Finally, the exact sequence 
$0 \longrightarrow \cF \longrightarrow \cE
 \stackrel{\kappa}{\longrightarrow} \cG  \longrightarrow 0$
gives the exact sequence
\[
0 \longrightarrow \cHom(\cG, \cG) \longrightarrow \cHom(\cE, \cG)
\longrightarrow \cHom(\cF, \cG) \,.
\]
Application of $\Gamma(Y,-)$ gives
\begin{conclusion}
   \label{conc:5.4}
The canonical morphism
 $\Hom(\cG, \cG) \stackrel{\kappa^*}{\longrightarrow} \Hom(\cE, \cG)$
defined by $u \mapsto u \circ \kappa$ is injective.
\hfill $\qed$
\end{conclusion}
All in all one obtains a diagram of natural homomorphisms:
\[
   \xymatrix{
      \Hom(\cE, \cG) & \; \Hom(\cG, \cG)  \ar@{>->}[l] &  \\
      \Hom(\cE, \cE)\; \ar@{->>}[u] \ar@{>->}[r] & \Hom(S_d \otimes \cO_Y, \cE)              &  \Hom_k(S_d, S_d) \ar@{->>}[l]
}
\]
\begin{conclusion}
   \label{conc:5.5}
   Each $\cO_Y$-homomorphism $\cG \to \cG$ is induced by a $k$-linear
   homomorphism $S_d \to S_d$.  \hfill $\qed$
\end{conclusion}

{\itshape Geometrical consequences}\\
We recall that $p: Y \to X$ and $\pi:Z\to Y$ are defined by $(\ell, h) \mapsto
\langle \ell \rangle$ respectively $(\ell, h, g) \mapsto (\ell, h)$. The
fibers are $\P(S_1 \otimes A / \ell A)\simeq \P^2_k \otimes A$ respectively
$\P(S_d \otimes A / \langle \ell, h\rangle S_{d-1} \otimes A)\simeq \P^d_k
\otimes A$, which shows that $p$ and $\pi$ are projective bundles.

If we take any $\phi \in \Aut_Y(Z)$, the diagram 
\begin{equation}
  \label{eq:5.6}
   \begin{aligned}
     \xymatrix{
           Z \ar[rr]^\phi \ar[dr]_\pi & & Z \ar[dl]^\pi \\
                     & Y & 
      }
   \end{aligned}  
\end{equation}
is commutative. We want to compute $\Pic(Z)$ and again use the
results~\cite[Chap II, Prop. 7.11; Ex. II 7.8, 7.9, 7.10; III 8.1, 8.3,
8.4]{H}.
\begin{align*}
  \Pic(Z) & \simeq \Pic(Y) \times \cO_Z(1) \cdot \Z \\
  & \simeq \cO_X(1) \cdot \Z \times \cO_Y(1) \cdot \Z \times \cO_Z(1) \cdot
  \Z\,.
\end{align*}
From the diagram~\eqref{eq:5.6}  it follows that 
\[
 \phi^*(\cO_Z(1)) \simeq \pi^* p^* \cO_X(m)  \otimes \pi^* \cO_Y(n) \otimes
 \cO_Z(\mu)\,,
\]
where $m, n \in \Set{ 0, \pm 1}$ and $\mu \in \Set{\pm 1}$, as $\phi^*$
induces an isomorphism of $\Pic(Z)$.

If $y \in Y(k)$ and $F:= \pi^{-1}(y) \hookrightarrow Z$ is the fiber, then
$\phi$ induces an isomorphism $\phi' = \phi | F$ and one has a commutative
diagram:
 \begin{equation}
   \label{eq:5.7}
   \begin{aligned}
     \xymatrix{
           F \ar[rr]^{\phi'} \ar[dr]_\pi & & F \ar[dl]^\pi \\
                     & \Spec k(y) & 
      }
   \end{aligned}
 \end{equation}
As $\pi^*p^* \cO_X(m)|F$ and $\pi^*\cO_Y(m)|F$ are trivial,
$\phi^*(\cO_Z(1)| F) \simeq \cO_Z(\mu)|F$, hence $\phi^*(\cO_Z(\nu)| F)
\simeq \cO_Z(\nu \cdot \mu)|F$ for all $\nu \in \N$. Now $\cG$ is a
$d+1$-bundle, therefore $F \simeq \P^d$ and $(\phi')^* \cO_F(\nu) \simeq
\cO_F(\mu \cdot \nu)$, which implies 
\[
H^0(\cO_F(\nu)) \simeq H^0((\phi')^*(\cO_F(\nu)) \simeq H^0(\cO_F(\nu \cdot
\mu))
\]
for all $\nu\in \N$. It follows that $\mu=1$ and 
\begin{align*}
\phi^* \cO_Z(1) & \simeq  \pi^* p^*\cO_X(m)\otimes \pi^*\cO_Y(n) \otimes
\cO_Z(1) \\
\pi_*\phi^* \cO_Z(1) & \simeq  p^*\cO_X(m)\otimes \cO_Y(n) \otimes \pi_* \cO_Z(1)\,.
\end{align*}
Putting $\cL:=p^*\cO_X(m)\otimes \cO_Y(n)$ and using $\pi_* = \pi_* \circ
\phi_*$ gives
\[
  \pi_* \cO_Z(1) \simeq \pi_* \cO_Z(1) \otimes \cL
\]
and hence $\cG \simeq \cG \otimes \cL$.  It follows that $\cG \simeq \cG
\otimes \cL^\nu$ for all $\nu\in \N$. Let $x \in X(k)$ and $F:=p^{-1}(x)
\hookrightarrow Y$. Then $F\simeq \P^2$ and $p^*\cO_X(m)|F$ is trivial. It
follows $\cG \otimes_{\cO_Y} \cO_F \simeq \cG \otimes_{\cO_Y}\cO_F(\nu n)$
for all $\nu\in \N$. If $n\neq 0$, then the dimension of $\supp(\cG \otimes
\cO_F)$ would be equal to $0$, contradiction, as $\cG$ is a locally free
$\cO_Y$-module of rank $d+1$. Thus we obtain $\cG \simeq \cG \otimes p^*
\cO_X(m)$ and
\[
\phi^* \cO_Z(1)\simeq  \pi^* p^*\cO_X(m) \otimes \cO_Z(1)\,. 
\]

 of $(p \circ \pi)_* = p_* \circ \pi_*$ gives: 
\begin{align*}
 p_* \pi_* \phi^* \cO_Z(1)   & \simeq \cO_X(m) \otimes p_* \pi_* \cO_Z(1) \\
 p_* \pi_* \phi_* \phi^* \cO_Z(1)   & \simeq \cO_X(m) \otimes p_*(\cG) \\
 p_* \pi_* \cO_Z(1)   & \simeq \cO_X(m) \otimes p_*(\cG) \\
 p_*(\cG)  & \simeq \cO_X(m) \otimes p_*(\cG)
\end{align*}
Hence $p_*(\cG) \simeq \cO_X(\nu m) \otimes p_*(\cG)$ for all $\nu \in \N$.
We now compute $p_*\cG$.  The sequence  
\[
 0 \longrightarrow \cL_2 \otimes S_{d-1}  \longrightarrow  \cE
 \longrightarrow  \cG  \longrightarrow 0
\]
is exact and defines $\cG$. Now $\cL_2 \simeq \cO_Y(-1)$ as
$\cO_Y$-modules. Application of $p_*$ gives 
\[
 0 \longrightarrow   p_*(\cO_Y(-1))\otimes S_{d-1} \longrightarrow p_*(\cE)
 \longrightarrow p_*(\cG)  \longrightarrow R^1 p_*(\cO_Y(-1))\otimes S_{d-1}
\]
and the first and last term in this exact sequence are zero~\cite[Ex. III
8.4]{H}. Hence $p_*(\cG) \simeq p_*(\cE)$ and we get $p_*(\cE) \simeq
\cO_X(\nu m)\otimes p_*(\cE)$ for all $\nu \in \N$. We will show again that
this is possible only  if $m=0$: $\cE$ is defined by the exact sequence 
\[
 0 \longrightarrow p^*(\cL_1) \otimes S_{d-1}  \longrightarrow S_d \otimes
 \cO_Y  \longrightarrow \cE \longrightarrow 0
\]
and $\cL_1 \simeq \cO_X(-1)$ as $\cO_X$-modules. It follows the exact
sequence
\[
0 \longrightarrow p_*(p^* \cO_X(-1) \otimes \cO_Y) \otimes S_{d-1}
\longrightarrow S_d \otimes p_*(\cO_Y) \longrightarrow p_*(\cE)
\longrightarrow R^1 p_*(p^* \cO_X(-1) \otimes \cO_Y)\,.
\]
Now 
\[
  p_*(p^* \cO_X(-1) \otimes \cO_Y) \simeq \cO_X(-1) \otimes p_*(\cO_Y)
  \simeq \cO_X(-1) \otimes \cO_X
\]
and 
\[
  R^1 p_*(p^* \cO_X(-1) \otimes \cO_Y) \simeq \cO_X(-1) \otimes
  R^1 p_*(\cO_Y) = (0)
\]
again by~\cite[Ex. III 8.4]{H}. We get an exact sequence 
\[
 0 \longrightarrow \cO_X(-1) \otimes S_{d-1}  \longrightarrow \cO_X \otimes
 S_d  \longrightarrow p_*(\cE)  \longrightarrow 0   \;,
\]
where $\cO_X(-1) \otimes S_{d-1} \simeq \cL_1 \otimes S_{d-1}$ as
$\cO_X$-modules. It follows that $p_*(\cE) \simeq S_d \otimes \cO_X / \cL_1
\otimes S_{d-1}$ is locally free of rank $\binom{d+2}{2}$, hence $m=0$
follows. Now from $\phi^*\cO_Z(1) \simeq \cO_Z(1)$ and $\pi_* = \pi_*
\phi_*$ it follows that $\phi$ induces an isomorphism of $\cO_Y$-modules
\[
  \cG \stackrel{\sim}{\longrightarrow} \pi_* \cO_Z(1) \simeq  \pi_* \cO_Z(1)
  \stackrel{\sim}{\longleftarrow} \cG\,.
\]

\begin{conclusion}
   \label{conc:5.6}
   Each $Y$-automorphism of $Z=\P(\cG)$ induces an $\cO_Y$-isomorphism of
   $\cG$, and conversely.
\end{conclusion}
\begin{proof}
  One direction follows from the preceding considerations and
  (loc.~cit.). Using~\cite[4.2.3]{EGA} gives the other direction.
\end{proof}

From Conclusions~\ref{conc:5.5} and~\ref{conc:5.6} follows that any $\phi\in
\Aut_Y(Z)$ is induced by a $k$-linear map $\psi:S_d \to S_d$ and from the
commutative diagram~\eqref{eq:5.7} follows that
\begin{equation}
  \label{eq:5.8}
  \psi(\langle \ell, h \rangle \cdot S_{d-1}) \subset \langle \ell, h \rangle \cdot S_{d-1}
\end{equation}
for all $\ell \in S_1 -(0)$ and all $h\in (S_1/\ell \cdot k) -(0)$. In order
to draw further conclusions from~\eqref{eq:5.8}, we need two simple
statements.

\begin{remark}
  \label{rem:5.5}  
 If $\ell \in S_1 -(0)$ and $I= (\ell, f)$ and $J = (\ell, g)$ are two
 ideals in $S$ such that $f \in \bar S_d$ and $g \in \bar S_\ell$ are
 relatively prime in the ring $\bar S = S /\ell S(-1))$, then $I \cap J =
 (\ell, \bar f \bar g)$.
\hfill $\qed$
\end{remark}

\begin{remark}
  \label{rem:5.6}  
Suppose that $\ell \in S_1 -(0)$, $\bar S = S /\ell S(-1)$, $I_{(i)}:= (\ell,
h_i)$, $h_i\in \bar{S}_1$, $1\leq i \leq m$, are relatively prime to
each other and $I:=\bigcap^m_1 I_{(i)}$.  Then $I_n = \ell S_{n-1}$, if $n <
m$, and $I_n = \ell S_{n-1}+ h_1\cdots h_m \cdot S_{n-m}$, if $n\geq m$.
\hfill $\qed$
\end{remark}

Now choose $h_1,\dots,h_{d+1} \in S_1$, which modulo $\ell$ are relatively
prime to each other and put $L_i:=\langle \ell, h_i\rangle S_{d-1}$, $1\leq i
\leq d+1$. Then from Remark~\ref{rem:5.6} and~\eqref{eq:5.8} it follows that
\[
\psi(\cap L_i) \subset \bigcap \psi(L_i) \subset \bigcap L_i = \ell \cdot
S_{d-1}\,.
\]
\begin{conclusion}
   \label{conc:5.7}
 $\psi(\ell S_{d-1}) \subset \ell S_{d-1}$ for all $\ell \in S_1$.
\hfill $\qed$
\end{conclusion}
From the lemma in Appendix~\ref{cha:B} follows

\begin{proposition}
   \label{prop:5.2}
  $\Aut_Y(Z) = \Set{\id}$. \hfill $\qed$
\end{proposition}

\pagebreak
\section{Normed automorphisms of $\HH$}
\label{sec:5.4}

\subsection{}
\label{sec:5.4.1}
 We start with a general situation:\\
Suppose one has a commutative diagram 
\begin{equation}
  \label{eq:5.9}
  \begin{aligned}
    \xymatrix{
      Z \ar[r]^\psi \ar[d]_q & Z \ar[d]^q \\
      Y \ar[r]^\phi \ar[d]_p & Y\ar[d]^p \\
      X & X }
  \end{aligned}
\end{equation}
where all schemes are reduced and projective over $k$, $\phi$ and $\psi$ are
automorphisms and the following conditions are fulfilled:
\begin{enumerate}[a)]
\item $\psi_*$ acts trivially on $A_1(Z)$.
\item $p$ locally has sections and $q$ is surjective.
\item If $x\in X(k)$ and $y_1, y_2 \in p^{-1}(x)$ are closed points, then
  there is a curve $C \subset p^{-1}(x)$ and a connected curve $D \subset Z$
  such that $y_1, y_2 \in C$ and $q(D) =C$.
\end{enumerate}

\begin{lemma}
  \label{lem:5.1}
 Suppose these conditions are fulfilled. Then one has:
\begin{enumerate}[(i)]
\item There is a morphism $\phi'$ such that
  \begin{equation}
    \label{eq:5.10}
    \begin{aligned}
      \xymatrix{
        Y \ar[r]^\phi \ar[d]_p & Y\ar[d]^p \\
        X \ar@{-->}[r]^{\phi'} & X }
    \end{aligned}
  \end{equation}
is commutative.
\item $\phi'$ is uniquely determined by $\psi$ and $\phi$ (Notation:
  $\phi'=(\psi, \phi))$.
\item $\phi'$ is an automorphism.
\item If one has two diagrams fulfilling the aforementioned conditions
  \begin{equation}
    \label{eq:5.11}
    \begin{aligned}
      \xymatrix{
        Z \ar[r]^{\psi_i} \ar[d]_q & Z \ar[d]^q \\
        Y \ar[r]^{\phi_i} \ar[d]_p & Y\ar[d]^p \\
        X \ar@{-->}[r]^{\phi'_i} & X }
    \end{aligned}
  \end{equation}
then $(\phi_1 \circ \phi_2)' = \phi'_1 \circ \phi'_2$.
\end{enumerate}
\end{lemma}
\begin{proof}
  (All points are closed points.)
\begin{enumerate}[(i)]
\item If $x\in X$, $y \in p^{-1}(x)$, then $\phi'(x):= p \phi(y)$ is well
  defined: If $y_1, y_2 \in C \subset p^{-1}(x)$ and $q(D) = C$, then
  $\psi_*[D] = [\psi(D)] = [D]$, hence 
\[
  \deg(pq|\psi(D)) \cdot [pq \psi(D)] = p_* q_* [\psi(D)] = p_* q_* [D] =0 \,.
\]
It follows that $pq\psi(D) = p\phi q(D) = p\phi(C)$ is a single point. If
$U\subset X$ is an open set and $s:U\to Y$ is a section of $p$, then
$\phi'|U = p\circ \phi \circ s$ is a morphism.
\item follows from the surjectivity of $p$.
\item[(iv)] follows from (ii), and (iii) follows from (iv) if one puts
  $\psi_1 = \psi$, $\psi_2 = \psi^{-1}$, $\phi_1 = \phi$, $\phi_2 =
  \phi^{-1}$.  \hfill $\qedhere$
\end{enumerate}
\end{proof}

\subsection{}
\label{sec:5.4.2}
Let be $\phi \in G:=\Aut_k(\HH)$. By Proposition~\ref{prop:4.2}, $\phi_*$
trivially acts on $A_1(\cG)$. Then Lemma~\ref{lem:5.1} should give a
commutative diagram:
  \begin{equation}
    \label{eq:5.12}
    \begin{aligned}
      \xymatrix{
        \cG \ar[r]^\phi \ar@{=}[d] & \cG \ar@{=}[d] \\
        \cG \ar[r]^\phi \ar[d]_{p_2} & \cG  \ar[d]^{p_2} \\
        \fX \ar[r]^{\phi_2} \ar[d]_\pi & \fX \ar[d]^\pi \\
        X \ar[r]^{\phi'_2} & X }
    \end{aligned}
  \end{equation}
  Here $p_2$ is a projective bundle and $\pi$ is a locally trivial fiber
  bundle (see Section~\ref{sec:5.1}). If $x = \langle \ell \rangle \in X$,
  $y_i = (\ell, \cK_i) \in \pi^{-1}(x)$, then there is a connected curve $B
  \subset \Hilb^c(\P^2)$, which contains $\cK_1$ and $\cK_2$, hence $C =
  \langle \ell \rangle \times B \subset \fX$ connects $y_1$ and $y_2$ and if
  $f$ is any suitable form, $D=(\ell, f) \times C \subset \cG$ is a
  connected curve such that $p_2(D) = C$. It follows that $\phi_2$ and
  $\phi'_2$ exist. In a similar way one gets
  \begin{equation}
    \label{eq:5.13}
    \begin{aligned}
         \xymatrix{
        \cG \ar[r]^\phi \ar@{=}[d] & \cG \ar@{=}[d] \\
        \cG \ar[r]^\phi \ar[d]_{p_1} & \cG  \ar[d]^{p_1} \\
        \cH \ar[r]^{\phi_1} \ar[d]_\pi & \cH \ar[d]^\pi \\
        X \ar[r]^{\phi'_1} & X }
    \end{aligned}
  \end{equation}
and running through the diagrams gives: 
\[
\phi'_1 = \phi'_2\,.
\]
\subsection{}
\label{sec:5.4.3}

As to $H_m$, one has the diagram:
  \begin{equation}
    \label{eq:5.14}
    \begin{aligned}
      \xymatrix{
        H_m \ar[r]^\phi \ar[d]_{p_2} & H_m\ar[d]^{p_2} \\
        Z \ar[r]^{\phi_2} \ar[d]_q & Z \ar[d]^q \\
        Y \ar[r]^{\phi_3} \ar[d]_p & Y\ar[d]^p \\
        X \ar[r]^{\phi'_2, \phi'_3} & X }
    \end{aligned}
  \end{equation}
The construction of $\phi_2$ and $\phi'_2= (\phi, \phi_2)$ goes as
in~\ref{sec:5.4.2} if one takes the irreducible subscheme $F$ instead of
$\Hilb^c(\P^2)$ (see Section~\ref{sec:5.1}).  If $y = (\ell, h) \in Y$ and
$z_i = (\ell, h, g_i) \in q^{-1}(y)$, then $C = \Set{ (\ell, h, \alpha g_1 +
  \beta g_2) }^- \simeq \P^1$ connects $z_1$ and $z_2$ and $p_2$ maps $D =
(\ell, f) \times C$ onto $C$. Hence $\phi_3 =(\phi, \phi_2)$ exists. 

If $x = \langle \ell \rangle \in X$, $y_1 = (\ell, h_1)$, $y_2 = (\ell,
h_2)$, take $z_i =(\ell, h_i, g_i) \in Z$ and a connected curve $B \subset
F$, which contains $(h_i, g_i)$. Then $D = (\ell, f) \times (\langle \ell
\rangle \times B)$ is a connected curve in $H_m$ such that $C:= q p_2(D)$ contains
$y_i$. It follows that $\phi'_3 =(\phi, \phi_3)$ exists and one checks that 
\[
  \phi'_2 = \phi'_3\,.
\]
If in the diagram~\eqref{eq:5.13} $\cG$ is replaced by $H_m$, one gets an
automorphism of $\cH$ and one checks again that it agrees with the $\phi_1$
of~\eqref{eq:5.13}.

\begin{conclusion}
   \label{conc:5.8}
$\phi \in \Aut(\HH)$ induces the same $\gamma \in \Aut_k(X)$ in all
diagrams. \hfill $\qed$
\end{conclusion}
\subsection{}
\label{sec:5.4.4}

Now we take this $\gamma$ and form all diagrams with $\gamma^{-1} \in
\Aut(\HH)$ instead of $\phi$. One obtains
diagrams~\eqref{eq:5.12}--\eqref{eq:5.14} such that all horizontal arrows
are equal to $\gamma^{-1}$.  Putting $\tilde{\phi} = \gamma^{-1}\phi$ (or
$\tilde{\phi} = \phi \circ \gamma^{-1}$) from Lemma~\ref{lem:5.1} Part (iv), it
follows that all diagrams, with $\tilde{\phi}$ instead of $\phi$, induce the
identical map of $X$. But then Proposition~\ref{prop:5.1} and
Corollary~\ref{cor:5.1} show that $\phi_1 = \id$ and $\phi_3 = \id$ in the
diagram~\eqref{eq:5.14}, hence $\phi_2=\id$ in the same diagram by
Proposition~\ref{prop:5.2}. As $H_m \stackrel{\sim}{\longrightarrow} \cH
\times_X Z$ one obtains

\begin{conclusion}
   \label{conc:5.9}
 $\tilde{\phi}| H_m =\id$. \hfill $\qed$
\end{conclusion}
\subsection{}
\label{sec:5.4.5}

From the commutative diagram 
\[
 \xymatrix{
 \cG \ar[r]^-\sim & \cH \times_X \fX \ar[d]_{p_1} \ar[r]^{\tilde{\phi}}  & \cH \times_X \fX \ar[d]^{p_1}\\
                 & \cH \ar[r]^{\id} & \cH 
 }
\]
it follows that $\tilde{\phi}$ induces an element of $\Aut_X(\fX)$, hence
for each $\langle \ell \rangle \in X$ an automorphism of $\Hilb^c(\Proj
S/\ell \cdot S(-1))$. By Theorem~\ref{thm:D.2} in Appendix~\ref{cha:D}, 
Section~\ref{sec:D.10}, it is induced by a linear map $\tau \in \Aut_k(S_1/
\ell \cdot k)$, if $c = b-a+1 \geq 6$. But as $\tilde{\phi}|H_m = \id$,
$\tau$ leaves fixed each ideal $(h,g) \subset S/\ell S(-1)$, hence $\tau$
acts as the identity on $\Proj(S/\ell S(-1))$ (cf. the Lemma in
Appendix~\ref{cha:B}).  One verifies that $b-a+1 \geq 6$, if $d \geq 6$ and
$g \leq g(d)$ is supposed.
\begin{proposition}
   \label{prop:5.3}
 Suppose that $d \geq 6$ and $g \leq g(d)$. Let be $\phi \in
 \Aut(\HH)$. Then there is a uniquely determined $\gamma \in \PGL(4;k)$
 such that $\phi| H_m$ and $\phi| \cG$ are induced by $\gamma$.
\hfill $\qed$
 \end{proposition}

From this result one easily gets:
 \begin{corollary}
 \label{cor:5.2}
  For each $\phi \in \Aut(\HH)$ the following conditions are equivalent:
     \begin{enumerate}[(i)]
     \item $\phi| H_m =\id$.
     \item $\phi|\cG =\id$.
     \item If $\gamma$ is the element of $\Aut_k(\P^3_k)$, which is determined by
       $\phi$ in the sense of Proposition~\ref{prop:5.3}, then $\gamma=\id$.
     \end{enumerate}
		 \end{corollary}
			\begin{definition}
			\label{def:5.1}
			\index{normed automorphism}
			We say $\phi\in\Aut(\HH)$ is \emph{normed}, if these conditions are fulfilled. 
			\end{definition}
			\begin{corollary}
			\label{cor:5.3}
      The set $N$ of all normed automorphisms is a normal subgroup of $\Aut(\HH)$
	    and $\Aut(\HH)$ is the semi-direct product of $\PGL(4;k)$ and $N$.
   \end{corollary}
  \begin{proof} It remains to show that $\PGL(4;k)\cap N = \Set{\id}$. A closed point $\xi\in
	H_m$ corresponds to an ideal of the form $(\ell,f(p,q))$ (see Appendix~\ref{cha:C}).
	If $g\in\GL(4;k)$ leaves all such $\xi$ fixed, it follows that $g(\ell\cdot k)=\ell\cdot k$
	for all non-zero linear forms $\ell$, i.e. all such forms are eigenvectors of $g$.
	But then $g$ has to be the identity in $\PGL(4;k)$.
	\end{proof}
	

\chapter{The action of $\Aut(\HH)$ on linear configuration ideals}
\label{cha:6}

\section{The case of simple lines}
\label{sec:6.1}

\subsection{Notations and assumptions}
\label{sec:6.1.1}

We recall from earlier chapters that $f:\HH \to \P$ is the so called
tautological morphism, which is defined by the globally generated line
bundle $\cL_1 \otimes \cL_2$ (respectively $\cM^{-1}_{n-1}\otimes \cM_n$, if
$n\geq d$ is any integer). In order to simplify the notation, if $\xi_1, \xi_2
\in \HH(k)$, then we write $\xi_1 \equiv \xi_2$ iff $f(\xi_1)=f(\xi_2)$.

$\phi$ is any \emph{normed} automorphism of $\HH$ and $\psi:\CC \to \CC$ is
the induced automorphism of the universal curve. As we will use the results
of Chapter~\ref{cha:5}, we have to assume $d\geq 6$.

\subsection{}
\label{sec:6.1.2}

\begin{lemma}
 \label{lem:6.1}
  Let be $g\in \GL(4,k)$ and $\phi\in \Aut(\HH)$ a normed
  automorphism. Then one has:
  \begin{enumerate}[(i)]
  \item $\xi_1 \equiv \xi_2 \iff g(\xi_1)\equiv g(\xi_2)$.
  \item $\xi_1 \equiv \xi_2 \iff \phi(\xi_1)\equiv \phi(\xi_2)$.
  \end{enumerate}
\end{lemma}
\begin{proof}
  Suppose $\xi_1 \equiv \xi_2$. By Lemma~\ref{lem:3.2} in
  Chapter~\ref{cha:3} there is a connected curve $C \subset \HH$ such that
  $\xi_i \in C, i=1,2$, and $C \sim \nu \cdot C_0$. For all curves $C
  \subset \HH$ one has $[g(C)]=[C]$ respectively $[\phi(C)]=[C]$ if $g\in
  \GL(4;k)$ respectively $\phi \in \Aut(\HH)$
  (cf.~Proposition~\ref{prop:4.1}). As $\deg(f|C) \cdot [f(C)] = f_*[C] =
  \nu \cdot \deg(f|C_0)\cdot [f(C_0)] = 0$, $f(C)$ is a single point, hence
  ``$\Rightarrow$'' is proved. Applying $g^{-1}$ respectively $\phi^{-1}$
  gives ``$\Leftarrow$''.
\end{proof}

\emph{Standard assumption} (\textbf{A}): \index{assumption! standard
  \textbf{A}} Given $d$ distinct, simple lines $\ell_i$ in $X:=\P^3_k$,
which are perpendicular to the plane $E=V(t)$, i.e. they run through the point $P_0= (0:0:0:1)$.
(The term "perpendicular" is used to give a somewhat geometric impression.) 
  Moreover, let $P_j$,$1\leq j \leq c$, be different simple points in $X$, such that no $P_j$ lies
on any $\ell_i$. It is assumed that this configuration defines a point $\xi
\leftrightarrow \ell_1 \cup \cdots \cup \ell_d \cup P_1 \cup \cdots \cup P_c
\in \HH(k)$.

Let $L$ be a plane ``perpendicular'' to $E$, i.e. not equal or parallel to $E$.The ``perpendicular''
projection $\pi_L= (Z,L)$ from a point $Z$ not in $L$ onto $L$ is defined by a suitable $\G_m$-operation $\tau$,
such that 
\[
   \pi_L(P) = \lim_{\lambda \to \infty} \tau(\lambda) P = \lim_{\lambda \to 0} \tau(\lambda^{-1}) P
\]
(see Appendix~\ref{cha:A}). 

\emph{Additional assumption} (\textbf{A1}): \index{assumption! \textbf{A1}}
Under the projection $\pi_L$
onto $L$, the images $\ell'_i = \pi_L(\ell_i)$, respectively $P'_i =
\pi_L(P_i)$, are different from each other and $P'_j \centernot\in \ell'_i$
for all $i$ and $j$.

We put $\xi(\lambda):= \tau(\lambda)\xi$ and get a curve $\cC :=\Set{
  \xi(\lambda) | \lambda \in k^*}^- \subset \HH$, which connects $\xi$ and
$\xi_\infty:= \lim_{\lambda \to \infty} \tau(\lambda)\xi$. Hence $\cD :=
\phi(\cC)$ connects the points $\phi(\xi)$ and $\phi(\xi_\infty)$.

\subsection{}
\label{sec:6.1.3}

\begin{auxlemma}
 \label{auxlem:6.1}
 Under the assumptions \emph{\textbf{(A)}} and \emph{\textbf{(A1)}} one has 
 $\phi(\xi_\infty) \equiv \xi_\infty$.
\end{auxlemma}
\begin{proof}
  If $\xi_\infty \in \cG$ (notation as in the last chapter), this would follow from
  Corollary~\ref{cor:5.2}. Now definitely $\xi_\infty \centernot\in \cG$,
  as is to be shown by the following consideration:

  The lines $\ell_1,\dots, \ell_d$ intersect in $P_0 \leftrightarrow
  (x,y,z)$ and they intersect $E = V(t)$ in the closed points $p_i
  \leftrightarrow \fp_i \subset k[x,y,z]$. Then $\xi \leftrightarrow \cI =
  \fp^*_1 \cap\dots\cap \fp^*_d \cap P_1 \cap\dots\cap P_c = \cN \cap \cR$,
  where $\fp^*_i$ is the ideal generated by $\fp_i$ in $k[x,y,z,t]$, $\cN =
  \fp^*_1 \cap\dots\cap \fp^*_d$ is the $\CM$-part of $\cI$ and $\cR = P_1
  \cap\dots\cap P_c$ is the punctual part of $\cI$. (We identify a closed
  point $P \in X$ with the corresponding ideal.)

 As $\xi\in \HH(k)$, $P(n) = \HP(\cO_X/\cN)-c$. Let $p'_i$ be the projection
 of $p_i$ on $L$, $\fp'_i \leftrightarrow p'_i$ the corresponding prime
 ideal. Then $\cN':=(\fp'_1)^* \cap\dots\cap (\fp'_d)^*$ is the $\CM$-part of
 $\cI_\infty \leftrightarrow \xi_\infty$ and one can write $\cI_\infty =
 \cN' \cap \cR'$, $\cR'$ the punctual part of $\cI_\infty$. Put
 $P'_i:=\pi(P_i) = \lim_{\lambda \to \infty}\tau(\lambda)P_i$. From
 $\tau(\lambda)\cI \subset \tau(\lambda)\cN \cap \tau(\lambda) P_1
 \cap\dots\cap \tau(\lambda)P_c$ follows $\cI_\infty \subset \cN' \cap P'_1
 \cap\dots\cap P'_c$.

Let be $\chi$ the Hilbert function of $\fp_1 \cap\dots\cap \fp_d$ and
$\chi'$ the Hilbert function of $\fp'_1 \cap\cdots\cap \fp'_d$. As the $p'_i$
lie on the line $L \cap E$, one has $\chi'(n)\geq \chi(n)$, hence the
Hilbert polynomial $\sum^n_{0}\chi'(i)$ of $\cN'$ is greater or equal the
Hilbert polynomial $\sum^n_{0}\chi(i)$ of $\cN$. It follows that 
\[
 \HP(\cO_X/\cN' \cap P'_1 \cap\dots\cap P'_c) = \HP(\cO_X/\cN') +c \leq 
\HP(\cO_X/\cN) +c = P(n)\,.
\]
From this we deduce that the punctual part of $\cI_\infty$ has the form $Q
\cap P'_1 \cap\dots\cap P'_c$, where $Q$ is primary to $P_0 =
(x,y,z)$. \\
(\textbf{N.B.} $P'_i \centernot\in \ell'_i$ by the choice of
$\pi$ and $P_0\in \ell'_i$ for all $i$ gives $P'_i \neq P_0$ for all $i$.)

Put $\mu:= \HP(\cN')-\HP(\cN) = \HP(\cO_X/\cN) - \HP(\cO_X/\cN')$. From
\[
  \cO_X/\cI_\infty \cong \cO_X/\cN' \cap Q \bigoplus^c_1 \cO_X/P'_i
\]
it follows that $P(n)= \HP(\cO_X/\cN' \cap Q) +c = \HP(\cO_X/\cN)+c$.
The exact sequence 
\[
 0 \longrightarrow \cN'/\cN' \cap Q  \longrightarrow \cO_X/\cN' \cap Q  \longrightarrow \cO_X/\cN' \longrightarrow 0
\]
gives $\length(\cN'/\cN' \cap Q)= \mu$.

Choose $\mu$ simple points $R_i \in L$, which do not lie on any line
$\ell'_j$ and are not equal to any of the points $P_0,P'_1,\dots,P'_c$. Put
$R_i(\lambda) = R_i + \lambda(P_0-R_i)$ and $\zeta(\lambda) \leftrightarrow
\cN' \cap R_1(\lambda) \cap\dots\cap R_\mu(\lambda) \cap P'_1 \cap\dots\cap
P'_c$. If $\lambda \in k$, $\lambda\neq 1$, this is a point of $\cG(k)$,
hence $\zeta(\lambda)$ is invariant under $\phi$. As $\phi$ is continuous,
it follows that $\zeta_1 := \lim_{\lambda \to 1}\zeta(\lambda)
\leftrightarrow \cI_1 = \cN' \cap Q_1 \cap P'_1 \cap\dots\cap P'_c$ is fixed
under $\phi$.  As $\supp(\cO_X/Q_1) = \{ P_0\}$, $\cN' \cap P'_1 \cap\dots\cap
P'_c/\cI_1 \simeq \cN'/\cN' \cap Q_1$ has the support $\{P_0\}$ and the
length $\mu$. Hence the $\CM$-parts of $\cI_\infty$ and $\cI_1$ are equal to
$\cN'$ and 
\[
\langle \cN'/ \cI_\infty\rangle = \mu \cdot P_0 + \sum^c_1 P'_i = \langle \cN'/ \cI_1\rangle\,.
\]
From Proposition~\ref{prop:3.1} it follows that $\xi_\infty \equiv \zeta_1$,
hence $\phi(\xi_\infty) \equiv \phi(\zeta_1) = \zeta_1 \equiv \xi_\infty$ by
Lemma~\ref{lem:6.1} and Proposition~\ref{prop:5.3}.  
\end{proof}

\subsection{}
\label{sec:6.1.4}
\textsc{Case 1}. 
 Suppose that $\xi$ fulfills \textbf{(A)}. Let $\cL$ be the set of planes
 $L \subset X$ such that $\pi_L$ fulfills \textbf{(A1)} and in addition:
 $\ell_1 \subset L$, $\ell_i \centernot\subset L$, if $i\geq 2$, $P_j
 \centernot\in L$ for all $1 \leq j \leq  c$.

 Let be $\xi(\lambda)\leftrightarrow \tau(\lambda)\ell_1 \cup\dots\cup
 \tau(\lambda)\ell_d \cup \tau(\lambda)P_1 \dcup\cdots\dcup
 \tau(\lambda)P_c$, where the $\G_m$-operation $\tau$ is defined by a
  ``perpendicular'' projection $\pi_L$ onto $L\in\cL$. Let be
 $\cC:=\Set{\xi(\lambda)}^-$ and $\cD=\phi(\cC)=\Set{\phi\xi(\lambda)}^-$.
 Let $p \in \ell_1$ be any point. Then $\cC^*:=\Set{(\xi(\lambda),p)}^-
 \subset \CC$ is a curve without $L^*$-component, hence
 $\psi(\cC^*):=\Set{\psi(\xi(\lambda),p)}^- =
 \Set{(\phi\xi(\lambda),\phi_{\xi(\lambda)}(p))}^-$ has no
 $L^*$-component, too (Prop.~\ref{prop:4.3} ) and according to the
notation introduced in Section~\ref{sec:4.3.1} we write:
\begin{equation}
  \label{eq:6.1}
  \phi_{\xi(\lambda)}(p) = \phi_\xi(p) \quad \text{ for all } 
   \lambda \in \P^1\,.
\end{equation}
Now $|\cC_\infty| = \ell'_1 \cup\dots\cup \ell'_d \dcup P'_1 \dcup\dots\dcup
P'_d$ and $\ell'_1 = \ell_1$ by construction, and $|\cC_\infty| =
|\cD_\infty|$ by Aux-Lemma~\ref{auxlem:6.1}. As $\phi_{\xi_\infty}$ induces an
isomorphism $|\cC_\infty| \simeq |\cD_\infty|$, from \eqref{eq:6.1} it follows 
$\phi_\xi(p) = \phi_{\xi_\infty}(p) \in \phi_{\xi_\infty}(\ell_1) =
\ell'_i$, where the index $i \in \{1,\dots,d\}$ does not depend on $p$. As
$\psi$ induces an isomorphism $|\cC_\lambda| \simeq |\cD_\lambda|$, for
$\lambda=1$ it follows that  $\phi_\xi(\ell_1) \subset \ell'_i$,
hence $\phi_\xi(\ell_1) =\ell'_i$. If $\cJ \leftrightarrow \phi(\xi)$, then 
\[
  |V(\cJ)| = |\cD_1| = \phi_\xi(\ell_1) \cup\dots\cup \phi_\xi(\ell_d) \cup
\phi_\xi(P_1)\cup\dots\cup \phi_\xi(P_d)\,,
\]
hence $\ell'_i \subset V(\cJ)$, where $i$ still depends on the projection
 $\pi_L$ on $L$. Hence there is an index $i \in \Set{1,\dots,d}$ such that 
\[
  \pi_L(\ell_i) \subset V(\cJ) \quad \text{for Zariski-many } L \in \cL\,. 
\]
 It follows that $i=1$, i.e. $\phi_\xi(\ell_1)=\ell_1$.

 If one chooses $\ell_2$ instead of $\ell_1$, the same argumentation shows
 $\phi_\xi(\ell_2) = \ell_2$, etc. As $\psi$ induces an isomorphism, one gets
\begin{conclusion}
  \label{conc:6.1}
  If $\xi$ fulfills \textbf{(A)} it follows that $\phi(\xi) \leftrightarrow
  \ell_1 \cup \cdots \cup \ell_d \dcup R_1 \dcup \cdots \dcup R_c$,
  where the $R_i$ are different, isolated simple points. \hfill $\qed$
\end{conclusion}

\subsection{}
\label{sec:6.1.5}
\textsc{Case 2}. Suppose $\xi$ fulfills \textbf{(A)} and in addition the
following assumption \textbf{(A2)}: \index{assumption!\textbf{A2}}
{\itshape 
 If one perpendicularly projects $P_i$ to the plane $E =V(t)$, one obtains
 $c$ different image points.
}\\  
\textbf{N.B.} One should mention here that condition \textbf{(A1)} refers to
the projection $\pi_L$, whereas condition \textbf{(A2)} refers to the point
$\xi$.

Let $a$ be the line through $P_1$ and $P_0 =(x,y,z)$. Let $\cL$ be the set
of planes $L$, which contain $a$ and fulfill the condition \textbf{(A1)}.

Let $\tau$ be again the $\G_m$-operation defined by the projection $\pi_L$,
$L \in \cL$. Because of $\tau(\lambda)P_1 = P_1$ one has
\[
 \xi(\lambda) = \tau(\lambda)\xi \leftrightarrow \tau(\lambda)\ell_1
 \cup\dots\cup \tau(\lambda)\ell_d \cup P_1 \cup \tau(\lambda)P_2
 \cup\dots\cup \tau(\lambda)P_c\,.
\]
From Conclusion~\ref{conc:6.1} follows
\[
 \phi\xi(\lambda) \leftrightarrow \tau(\lambda)\ell_1
 \cup\dots\cup \tau(\lambda)\ell_d \dcup \cP_1(\lambda) \dcup \cdots \dcup \cP_c(\lambda)\,,
\] 
where $\cP_i(\lambda):=\phi_{\xi(\lambda)}(\tau(\lambda)P_i)$ are $c$
distinct simple points.

If $\cC := \Set{\xi(\lambda)}^-$ and $\cC^*_i := \Set{(\xi(\lambda),\tau(\lambda)P_i)}^-$, then one can write $\cC \sim q_2C_2 + q_1 C_1 +q_0 C_0$ and
$\cC^*_i \sim q^*_2 C^*_2 + q^*_1 C^*_1 +q^*_0 C^*_0 + q \cdot L^*$ (cf.~Theorem~\ref{thm:1.2}). Applying $\pi_*$ and $\kappa_*$ (see Section~\ref{sec:1.1})
shows that $q_j = q^*_j$, $0 \leq j \leq 2$, and $q=0$, if $i=1$ respectively $q=1$, if $i\geq 2$.
If $\cD^*_i:=\psi(\cC^*_i) = \Set{ ( \phi\xi(\lambda), \cP_i(\lambda))}^-$,
then $[\cD^*_i] = [\cC^*_i]$ by Proposition~\ref{prop:4.3}.  It follows that
\[
 \kappa_*[\cD^*_i] = \deg(\kappa|\cD^*_i) \cdot [\{ \cP_i(\lambda)\}^-] =
 \kappa_*[L^*] = [L]\,,
\] 
hence $\deg(\kappa|\cD^*_i)=1$ and $\Set{ \cP_i(\lambda)}^-\subset X$ is a
line, if $i\geq 2$.  As $\tau(\lambda)P_1 = P_1$, in $[\cD^*_1]$ the term
$[L^*]$ is missing, hence $\kappa_*[\cD^*_1]=0$.  From this one deduces that
$\cP_1(\lambda)=:\cP_1$ is independent of $\lambda$ and hence $\cP_1=
\phi_\xi(P_1)$. As one has $\phi\tau(\lambda) \leftrightarrow : \cJ_\lambda
\subset \cP_1(\lambda)$ for all $\lambda$, it follows that $\cJ_\infty =
\lim_{\lambda\to \infty} \cJ_\lambda\leftrightarrow \phi(\xi_\infty)$ is
contained in $\cP_1$. Now by Aux-Lemma~\ref{auxlem:6.1}
\[
 \phi(\xi_\infty) \equiv \xi_\infty \leftrightarrow \cI_\infty = (\fp'_1)^*
 \cap\dots\cap  (\fp'_d)^* \cap P'_1 \cap\dots\cap P'_c \cap Q\,,
\]
where $Q$ is primary to $P_0 = (x,y,z)$ (see Section~\ref{sec:6.1.3}). It
follows that $\cP_1 \in V(\cI_\infty)\subset L$. This holds true for the
Zariski-many planes $L\in \cL$, hence $\cP_1$ is contained in the
intersection of these planes and it follows that $\cP_1 \in a$. Now $P_1 =
P'_1\in a$ and $P'_1 \centernot\in \ell'_i = V((\fp'_i)^*)$ by
Assumption~\textbf{(A1)}, hence $a \neq \ell'_i$.
If one assumes $P_0 = \cP_1 = \phi_\xi(P_1)$, this gives a contradiction of
Conclusion~\ref{conc:6.1}. It follows that $\cP_1$ does not lie on any line
$\ell'_i$, hence $(\fp'_i)^*\centernot\subset \cP_1$. From $\cI_\infty \subset
\cP_1$ it follows that $\cP_1 \in \Set{P'_1,P'_2,\dots,P'_c,P_0}$. As $P'_i
\centernot\in a$ if $i\geq 2$ by Assumption \textbf{(A2)}, it follows that
$\phi_\xi(P_1)= \cP_1 = P'_1 = P_1$.  The same argumentation with
$P_2,\dots,P_c$ gives $\phi_\xi(P_i)= \cP_i$ for all $i$.

\begin{conclusion}
    \label{conc:6.2}
 If $\xi$ fulfills \textbf{(A)} and \textbf{(A2)}, then $\phi(\xi) = \xi$.
\hfill $\qed$
\end{conclusion}
\subsection{}
\label{sec:6.1.6}
If $P_j$ are any different simple points such that $P_j \centernot\in
\ell_i$ for all $i$ and $j$, then one chooses points $R_i\in X$ in general
position and puts $P_i(\lambda) = P_i + \lambda (R_i-P_i)$. Then for almost
all $\lambda \in k$ one has a point $\xi(\lambda) \leftrightarrow \ell_1
\cup \cdots \cup \ell_d \dcup P_1(\lambda) \dcup\cdots \dcup P_c(\lambda)
\in \HH(k)$, which fulfills \textbf{(A)} and \textbf{(A2)}. Then from
Conclusion~\ref{conc:6.2} follows:
\[
  \phi(\xi) = \phi(\xi(0)) := \phi\bigl(\lim_{\lambda\to
    0}\xi(\lambda)\bigr) = \lim_{\lambda\to 0}\phi\xi(\lambda) = 
\lim_{\lambda\to
    0}\xi(\lambda) = \xi\,.
\] 

\begin{conclusion}
    \label{conc:6.3}
 If $\xi$ fulfills the assumption \textbf{(A)}, then $\phi(\xi) = \xi$.
\hfill $\qed$
\end{conclusion}

\begin{lemma}
  \label{lem:6.2}
  Let be $\ell_i$, $1 \leq i \leq d$, different simple lines, all running
  through one and the same point $P$. Let $P_j$, $1 \leq i \leq c$, be
  simple points, different from each other and none of them lying on a line
  $\ell_i$. Assume that $\xi \leftrightarrow \ell_1 \cup \cdots \cup \ell_d
  \dcup P_1 \dcup\cdots \dcup P_c$ is in $\HH(k)$. Then $\phi(\xi) =
  \xi$ for each normed automorphism $\phi$.
\end{lemma}
\begin{proof}
  Take a linear form $\ell$ such that $P \centernot\in V(\ell)$ and $\ell_i$
  is not contained in $V(\ell)$ for all $i$. Choose $g \in \GL(4,k)$ such
  that $g(\ell) = t$ and $g(P)= P_0 = (0:0:0:1)$. Then $g(\xi)$ fulfills the
  assumption \textbf{(A)}, hence $\phi g(\xi) = g(\xi)$ for all $\phi\in N$
  by Conclusion~\ref{conc:6.3}. As $g^{-1}N g = N$ for all $g \in \GL(4,k)$ 
	(cf. Definition in Chapter~\ref{cha:5}) the assertion follows.
\end{proof}

\section{The case of multiple lines}
\label{sec:6.2}

In order to simplify the notation, we put $X = \P^3_k= Proj(k[x,y,z,t]$, $Y= \P^2_k= Proj(k[x,y,z]$, $H^d =
\Hilb^d(Y)$. Let $\cJ \subset \cO_Y$ be an ideal of colength $d$ and Hilbert
function $\psi$. If $\cJ^* \subset \cO_X$ is the ideal generated by $\cJ$,
then $H^0(\cJ^*(n)) = \bigoplus^n_{i=0} t^{n-i} H^0(\cJ(i))$, hence $\cJ^*$
has the Hilbert polynomial $\Psi(n) = \sum^n_{i=0}\psi(i)$ and $\cO_X/\cJ^*$
has the Hilbert polynomial $p(n) = \binom{n+3}{3} -\Psi(n)$.
\begin{lemma}
  \label{lem:6.3}
  If $p(n)-P(n)=:s \geq 0$ and $P_i \in X - V(\cJ^*)$ are $s$ distinct simple
  points, then $\cJ^* \cap P_1 \cap \cdots \cap P_s$ defines a point $\xi
  \in \HH(k)$ and $\phi(\xi)=\xi$ for all normed $\phi \in \Aut(\HH)$.
\end{lemma}
\begin{proof}
  Let $H_\psi$ be the subscheme (with the induced reduced structure) of
  $H^d$, whose closed points correspond to ideals $\cJ \subset \cO_Y$ 
  colength $d$ with Hilbert function $\psi$. By a theorem of Davis~\cite{D}
  one has: $H_\psi \neq \emptyset \Rightarrow H_\psi \cap H^{(d)} \neq
  \emptyset$. As $H_\psi$ is irreducible~\cite[p. 539]{G88}, $H^{(d)}\cap
  H_\psi$ is dense in $H_\psi$, where $H^{(d)}$ is the open subscheme of $H^d$ introduced in Appendix~\ref{cha:H}.

  We still have to take into account the points $p_i:=\pi(P_i)$ where $\pi:X
  - \Set{(0:0:0:1)} \to Y$ is the projection onto the plane $V(t)$. For this
  reason we replace $Y$ by $E:=Y - \Set{p_1, \dots, p_s}$ and obtain an open
  subscheme $\Hilb^d(E)$ of $H^d$ and an open subscheme $U = H^{(d)} \cap
  H_\psi \cap \Hilb^d(E)$ of $H_\psi$, which is dense in $V:= H_\psi \cap
  \Hilb^d(E)$.

  Let be $\zeta \in V(k)$, i.e.~$\zeta \leftrightarrow \cJ \subset \cO_Y$
  has the Hilbert function $\psi$. If one defines $\cJ^* \subset \cO_X$ by
  $H^0(\cJ^*(n)) = \bigoplus^n_{i=0} t^{n-i} H^0(\cJ(i))$, then
  $\cO_X/\cJ^*$ has the Hilbert polynomial $p(n)$. Let $\eta
  \leftrightarrow P_1 \dcup\cdots \dcup P_s$ and define $\zeta^*
  \dcup \eta \in \HH(k)$ in the obvious manner. Let $f$ be the
  tautological morphism of Chapter~\ref{cha:3} and define morphisms
  $g_i:V \to \P$ by $g_1:\zeta \mapsto f(\zeta^* \dcup \eta)$,
  respectively $g_2:V \to \P$ by $g_2:\zeta \mapsto f(\phi(\zeta^* \dcup
  \eta))$. $g_1$ and $g_2$ agree on the open dense subset $U \subset V$,
  because $\phi(\zeta^* \dcup \eta) = \zeta^* \dcup \eta$ by
  Lemma~\ref{lem:6.2}, hence they agree on $V$.
\end{proof}

\section{The case of multiple points}
\label{sec:6.3}

Set $\cN \subset \cO_X$ be an ideal such that $\HP(\cO_X/\cN) = P(n)-s$, $s$
a positive integer. If $P_i \in X - V(\cN)$, $1 \leq i \leq s$, are distinct
simple points, then $\eta \leftrightarrow \cN \cap P_1 \cap \cdots \cap P_s$
is a point of $\HH(k)$.

\begin{lemma}
  \label{lem:6.4}
Suppose there is an open, non-empty set $U \subset X - V(\cN)$ such that for
all different simple points $\cP_i\in U$, $1\leq i \leq s$, the point
$\zeta \leftrightarrow \cN \cap \cP_1  \cap \cdots \cap \cP_s \in \HH(k)$
fulfills the condition $\phi(\zeta)\equiv \zeta$. If $\cI = \cN \cap 
Q_1 \cap \cdots \cap Q_r$, $Q_i$ is $P_i$-primary, $P_i$ distinct points in
$X(k)$ such that $\cI$ defines a point $\xi \in \HH(k)$, then
$\phi(\xi)\equiv \xi$.
\end{lemma}
\begin{proof}
  If $Q_1 \cap \dots \cap Q_r=: \cR$, then $\cN/\cN \cap\cR =\bigoplus^r_1
  \cN/\cN \cap Q_i$. If $\cN/\cN \cap Q_i$ has the length $\mu_i$, then
  $\sum^r_1 \mu_i =s$. Choose distinct simple points $P^j_i \in U$, $1 \leq
  j \leq \mu_i$, $1 \leq i \leq r$. Then $P^j_i(\lambda):= P_i +
  \lambda(P^j_i-P_i)$ is in $U$ for almost all $\lambda\in k$ and
  $\xi(\lambda) \leftrightarrow \cN \bigcap_{i,j} P^j_i(\lambda) \in \HH(k)$
  for almost all $\lambda$.  Then $\xi_0:=\lim_{\lambda \to 0} \xi(\lambda)
  \in \HH(k)$ and the corresponding ideal is $\cN \cap R_1 \cap \dots \cap
  R_r$, where $R_i$ is $P_i$-primary and $\cN/\cN \cap R_i$ has the length
  $\mu_i$. By Proposition~\ref{prop:3.1} one has $f(\xi) = f(\xi_0)$. By
  assumption one has $f(\phi\xi(\lambda)) = f(\xi(\lambda))$ for almost all
  $\lambda$ and because $f$ and $\phi$ are continuous $f(\phi(\xi_0)) =
  f(\xi_0)$ follows. Using Lemma~\ref{lem:6.1} we get $f(\phi(\xi))
  =f(\phi(\xi_0)) = f(\xi_0) = f(\xi)$.
\end{proof}  

Let be $\cI \leftrightarrow \xi \in U(t)$, $\cI':= \cI + t \cO_X(-1)/ t
\cO_X(-1)$, $\cI_0 \leftrightarrow \xi_0 = \lim_{\lambda \to
  0}\sigma(\lambda)\xi$.  Let $(\cI')^* \subset \cO_X$ be the ideal
generated by $\cI'$. Then $\cI_0 =(\cI')^* \cap \cR_0$ and $\cR_0$ is
$(x,y,z)$-primary (cf.~Appendix~\ref{cha:G}, Lemma~\ref{lem:G.3}). The
$\CM$-part $(\cI')^*$ fulfills the assumption of Lemma~\ref{lem:6.3}, and by
Lemma~\ref{lem:6.4} we get:

\begin{proposition}
  \label{prop:6.1}
  If $\xi \in U(t)$, then $\phi(\xi_0)\equiv \xi_0$ for all $\phi \in N$.
\hfill $\qed$
\end{proposition}
\section{Limits of image points}
\label{sec:6.4}
Let $\xi \in U(t)$ be a closed point and $C \subset X$ the corresponding
curve. Let be $P\in C(k)-V(t)$ and $P\neq P_0=(0:0:0:1)$. Then 
\begin{lemma}
  \begin{enumerate}[(a)]
    \item $\phi_{\sigma(\lambda)\xi}(\sigma(\lambda)P) \xrightarrow[\lambda\to
      0]{} P_0$\,.
    \item $\Set{\phi_{\sigma(\lambda)\xi}(\sigma(\lambda)P)}^-$ is a line in
      $X$ through $P_0$.
   \end{enumerate}
  \label{lem:6.5}
\end{lemma}
\begin{proof}
  $1^\circ$ We modify the proof of Lemma~\ref{lem:6.3} and use the same
  notations. We first treat the case $\xi \leftrightarrow \cJ^* \cap P_1
  \cap \cdots \cap P_s$, where $\cJ \subset \cO_Y$ has the Hilbert function
  $\psi$.  Then $\cJ \in \bar U = H_\psi$ (closure in $H_\psi$). Now
  $H_\psi$ is a rational variety (see~\cite[proof of the theorem on page
  544]{G88}), hence there is a connected curve $A$, which connects the point
  in $H_\psi$, which corresponds to $\cJ$, to a point in $U$.  Hence $\cJ_b
  \leftrightarrow b \in U$ for all $b \in A - \Set{\text{finitely many
      points}}=:B$. Hence $F:= \bigcup \Set{ V(\cJ^*_b) | b \in B}^-$ is a
  surface in $X$.  Let be $P_i \in X - (C \cup F)$ $s$ distinct, simple
  points. Then $\cJ^*_b \cap P_1 \cap \cdots \cap P_s$ fulfills the
  assumptions of Lemma~\ref{lem:6.3}   for all $b\in B$. It follows that 
\[
    \cB:= \Set{\cJ^*_b \cap P_1 \cap \cdots \cap P_s | b \in B }^- \subset \HH
\]
is a connected curve, which contains $\xi$. If $b \in B$, then
$P_0=(0:0:0:1) \in V(\cJ^*_b \cap P_1 \cap \cdots \cap P_s)$, i.e.~$\cJ^*_b
\cap P_1 \cap \cdots \cap P_s \subset (x,y,z)$. As this is a closed
condition, $P_0 \in C_\eta \leftrightarrow \eta$ for all $\eta\in \cB$,
hence $\cB^*:=\Set{(\eta,P_0)| \eta \in \cB} \subset \CC$ is a closed
curve, such that $\cB^* \sim q_2 C^*_2 + q_1 C^*_1 + q_0C^*_0 + 0 \cdot
L^*$. It follows that $\psi(\cB^*)=\Set{(\phi(\eta),\phi_\eta(P_0))| \eta
  \in \cB} \subset \CC$ is a curve in $\CC$ without an $L^*$-component,
hence $\phi_\eta(P_0)$ is independent of $\eta\in \cB$. If $\eta
\leftrightarrow \cJ^*_b \cap P_1 \cap \cdots \cap P_s$ and $b\in B$, then
$\phi(\eta)=\eta$ by Lemma~\ref{lem:6.3}, hence $\psi$ induces an
automorphism of $|C_\eta|= (C_\eta)_\red$ as a set of closed points in $X$,
which is described by $p \mapsto \phi_\eta(p)$. If $b\in B$, then $|C_\eta|$
consists of $d$ distinct lines $\ell_1,\dots,\ell_d$, which all run through
$P_0$, and points $P_i \centernot\in\bigcup \ell_i$. It follows that
$\phi_\eta$ permutes the points and lines and hence $\phi_\eta(P_0)=P_0$. As
has been noted above, $\phi_\eta(P_0)$ is independent of $\eta \in \cB$,
hence $\phi_\eta(P_0)=P_0$ for all $\eta \in \cB$, hence
$\phi_\xi(P_0)=P_0$.

$2^\circ$ Let now be $\cI \leftrightarrow \xi \in U(t)$. Then
$\xi_0:=\lim_{\lambda \to 0}\sigma(\lambda)\xi \leftrightarrow \cJ \cap
\cR$, where $\cJ = (\cI')^*$ and $\cR$ is $P_0 = (x,y,z)$-primary with
$\length(\cJ/\cI)=:s$ (see Appendix~\ref{cha:G}). We take $s$ different
simple points $P_i\in X- V(\cJ)$ and put $P_i(\lambda) = P_0 + \lambda(P_i -
P_0)$. Then $\eta(\lambda) \leftrightarrow \cJ \cap P_1(\lambda)\cap\cdots
\cap P_s(\lambda) \in \HH(k)$ for almost all $\lambda \in k$. If $\cC =
\Set{\eta(\lambda)}^-$, then $\cC^* = \Set{(\eta(\lambda),P_0)}^- \subset
\CC$ is a curve without $L^*$-component, hence $\psi(\cC^*) =
\Set{(\phi(\eta), \phi_\eta(P_0)) | \eta \in \cC }$ also has no
$L^*$-component. This means that $\phi_\eta(P_0)$ is independent of $\eta
\in \cC$. But $\eta(\lambda)$ fulfills the assumption of Part $1^\circ$,
hence $\phi_\eta(P_0) = P_0$ for all $\eta \in \cC$. If $\xi_1 :=
\lim_{\lambda \to 0}\eta(\lambda)$, then $\xi_1 \in \cC$, hence
$\phi_{\xi_1}(P_0) = P_0$.

 $3^\circ$ Now by construction $\xi_1 \leftrightarrow \cJ \cap \cR_1$, where
 $\cR_1$ is $P_0$-primary and $\cJ/\cJ \cap \cR_1$ has length $s$, hence
 $\xi_0 \equiv \xi_1$ by Proposition~\ref{prop:3.1}. By Lemma~\ref{lem:3.2}
 the points $\xi_0$ and $\xi_1$ can be connected by a curve $D \sim \nu
 \cdot C_0$. It follows that $f(\eta) = f(\xi_0)$  for all $\eta \in D$,
 hence the ideal $\cI_\eta \leftrightarrow \eta$ has the same $\CM$-part
 $\cJ$ for all $\eta \in D$, hence $D^*:=\Set{ (\eta, P_0) | \eta \in D}
 \subset \CC$ is a curve without $L^*$-component. It follows that $\psi(D^*)
 = \Set{ (\phi(\eta), \phi_\eta(P_0)) | \eta \in D}$ has no $L^*$-component,
 too. It follows that $\phi_\eta(P_0)$ is independent of $\eta \in D$. Now
 $\phi_{\xi_1}(P_0)=P_0$ by Part  $2^\circ$, hence $\phi_\eta(P_0)=P_0$ for
 all $\eta \in D$, thus $\phi_{\xi_0}(P_0)=P_0$.

 $4^\circ$ If $P \in C - V(t)$ then 
\begin{align*}
 \psi(\sigma(\lambda)\xi,\sigma(\lambda)P) 
= {} & (\phi\sigma(\lambda)\xi, \phi_{\sigma(\lambda)\xi}(P)) \xrightarrow[\lambda\to 0]{}
  \psi\bigl(\lim_{\lambda\to 0}(\sigma(\lambda)\xi, \sigma(\lambda)P)\bigr) \\
 = {} & \psi(\xi_0, P_0) = (\phi(\xi_0),\phi_{\xi_0}(P_0)) = (\phi(\xi_0),P_0)
\end{align*}
 by Part $3^\circ$. Hence one gets (a).

 $5^\circ$ If $\cC :=\Set{\sigma(\lambda)\xi}^-$, $P \in C - V(t)$ and
 $P\neq P_0$, then $\cC ^*:=\Set{(\sigma(\lambda)\xi, \sigma(\lambda)P)}^-$
 has the $L^*$-component $1 \cdot L^*$, hence $\psi(\cC^*) \sim \cC^*$ has
 the same $L^*$-component. Applying $\kappa_*$ shows that
 $\Set{\phi_{\sigma(\lambda)\xi} (\sigma(\lambda)P)}^-\subset X$ is a line,
 which runs through $P_0$, because of Part~$(a)$.
\end{proof}  


\chapter{Automorphisms of $\HH$ and the Hilbert--Chow morphism}
\label{cha:7}

The aim of this chapter is to show:

\begin{theorem}
  \label{thm:7.1}
  If $h: \HH \to \P$ is the Hilbert--Chow morphism, then for all $\xi \in
  \HH(k)$ and all normed morphisms $\phi$ of $\HH$ one has $h(\phi(\xi))=
  h(\xi)$.
\end{theorem}

\section{Notations}
\label{sec:7.1}

$N$ denotes the group of all normed automorphisms of $\HH$; $N$ is
normalized by any $g\in G:=\GL(4,k)$; $|-|$ denotes the set of points, where
``point'' means ``closed point''; $S = k[x,y,z,t]$, $X =\Proj S$; and $\G_m$
acts by $\sigma(\lambda): x\mapsto x,\; y\mapsto y, \; z\mapsto z, \;
t\mapsto \lambda t$.

If $C\subset X$ is a curve and $P \in X - C$, then the cylinder $Z(P,C)$ is
defined to be the union of all lines in $X$, which join a point in $C$ with $P$.
Each $\xi\in \HH(k)$ corresponds to a curve $C\subset X$, and we write $\xi
\leftrightarrow C_1 \cup\dots\cup C_r \cup \{ \text{points} \}$, where $C_i$
are the irreducible components of dimension $1$, $(C_i)_\red = V(\fp_i)$,
$\fp_i \subset S$ graded prime ideal, $C_i$ has degree $d_i$ and
multiplicity $\mu_i$, $\{ \text{points} \}$ denotes the $0$-dimensional
components, embedded or not. A linear form $\ell \in S_1$ is \emph{very
  general} for $C$, if $V(\ell) \cap V(\fp_i)$ consists of $d_i$ simple
points $P_{ij}$ of multiplicity $\mu_i$, $P_{ij} \centernot\in C - C_i$, and
$P_{ij} \centernot\in \{ \text{points} \}$ for all $i$ and $j$. 

In the same way we write $\phi(\xi) \leftrightarrow D = D_1 \cup\dots\cup
D_s \cup \{ \text{points} \}$, $D_j$ the $1$-dimensional components of
multiplicity $\nu_j$.

\section{The irreducible components}
\label{sec:7.2}

\subsection{}
\label{sec:7.2.1}

We want to show that $\{| C_i|\} = \{| D_j|\}$ and assume that there is
an index $i$ such that $C_i \centernot\subset D$. (For simplicity we write
$C$, $D$, $C_i$, $D_j$ instead of $|C|$, $|D|$, $|C_i|$, $|D_j|$ etc.)
Without restriction we assume $C_1 \centernot\subset D$, hence $C'_1:=C_1 -
D$ is open in $C_1$. \\
Then there is a point $P \in X - (C \cup D)$ such that $Z(P,D)\cap C'_1 =
\emptyset$. Then one can find $\ell \in S_1$ such that $\ell$ is very
general for $C$ and $P_{1j}\in C'_1 \cap V(\ell)$ for $1\leq j \leq
d_1$. (The set of such $\ell \in S_1$ forms a Zariski-dense subset of $S_1$.)

Let $\pi$ be the projection from $P$ onto $V(\ell)$. According to
Appendix~\ref{cha:A}, $\pi$ is defined by a $\G_m$-operation $\tau(\lambda)$
and one can find a $g \in G$, such that $g(\ell)=t$, $g(P)= P_0 = (0:0:0:1)$
and $\tau(\lambda)= g^{-1}\sigma(\lambda)g$. From the assumptions follows
that $\xi$ is not invariant under the $\G_m$-operation $\tau(\lambda)$ and
one obtains the following curves in $\HH$: $\cC = \Set{\tau(\lambda)\xi}^-$
and $\cD = \Set{\phi\tau(\lambda)\xi}^-$. If one applies $g$, one obtains
the curves
\[
g(\cC) = \Set{\sigma(\lambda)g\xi}^- \quad \text{and} \quad g(\cD) = 
\Set{g \phi g^{-1}g\tau(\lambda)\xi}^- =
\Set{\tilde{\phi}\sigma(\lambda)g\xi}^-,
\]
where $\tilde{\phi}:= g \phi g^{-1} \in N$. Put $\tilde{\xi}:=
g(\xi)\leftrightarrow g(C)=:\tilde{C}$. \\
Then
\[
  g(D) = g \phi(C) = g \phi g^{-1}g(C) = \tilde{\phi}(\tilde{C})
  \leftrightarrow g \phi(\xi) = \tilde{\phi}(\tilde{\xi})\;.
\]
Now
\begin{align*}
  h(\tilde{\phi}(\tilde{\xi})) = h(\tilde{\xi}) & \Leftrightarrow h(g \phi
  g^{-1}g\xi) = h(g(\xi)) 
   \Leftrightarrow h(g\phi(\xi)) = h(g\xi) \\
  & \Leftrightarrow gh(\phi(\xi)) = gh(\xi)  \Leftrightarrow h(\phi(\xi)) =
  h(\xi) \qquad  \text{ for all } \phi \in N\,.
\end{align*}
Hence it suffices to show the assertion for
$\tilde{\xi}$ and all $\phi \in N$. Now clearly $t$ is very general for
$\tilde{C}$, and as $g(Z(P,D)) = Z(g(P),g(D))$, $g(C)$, $g(D)$, $t$ fulfill
all assumptions as before. Hence we can assume $P = P_0$, $\ell =t$ and the
projection is defined by the $\G_m$-operation $\sigma(\lambda)$. By
construction $P_{1j}\centernot\in Z:=Z(P_0,D)$, $1\leq j \leq d_1$.
\subsection{}
\label{sec:7.2.2}
 Let be $\xi \leftrightarrow \cI$. Then $\xi_0= \lim_{\lambda\to
  0}\sigma(\lambda)\xi \leftrightarrow \cI_0 = (\cI')^* \cap \cR_0$, where
$(\cI')^*$ is the $\CM$-part of $\cI_0$ and $\cR_0$ is $(x,y,z)$-primary
(see Appendix~\ref{cha:G}). Hence the curve $\cC_0 \leftrightarrow \xi_0$
contains the line $\ell_{1j}$, $1 \leq j \leq d_1$, which connects $P_{1j}$
and $P_0$ (at the moment the multiplicities are irrelevant). Let be
$\cD_\lambda \leftrightarrow \phi\sigma(\lambda)\xi$. Then $D = \cD_1$ and
$C = \cC_1$. Now $\phi$ defines an automorphism $\psi$ of $\CC$, which
induces an isomorphism $|C| \simeq |D|$ denoted by $p\mapsto
\phi_\xi(p)$. If $p$ runs through the points of $C$, then $\phi_\xi(p)$
runs through the points of $D$, and the same holds true for $\cC_\lambda$
and $\cD_\lambda$ for all $\lambda \in \P^1$. From Lemma~\ref{lem:6.5} it
follows that all curves $\cD_\lambda$ lie on the cylinder $Z = Z(P_0,D)$. By
Proposition~\ref{prop:6.1} one has $\phi(\xi_0)\equiv \xi_0$, hence $|\cC_0|
= |\cD_0|$. But by construction, the lines $\ell_{1j} \subset |\cC_0|$ do
not lie on $Z$, contradiction.

It follows that each $C_i(k)$ is equal to a $D_j(k)$. As $C(k)\simeq D(k)$,
one has: 

\begin{conclusion}
    \label{concl:7.1}
  $\Set{C_i(k)} = \Set{D_i(k)}$. \hfill $\qed$
\end{conclusion}

As the triple $(\sigma(\lambda)C,P_0,t)$ fulfills the same assumptions as
$(C,P_0,t)$ for all $\lambda \in k^*$, $\cC_\lambda(k) = \cD_\lambda(k)$
except isolated points. Hence the same is true for $\lambda = \infty$, and
because of $\phi(\xi_0) \equiv \xi_0$ one gets:
\begin{conclusion}
  \label{concl:7.2}
  With the exception of isolated points $\cC_\lambda(k) = \cD_\lambda(k)$
  for all $\lambda \in \P^1$, and all the curves lie on the cylinder
  $Z(P_0,C)=Z(P_0,D)$.  \hfill $\qed$
\end{conclusion}

\section{The multiplicities}
\label{sec:7.3}

Choose $\ell \in S_1$ very general for $C$ and $D$. Then again $\ell = t$
without restriction. As $\phi(\xi_0) \equiv \xi_0$
(Proposition~\ref{prop:6.1}) and $\xi_0 \in U(t)$, there is an open set $T
\subset \P^1$ such that $0 \in T$ and $\phi\sigma(\lambda)\xi \in U(t)$ for
all $\lambda \in T$. Let $r$ be the restriction morphism defined by $t$
(cf.~Appendix~\ref{cha:G}). Then $\lambda \mapsto r(\phi \sigma
(\lambda)\xi)$ defines a morphism $T \to \Hilb^d(\P^2)$, i.e.~ a closed
subscheme $Y \subset \P^2 \times_k T$ over $T$, such that for all $\lambda
\in T$
\[
| Y \otimes k(\lambda) | = | \cD_\lambda | \cap V(t) = | \cC_\lambda | \cap
V(t) = \Set{ P_{ij} | 1\leq i \leq r,\, 1\leq j \leq d_i}\,.
\]
It follows that $Y = \coprod Y_{ij}$, $Y_{ij}$ flat over $T$. 

Now the multiplicity of $P_{ij}$ in $\cD_\lambda\cap V(t)$ is equal to the Hilbert polynomial of $Y_{ij}$, and this is equal to a constant $c_{ij}$, independent of $\lambda$. 

From $\phi\sigma(\lambda)\xi \to \phi(\xi_0)$, it follows that
$r(\phi\sigma(\lambda)\xi) \to r(\phi(\xi_0))$. Now by
Lemma~\ref{lem:3.1}, from $\phi(\xi_0)\equiv \xi_0$ it follows that
$r(\phi(\xi_0)) = r(\xi_0)$.  But as $r(\xi_0)= \lim_{\lambda\to
  0}r(\sigma(\lambda)\xi)=r(\xi)$, it follows that $Y \otimes k(0)
\leftrightarrow r(\xi)$. By construction $r(\xi) \leftrightarrow \bigcap
Q_{ij}$, $Q_{ij}$ is $P_{ij}$-primary with multiplicity $\mu_i$. The points
$P_{ij}$, $1 \leq j \leq d_i$ lie on $|D_i| = | C_i|$, hence the
multiplicity of $D_i$ is equal to the multiplicity of $C_i$. It follows that
$h(\phi\sigma(\lambda)\xi)=h(\xi)$ for all $\lambda \in T$, hence for all
$\lambda \in \P^1$ and $h(\phi(\xi))=h(\xi)$ follows.

\chapter{Automorphisms of $\HH$ and the tautological morphism}
\label{cha:8}

\section{Preliminaries}
\label{sec:8.1}

The so called tautological morphism $f_n : \HH \to \P$ is defined by the
globally generated line bundle $\cM^{-1}_{n-1} \otimes \cM_n$, if $n\geq
d$.  If $n=d$, we wrote $f$ instead of $f_n$, but for simplification we now
write $f$ instead of $f_n$, if $n\geq d$ is any number.
We denote $h$ the Hilbert-Chow morphism $\HH \to \P$. The normed
automorphism $\phi$ of $\HH$ induces an automorphism $\psi$ of the universal
curve $\CC$.  We again suppose $d\geq 6$.

We write $S = k[x,y,z,t]$ and $X = \Proj S$. In order to avoid formulas of
too awkward size, 
we often write $S = A[x,y,z,t]$, $X=\Proj(S\otimes A)$ etc., if $A$ is a
$k$-algebra. The letter $T$ stands for $\P^1_k$ or an open subset of
$\P^1_k$. $\G_m$ operates by $\sigma(\lambda): x\mapsto x,\; y\mapsto y, \;
z\mapsto z, \; t\mapsto \lambda t$.

Let $N$ be the subgroup of all normed automorphisms of $\HH$.  An essential
property of $N$ is $g^{-1}Ng= N$ for all $g\in G:=\Aut_k(S_1)$. 

As usual the idea is to produce, by means of a suitable projection, for a
point $\xi \in \HH(k)$ a curve $\cC \subset \HH$ (respectively curves $\cC^*
\subset \CC$). If $\phi \in N$, then $[\cC] = [\phi(\cC)]$ (respectively
$[\cC^*] = [\psi(\cC^*)])$ by Proposition~\ref{prop:4.1} and
Proposition~\ref{prop:4.2}. (This is the reason why one has to require
$d\geq 6$.) More concretely, the procedure goes as follows: Let be $\xi \in
\HH(k)$ and $C \subset X$ the corresponding curve.  Choose a point $P
\centernot\in C$ and $\ell \in S_1$ such that $P \centernot\in V(\ell)$ and
$\ell$ is general for $C$. Take $g \in G$ such that $g(\ell) = t$ and $g(P)=
P_0 = (0:0:0:1)$ and put $g(\xi)=\tilde{\xi}$. By Lemma~\ref{lem:6.1} one
has $\phi(\tilde{\xi}) \equiv \tilde{\xi}$ for all $\phi \in N
\Leftrightarrow g^{-1}\phi g\xi\equiv \xi$ for all $\phi \in N
\Leftrightarrow \phi(\xi)\equiv \xi$ for all $\phi\in N$. Hence we can
assume that $\xi\in U(t)$ and $P_0 \centernot\in C \leftrightarrow \xi$.

Let be $\xi \leftrightarrow \cI \subset \cO_X$ and take any $\ell \in
k[x,y,z]_1$ such that $\ell$ is not a zero-divisor of $\cO_X/\cI$ (the set
of such linear forms is dense in $k[x,y,z]_1$). Then $t + \alpha \ell$ is
not a zero-divisor of $\cO_X/\cI$ for almost all $\alpha \in k$.  Define
$u_\alpha \in G$ by $x\mapsto x,\; y\mapsto y, \; z\mapsto z, \;
t\mapsto t+\alpha \ell$. Then one still has $P_0 \centernot\in u_\alpha(C)$
and $u_\alpha(\xi) \in U(t)$ for almost all $\alpha \in k$. By
Corollary~\ref{cor:A.2} of Appendix~\ref{cha:A} the $\G_m$-isotropy of
$h(u_\alpha(\xi))$ is trivial for almost all $\alpha \in k$, i.e.\
$\sigma(\lambda)h(u_\alpha(\xi)) = h(u_\alpha(\xi)) \Rightarrow \lambda =1$.

Take any such $\alpha$ and put $\tilde{\xi} = u_\alpha(\xi)$. If we can
prove $\phi(\tilde{\xi}) \equiv \tilde{\xi}$ for all $\phi \in N$, the same
argumentation as before shows $\phi(\xi)\equiv \xi$ for all $\phi\in N$.
Hence we can assume that $\xi \in \HH(k)$ is ``adapted'' in the following
sense:
\begin{definition}
  \label{def:8.1}
\index{adpated point}
 A point $\xi \in \HH(k)$ is adapted, if $\xi \in U(t)$, $P_0 = (0:0:0:1)
 \centernot\in C \leftrightarrow \xi$ and $h(\xi)$ has trivial $\G_m$-isotropy.
\end{definition}

So we assume from Section~\ref{sec:8.2} to Section~\ref{sec:8.6.2} that
$\xi$ is adapted, but from Section~\ref{sec:8.6.3} we do not need this
assumption. 

If $\xi \in \HH(k)$, then $\xi_0 :=\lim_{\lambda \to 0} \sigma(\lambda)\xi \in
\HH(k)$ and $\xi \in U(t) \Leftrightarrow \xi_0 \in U(t)$ (see
Appendix~\ref{cha:G}). If $M \subset X$ is any set, $| M |$ denotes the set
of its closed points. For example, if $C \subset X$ is a curve, in order to
simplify the notation, we write $|C|$ instead of $C(k)$, etc. If $\cC \subset
\HH$ is a curve, one can write $[\cC] = q_2(\cC) \cdot [C_2] + q_1(\cC)
\cdot [C_1] +q_0(\cC) \cdot [C_0]$ where $q_i(\cC) \in \N$
(cf.~Theorem~\ref{thm:1.2}).  If $\xi \in \HH(k)$ and $\cC:=\Set{
  \sigma(\lambda)\xi}^-$, then $q_i(\xi):=q_i(\cC)$ is called the
\emph{complexity} of $\xi$ with regard to $C_i$ (see Appendix~\ref{cha:F}).

\section{Composition series of ideal sheaves}
\label{sec:8.2}
\subsection{Preliminaries}
\label{sec:8.2.1}

Let be $T = \P^1_k$, $\xi \in \HH(k)$, $\alpha: T \to \HH$ the uniquely
determined extension of the morphism $\lambda \mapsto \sigma(\lambda)\xi$,
$\lambda \in k^*$. \\
 We put $\xi(\lambda) = \sigma(\lambda)\xi$,\; $\xi(0)=
\lim_{\lambda \to 0}\sigma(\lambda)\xi$,\; $\xi(\infty) =\lim_{\lambda \to
  \infty}\sigma(\lambda)\xi$. The image of $\alpha$ is the curve $\cC =\Set{
  \xi(\lambda) | \lambda \in T} = \Set{ \sigma(\lambda)\xi | \lambda \in
  k^*}^- \subset \HH$, which, at the same time, is a curve in $X \times T$,
flat over $T$. $\cC$ is defined by an ideal $\cI \subset \cO_{X \times T}$
and $\cI(\lambda):= \cI \otimes_T k(\lambda) \leftrightarrow \xi(\lambda)$.

There is a filtration
\begin{equation}
  \label{eq:8.1}
   0 = \cM^0 \subset \cdots \subset  \cM^\ell = \cO_X/\cI(1)
\end{equation}
such that, possibly after renumbering,
\begin{equation}
  \label{eq:8.2}
  \cM^i / \cM^{i-1} \simeq f_i(S/p_i)(-d_i), \quad 1 \leq i \leq  r,
\end{equation}
and for the remaining indices
\begin{equation}
  \label{eq:8.3}
  \cM^i / \cM^{i-1} \simeq g_i(S/P_i)(-e_i)\,.
\end{equation}
Here $f_i$ and $g_i$ are forms in $S$ of degree $d_i$ respectively $e_i$;
$p_i \subset S$ is a graded prime ideal, which defines a curve in $X$, and
$P_i \subset S$ is a prime ideal, which is generated by a $3$-dimensional
linear subspace of $S_1$, i.e.~$P_i$ is a point in $X(k)$. And in order 
to simplify the notation, we delete $\sim $ (sheafification).

If $\P^1_k = \Proj k[x,y]$, then we write $k[\lambda] = k[x/y]$,
i.e.~$k(\lambda) = k[x/y]/(x/y -\lambda)$, if we take $\lambda$ as a parameter
in $k$.

If we apply $\sigma(\lambda)$ to eq.~\eqref{eq:8.1}, then we get a
filtration of $\cO_{X \times T}/ \cI (\lambda)$ over $T = \Spec k[\lambda]$
with quotients
  \begin{align}
    \label{eq:8.4}
 & f_i(\lambda)(S \otimes T/\sigma(\lambda) p_i)(-d_i), \quad 1 \leq i \leq r,  \\
& g_i(\lambda)(S \otimes T/\sigma(\lambda) P_i)(-e_i)\,,
    \label{eq:8.5}
  \end{align}
where $f_i(\lambda) = \sigma(\lambda)f_i$,\; $g_i(\lambda) = \sigma(\lambda)g_i$
are forms in $S \otimes k[\lambda]$ of degree $d_i$ respectively $e_i$. 

\subsection{Applying an automorphism}
\label{sec:8.2.2}

In the following considerations $\phi$ is any normed automorphism of
$\HH$. $\beta: T \to \HH$ is defined by $\lambda \mapsto \phi
\xi(\lambda)$.  The image of $\beta$ is the curve $\cD := \phi(\cC) = \Set{
  \phi \sigma(\lambda)\xi | \lambda \in k^* }^-$. One can conceive $\cD$ as
a curve in $X \times T$, flat over $T$, which is defined by an ideal $\cJ
\subset \cO_{X \times T}$.

Now we replace $\P^1_k$ by a suitable small open affine subset $T = \Spec A
\subset \P^1_k - \Set{ 0, \infty }$, $A= k[\lambda]_f$, $f \in
k[\lambda]-(0)$ and for simplicity write $X$, $\cJ$, $S$, etc.~instead of
$X\times T$, $\cJ \otimes \cO_T$, $S \otimes A$, etc. Then by the lemma in
Appendix~\ref{cha:E} we get a filtration
\begin{equation}
  \label{eq:8.5second}
   0 = \cM^0 \subset \cdots \subset  \cM^\ell = \cO_X/\cJ
\end{equation}
such that the quotients have the form:
  \begin{equation}
    \label{eq:8.6}
  f_i(S/\fp_i)(-\ell_i), \quad 1 \leq i \leq s,   \\
 \end{equation}
and for the remaining indices
\begin{equation}
   \label{eq:8.7}
 g_i(S/\cP_i)(-m_i)\,.
  \end{equation}
Here all quotients are flat over $T= \Spec A$; \; $\fp_i \subset S$ is a
graded prime ideal, which defines a curve in $X$,\;  $\cP_i \subset S$ is 
a graded prime ideal, generated by a subbundle $\cL_i \subset S_1$ ($= S_1
\otimes A$\,!) of rank $3$; and $f_i$ and $g_i$ are forms in $S$ ($= S
\otimes A$\,!) of degree $\ell_i$ respectively $m_i$.

Let $\fq_1 \cap\dots\cap \fq_n \cap \cR$ be a reduced primary decomposition
of $\cJ$ ($= \cJ \otimes A$\,!), where $\fq_i$ is primary to $\fp_i$ with
multiplicity $\nu_i$, and $\cR$ is the punctual part, i.e.~$\cR$ is the
intersection of ideals, which are primary to associated primes of $\cO_X /
\cJ$ and occur among the $\cP_i$ in eq.~\eqref{eq:8.7}. If $\cP_i$ occurs in
eq.~\eqref{eq:8.7}, then $\cJ \subset \cP_i$ (cf.~\cite[Prop. 7.4,
p. 50]{H}).

As one can choose $T$ sufficiently small, all quotients are flat over $T$
and $f_i$ and $g_i$ generate subbundles of $S_n \otimes A$, and $\cO_X / \cR
$ is flat over $T$ with constant Hilbert polynomial. Hence $\dim_k | V(\cR)|
\leq 1$.
 
Put $D:=\cD \otimes \cO_T$. Then
\begin{equation}
  \label{eq:8.8}
  |D| = \bigcup | V(\fp_i)| \cup | V(\cR)| \cup M\,,
\end{equation}
where $M \subset | X \times T|$ is a finite set. Now $\phi\xi(\lambda)
\leftrightarrow \cD_\lambda = \cD \otimes k(\lambda)$ by definition and
$\cD_\lambda \subset X \otimes k(\lambda)= \P^3$ is defined by
$\cJ(\lambda):=\cJ \otimes k(\lambda)$.

Let be $\xi= \sigma(1)\xi \leftrightarrow \cI (1) = \bigcap q_i \cap R$ a
reduced primary decomposition, $q_i$ primary to a $p_i$ as in
eq.~\eqref{eq:8.2} with multiplicity $\mu_i$, and $R$ the punctual part. It
follows
\[
  | \cC_1 | = \bigcup | V(p_i)| \dcup \Set{P_i}\,,
\]
where $\Set{P_i}$ is a finite set of isolated points in $|V(R)|$, which
therefore are among the $P_i$ of~\eqref{eq:8.3}. It follows that
\[
| \cC_\lambda | = \bigcup | V(\sigma(\lambda)p_i)| \dcup
\Set{\sigma(\lambda)P_i}\,.
\]
Because of $| D_\lambda | \simeq | \cC_\lambda |$ and $h(\xi(\lambda)) =
h(\phi\xi(\lambda))$ (cf.~Thm.~\ref{thm:7.1}) it follows that
\[
| \cD_\lambda | = \bigcup | V(\sigma(\lambda)p_i)| \dcup
\Set{\phi_{\xi(\lambda)}\sigma(\lambda)P_i}\,,
\]
where the $\phi_{\xi(\lambda)}(\sigma(\lambda)P_i)$ again are different
isolated points (as $| \cD_\lambda | \simeq | \cC_\lambda |$).

Suppose there is a $P_i$, which really occurs, e.g. $P_1$. Then
$\cC^*_1:=\Set{(\xi(\lambda), \sigma(\lambda)P_1}^- \subset \CC$ is a curve
with $L^*$-component $1\cdot L^*$. Then $\cD^*_1 = \psi(\cC^*_1)
=\Set{(\phi\xi(\lambda), \phi_{\xi(\lambda)}(\sigma(\lambda)P_1))}^- \subset
\CC$ has the $L^*$-component $1\cdot L^*$, too. Hence
$\kappa(\cD^*_1)=\Set{\phi_{\xi(\lambda)}(\sigma(\lambda)P_1)}^-$ is a line
$L_1\subset X$.  In any case one has
\[
  |\cD \otimes \cO_T | = \bigcup_\lambda | \cD_\lambda| = \bigcup_i
  \bigcup_\lambda  | V(\sigma(\lambda)p_i)| \cup \ell_1 \cup\dots\cup \ell_m \cup E\,, 
\]
where $\lambda$ runs through $T$, $\ell_i \subset X$ are lines minus
finitely many points, $E$ finite set of points (possibly there are no such
lines and $E = \emptyset$). 

If one takes $\lambda$ as a variable, then $\sigma(\lambda)p_i=: \mathbf{p}_i$
is a graded prime ideal in $S \otimes A$ and the set of its closed points
$|V(\mathbf{p}_i)| \subset | X \times T|$ is equal to $\bigcup_\lambda  |
V(\sigma(\lambda)p_i)|$. It follows that 
\begin{equation}
  \label{eq:8.9}
  | \cD \otimes \cO_T| = \bigcup_i |V(\mathbf{p}_i)| \cup \ell_1 \cup\dots\cup \ell_m \cup E\,.
\end{equation}
Comparing eq.~\eqref{eq:8.8} and eq.~\eqref{eq:8.9}, it follows that
$\Set{\fp_i} = \Set{\mathbf{p}_i}$, hence 
\begin{equation}
  \label{eq:8.10}
  (S \otimes A/\fp_i) \otimes k(\lambda) = S/\sigma(\lambda)p_i\,.
\end{equation}
Because of $h(\phi(\xi(\lambda))) = h(\xi(\lambda))$ the prime ideal
$\sigma(\lambda)p_i$ occurs in the filtration of $\cO_X/\cI (\lambda)$ as
many times as in the filtration of $\cO_X/\cJ (\lambda)$ and $r=s$.

Put $I =\bigoplus H^0(X,\cI(n))$ and $J= \bigoplus H^0(X,\cJ(n))$. The
essential point is: Although the $f_i$ and $g_i$ and their degrees
in~\eqref{eq:8.4}\ \&~\eqref{eq:8.5} respectively ~\eqref{eq:8.6}\
\&~\eqref{eq:8.7} do not agree, one obtains (with the abbreviation $S = S
\otimes A$):
\begin{equation}
 \label{eq:8.11}
\begin{multlined}
  (\det(S/I)_{n-1})^{-1} \otimes \det(S/I)_n =  \\
  \Big[ \bigotimes^r_1( \det(S/\sigma(\lambda)p_i)_{n-d_i-1})^{-1} 
  \otimes \det(S/\sigma(\lambda)p_i)_{n-d_i} \Big] \otimes \Big\{
  \bigotimes_i S_1/\sigma(\lambda)L_i\Big\}\,,
\end{multlined}
\end{equation}
where $L_i \subset k[x,y,z,t]_1$ is a $3$-dimensional subspace and $P_i
=(L_i)$ the generated ideal.
\begin{equation}
  \label{eq:8.12}
\begin{multlined}
  (\det(S/J)_{n-1})^{-1} \otimes \det(S/J)_n = \\
  \Big[ \bigotimes^r_1( \det(S/\sigma(\lambda)p_i)_{n-e_i-1})^{-1} 
  \otimes \det(S/\sigma(\lambda)p_i)_{n-e_i} \Big] \otimes \Big\{
  \bigotimes^s_1 S_1/\cL_i\Big\}\,,
\end{multlined}
\end{equation}
where $\cL_i \subset A[x,y,z,t]_1$ is a rank $3$-subbundle and $\cP_i =
(\cL_i)$. Now $(S/\sigma(\lambda)p_i)_n$ is globally generated on $T$ by
the monomials in $S_n$, hence $\det(S/\sigma(\lambda)p_i)_n$ can be extended to
a uniquely determined line bundle on $\P^1_k$. The same is true for
$S_1/\cL_i$ and $S_1/\sigma(\lambda)L_i$. The extensions of $[ \otimes
\cdots]$ in eq.~\eqref{eq:8.11} and eq.~\eqref{eq:8.12} to line bundles on
$\P^1_k$ are denoted by $A_n$ resp.~$\cA_n$.  The extensions of $\{ \otimes
\cdots\}$ to line bundles on $\P^1_k$ are denoted by $B$
resp.~$\cB$. Hence
\begin{align}
  \label{eq:8.13}
  \bigl(\det(S/I)_{n-1}\bigr)^{-1} \otimes \det(S/I)_n & = A_n \otimes B\,,
  \\
  \bigl(\det(S/J)_{n-1}\bigr)^{-1} \otimes \det(S/J)_n & = \cA_n \otimes \cB\,.
\end{align}
We now write $U = \Spec A$ and $T = \P^1_k$.  Let $\lambda_0 \in U$,
$\zeta_0:=\phi \sigma(\lambda_0)\xi$ and $\cD_0= \Set{
  \sigma(\lambda\lambda^{-1}_0)\zeta_0}^{-1}$. Then instead of
eq.~\eqref{eq:8.11} one gets:
\begin{equation}
  \label{eq:8.15}
  \bigl(\det(S/J(\lambda_0))_{n-1}\bigr)^{-1} \otimes \det(S/J(\lambda_0))_n
   = \cA_n \otimes B(\lambda_0)\,,
\end{equation}
where
\[
   B(\lambda_0):= \bigotimes^s_1S_1 \otimes A /\sigma(\lambda
   \lambda^{-1}_0)\cL_i(\lambda_0)\quad \text{ and } \cL_i(\lambda_0)=\cL_i
   \otimes k(\lambda_0)\,.
\]
Let $\cP_i(\lambda)$ be the prime ideal generated by $\cL_i(\lambda)=\cL_i
\otimes k(\lambda)$; let $P_i(\lambda)$ be the prime ideal generated by
$\sigma(\lambda \lambda^{-1}_0)\cL_i(\lambda_0)$. Then $\cP_i(\lambda)$ is
a closed point on $\cD_\lambda$ and $P_i(\lambda)$ is a closed point on
$\sigma(\lambda \lambda^{-1}_0)\cD_{\lambda_0}$.

\textsc{Case 1}. $P_i(\lambda)$ is independent of $\lambda \Leftrightarrow
(\cL_i \otimes k(\lambda_0)) = \cP_i(\lambda_0)$ is fixed under the
$\G_m$-operation $\sigma(\lambda)$. \\

\textsc{Case 2}. $P_i(\lambda)$ depends on $\lambda$. Then $P_i(\lambda)$
moves on a line and the intersection number (of the extension) of $S_1
\otimes A/ \sigma(\lambda \lambda^{-1}_0) \cL_i(\lambda_0)$ with $T = \P^1$
is equal to $1$.

\textbf{Assumption $A(0)$:} \index{assumption! $A(0)$}
Until the end of Section~\ref{sec:8.6} the curve
$C \leftrightarrow \xi \in \HH(k)$ has no isolated point.

\textbf{N.B.} Hence $\cC_\lambda \leftrightarrow \sigma(\lambda)\xi$ and
$\cD_\lambda \leftrightarrow \phi\sigma(\lambda)\xi$ have no isolated
points, for all $\lambda \in T = \P^1_k$. 

Suppose that $\cP_i(\lambda)$ really occurs, but is independent of $\lambda$,
i.e.\ $\cP_i(\lambda) = \cP_0$ for all $\lambda \in U$.  As $|\cC_\lambda| =
|\cD_\lambda|$ for all $\lambda \in T=\P^1$ by Theorem~\ref{thm:7.1}, it
follows that $\cP_0 \in \cC_\lambda$ for all $\lambda \in U$, hence for all
$\lambda \in T = \P^1$. But then $\cP_0 = P_0 = (0:0:0:1)$ or $\cP_0 \in
V(t)$.  From this one deduces:
\begin{conclusion}
  \label{conc:8.1}
  Let be $T = \P^1_k$. Then $(B(\lambda_0)\cdot T) \leq (\cB \cdot T)$ and
  equality if and only if, for each index $i$, one of the following cases
  occurs:
  \begin{enumerate}[1.]
  \item $\Set{ \cP_i(\lambda) | \lambda \in T}$ consists of one and the same
    point either equal to $P_0 = (0:0:0:1)$ or lying on $E=V(t)$.
  \item $\Set{ P_i(\lambda)}^-$ is a line in $X$ and $\Set{
      \cP_i(\lambda)}^-$ is a line in $X$. \hfill $\qed$
  \end{enumerate}
\end{conclusion}

\section{Additional assumption}
\label{sec:8.3}

Let be $T=\P^1_k$, $U = \Spec A$ as in Section~\ref{sec:8.2}. 

\textbf{Assumption} $A(\lambda_0)$: \index{assumption!$A(\lambda_0)$}
$\lambda_0 \in U = \Spec A$ and in
Conclusion~\ref{conc:8.1} one has equality.

\subsection{}
 \label{sec:8.3.1}

Suppose \An and \Aln are fulfilled. Because of $|\cC_\lambda| \simeq
|\cD_\lambda|$, for all $\lambda \in k^*$, the curve $\cD_\lambda$ has no
isolated points. The assumption \Aln implies that either $\cP_i(\lambda)$ is
a single point, independent of $\lambda$, or $\cP_i(\lambda)$ moves on a
line $\ell_i$.  We consider this last case. Now $\cP_i(\lambda) \in
|\cD_\lambda|$ and $|\cC_\lambda| = |\cD_\lambda|$ for all $\lambda \in U$
(cf.~Thm.~\ref{thm:7.1}). It follows that $\cP_i(\lambda)$ moves on a line
$\ell_i$, which lies on the cylinder over $\pi(\cC_1) = \pi(\cD_1)$, where
$\pi$ is the projection from $P_0$ onto $E = V(t)$, defined by
$\sigma(\lambda)$. Hence $\ell_i$ is a line through the point $P_0$. 

Now by assumption (c.f.~Section~\ref{sec:8.1}) $P_0 \centernot\in C$, hence
a line through $P_0$ intersects the curve $C$ in at most finitely many
points. Let be $\ell_i \cap |\cC_1| = \Set{ R_\nu }$, $R_\nu \in X(k)$
distinct from each other. Because of $\sigma(\lambda)\cC_1 = \cC_\lambda$
one gets $\ell_i \cap |\cC_\lambda| = \Set{ \sigma(\lambda)R_\nu }$; as 
$|\cC_\lambda| = |\cD_\lambda|$ it follows that $\ell_i \cap |\cD_\lambda| =
\Set{ \sigma(\lambda)R_\nu }$ for all $\lambda \in U$.

Suppose $\cP_i(\lambda_0) = \sigma(\lambda_0)R_j$ and $I:=\Set{\lambda \in U
|\; |\lambda-\lambda_0| < \varepsilon}$. Choose $\varepsilon \in \R$ so that 
$\Set{ \sigma(\lambda)R_\nu| \lambda \in I } \cap \Set{
  \sigma(\lambda)R_\mu| \lambda \in I }= \emptyset$ for all $\mu \neq \nu$.
As $\cP_j(\lambda)$ continuously depends on $\lambda$, from
$\cP_i(\lambda_0) = \sigma(\lambda_0)R_j$ it follows that $\cP_i(\lambda) =
\sigma(\lambda)R_j$ for all $\lambda \in I$, hence
\begin{equation}
  \label{eq:8.16}
  \cP_i(\lambda) = \sigma(\lambda\lambda^{-1}_0)\cP_i(\lambda_0) \quad
  \text{ for all } \lambda \in U\,.
\end{equation}

\begin{conclusion} 
  \label{conc:8.2}
 If \An and \Aln are fulfilled:
  \begin{enumerate}[(a)]
  \item $\cL_i(\lambda):=\cL_i \otimes k(\lambda) =
    \sigma(\lambda\lambda^{-1}_0)\cL_i(\lambda_0)$ for all $i$ and all
    $\lambda \in U$.
  \item If $\cB:=\bigotimes^s_1 S_1 \otimes T/\cL_i$, $T= \P^1_k$, then
    $\cB(\lambda):= \cB \otimes k(\lambda) = \bigotimes^s_1
    S_1/\sigma(\lambda\lambda^{-1}_0)\cL_i(\lambda_0)$ for all $\lambda \in
    T$.\\
    Here $\sigma(0)\cL_i(\lambda_0):= \lim\limits_{\lambda \to
      0}\sigma(\lambda)\cL_i(\lambda_0)$ and
    $\sigma(\infty)\cL_i(\lambda_0):= \lim\limits_{\lambda \to
      \infty}\sigma(\lambda)\cL_i(\lambda_0)$.  \hfill $\qed$
  \end{enumerate}
\end{conclusion}

\section{The morphisms $\alpha$, $\beta$, $\gamma$}
\label{sec:8.4}

\subsection{}
\label{sec:8.4.1}

Let be $T = \P^1_k - \Set{ 0, \infty}$ and $\alpha:\P^1 \to \HH$,
$\beta:\P^1 \to \HH$, defined as in~\ref{sec:8.2.1} resp.~\ref{sec:8.2.2} by
$\lambda \mapsto \xi(\lambda) = \sigma(\lambda)\xi$ resp.~$\beta = \phi
\circ \alpha$. If $\lambda_0 \in T$ (not necessarily
$\lambda_0 \in U$), then $\gamma: \P^1 \to \HH$ is defined by $\lambda \mapsto
\sigma(\lambda\lambda_0) \phi(\sigma(\lambda_0)\xi)$. The injectivity of
$\alpha$ follows from the assumption in Section~\ref{sec:8.1}. If 
$\sigma(\lambda\lambda^{-1}_0) \phi(\sigma(\lambda_0)\xi) =
\sigma(\mu\lambda^{-1}_0) \phi(\sigma(\lambda_0)\xi)$, applying $h$ and
using Theorem~\ref{thm:7.1} gives $\sigma(\lambda)h(\xi)=
\sigma(\mu)h(\xi)$, hence $\lambda = \mu$.
It follows that $\alpha, \beta, \gamma$ are injective.

As we had put
$\zeta_0 = \phi \sigma(\lambda_0) \xi$ and $\cD_0 = \Set{
  \sigma(\lambda\lambda^{-1}_0)\zeta_0}^-$, the argumentation also shows:
\[
  \deg(h|\cC) = \deg(h|\cD) = \deg(h|\cD_0) = 1 \quad \text{ and }\quad
 h(\cC) = h(\cD) = h(\cD_0)\,.
\]
If 
\[
  \cD_0 \sim q_2(\cD_0)\cdot C_2 + q_1(\cD_0)\cdot C_1 + q_0(\cD_0)\cdot C_0\,,
\]
then 
\[
  [h(\cD_0] =  q_2(\cD_0)\cdot [h(C_2] =  [h(\cC)] = q_2(\cC)[h(C_2)]\,,
\]
and ditto with $\cD$ instead of $\cD_0$. Finally we can interpret
eq.~\eqref{eq:8.13}--eq.~\eqref{eq:8.15} by means of $\alpha, \beta, \gamma$
and we get:
\begin{conclusion}
  \label{conc:8.3}
Even if \An or \Aln is not assumed, one has for every $\lambda_0 \in T =
\P^1_k - \{0, \infty\}$:
\begin{enumerate}[(a)]
\item $\alpha|T$, $\beta|T$, $\gamma|T$ are injective,
\item $q_2(\cC) = q_2(\cD) = q_2(\cD_0)$,
\item $\cA_n \otimes B = \alpha^*(\cM^{-1}_{n-1}\otimes \cM_n)$; 
$\cA_n \otimes \cB = \beta^*(\cM^{-1}_{n-1}\otimes \cM_n)$ and  $\cA_n
\otimes B(\lambda_0) = \gamma^*(\cM^{-1}_{n-1}\otimes \cM_n)$\,.
 \hfill $\qed$
\end{enumerate}
\end{conclusion}
\vfill
\subsection{}
\label{sec:8.4.2}

Let be $T = \P^1_k$. Then Conclusion~\ref{conc:8.1}  gives 
\begin{align*}
  & (\cA_n \cdot T) + (B(\lambda_0)\cdot T) = (\cA_n \otimes
  B(\lambda_0)\cdot
  T) = (\gamma^*\cM^{-1}_{n-1}\otimes \cM_n \cdot T) \\
  = {} & \deg(\gamma)\cdot (\cM^{-1}_{n-1}\otimes \cM_n\cdot\cD_0) =
  q_2(\cD_0)(n-d +1) + q_1(\cD_0) \\
  \leq {} & (\cA_n \cdot T) + (\cB \cdot T)= (\cA_n \otimes \cB \cdot T)=
  (\beta^*\cM^{-1}_{n-1}\otimes
  \cM_n \cdot T) \\
  = {} & \deg(\beta)\cdot (\cM^{-1}_{n-1}\otimes \cM_n \cdot \beta(T)) =
  (\cM^{-1}_{n-1}\otimes \cM_n\cdot\cD)\\
  = {} & q_2(\cD)(n-d+1) +q_1(\cD)\,.
\end{align*}
We sum up:
\begin{lemma}
  \begin{enumerate}[(a)]
  \item Without assuming \An or \Aln, for each $\lambda_0 \in T =
\P^1_k - \{0, \infty\}$ one has: \\
$\alpha|T$, $\beta|T$, $\gamma|T$ are injective, $q_2(\cC) = q_2(\cD)
=q_2(\cD_0)$ and $q_1(\cD_0) \leq q_1(\cD)$.
\item If \An is supposed and $T:=\P^1_k$, the following statements are
  equivalent:
  \begin{enumerate}[(i)]
  \item \Aln is fulfilled, i.e.\ $(B(\lambda{_0})\cdot{ T}) = (\cB\cdot T)$.
  \item $q_1(\cD_0) = q_1(\cD)$.
  \item The line bundles $\cB = \bigotimes^s_1 S_1 \otimes T / \cL_i$ and 
$B(\lambda_0) = \bigotimes^s_1 S_1 \otimes T /
\sigma(\lambda\lambda{^{-1}_0})  \cL_i(\lambda{_0})$ on $T$ are indentical,
i.e.\\
 $\cL_i(\lambda) = \cL_i\otimes k(\lambda) =
\sigma(\lambda\lambda{^{-1}_0})  \cL_i(\lambda{_0})$ for all $\lambda\in T$
and all $i$.
  \end{enumerate}
\item If $n \geq d$, denote by $f$ the tautological morphism, which is defined
  by $\cM^{-1}_{n-1}\otimes \cM_n$. \\
 If \An is fulfilled and \Aln is fulfilled by an element $\lambda_0 \in U$,
 then $f(\cD) = f(\cD_0)$.
  \end{enumerate}
  \label{lem:8.1}
\end{lemma}
\begin{proof}
  (a) has just been stated before, and in Part (b) the equivalence of (i)
  and (ii) follows from the preceding computation. If we assume (i), then
  (iii) follows from Conclusion~\ref{conc:8.2} and (iii) $\Rightarrow$ (i)
  is trivial. As to Part (c), as $\cA_n \otimes \cB \otimes k(\lambda) =
  \cA_n \otimes B(\lambda_0) \otimes k(\lambda)$ by
  Conclusion~\ref{conc:8.2}, one has $f(\phi \sigma(\lambda)\xi) =
  f(\sigma(\lambda\lambda^{-1}_0)\zeta_0)$ for all $\lambda$ in an open
  subset of $T$, hence $(c)$ follows.
\end{proof}

We still suppose \An and \Aln. By Lemma~\ref{lem:8.1}c), for $\lambda \in
k^*$ there is a $\mu \in \P^1$ and for $\mu \in k^*$ there is a $\lambda \in
\P^1$ such that 
\begin{equation}
  \label{eq:8.17}
  f(\phi \sigma(\lambda)\xi) = f(\sigma(\mu) \phi \sigma(\lambda_0)\xi)
\end{equation}
and hence $h(\phi \sigma(\lambda)\xi) = h(\sigma(\mu) \phi
\sigma(\lambda_0)\xi)$. We show that then $\mu$ (resp.~$\lambda$) is in
$k^*$, too: If $\mu =0$ or $\mu = \infty$, then $\zeta_{0/\infty}:=
\sigma(\mu)\phi(\sigma(\lambda_0)\xi)$ would be fixed by $\G_m$, hence
$h(\zeta_{0/\infty})$ would be fixed by $\G_m$, as $h$ is $\G_m$-equivariant.
It would follow that $h(\phi \sigma(\lambda)\xi) = h(\sigma(\lambda)\xi) =
\sigma(\lambda) h(\xi)$ is invariant under $\G_m$, and hence $h(\xi)$ would
be invariant under $\G_m$, which is not the case by assumption
(cf.~Section~\ref{sec:8.1}). On the other hand, if one starts with $\mu \in
k^*$ and supposes $\lambda =0$ or $\lambda = \infty$, then
\begin{align*}
  & h(\xi_{0/\infty}) = h(\sigma(\lambda)\xi_{0,\infty}) = \sigma(\lambda)
  h(\xi_{0/\infty}) = \sigma(\lambda) h(\sigma(\mu) \phi
  \sigma(\lambda_0)\xi)  \\
  = {} & \sigma(\lambda\mu) h(\phi\sigma(\lambda_0)\xi) = \sigma(\lambda\mu)
  h(\sigma(\lambda_0)\xi) = \sigma(\lambda\mu\lambda_0) h(\xi) \quad \text{
    for all } \lambda \in k^*\,.
\end{align*}
Then $h(\xi)$ again would  be $\G_m$-fixed. We get
\begin{conclusion}
  \label{conc:8.4}
  If in eq.~\eqref{eq:8.17} $\lambda\in k^*$ (resp.~$\mu\in k^*$), then $\mu\in k^*$ (resp.~$\lambda\in k^*$).
 \hfill $\qed$
\end{conclusion}

We draw further consequences: As $h$ is equivariant with respect to $\G_m$
and $\phi$ leaves invariant the fibers of $h$ (cf.~Thm.~\ref{thm:7.1}), one
has $h(\phi \sigma(\lambda)\xi) = h(\sigma(\lambda)\xi) = \sigma(\lambda)
h(\xi)$ and $h(\sigma(\mu) \phi \sigma(\lambda_0)\xi) =
\sigma(\mu\lambda_0)h(\xi)$, hence $\lambda = \mu \lambda_0$, and
eq.~\eqref{eq:8.17} can be written as
\begin{equation}
  \label{eq:8.18}
    f(\phi \sigma(\lambda)\xi) = f(\sigma(\lambda \lambda^{-1}_0)\phi
    \sigma(\lambda_0) \xi) \quad \text{ for all } \lambda \in \P^1\,.
\end{equation}
If $\lambda =1$, one gets $f(\phi(\xi))=  f(\sigma(\lambda^{-1}_0) \phi
\sigma(\lambda_0)\xi)$, hence $f(\sigma(\lambda)\phi(\xi)) =
 f(\sigma(\lambda \lambda^{-1}_0) \phi\sigma(\lambda_0)\xi)$ for all
 $\lambda$.

\begin{conclusion}
  \label{conc:8.5}
Assume that \An and \Aln\ are fulfilled. Put $\cE := \Set{ \sigma(\lambda)
  \phi(\xi)}^-$. Then
\begin{enumerate}[(a)]
\item $f(\sigma(\lambda)\phi(\xi)) = f(\phi \sigma(\lambda)\xi)$
for all $\lambda \in \P^1$,
\item $f | \oring{\cD}$ and $f | \oring{\cE}$ is injective,
\item $q_i(\phi(\xi)) = q_i(\xi)$ for $i=1, 2$.
\end{enumerate}
\end{conclusion}
\begin{proof}
\begin{enumerate}[(a)]
\item follows from eq.~\eqref{eq:8.18} and the foregoing  computation.
\item follows by the same argumentation as in Section~\ref{sec:8.4.1}, as
  the isotropy of $h(\xi) = h(\phi(\xi))$ is trivial by assumption.
\item follows from Lemma~\ref{lem:8.1}a, if $i=2$.\\
 Write 
\[
\cC \sim q_2 C_2 + q_1 C_1 + q_0 C_0 \sim \cD \quad \text{ and } \quad  
\cE \sim p_2 C_2 + p_1 C_1 + p_0 C_0\,.
\]
Let be $g: \HH \to \P$ the morphism
defined by $\cL = \cM^{-1}_{n-1} \otimes \cM_n$, $n \geq d$. By
Corollary~\ref{cor:3.2} the restrictions $g | \oring{\cD}$ and $g|
\oring{\cE}$ are injective, too.\\ 
Hence
\begin{align*}
     &   q_2(\cL \cdot C_2)  + q_1(\cL \cdot C_1)  + q_0(\cL \cdot C_0) =
     (\cL \cdot \cC) \\
= {} &   (\cL \cdot \cD) = (f^* \cO_\P(1) \cdot \cD) = (\cO_\P(1) \cdot
f(\cD)) = (\cO_\P(1) \cdot f(\cE))\\
 = {} &  (\cL \cdot \cE) = p_2(\cL \cdot C_2) + p_1(\cL \cdot C_1) + p_0(\cL \cdot C_0)\,. 
\end{align*}
It follows that $q_2(n-d+1)+ q_1 = p_2(n-d+1) + p_1$, hence $q_1=p_1$.
\qedhere
\end{enumerate}
\end{proof}
\section{Eliminating the assumption \Aln}
\label{sec:8.5}

We now consider the case that \An is fulfilled but \Aln is not fulfilled for
any $\lambda_0 \in U$.

Let be $V:=\beta(U)$; this is an open non-empty subset of $\cD$ and from
Lemma~\ref{lem:8.1} it follows that $q_1(\zeta) < q_1(\cD)$ for all $\zeta
\in V$. \\
Suppose there is $\eta_0 \in \oring{\cD} = \cD - \Set{ \phi(\xi_0),
  \phi(\xi_\infty)}$ such that $q_1(\cD) \leq q_1(\eta_0)$. As $q_2(\eta) =
q_2(\eta_0)$ for all $\eta \in \oring{\cD}$ by Lemma~\ref{lem:8.1}, from
Conclusion~\ref{conc:F.2} in Appendix~\ref{cha:F} it follows that there is
an open neighborhood $V_0$ of $\eta_0$ in $\cD$ such that $q_1(\cD) \leq
q_1(\eta_0) \leq q_1(\eta)$ for all $\eta\in V_0$.  As $V \cap V_0 \neq
\emptyset$, this gives a contradiction.

It follows that $q_1(\zeta) < q_1(\cD)$ for all $\zeta \in \oring\cD$. Now
from $q_1(\cD) = q_1(\cC) = q_1(\xi)$ follows that 
\begin{equation}
  \label{eq:8.19}
   q_1(\zeta) < q_1(\xi) \quad \text{ for all } \zeta \in \oring\cD\,.
\end{equation}
\begin{conclusion}
  \label{conc:8.6}
  Assume \An is fulfilled, but if $\lambda_0 \in U$, then \Aln\  is not
  fulfilled. Then $q_1(\phi(\xi)) < q_1(\xi)$.  \hfill $\qed$
\end{conclusion}

From Conclusion~\ref{conc:8.5} and Conclusion~\ref{conc:8.6} one obtains: 
\begin{conclusion}
  \label{conc:8.7}
  Assume \An.  Then either $q_1(\phi(\xi))= q_1(\xi)$ or $q_1(\phi(\xi)) <
  q_1(\xi)$.  \hfill $\qed$
\end{conclusion}

Let $\xi \leftrightarrow C$ and $\phi(\xi) \leftrightarrow D$. As $C(k)$ and
$D(k)$ are isomorphic, $D$ has no isolated points. As $h(\phi(\xi)) =
h(\xi)$ by Theorem~\ref{thm:7.1}, $t$ is general for $D$ and $\phi(\xi)$
fulfills the assumptions of Section~\ref{sec:8.1}. Hence from
Conclusion~\ref{conc:8.7}  applied to $\phi^{-1}$ and $\phi(\xi)$ instead of
$\phi$ and $\xi$ it follows that $q_1(\phi^{-1}\phi(\xi)) = q_1(\phi(\xi))$
or $q_1(\phi^{-1}\phi(\xi)) < q_1(\phi(\xi))$ i.e.\ $q_1(\xi) =
q_1(\phi(\xi))$ or $q_1(\xi) < q_1(\phi(\xi))$. It follows that 
$q_1(\phi(\xi)) = q_1(\xi)$.
\begin{lemma}
  \label{lem:8.2}
 Assume \An. Then one has:
 \begin{enumerate}[(a)]
 \item $q_i(\phi(\xi))= q_i(\xi)$ if $i=1,2$.
 \item $f(\sigma(\lambda)\phi(\xi)) = f(\phi\sigma(\lambda)\xi)$ for all
   $\lambda \in \P^1$.
 \end{enumerate}
\end{lemma}
\begin{proof}
  (a) has just been proved. If there were no $\lambda_0 \in U$ such that
  \Aln is true, then Conclusion~\ref{conc:8.6} gives $q_1(\phi(\xi)) <
  q_1(\xi)$, contradicting (a). Then (b) follows from
  Conclusion~\ref{conc:8.5} .
\end{proof}
\section{The restriction morphism}
\label{sec:8.6}
\index{restriction morphism}
 It is defined by $r: U(t) \to H^d = \Hilb^d(\P^2)$, $\cI \mapsto \cI' =
 \cI + t \cO_X(-1) / t\cO_X(-1)$.
\subsection{}
\label{sec:8.6.1}

Let $\xi \in \HH(k)$ be adapted. We assume \An.  As $\xi_0 \in U(t)$ and
$\phi(\xi_0)\equiv \xi_0$ by Proposition~\ref{prop:6.1}, it follows that
$\phi\sigma(\lambda)\xi \in U(t)$  for almost all $\lambda \in k$,
including $\lambda=0$. By Lemma~\ref{lem:8.2} $f(\sigma(\lambda)\phi(\xi)) =
f(\phi\sigma(\lambda)\xi)$, hence from Lemma~\ref{lem:3.1} it follows that 
\[
  r(\phi(\xi)) =   r(\sigma(\lambda)\phi(\xi)) = r(\phi\sigma(\lambda)\xi) 
 \xrightarrow[\lambda \to 0]{} r(\phi(\xi_0)) = r(\xi_0) = r(\xi)\,.
\] 
\begin{conclusion}
  \label{conc:8.8}
 If $\xi \in \HH(k)$ is adapted, \An is fulfilled, $\phi \in N$, $\xi
 \leftrightarrow \cI$, $\phi(\xi) \leftrightarrow \cJ$, then $\cI' = \cJ'$.
 \hfill $\qed$
\end{conclusion}

\subsection{}
\label{sec:8.6.2}

Let $\xi \leftrightarrow C$ be as before. As $P_0 = (0:0:0:1) \centernot\in C$,
the ideal $(x,y,z)$ is not associated to $\cI$. Put $R = k[x,y,z]$ and let
be $L$ the set of $\ell \in R_1$ such that $\ell$ is not in any associated
prime of $\cI$. Then $L$ is Zariski-open in $R_1$.  Fix $\ell \in L$ and
define $u_\alpha: x\mapsto x,\; y\mapsto y, \; z\mapsto z, \; t\mapsto t -
\alpha \ell$.  Then for almost all $\alpha \in k$, including $\alpha =0$,
$t+\alpha \ell$ is general for $\xi$, i.e.~$t+\alpha \ell$ is not in any
associated prime of $\cI$, i.e.~one has $u_\alpha(\xi)\in U(t)$. As $P_0$ is
fixed by $u_\alpha$, one has $P_0 \centernot\in u_\alpha(C)$. Finally from
Corollary~\ref{cor:A.2} in Appendix~\ref{cha:A} it follows that
$h(u_\alpha(\xi))$ has trivial $\G_m$-isotropy, hence $u_\alpha(\xi)$ is
adapted for all $\alpha \in A$, where $A$ is a set, which depends on $\xi$
and $\ell$ and is equal to $k$ minus finitely many elements. If $\alpha \in
A$, then by Conclusion~\ref{conc:8.8} one has $r(u_\alpha(\xi)) = r(\phi
u_\alpha(\xi))$ for all $\phi \in N$.

\subsection{}
\label{sec:8.6.3}

In Section~\ref{sec:8.1} we started from an \emph{arbitrary} $\eta\in
\HH(k)$, took a suitable $g\in G = \GL(4,k)$ to get an adapted $\xi =
g(\eta)$. It follows that 
\[
r(u_\alpha g(\eta)) = r(\phi u_\alpha g(\eta)) = 
r(u_\alpha g g^{-1} u^{-1}_\alpha \phi u_\alpha g(\eta)) \quad \text{ for all } \phi\in
N\,,
\] 
hence
\[
  r(u_\alpha g(\eta)) = r(u_\alpha g \phi (\eta))
\]
for all $\phi \in N$. If $\eta \leftrightarrow \cI$, $\phi(\eta)
\leftrightarrow \cJ$, this equation can be written as 
\[
   u_\alpha g(\cI) \equiv u_\alpha g(\cJ) \mod t\,,
\]
which is equivalent to 
\[
  \cI \equiv \cJ \mod g^{-1}u^{-1}_\alpha(t)\,.
\]

Now $\Set{ g^{-1}u^{-1}_\alpha(t)}$ is a Zariski-dense set of linear forms
in $\P(S_1)$, and one obtains:

\begin{conclusion}
  \label{conc:8.9}
  Let $\xi\in \HH(k)$ be any point such that the curve $C \leftrightarrow
  \xi$ has no isolated points. If $\phi$ is any normed automorphism and $\xi
  \leftrightarrow \cI$ and $\phi(\xi) \leftrightarrow \cJ$ are the ideals
  corresponding to $\xi$ resp.~$\phi(\xi)$, then $\cI + \ell \cO_X(-1) /
  \ell \cO_X(-1) = \cJ + \ell \cO_X(-1) / \ell \cO_X(-1)$ for all linear
  forms $\ell$ in a Zariski-dense subset of $S_1$.  \hfill $\qed$
\end{conclusion}
\textbf{N.B.} In this conclusion, there is no assumption that $\xi$ is
adapted, so in the rest of this Chapter~\ref{cha:8} $\xi$ is not assumed to
be adapted.

\subsection{}
\label{sec:8.6.4}

 We need a simple general lemma and first have to introduce some
 notations. Let be $S = k[x,y,z,t]$. We say a statement is true for
 Zariski-many linear forms $\ell \in S_1$, if there is a set $L \subset S_1$,
 which is dense in $S_1 \simeq \A^4$ in the Zariski-topology, such that the
 statement is true for all $\ell \in L$.

 Let be $X = \Proj S$ and $\cI \subset \cO_X$ an ideal, and $\ell \in S_1 -
 0$. We write $\cI \in U(\ell)$, if $\ell$ is a non-zero divisor of
 $\cO_X/\cI$, or equivalently, if $\ell$ does not lie in an associated prime
 ideal of $\cO_X/ \cI$. We write $\cI' = \cI + \ell \cO_X(-1)/ \ell
 \cO_X(-1)$ only if $\cI \in U(\ell)$. If $\cI \subset \cO_X$ is an ideal,
 which defines a curve $C \subset X$, then one can write $\cI = \cN \cap
 \cR$, where $\cN$ is a $\CM$-ideal and $\cR$ is the punctual part. We write
 $\cN = \CM(\cI)$.

\begin{lemma}
  \label{lem:8.3}
  Let be $\cI, \cJ \subset \cO_X$ two ideals, which define curves in $X$. If
  $\cI' = \cJ'$ for Zariski-many $\ell \in S_1$, then the $\CM$-parts of
  $\cI $ and $\cJ$ are equal.
\end{lemma}
\begin{proof}
Put $M = \Set{ \cK \subset \cO_X | \cI' = \cJ' = \cK' \text{ for Zariski-many }
 \ell \in S_1 }$. Let $\cM \in M$ be a maximal element. Then $\cI + \cJ
\subset \cM$ and $\cM$ is a $\CM$-ideal. The sequence
\[
 0 \longrightarrow (\cO_X/\cM)(-1) \stackrel{\cdot \ell}{\longrightarrow}
   \cO_X/\cM \longrightarrow \cO_Y/\cM' \longrightarrow 0
\]
where $Y:=\Proj(S/\ell S(-1))$, is exact for Zariski-many $\ell$, and $\cM' =
\cI'$. If $P(n)= dn -g +1$ and $p(n)= \delta n -\gamma +1$ are the Hilbert
polynomials of $\cO_X / \cI$ respectively of $\cO_X / \cM$, then from
$p(n)-p(n-1) = \delta = d$ if $n \gg 0$ it follows that $P(n)-p(n) =c$ is a
constant $\geq 0$, and $\cM / \cI$ has the Hilbert polynomial $c$. Thus $\cM
/ \cI$ is artinian and we can write $\cI = \cM \cap \cR$, $\cR$ the punctual
part. In the same way we get $\cJ = \cM \cap \cS$, $\cS$ the punctual part.
\end{proof}

\subsection{}
\label{sec:8.6.5}

We apply Lemma~\ref{lem:8.3} to the situation of Conclusion~\ref{conc:8.9}
and we get:
\begin{lemma}
  \label{lem:8.4}
  Let $\xi\in \HH(k)$ be a point such that the curve $C \leftrightarrow \xi$
  has no isolated points.  Let be $\phi \in N$ and $\cI \leftrightarrow \xi$
  and $\cJ \leftrightarrow \phi(\xi)$.  Then the $\CM$-parts of $\cI $ and
  $\cJ$ are equal.  \hfill $\qed$
\end{lemma}

\begin{proposition}
    \label{prop:8.1}
    If $\xi \in \HH(k)$ and $\cI \leftrightarrow \xi$ is a $\CM$-ideal, then
    $\phi(\xi)=\xi$ for all normed automorphisms $\phi$ of $\HH$.
\end{proposition}
\begin{proof}
  From Lemma~\ref{lem:8.4} it follows that $\phi(\xi)\leftrightarrow \cJ =
  \cI \cap \cR$, where $\cR$ is the punctual part of $\cJ$. As the Hilbert
  polynomials of $\cI$ and $\cJ$ are equal, it follows that $\cI = \cJ$.
\end{proof}
\begin{remark*}
In Corollary~\ref{cor:8.1} there is a more general formulation of
Proposition~\ref{prop:8.1}.
\end{remark*}
\section{Eliminating the assumption \An}
 \label{sec:8.7}

\subsection{First step}
\label{sec:8.7.1}
 The isolated points are simple points. \\
This means we can write $\xi \leftrightarrow \cI = \cK_0 \cap \cR$ where the
curve $C_0$ defined by $\cK_0 \subset \cO_X$ has no isolated points and the
subscheme of $X$ defined by $\cR$ consists of $s$ simple points $P_i$ not on
$C_0$. Choose a fixed point $P_0\in C$ and put $P_i(\lambda) = P_0 +
\lambda(P_i - P_0)$. If $P_0$ is general enough, then $\cI(\lambda):=\cK_0
\bigcap^s_1 P_i(\lambda)$ defines a curve $\cC_\lambda \subset X$ with
Hilbert polynomial $P(n)$ for all $\lambda \in U$, where $U \subset
T:=\P^1_k$ is open and non-empty. Then $\lambda \mapsto \cI(\lambda)$
defines a map $U \to \HH$, which uniquely extends to a map $\alpha:T \to
\HH$, which is injective on $U$. Denote $\xi(\lambda)= \alpha(\lambda)$ and
$\cC = \Set{ \xi(\lambda)| \lambda \in T}$. Now $\cO_X/\cI(\lambda) = \cO_X
/\cK_0 \bigoplus^s_1 \cO_X/P_i(\lambda)$ for all $\lambda \in U$ and
$L_i=\Set{\cO_X/P_i(\lambda)| \lambda \in T }$ is a line in $X$. It follows
that 
\begin{equation}
  \label{eq:8.20}
  (\cM_n \cdot \cC) = s\cdot n\;.
\end{equation}
Let be $\beta = \phi \circ \alpha$, $\phi\xi(\lambda) \leftrightarrow
\cJ(\lambda)$ and $\phi(\cC) = \cD = \Set{\cD_\lambda}$, where
$\cD_\lambda\subset X$ is defined by $\cJ(\lambda)$. Let $\psi$ be the
automorphism of the universal curve $\CC$, which is induced by $\phi$. As $\psi$ induces an isomorphism  
$|\cC_\lambda| \simeq |\cD_\lambda|$, if $\lambda\in U$ then one has 
\[
  \cJ (\lambda) = \cK (\lambda) \bigcap^s_1 \cP_i(\lambda)\,,
\]
where the $\cP_i(\lambda)$ are $s$ distinct points not on $V(\cK
(\lambda))$. Here $\cK \subset \cO_{X\times U}$ is an ideal such that $\cL
:= \cO_{X\times U}/\cK$ is flat over $U$, hence has a unique extension all
over $T$, which we denote by the same letter. $\cP_i \subset \cO_{X\times
  U}$ is an ideal such that $\cL_i = \cO_{X\times U}/\cP_i$ is flat over
$U$, hence has a unique extension all over $T$, which we again denote by the
same letter.

Now $h(\xi) = h(\xi(\lambda)) = h(\phi(\xi(\lambda)))$(cf.~Theorem~\ref{thm:7.1}), hence
$V(\cK(\lambda)) = V(\cK_0) = C_0$ for all $\lambda \in T$.
If one puts $\cC^* = \Set{ (\xi(\lambda), P_1(\lambda)) | \lambda \in T}
\subset \CC$, then $[\cC^*]$ has the component $1\cdot L^*$, hence the same
is true for $\psi(\cC^*)= \Set{ (\phi\xi(\lambda), \cP_1(\lambda)) | \lambda
  \in T}$. The usual argumentation shows that $\cP_1(\lambda)$ moves on a
line in $X$. If $\cF$ is the structure sheaf of $\cD$, one has $\cF
\otimes \cO_U = (\cL \bigoplus^s_1 \cL_i) \bigotimes \cO_U$. It follows that
$\cL$ has the Hilbert polynomial $P(n)-s$ and $\cL_i$ has the Hilbert
polynomial $1$, hence $\cP_i(\lambda)$ is a simple point moving on a line,
which we denote by $\cL_i$, too.

Let be $p:X \times T \to T$ the projection. Then 
\[
  p_* \cF(n) \otimes \cO_U = p_* \cL(n) \otimes \cO_U \bigoplus_i p_*\cL_i(n) \otimes \cO_U\,,
\]
hence
\begin{equation}
  \label{eq:8.21}
 \dot\bigwedge p_* \cF(n) \otimes \cO_U = \dot\bigwedge p_*\cL(n) \otimes \cO_U \bigotimes_i p_*\cL_i(n) \otimes \cO_U\,.
\end{equation}
As $p_*\cF(n)$, $p_*\cL(n)$, $p_*\cL_i(n)$ are globally generated by the
monomials in $S_n$, if $n \gg 0$, the extensions of the single factors in
eq.~\eqref{eq:8.21} are uniquely determined line bundles, and
eq.~\eqref{eq:8.21}  holds true, if $U$ is replaced by $T$. It follows that 
\begin{align*}
 &   (\cM_n \cdot \cD) = \deg(\beta) (\cM_n \cdot \cD) = (\beta^*\cM_n \cdot
  T) \\
  = {} & \bigl(\dot\bigwedge p_* \cF(n)\cdot T\bigr) = 
\bigl(\dot\bigwedge p_* \cL(n)\cdot T\bigr) + \sum (p_* \cL_i(n)\cdot T)\;.
\end{align*}
$\cL_i$ is a line in $X$, hence $(p_* \cL_i(n)\cdot T) \geq n$. As
$\cC \sim \cD$, one has $(\cM_n \cdot \cC) =(\cM_n \cdot \cD)$. And then
from eq.~\eqref{eq:8.20} one deduces that $(p_* \cL_i(n)\cdot T)=n$ and
$\bigl(\dot\bigwedge p_* \cL(n)\cdot T\bigr) = 0$. It follows that $\cK
\subset \cO_X$ is a fixed ideal such that $\cO_X / \cK$ has the Hilbert
polynomial $P(n)-s$. If $\cJ(\lambda) \subset \cO_X$ is the ideal, which
defines $\cD_\lambda$, then $\cJ(\lambda) \subset \cK$ for all $\lambda \in
T$. If $\cM := \CM(\cK)$ is the $\CM$-part of $\cK$, then $\cJ(\lambda)
\subset \cM$, hence $\CM(\cJ(\lambda)) = \cM$ for all $\lambda \in T$, hence
$\cJ_0 = \lim_{\lambda \to 0} \cJ(\lambda)$ also has the $\CM$-part $\cM$.

Put $\cN := \CM(\cI) = \CM(\cK_0)$. As $\cI (\lambda) \subset \cN$ for all
$\lambda \in U$, one has $\cI_0:=\lim_{\lambda \to 0} \cI(\lambda) \subset
\cN$ and hence $\CM(\cI_0)= \cN$. Now $\xi_0 = \lim_{\lambda \to 0}
\xi(\lambda) \leftrightarrow \cI_0$ corresponds to a curve without isolated
points, and as $\phi\xi(\lambda) \to \phi(\xi_0) \leftrightarrow \cJ_0$, by
Lemma~\ref{lem:8.4} it follows that $\CM(\cI_0) = \CM(\cJ_0)$. 

\begin{conclusion}
  \label{conc:8.10}
 Assume that $\xi \in \HH(k)$ corresponds to an ideal $\cI$ such that the 
isolated points of $V(\cI)$ are simple. Then $\phi(\xi)$ corresponds to an
ideal $\cJ$ such that $V(\cJ)$ has the same number of isolated points and 
$\CM(\cI) = \CM(\cJ)$.  \hfill $\qed$
\end{conclusion}

\subsection{Second step}
\label{sec:8.7.2}

Let $\cN \subset \cO_X$ be a $\CM$-ideal. Suppose that for the Hilbert
polynomial $p(n)$ of $\cO_X / \cN$ one has $P(n)-p(n)= s > 0$ is a fixed
number. Suppose that $t$ is not a zero-divisor of $\cO_X / \cN$. Let $\G_m$
operate by $\sigma(\lambda): x\mapsto x,\; y\mapsto y, \; z\mapsto z, \;
t\mapsto \lambda t$ and put $\cN' = \cN + t\cO_X(-1) / t\cO_X(-1)$. Then
$\cN_0:= \lim\limits_{\lambda \to 0} \sigma(\lambda) \cN = (\cN')^* \cap
\cR_0$ and $\cR_0$ is primary to $\cP_0=(x,y,z)$ (see Lemma~\ref{lem:G.3}).

Let $P_1,\dots,P_s \in X- \bigl[\bigcup_{\lambda \in k^*} V(\sigma(\lambda)\cN)
\cup V(\cN_0) \cup V(t)\bigr]$ be distinct closed points and put $\cR = \bigcap
^s_1 P_i$ (we identify the points with the corresponding ideals, as usual).

Put $\cI = \cN \cap \cR$ and $\cI(\lambda)=(\sigma(\lambda)\cN) \cap \cR
\leftrightarrow \xi(\lambda)$. This is a closed point of $\HH$ for all
$\lambda \in k^*$.  Then from Conclusion~\ref{conc:8.10} one gets
$\phi\xi(\lambda) \leftrightarrow \sigma(\lambda)\cN \cap \cS_\lambda$ for
all $\lambda$, where $\cS_\lambda \subset \cO_X$ has colength $s$ and
$V(\cS_\lambda)$ consists of $s$ distinct points not on
$V(\sigma(\lambda)\cN)$. Let $P\in V(\cR)$ be a fixed point, let $\cC^* =
\Set{(\xi(\lambda),P)}^-$. Then $\psi(\cC^*)
=\Set{(\phi\xi(\lambda),\phi_{\xi(\lambda)}(P))}^-$ and $[\psi(\cC^*)]$ has
no $L^*$-component, as $[\cC^*]$ has no $L^*$-component. Now $\psi$ induces
an isomorphism $|V(\sigma(\lambda)\cN \cap \cR)| \simeq
|V(\sigma(\lambda)\cN \cap \cS_\lambda)|$ and hence $\phi_{\xi(\lambda)}(P)
\in V(\cS_\lambda)$ is independent of $\lambda$. As $\phi_{\xi(\lambda)}$
induces an isomorphism, $\phi_{\xi(\lambda)}(P_i) =
\phi_{\xi(\lambda)}(P_j)$ implies $P_i = P_j$. It follows that $\cS_\lambda
= \cS$ for all $\lambda$, where $\cS = P'_1 \cap\dots\cap P'_s$ and the
$P'_i$ are distinct and not in $V(\sigma(\lambda)\cN)$. Suppose there is a
number $1 \leq r \leq s$ such that $P'_i \in V(\cN_0)$, if $1 \leq i \leq r$
and $P'_i \centernot\in V(\cN_0)$, if $r<i$.

Put $\cS_1 = \bigcap^r_1 P'_i$, $\cS_2 = \bigcap^s_{r+1} P'_i$, and take $p \in I:=\prod^r_1 P'_i$.  It follows that
\begin{align*}
   & p\cdot \sigma(\lambda) \cN \subset (\sigma(\lambda) \cN) \cap \cS_1
   \subset \sigma(\lambda) \cN \\
\Rightarrow  {} & \lim_{\lambda \to 0} p\cdot \sigma(\lambda) \cN 
 \subset \lim_{\lambda \to 0}[(\sigma(\lambda) \cN) \cap \cS_1]
\subset  \lim_{\lambda \to 0} \sigma(\lambda) \cN \\
\Rightarrow  {} & p\cdot \cN_0 \subset  \cL := \lim_{\lambda \to 0} 
[(\sigma(\lambda) \cN) \cap \cS_1] \subset \cN_0 \\
\Rightarrow  {} & V(p\cdot \cN_0) \supset  V(\cL) \supset V(\cN_0) \quad
\text{ for all }  p \in I \\
\Rightarrow  {} & V(I\cdot \cN_0) \supset  V(\cL) \supset V(\cN_0)\,.
\end{align*}
As $V(I\cdot \cN_0) =  V(I) \cup V(\cN_0) = V(\cN_0)$,
one has $V(\cL) = V(\cN_0)$. Now
\[
  \phi(\xi(\lambda)) \leftrightarrow \cJ (\lambda) = \sigma(\lambda) \cN
  \cap \cS_1 \cap \cS_2
\]
and $V(\sigma(\lambda) \cN \cap \cS_1) \cap V(\cS_2)= \emptyset$ for 
all $\lambda$. Clearly one has 
\[
  \cI_0 \leftrightarrow \xi_0 := \lim_{\lambda \to 0} \xi(\lambda)
  \leftrightarrow \cN_0 \cap \cR\,,
\]
hence 
$\phi(\xi_0) = \lim_{\lambda \to 0} \phi(\xi(\lambda)) \leftrightarrow \cJ_0
 := \lim_{\lambda \to 0} \cJ(\lambda) = \cL \cap \cS_2$,
as $V(\cL) \cap V(\cS_2) = \emptyset$.

On the other hand, $\xi_0
\leftrightarrow \cN_0 \cap \cR = (\cN')^* \cap \cR_0 \cap \cR$, and
from Lemma~\ref{lem:6.3} and Lemma~\ref{lem:6.4}  it follows that that
$\phi(\xi_0) \equiv \xi_0$. But this implies that $|V(\cI_0)| = |V(\cN_0)|
\dcup |V(\cR)|$ is equal to 
 $|V(\cJ_0)| = |V(\cL)| \dcup |V(\cS_2)| =  |V(\cN_0)| \dcup |V(\cS_2)|$
from which $s=s-r$ follows, contradiction.

It follows that $V(\cN_0) \cap V(\cS) = \emptyset$, hence $\cJ_0 =
\lim_{\lambda \to 0}(\sigma(\lambda) \cN \cap \cS) = \cN_0 \cap \cS$.  Now
from $|V(\cI_0)| = | V(\cJ_0)|$ follows $\cS = \cR$, hence $\phi\xi(\lambda)
\leftrightarrow \sigma(\lambda)\cN \cap \cR$ for all $\lambda \in k^*$.

\begin{conclusion}
  \label{conc:8.11}
 Assume $\xi\in \HH(k)$ corresponds to an ideal $\cI = \cN \cap P_1
 \cap\dots\cap P_s$, where $\cN$  is a Cohen-Macaulay ideal, $t$ is not a
 zero-divisor of $\cO_X/\cN$ and $P_i \in X$ are distinct closed points not
 in $\bigl[\bigcup_{\lambda \in k^*} V(\sigma(\lambda)\cN)
\cup V(\cN_0) \cup V(t)\bigr]$,
$\cN_0:=\lim\limits_{\lambda\to 0}\sigma(\lambda)\cN$. Then for each normed
automorphism $\phi$ one has $\phi(\xi)=\xi$.  \hfill $\qed$
\end{conclusion}

\section{The result}
  \label{sec:8.8}

\begin{theorem}
  \label{thm:8.1}
 Let $ k =\C $ be the ground field, $\HH = \Hilb^P(\P_k^3)$, $P(n) = dn -g +1$, $d\geq 6$ and $g \leq
 g(d)$. Let $f$ be the morphism $\HH \to \P$ defined by
 $\cM^{-1}_{n-1}\otimes \cM_n$ for any $n\geq d$. If $\xi \in \HH(k)$, then
 for every normed automorphism $\phi$ of $\HH$ one has $f(\phi(\xi)) =
 f(\xi)$. 
\end{theorem}
\begin{proof}
  $1^\circ$ Suppose $\xi\in U(t)$. If $\xi \leftrightarrow \cI = \cN \cap
  \cR$, $\cN$ is the $\CM$-part, $\cR$ the punctual part. Then from
  Conclusion~\ref{conc:8.11} it follows that $\cN$ fulfills the assumption
  of Lemma~\ref{lem:6.4}, hence $\phi(\xi)\equiv\xi$. \\
  $2^\circ$ Suppose $\xi\in \HH(k)$ arbitrary. Take a $g \in \GL(4,k)$ such
  that $g(\xi) \in U(t)$, hence $\phi g(\xi) \equiv g(\xi)$ for all $\phi\in
  N$. By Lemma~\ref{lem:6.1}  one has $g^{-1}\phi g (\xi) \equiv \xi$ for
  all $\phi \in N$, hence $\phi(\xi)\equiv \xi$ for all $\phi \in N$.
\end{proof}

\begin{remark*}
  (Notations and assumptions as before.) Theorem~\ref{thm:8.1} says that
  each normed automorphism $\phi$ leaves invariant the reduced fibers of
  each tautological morphism. Then from Proposition~\ref{prop:3.1} and
  Corollary~\ref{cor:3.2} one obtains the following formulation of
  Theorem~\ref{thm:8.1}:
\noindent
{\itshape 
 Let $\xi \in \HH(k)$ correspond to the ideal $\cI = \cJ \cap \cR$, where
 $\cJ$ is the Cohen-Macaulay part and $\cR = \bigcap Q_i$ is the punctual part
 such that the $Q_i$ are primary to ideals $P_i$, which correspond to
 different closed points of $\P^3_k$. Then $\phi(\xi)$ corresponds to the
 ideal $\cJ \cap \cR'$, where $\cR' = \bigcap Q'_i$, the $Q'_i$ are
 $P_i$-primary and $\length(\cJ/\cJ \cap Q_i) = \length(\cJ/\cJ \cap Q'_i)$
 for all $i$.}
\end{remark*}
\begin{corollary}
  \label{cor:8.1}
 Assume as before $k=\C$ is the ground field, $d\geq 6$ and $g\leq (d-2)^2/4$. Let $\HH $ be the Hilbert scheme, 
which parametrizes the curves with degree $d$ and genus $g$ in the projective space $\P^3_k$ . Let $\HHCM$,
 respectively $\HHcm$, be the open, non-empty subscheme of $\HH$, whose closed points correspond
to curves without embedded or isolated points, respectively to curves without embedded points.
Then the restriction of a k- automorphism of $\HH$ to $\HHCM$, respectively to $\HHcm$, is induced by
a linear transformation of $\P_k^3$, which is uniquely determined by the automorphism.
\end{corollary}
\begin{proof}                                           
With the same notations as before let $\xi\in\HH(k)$ correspond to the ideal $\cI\subset\cO_{\P^3}$, which does \emph{not}
define a pure curve. Then there is an ideal $\cJ\subset\cO_{\P^3}$ such
that $\cI\subset\cJ$ has finite colength, hence there is such an ideal with
Hilbert polynomial $q(T) = Q(T) + 1$. Let $\FF$ be the Flag-Hilbert scheme
as in the proof of Conclusion 2.7, now with $ q(T)$ instead of $Q^*(T)$.
One defines $\HHCM $ as the complement of the image of the projection $\pi$ from $\FF$
 to $\HH$. Then Theorem 8.1 implies that $\HHCM(k)$ is pointwise invariant under each
normed automorphism of $\HH$, and hence the same is true for $\HHCM$. As for the non-emptiness 
of $\HHCM$, this follows from a theorem of Hartshorne~\cite{H2}. 
The condition c) on p.\,3 of this paper is fulfilled for $d \geq 6$, and thus $\HHCM $ is not empty. 
 As to $\HHcm$, let $ U_0 $ be the set of $\xi\in\HH $ such that $\dim(f^{-1}f(\xi))= 0 $. 
This is an open and nonempty subset of $\HH$  (see R.\,Vakil, FOAG, thm.11.4.2 ), and from Aux- Lemma 3.3 it follows that $U_0(k)$
 is the set of points $\xi\in\HH(k)$ such that the fiber $f^{-1}(f(\xi))$ consist only of the point $\xi$. Theorem 8.1
 says that each $\phi\in N$ leaves the fibers of $f$ fixed, hence $U_0(k)$ is pointwise fixed by $N$, and the same is true for $U_0$.
If I understand correctly, the answer of J.\,Starr to the mathoverflow question: "`Being Cohen-Macaulay open in Hilbert scheme?"'
Aug. 2, 2016 shows that $\HHcm$ is open in $\HH$, too.
\end{proof}

\begin{remark*}
Probably $U_0$ is equal to $\HHcm $, but I can not prove this. 
\end{remark*}
\appendix

\chapter{Linear projections and $\G_m$-actions}
\label{cha:A}

\section{Description of the linear projection}
\label{sec:A.1}

Let $k$ be an algebraically closed field, $S = k[x,y,z,t]$ and $X= \Proj(S)=
\P^3$. Each maximal graded prime ideal $\cP \neq S_+$ of $S$ corresponds to
a point $P \in \P^3$, which is denoted by $\cP \leftrightarrow P$. If $\cP =
(\ell_1, \ell_2, \ell_3)$, $\ell_i \in S_1$, $1 \leq i \leq 3$, linearly
independent, then $\cP \leftrightarrow P = (p_0: \cdots : p_3)$ is the point in
$X(k)$ such that $\ell_i(P)=0$, $1 \leq i \leq 3$. We often identify $\cP$
and $P$, i.e.~we also denote with $\cP$ (respectively with $P$) the
corresponding point (respectively the corresponding prime ideal).

Let $\pi = (P,E)$ be the linear projection from the point $\cP
\leftrightarrow P$ onto the plane $E = V(\ell)$, where $\ell \in S_1$  is a
linear form. We want to describe, how $\pi$ can be defined by a $\G_m$-action
on $X$.

If $g \in G:= \Aut_k(S_1)$, then $g$ acts as an automorphism of $S$ on $X$
and the action on $X(k)$ is defined by $g(p_0: \cdots : p_3) = (p_0: \cdots
: p_3)M(g^{-1})$, where $M(-)$ denotes the corresponding matrix with respect
to the $k$-basis $\{x,y,z,t\}$ of $S_1$.

As $P \centernot\in E$, one has $\langle \ell_1,\ell_2,\ell_3,\ell \rangle =
S_1$ and we take any $g\in G$ such that $g(\cP) = \cP_0 := (x,y,z)
\leftrightarrow (0:0:0:1)=:P_0$ and get $g(\ell) = t$. If
$\sigma(\lambda):x\mapsto x$, $y\mapsto y$, $z\mapsto z$, $t \mapsto \lambda
t$, $\lambda\in k^*$, is the ``usual'' $\G_m$-action on $S$, and if $g$ is
the linear transformation just mentioned, we put $\tau(\lambda):= g^{-1}\circ
\sigma(\lambda) \circ g$. Then one has the following simple

\begin{lemma}
  \label{lem:A.1}
  If $Q \in \P^3(k) - \{ P \}$, then $\pi(Q) = \lim\limits_{\lambda \to \infty}
  \tau(\lambda) (Q)$\,.
\end{lemma}
\begin{proof}
  If $\ell = (P,Q)$ is the line connecting $P$ and $Q$, then the intersection
  $R = \ell \cap E$ is equal to $\pi(Q)$, hence $g(\ell) = (g(P),g(Q))$ and
  $g(R) = g(\ell) \cap V(t)$. \\

\textsc{Case 1}.  $Q \in E$. Then $g(Q) \in V(t)$ and $g^{-1}
\sigma(\lambda) g(Q) = g^{-1}g(Q) = Q$ for all $\lambda \in k^*$. \\

\textsc{Case 2}.  $Q \centernot\in E$. Then $g(Q)\centernot\in V(t)$ and if
one assumes $g(Q)=P_0 = (0: 0: 0: 1) = g(P)$, then $Q = P$ follows,
contradiction. It follows that $g(Q)$ is not invariant under
$\sigma(\lambda)$, hence $h:=\Set{\sigma(\lambda)g(Q) | \lambda \in k^*}^-$
is a line in $X$, which connects $\sigma(1)g(Q) = g(Q)$ and $\lim_{\lambda\to
  0} \sigma(\lambda)g(Q)$. As $\sigma(\lambda)g(Q) =
g(Q)M(\sigma(\lambda^{-1})) \to P_0 = (0:0:0:1)$ if $\lambda \to 0$, $h$ is
the line through $g(Q)$ and $g(P)$, i.e.~$h=g(\ell)$. It follows that
$\Set{g^{-1}\sigma(\lambda)g(Q) | \lambda \in k^*}^- = \ell$ and $g(R) =
g(\ell) \cap g(E) = \Set{\sigma(\lambda) g(Q)}^- \cap V(t)$. If one assumes
$\sigma(\lambda) g(Q) \in V(t)$ with $\lambda \in k^*$, then $g(Q) \in V(t)$
and $Q \in E$ would follow. As we have already noted above
$\sigma(\lambda)g(Q) \to P_0 \centernot\in V(t)$ if $\lambda \to 0$, it
follows that $g(R) = \lim\limits_{\lambda\to \infty} \sigma(\lambda)g(Q)$,
which implies $R =\lim\limits_{\lambda \to \infty} \tau(\lambda)Q$.
\end{proof}

\section{Notations}
\label{sec:A.2}
 Let the curve $C \subset X$ be defined by the ideal $\cI \subset \cO_X$. We
 say that the linear form $\ell \in S_1$ is general for $C$ (or $\cI$, or
 $\cO_X/\cI$), if the sequence 
\begin{equation}
\label{eq:seq}
 0 \longrightarrow (\cO_X /\cI)(-1)  \stackrel{\mu}{\longrightarrow} \cO_X /\cI   \longrightarrow \cO_{X'} /\cI'  \longrightarrow 0
\end{equation}
is exact. Here $\mu$ is the multiplication by $\ell$, $S' = S/ \ell S(-1)$,
$X' = \Proj S' \simeq \P^2_k$, $\cI' = \cI  + \ell \cO_X(-1)/ \ell \cO_X(-1)$
is an ideal on $X'$.  An equivalent condition is that $\ell \centernot\in
\bigcup \cP_i$, where $\cP_i$ are the associated prime ideals of $\cI$
(i.e.~ associated prime ideals of $\cO_X/\cI$), which may have the dimension
$0$ or $1$. It follows that there are Zariski-many linear forms, which are
general for $C$. And the same is true, if one simultaneously considers
finitely many such curves. 

The sequence~\eqref{eq:seq} will occur quite often and we denote $\cI'$ the
restriction ideal with respect to the hyperplane section $V(\ell)$ (or with
respect to the canonical restriction morphism $r:S \to S'$ etc.).

\section{Varying the plane of projection}
\label{sec:A.3}

If we replace $V(t)$ by $V(t-\alpha \ell)$, where $\ell = ax +by +cz$ and
hold the point $P_0 = (0:0:0:1)$ fixed, according to~\ref{sec:A.1}, the
projection $\pi_\alpha$ from $P_0$ to $E_\alpha = V(t-\alpha \ell)$ is
defined by the $\G_m$-operation $\tau(\lambda)= u^{-1} \circ \sigma(\lambda)
\circ u$, where $u: x\mapsto x$, $y\mapsto y$, $z\mapsto z$, $t \mapsto
t+\alpha \ell$.  If $P = (p_0:p_1:p_2:p_3)$, a simple computation gives
$\tau(\lambda)P = (p_0:p_1:p_2:\gamma a p_0 + \gamma b p_1 + \gamma c p_2 +
\lambda^{-1}p_3)$, where $\gamma:= \alpha(1- \lambda^{-1})$, hence
$\pi_\alpha(P) = (p_0:p_1:p_2: \alpha (a p_0 + b p_1 + c p_2))$. If $a p_0 +
b p_1 + c p_2\neq 0$, i.e.~if $\pi_\alpha(P) \centernot\in V(t) \cap V(t +
\alpha \ell)$, then the points $\pi_\alpha(P)$ all lie on the line
connecting $(p_0:p_1:p_2:0 )$ and $(0:0:0:1)$.

\begin{lemma}
  \label{lem:A.2}
  If $t$ is general for the curve $C$, then $t + \alpha \ell$ is general for
  $C$ for almost all $\alpha \in k$, and the cylinders over $\pi_\alpha(C)$
  perpendicular to $V(t)$ are equal. \hfill $\qed$
\end{lemma}

\section{Auxiliary lemmas}
\label{sec:A.4}

Let $I \subset P = k[x,y,z,t]$ be a saturated graded ideal, i.e.~$I_n =
H^0(X,\tilde{I}(n))$, $X = \P^3_k$. Suppose that $(x,y,z)$ is not associated
to $I$. Let be $S= k[x,y,z]$ and $\ell \in S_1$ a non-zero divisor of
$P/I$. Suppose that the following condition is fulfilled: $ f\in I_d
\Rightarrow \ell \partial f / \partial t \in I_d$.

\begin{auxlemma}
  \label{auxlem:A.1}
  If $f = f^0 + t f^1 + \cdots + t^d f^d \in I_d$, $f^i \in S_{d-i}
  \Rightarrow f^0 \in I_d$.
\end{auxlemma}
\begin{proof}
  Write $f = f^0 + t^m f^m + \cdots + t^d f^d$, where $m \geq
  1$. $\Rightarrow \ell \partial f /\partial t = m \ell t^{m-1} f^m + \cdots
  + d \ell t^{d-1} f^d \in I_d$  $\Rightarrow  \partial f /\partial t \in
  I_{d-1}$ $\Rightarrow g:=f - \tfrac{1}{m} \cdot t \cdot \partial f
  /\partial t = f^0 + t^{m+1} g^{m+1} + \cdots + t^d g^d \in I_d$, and by an
  induction argument, $f^0 \in I_d$ follows.
\end{proof}

\begin{auxlemma}
  \label{auxlem:A.2}
  $I_d$ is generated by forms of the shape $f^i t^{d-i}$ with $f^i \in S_i$.
\end{auxlemma}
\begin{proof}
  If $f = f^0 + t^m f^m + \cdots + t^d f^d \in I_d$, then by
  Aux-Lemma~\ref{auxlem:A.1}, $f^0 = 0$ and $m \geq 1$, without
  restriction. Then one has 
\[
   g:= f - \tfrac{1}{m} \cdot t \partial f/\partial t =\frac{1}{m}
   \sum^d_{m+1} (m-i)t^i f^i \in I_d\,.
\]
Now we use an induction argument and may suppose that $t^i f^i \in I_d$, if
$i\geq m+1$. But then $t^m f^m \in I_d$, too.
\end{proof}

\section{Isotropy groups of linear projections}
\label{sec:A.5}

Let be $P=k[x,y,z,t]$, $S=k[x,y,z]$, $X=\P^3$, $\cI \subset \cO_X$ an ideal
such that the ideal $(x,y,z)$ is not associated to the saturated ideal $I =
\bigoplus_{n\geq 0} H^0(X,\cI (n))$. Choose any number $d \geq \reg(\cI)$
and any $\ell \in S_1$ such that $\ell$ is not a zero-divisor of $P/I$. Let
$\G_m$ operate by $\sigma(\lambda):x\mapsto x$, $y\mapsto y$, $z\mapsto z$,
$t \mapsto \lambda t$.
Suppose $\cI$ is not $\G_m$-invariant. Let $m = \dim I_d$ and $I_d
\leftrightarrow \xi \in W := \Grass_m(P_d)$. Let $U \subset{ U(4;k)}$ be the
subgroup of linear transformations $u_\alpha: x\mapsto x$, $y\mapsto y$,
$z\mapsto z$, $t \mapsto t +\alpha \ell$, $\alpha \in k$, ($\ell$ is fixed!).

\begin{lemma}
  \label{lem:A.3}
  For nearly all $\alpha \in k$ the isotropy group of $u_\alpha(\xi)$ in
  $\G_m$ is trivial, that is $\sigma(\lambda)u_\alpha(\xi) =
  \sigma(\mu)u_\alpha(\xi) \Rightarrow \lambda = \mu$.
\end{lemma}
\begin{proof}
 Let $G$ be the isotropy group of $\xi$ in $\GL(4;k)$. If $u = u_\alpha \in
 U$ let be  $T(\alpha) = \Set{ \lambda \in k^* |
   \sigma(\lambda)u_\alpha(\xi) = u_\alpha(\xi)}$.  From the proof
 of~\cite[Hilfssatz 5, pp. 8]{T2} it follows that $T(\alpha) = E_n :=
 \Set{\varepsilon \in \C | \varepsilon^n =1}$, where $n\geq 1$ depends on
 $\alpha$, but $n\leq d$.

 \textbf{Assumption:} {\it For infinitely many $\alpha \in k$ the isotropy
   group $T(\alpha)$ is not trivial.}

 Then there are infinitely many $\alpha$ such that $T(\alpha)=E_n=: E$,
 where now $n> 1$ is independent of these $\alpha$. It follows that
 $u_\alpha(\xi)$ lies in $W^E$. As this fixed point scheme is closed in $W$,
 it follows that $u_\alpha(\xi)$ is fixed by $E$ for all $\alpha \in k$. But
 from $\sigma(\lambda)u_\alpha(\xi)= u_\alpha(\xi)$ it follows that
 $g(\lambda, \alpha):= u^{-1}_\alpha \sigma(\lambda) u_\alpha \in G$ for all
 $\alpha \in k$, all $\lambda \in E$. Now $g(\mu,\beta) \circ g(\lambda,
 \alpha)$ leaves $x,y,z$ invariant and maps $t$ to $\lambda \mu t + [
 \lambda(1-\mu)\beta +(1-\lambda)\alpha]\ell$. If \emph{not} $\lambda=\mu
 =1$, it follows that the transformation $x\mapsto x$, $y\mapsto y$,
 $z\mapsto z$, $t \mapsto \lambda \mu t + \alpha \ell$ is in $G$ for all
 $\alpha \in k$. If one chooses $\lambda \neq 1$, $\mu = \lambda^{-1}$, it
 follows that $U < G$. But then $I_d$ is invariant under $f \mapsto \ell
 \cdot \partial f/\partial t$, $f \in I_d$~\cite[proof of Hilfssatz 1, p.
 142]{T2}. Now $\ell$ is a $\NNT$ of $P/I$ by assumption, and the
 Aux-Lemma~\ref{auxlem:A.2} shows that $I_d$ is $\G_m$-invariant. As $d\geq
 \reg(\cI)$, it follows that $\cI$ is fixed by $\G_m$, contradiction.
 Hence the intermediate assumption is not possible, i.e.~for nearly all
 $\alpha \in k$, $T(\alpha) = \{ 1 \}$. \hfill
\end{proof}

\begin{corollary}
  \label{cor:A.1}
  Let $C \subset X $ be a curve not containing the point $P_0 = (0:0:0:1)$,
  such that $t$ is general for $C$. Let $\pi$ be the projection from $P_0$
  onto $E = V(t)$ defined by the $\G_m$-action $\sigma(\lambda):x\mapsto x$,
  $y\mapsto y$, $z\mapsto z$, $t \mapsto \lambda t$. Let $C_\red =
  \bigcup^r_1 Z_i$ be the decomposition into irreducible components,
  i.e.~either $Z_i$ is a reduced and irreducible curve or $Z_i$ is a single
  point not lying on any other $Z_j$. Let $\ell \in S_1$ be a linear form
  such that no $Z_i$ is contained in $V(\ell)$, i.e.~$\ell \centernot\in
  \fp_i$, where $\fp_i$ is the prime ideal such that $Z_i = V(\fp_i)$, $1
  \leq i \leq r$. Put $\cJ : = \bigcap \fp_i $ and let $u_\alpha$ be
  the linear transformation $x\mapsto x$, $y\mapsto y$, $z\mapsto z$, $t
  \mapsto t + \alpha \ell$. Then for nearly all $\alpha \in k$ one has:
  $\sigma(\lambda)u_\alpha(\cJ) = \sigma(\mu)u_\alpha(\cJ) \Rightarrow
  \lambda = \mu$. 
\mbox{} \hfill $\qed$
\end{corollary}

Let $\xi \in \HH(k)$ correspond to a curve $C \subset X$, which is defined
by an ideal $\cI \subset \cO_X$. Suppose that $\xi \in U(t)$ and $P_0 =
(0:0:0:1) \centernot\in C$. It follows that $(x,y,z)$ is not among the 
associated primes of $\cI$, which we denote by $\fp_i$. Hence the set
$L:=\Set{ \ell \in S_1 | \ell \centernot\in \bigcup \fp_i}$ is non-empty and
Zariski-open in $\A^3$. Let $\sigma(\lambda)$ be the usual $\G_m$-operation
and $u_\alpha$ be the transformation $x\mapsto x$, $y\mapsto y$, $z\mapsto
z$, $t \mapsto t + \alpha \ell$, if $\ell \in L$ is fixed.
\begin{corollary}
  \label{cor:A.2}
  For almost all $\alpha \in k$ one has $\sigma(\lambda)h(u_\alpha(\xi)) =
  h(u_\alpha(\xi)) \Rightarrow \lambda = 1$\,.
\end{corollary}
\begin{proof}
  Let $\fp_i$ be the associated primes of $\cI$ such that $\dim V(\fp_i)
  =1$, let $\nu_i$ be their multiplicity in $\cI$. Then 
\[
  h(\sigma(\lambda)u_\alpha(\xi)) = \sum\nu_i \langle
  \sigma(\lambda)u_\alpha(\fp_i)\rangle = h(u_\alpha(\xi)) = \sum \nu_i 
\langle u_\alpha(\fp_i)\rangle
\]
shows that $\sigma(\lambda)$ is a permutation of $\{ u_\alpha(\fp_i)\}$.
If $\cJ :=\bigcap \fp_i$, then it follows that $\sigma(\lambda)u_\alpha(\cJ)
= u_\alpha(\cJ)$.  Now $\cJ$ fulfills the
assumptions of Corollary~\ref{cor:A.1}, hence for   almost all $\alpha \in
k$ it follows that $\lambda =1$.
\end{proof}


\chapter{A linear algebra lemma}
\label{cha:B}

\begin{lemma}
  \label{lem:B1}
 Let $S=k[X_0,\dots,X_r]$ and $\psi$ be a $k$-linear endomorphism  of $S_d$
 ($d\geq 1$ is a fixed integer), such that $\psi(\ell \cdot S_{d-1}) \subset
 \ell S_{d-1}$ for all $\ell \in S_1$. Then there is a fixed element $\alpha
 \in k$ such that $\psi(f) = \alpha \cdot f$ for all $f\in S_d$. $\psi$ is
 not the zero-map iff $\alpha \neq 0$.
\end{lemma}
\begin{proof}
  $1^\circ$ Let $\ell_1,\dots, \ell_d \in S_1 - (0)$ be relatively prime to
  each other.\\ Then 
\[
\psi(\bigcap^d_1 \ell_i S_{d-1}) \subset \bigcap^d_1
  \psi(\ell_i S_{d-1}) \subset \bigcap^d_1
  \ell_i S_{d-1} \subset \ell_1 \cdots \ell_d \cdot k\,.
\]
We get: 

\textsc{Conclusion 1}. 
If  $\ell_1,\dots, \ell_d \in S_d$ are relatively prime to 
  each other,\\ then $\psi(\ell_1 \cdots \ell_d) = \alpha \cdot \ell_1
  \cdots \ell_d$, where $\alpha \in k$ possibly depends on
  $\ell_1,\dots,\ell_d$. 

$2^\circ$ Let $V$ be an $m$-dimensional vector space, let be $\ell_i \in V$,
$1 \leq i \leq n$, such that $\ell_i$ and $\ell_j$ are linearly independent
for $i \neq j$. Let $h_i \in V$, $1 \leq i \leq n$ be any vectors. Put
$L_i:=\ell_i + \lambda_i h_i$, $1 \leq i \leq n$. Then the set 
\[
U:= \Set{
  \lambda = (\lambda_1, \dots, \lambda_n) \in k^n | L_i \text{ and } L_j
  \text{ are linearly independent for all } i\neq j}
\]
is non-empty and Zariski-open in $k^n$. To prove this, let $e_1, \dots,e_m$
be a basis of $V$ and write $\ell_i = \sum^m_{\nu=1}a_{i\nu}e_\nu$, $h_i =
\sum^m_{\nu=1}b_{i\nu}e_\nu$, $a_{i\nu}, b_{i\nu} \in k$. Then $L_i$ and
$L_j$ are linearly independent 
\[
\iff 
  D(i,j,\nu,\mu, \lambda_i,\lambda_j):=
 \det \begin{pmatrix}
   a_{i\nu}+ \lambda_i b_{i\nu} & a_{i\mu}+ \lambda_i b_{i\mu} \\
   a_{j\nu}+ \lambda_j b_{j\nu} & a_{j\mu}+ \lambda_j b_{j\mu}
 \end{pmatrix} = 0 \quad 
\text{ for all $\nu$ and $\mu$ }.
\]
Put $D(i,j,\nu,\mu):= \Set{ \lambda \in k^n | D(i,j,\nu,\mu,
  \lambda_i,\lambda_j) \neq 0}$, $D(i,j):= \bigcup\limits_{\nu \neq \mu}
 D(i,j,\nu,\mu)$.\\
 As  $\lambda = (0,\dots,0) \in D(i,j)$, we
get: 

\textsc{Conclusion 2}.
 \label{conc:B.2}
  $U = \bigcap_{i\neq j} D_{ij} \neq \emptyset$. 

$3^\circ$ Now take $\ell_i\in S_1$, $1 \leq i \leq d$, relatively prime to
each other and $h_i \in S_1 - (0)$, $1 \leq i \leq  d$, arbitrary. Then by
Conclusion~1 
and  Conclusion~2
one has $\psi(L_1 \cdots L_d) = \alpha(\lambda) L_1 \cdots L_d$, with $L_i =
\ell_i + \lambda_i h_i$, for Zariski-many $\lambda \in k^n$. As $\psi \neq
0$ has only finitely many eigenvalues it follows
\[
   \psi(L_1 \cdots L_d) = \alpha L_1 \cdots L_d\,,
\]
$\alpha \in k$ independent of $\lambda$, if $\lambda$ is in an open subset
$\Lambda \neq \emptyset $ of $k^n$. It follows that 
\[
  \sum_{(i),(j)} \psi(\ell_{i_1}\cdots \ell_{i_r} \cdot h_{j_1} \cdots
  h_{j_s}) \cdot \lambda_{j_1} \cdots \lambda_{j_s} = 
    \sum_{(i),(j)} \alpha \ell_{i_1}\cdots \ell_{i_r} \cdot h_{j_1} \cdots
  h_{j_s}\cdot \lambda_{j_1} \cdots \lambda_{j_s}\,,
\]
where $(i)$ runs over all sequences $1 \leq i_1 < \cdots < i_r \leq d $, for
all $0 \leq r \leq  d$ and $(j)$ runs over all the complementary sequences,
$r+s=d$. As this is to hold for all $\lambda \in \Lambda$, one deduces $\psi(\ell_{i_1}\cdots \ell_{i_r} \cdot h_{j_1} \cdots
  h_{j_s}) = \alpha(\ell_{i_1}\cdots \ell_{i_r} \cdot h_{j_1} \cdots
  h_{j_s})$. Choosing  $r=0$ one obtains $\psi(h_1 \cdots h_d) = \alpha h_1
  \cdots h_d$ for arbitrary $h_i \in S_1$. It follows that $\psi(m) = \alpha
  m$ for all monomials $m \in S_d$ and the lemma is proved.
\end{proof}

\chapter{Some special schemes}
\label{cha:C}

\section{The scheme $\cH$}
\label{sec:C.1}

We write $S=k[X_0,\dots, X_3]$ and fix the Hilbert polynomial $q(n) =
\tbinom{n-1+3}{3} + \tbinom{n-d+2}{2}$, $d\geq 3$ an integer and $\cH =
\Hilb_q(X)$ the Hilbert scheme, which parametrizes the ideal sheaves on $X =
\P^3_k$ with Hilbert polynomial $q$.

\begin{lemma}
 \label{lem:C.1}
  If $Y/k$ is a scheme, $\cH(Y)$ consists of the ideals $\cI \subset
  \cO_{X\times Y}$, which are generated by a subbundle $\cL_1 \subset S_1
  \otimes \cO_Y$ and and by a subbundle $\overline{\cF_d} \subset S_d
  \otimes \cO_Y /\cL_1 \cdot S_{d-1}$, each of rank $1$.
\end{lemma}
\begin{proof}
  $\cI \otimes k(y)$ is $d$-regular, $y \in Y$~\cite[Lemma 2.9]{G78} and we
  put $\cF := \cO_{X\times Y}/\cI$. From standard results on cohomology and
  flatness~\cite[Lect. 11, 14]{M2} it follows that
  $\cL_{d-1}:=\pi_*\cI(d-1) \subset S_{d-1} \otimes \cO_Y$ is a subbundle of
  rank $q(d-1) = \binom{d-2+3}{3}$ and
\[
   H^0(X \otimes k(y), \cI(d-1)\otimes k(y)) \simeq \cL_{d-1} \otimes k(y)
\]
for all $y\in Y$.  Now from~\cite[Prop. 2]{G89} it follows that $\cI \otimes
K = (\ell,f)$, $K:=\overline{k(y)}, \ell \in S_1 \otimes K$, $f \in S_d
\otimes K / \ell S_{d-1} \otimes K$. It follows that $H^0(\cI(d-2)\otimes K)
= \ell S_{d-3} \otimes  K$. As
\[
h^0(\cI(d-2) \otimes  K) - h^1(\cI(d-2) \otimes  K) =
 q(d-2) = \tbinom{d-2-1+3}{3} + \tbinom{d-2-d+2}{2} = \tbinom{d-2-1+3}{3}\,,
\]
it follows that $h^1(\cI(d-2)\otimes K) = 0$ and therefore 
$h^1(\cI(d-2)\otimes k(y)) = (0)$ for all $y \in Y$. Let $\pi:
X \times Y \to Y$ be the projection. Then from~\cite[Chap. III,
Thm. 12.11]{H} it follows that $R^1 \pi_*\cI(d-2)\otimes k(y) =(0)$, $y\in
Y$, and therefore $R^1 \pi_* \cI(d-2) =(0)$.

Now $\cI \otimes k(y)$ defines a curve in $\P^3 \otimes k(y)$ and the same
argumentation as in Chapter~\ref{cha:1}, proof of Lemma~\ref{lem:1.1} shows that $\cF
\otimes k(y)$ is $(d-1)$-regular and therefore $H^1(\cF(d-2)\otimes k(y))
=(0)$. Then from~\cite[Lecture 7, Corollary 1]{M2} it follows that
$\pi_*\cF(d-2)\otimes k(y) \simeq H^0(\cF(d-2)\otimes k(y))$. Then (loc.\
cit., Corollary 2) gives that $\pi_*\cF(d-2)$ is locally free of rank 
$\tbinom{d-2+3}{3}-q(d-2)$.
From the exact sequence 
\[
 0 \longrightarrow \pi_* \cI(d-2) \longrightarrow  S_{d-2}\otimes \cO_Y \longrightarrow \pi_* \cF(d-2)  \longrightarrow 0\,
\]
it follows that $\cL_{d-2} := \pi_* \cI(d-2)$ is a subbundle of rank $q(d-2)
= \tbinom{d-2-1+3}{3} + \tbinom{d-2-d+2}{2} =
\tbinom{d-1-2+3}{3}$. From~\cite[Korollar 3.8]{G78} it follows that
$\cL_{d-2}$ generates an ideal $\cL \subset \cO_{X \times Y}$ with
Hilbert polynomial $\tbinom{n-1+3}{3}$ such that $\cO_{X \times Y}/\cL$ is
flat over $Y$.

From the $1$-regularity of $\cL$ it follows that $\cL$ is generated by a
subbundle $\cL_1 \subset S_1 \otimes \cO_Y$ of rank $1$.  If $U = \Spec(A)
\subset Y$ is sufficiently small, one can make an $A$-linear transformation
such that $\cL_1 \otimes A = X_0 \cdot A$ and one can write $H^0(\P^3
\otimes A, \cI(d)) = X_0 \cdot S_{d-1} \otimes A \oplus f\cdot A$, $f \in
R_d \otimes A$, $R = k[X_1,X_2, X_3]$. It follows that $H^0(\P^3 \otimes A,
\cI(n)) = X_0 S_{n-1} \otimes A \oplus f \cdot R_{n-d} \otimes A$ is a
subbundle of $S_n \otimes A$ of rank $q(n)$.
\end{proof}

\section{The scheme $\fX$}
\label{sec:C.2}

We first describe a general situation. Let be $S = k[x_0,\dots,x_r]$, $S_{(i)} =
k[x_0,\dots,\hat{x_i},\dots,x_r]$, $X = \P(S_1)$, $D_i:=\Set{\ell = a_0 x_0
  +\cdots + a_r x_r \in S_1 | a_i \neq 0 }$, $H_i:=
\Hilb^c(\Proj S_{(i)} )$, $\fX_i := D_i \times_k H_i$, $\phi_{ij}: \fX_i \to
\fX_j$, $i \neq j$, defined by the automorphism $x_i \mapsto x_j$, $x_j
\mapsto x_i$, and $x_k \mapsto x_k$, $k \centernot\in \{ i,j\}$.  If one
puts $U_{ij}:= D_i \times H_i \cap D_j \times H_i = D_i \cap D_j \times
H_i$, then one sees that the $\fX_i$ glue together to a scheme $\fX$ with
the following property: 

Let $U = \Spec A $ be sufficiently small and let $\ell\in S_1 \otimes A $ 
generate a direct summand of $ S_1 \otimes A $, $\overline{S \otimes A}:= S
\otimes A / \ell S(-1) \otimes A$, $\overline{X} = \Proj \overline{S \otimes
  A} \simeq \P^{r-1}\otimes A$. Then $\fX(A)$ is the set of pairs $(\ell,
\cK)$, where $\cK \subset \cO_{\bar X}$ is an ideal such that
$\cO_{\overline{X}}/\cK$ is flat over $U$ with Hilbert polynomial $c$. If
$\pi: \fX \to X$ is defined by $(\ell, \cK) \mapsto \langle \ell \rangle$,
then the fibers of $\pi$ are isomorphic to $\Hilb^c(\P^{r-1})$.
 
\begin{lemma}
  \label{lem:C.2}
  If $r=3$, then $\fX$ is smooth over $X = \P(S_1)$ with fibers isomorphic
  to $\Hilb^c(\P^2)$. \hfill $\qed$
\end{lemma}

\section{The schemes $\cG, X, Y, Z$}
\label{sec:C.3}

\subsection{}
\label{sec:C.3.1}

Let be $S = k[x,y,z,t]$ and $\cH$ the Hilbert scheme $\Hilb_q(\P^3)$, $q(n)
= \tbinom{n-1+3}{3} + \tbinom{n-d+2}{2}$, $d\geq 3$,  as in
Section~\ref{sec:A.1}. The projection $\kappa: \cH \to X = \P(S_1)$,
defined by $(\ell, f) \mapsto \langle \ell \rangle$, makes $\cH$ a
projective bundle over $X$, hence $\cF := \cH \times_X \fX$ is projective and
smooth over $X$.

As usual $\HH = \HH_Q$, $Q(n) = \tbinom{n-1+3}{3} +\tbinom{n-a+2}{2}+
\tbinom{n-b+1}{1}$, $a= d+1$.  Let be $c:=b-a+1$. The morphism $\gamma: \cF
\to \HH$ is defined by mapping $[ (\ell, f), (\ell, \cK))]\in \cF(A)$ to
$(\ell, f \cdot \cK) \in \HH(A)$. We show that $\gamma(A)$ is injective:
$(\ell_1, f_1 \cdot \cK_1) = (\ell_2, f_2 \cdot \cK_2) \Rightarrow \ell_1 A =
\ell_2 A$ and $f_1 \cK_1 = f_2 \cK_2$ in $H_p(A)$. Here $H_p \simeq
\Hilb_p(\P^2_k)$, $p(n) = \tbinom{n-a+2}{2} + \tbinom{n-b+1}{1}$, if $\Spec(A)$
 is sufficiently small, such that without restriction $\ell_1 = \ell_2 = 
 ax + by + cz + t$, $\P^2 = \Proj(R)$, $R = k[x,y,z]$. Now by a result
of Fogarty~\cite[Theorem 1.4, p. 514]{F2}, $H_p \xrightarrow{\sim} \P(R_d)
\times_k \Hilb^c(\P^2)$, where $a=d+1$ and $c = b-a+1$. It
follows that $\gamma(A)$ is injective, i.e.~$\gamma$ is a monomorphism. Now
$\cF/X$ is projective, hence $\cF$ projective over $\Spec(k)$, hence
$\gamma$ projective. It follows that $\gamma$ is a closed immersion of $\cF$
into $\HH$ and we identify $\cF$ with the corresponding closed subscheme
$\cG \subset \HH$.  Thus $\cG(A)$ is the set of ideals $(\ell, f\cdot \cK)
\in \HH(A)$, where $\ell \cdot A \subset S_1 \otimes A$ respectively $f
\cdot A \subset S_d \otimes A / \ell S_{d-1} \otimes A$ are $1$-subbundles
and $\cK \subset \cO_Y$, $Y:= \Proj(S \otimes A/ \ell S(-1)\otimes A)$, is
an ideal such that $\cO_Y / \cK$  is flat over $A$ with Hilbert polynomial
$c$ (where $\Spec(A)$ is sufficiently small). 

As $\cF /k$ is smooth, it follows that $\cG /k$ also is smooth. As $\dim_k
\Hilb^c(\P^2)=2c$, it follows that 
\[
  \dim_k \cG = \tbinom{d+2}{2} +2(b-a) +4\,.
\]

\subsection{}
\label{sec:C.3.2}

For a moment we write $R = k[x,y,z]$. The same argumentation as in
Section~\ref{sec:C.1} shows that there is a closed subscheme $F \subset
\Hilb^c(\P^2)$ such that 
\[
  F(A) = \Set{ (h,g) | h \in R_1 \otimes A \text{ and } g \in R_c \otimes
    A/h\cdot R_{c-1}\otimes A \text{ generate $1$-subbundles} }\,.
\]
It follows that there is a closed subscheme $Z \subset \fX$ such that 
\[
  Z(A) = \Set{ (\ell,h,g) | 
    \begin{aligned}
      & \ell \in S_1 \otimes A,\, h \in S_1 \otimes A/ \ell \cdot A,\\
      & g \in S_c \otimes A/(\ell,h)\cdot S_{c-1}\otimes A \text{ generate
        $1$-subbundles}
    \end{aligned}
 }\,.
\]
$Y = \Flag(1,2,S_1)$ is the scheme such that 
\[
  Y(A) = \Set{ (\ell,h) | \ell \in S_1 \otimes A,\, h \in S_1 \otimes A/
    \ell \cdot A \text{ generate $1$-subbundles} }\,.
\]
$p:Y \to X = \P(S_1)$ defined by $(\ell, h) \mapsto \langle \ell \rangle$
makes $Y$ a projective bundle over $X$. The same holds true for the
projection $q:Z \to Y$ defined by $(\ell, h, g)  \mapsto (\ell,h)$.

\section{The scheme $H_m$}
\label{sec:C.4}

The notations are as before. From~\ref{sec:C.2} follows that $\cH \times_X Z$
is a closed subscheme of $\cH \times_X \fX$. Its image under $\gamma$ is
denoted by $H_m$. It follows that $H_m \xrightarrow{\sim} \cH \times_X Z$ is
a closed subscheme of $\cG$, which is smooth over $k$, such that 
\[
 H_m(A) = \Set{ (\ell,\, f(h,g)) |
    \begin{aligned}
 &  \ell \cdot A \subset S_1 \otimes A; \, f
   \cdot A \subset S_d \otimes A / \ell S_{d-1} \otimes A; \,  h \cdot A \subset S_1 \otimes A / \ell \cdot A; \,\\
&  g \cdot A \subset S_c \otimes A /
 \langle \ell,h \rangle S_{c-1} \otimes A \text{ are $1$-subbundles}
\end{aligned}
}.
\]
It follows from this description that $\dim_k H_m = \tbinom{d+2}{2} + (b-a)
+5$\,.
\section{Ideals with maximal regularity}
\label{sec:C.5}

Let be $P = k[x,y,z,t]$, $R = k[y,z,t]$, $X = \Proj(P)$, $Q(n) =
\tbinom{n-1+3}{3} + \tbinom{n-a+2}{2} +  \tbinom{n-b+1}{1}$, $a<b$.

\begin{auxlemma}
  \label{auxlem:C.1}
  If $\cI \subset \cO_X$ has the Hilbert polynomial $Q(n)$, $\reg(\cI) =b$
  and $\cI$ is fixed by the Borel group $B = B(4;k)$, then $\cI$ is equal to
  the lexicographic ideal $(x,y^a,y^{a-1}z^{b-a+1})$ with Hilbert polynomial
  $Q$.
\end{auxlemma}
\begin{proof}
  Let be $I_n = H^0(X,\cI(n))$, $I = \oplus I_n$. Let be $J \subset P$ the
  ideal generated by $I_{b-1}$, i.e.~$J_n = I_n$, if $n< b$, and $J_n = P_1
  J_{n-1}$, if $n\geq b$. Put $q(n) = \tbinom{n-1+3}{3}+
  \tbinom{n-a+2}{2}$. Then $Q(b-1) = q(b-1)$ and $Q(b) = q(b)+1$. Let be
  $\cJ = \tilde{J}$. Then $J$ is $d$-regular with $d \leq b-1$. This means,
  $\cJ$ is $d$-regular and $J_n = H^0(\cJ(n))$, if $n \geq d$
  (see~\cite[Prop. 2.6]{Green}). \\

\textsc{Case 1}:
$P_1 J_{b-1} = I_b$. Then $\reg(I)<b$, contradiction (see~\cite[Thm. 2.27]{Green}). \\

\textsc{Case 2}:
$P_1 J_{b-1} \subset I_b$ is a strict inclusion. Then $\dim P_1 J_{b-1} \leq
Q(b)-1 = q(b)$ and hence $\dim J_n = q(n)$ for all $n \geq b-1$
(cf.~\cite[Korollar 3.8]{G78}). From the special form of the Hilbert
polynomial $q(n)$ one deduces that $\cJ = (\ell,f)$, $\ell \in P_1 - (0)$
and $f \in P/\ell P(-1)$ of degree $a$ (for example, see~\cite[Abschnitt
2.8]{G82}). Because of the $B$-invariance of $\cJ$ it follows that $\cJ =
(x,y^a)$. Hence on can write $I_b=xP_{b-1} \oplus y^a R_{b-a} \oplus f\cdot
k$, where $f\in R$ is a monomial of degree $b$. Because of the
$B$-invariance of $I_b$ and $J_b$ it follows that $z \partial f/\partial t
\in J_{b-1}$, hence $\partial f/\partial t =0$ (see~\cite[Hilfssatz
1, p. 142]{T2}). Therefore one can write $f=y^iz^j$, where $i \leq a-1$,
$i+j=b$. If $i\leq a-2$, then it follows that $y \partial f/\partial z = j
y^{i+1}z^{j-1}\in J_b$ (cf.~loc.~cit.), which is not possible. Hence one has
$f = y^{a-1}z^{b-a+1}$, i.e.~$\cI$ is the lexicographical ideal.
\end{proof}

Now let be $\cI \subset \cO_X$ any ideal with Hilbert polynomial $Q(n)$ and
$\reg(\cI) = b$. Let be $I_n = H^0(\cI(n))$ and $I = \bigoplus I_n$. By
applying a suitable $g \in \Aut_k(P_1)$, one can achieve that
$\inn(g(I_n))$ is invariant under $B$, hence without restriction one can suppose
that $\inn(I_n)$ is invariant under $B$, for all $n\geq 0$. Let be $M =
\bigoplus \inn(I_n)$ and $\cM = \tilde{M}$. Then $\reg(\cI) =
\reg(\cM)$~\cite[Thm. 2.27]{Green}. Hence $\reg(\cM)=b$ is maximal and $\cM$
is equal to the lexicographical ideal, by the Aux-Lemma~\ref{auxlem:C.1}. But
then $h^0(\cI(1))=1$, and without restriction $x \in H^0(\cI(1))$. Then one
can write $\cI = x \cO_X(-1) \oplus \cL$, $\cL \subset \cO_Y$, $Y =
\Proj(R)$ and the Hilbert polynomial of $\cL$ is equal to
$\tbinom{n-a+1}{2}+\tbinom{n-b+1}{1}$. It follows that $\cL = f \cdot \cK$,
$\cK \subset \cO_Y$ has the Hilbert polynomial
$\tbinom{n-1+2}{2}+\tbinom{n-c+1}{1}$, $c = b-a+1$, $f \in R_d$, $d=a-1$
(cf.~\cite[Abschnitt 2.8]{G82}). Now $h^0(\cI(n)) = h^0(\cM(n))$ (see
\cite{Green} or \cite[Remark 2, p. 543]{G88}).  Hence $h^0(\cI(n))
=\tbinom{n-1+3}{3}$, if $n<a$; $h^0(\cI(n)) =\tbinom{n-1+3}{3}+
\tbinom{n-a+2}{2} $, if $a\leq n \leq b-1$; $h^0(\cI(n)) = Q(n)$, if $b \leq
n$. It follows that $h^0(\cL(n)) =0$, if $n<a$; $h^0(\cL(n))
=\tbinom{n-a+2}{2}$, if $a\leq n \leq b-1$; $h^0(\cL(n)) =
\tbinom{n-a+1}{2}+\tbinom{n-b+1}{1}$, if $b \leq n$. If one puts $c =
b-a+1$, one sees that $h^0(\cK (n)) = \tbinom{n-1+2}{2}$, if $0 \leq n \leq
c-1 $ and $h^0(\cK(n)) = \tbinom{n-1+2}{2} +\tbinom{n-c+1}{1}$, if $c \leq
n$. It follows that $\cK = (h, g)$, where $h$ is a linear form in $R$ and $g
\in R/hR(-1)$ is a form of degree $c$. We get:

\begin{proposition}
  \label{prop:C.1}
  If $\cI \subset \cO_X$ is an ideal with Hilbert polynomial $Q(n)$ and
  $\reg(\cI) =b$, then $\cI = (\ell, f(h,g))$, $\ell  \in P$ a linear form, $f
  \in P/\ell P(-1)$ a form of degree $d= a-1$, $h \in P / \ell P(-1)$ a
  linear form and $g \in  P/(\ell, h) P(-1)$ a form of degree $b- a+1$.
\hfill $\qed$
\end{proposition}

\begin{corollary}
    \label{cor:C.1}
 Let $\HH_Q$ be the Hilbert scheme, which parametrizes the ideals $\cI \subset
 \cO_X$ with Hilbert polynomial $Q(n)$ as above. The following statements
 are equivalent:
 \begin{enumerate}[(i)]
 \item $\xi \in H_m(k)$;
 \item The ideal $\cI \leftrightarrow \xi \in \HH_Q(k)$ has maximal
   regularity $b$.
 \item The ideal $\cI \leftrightarrow \xi \in \HH_Q(k)$ has maximal Hilbert
   function.
 \end{enumerate}
 \hfill $\qed$
\end{corollary}

\section{The first Chow group of $\cG$}
\label{sec:C.6}

We write $ R = k[y,z,t]$, $S = k[x,y,z,t]$ and we let $\G_m$ operate by
$\sigma(\lambda):x\mapsto \lambda^{g^3} x$, $y\mapsto \lambda^{g^2}y$,
$z\mapsto \lambda^g z$, $t \mapsto t$, where $g$ is a sufficiently great
natural number.

Let be $c \geq 3$, $H^c = \Hilb^c( \Proj R)$. According to~\cite{E-S} one
has: 

There is exactly one $0$-dimensional cell of the B-B-decomposition of $H^c$,
\index{Bialynicki-Birula decomposition},
which belongs to a monomial ideal $\cK_0$. There are exactly two
$1$-dimensional cells, which we denote $W_1$ and $W_2$. 

It is not difficult to see that there are four $1$-dimensional cells in  the
B-B-decomposition of $\cG$, namely:
\begin{alignat*}{2}
  Z_1 & = \Set{ (x,y^d \cdot \cK) | \cK \in W_1}\,, &\qquad Z_2 & = \Set{ (x,y^d
    \cdot \cK) | \cK \in W_2}\,, \\
  Z_3 & = \Set{ (x,y^{d-1}(\alpha y +z) \cdot \cK_0)}^-\,, & \qquad Z_4 & = \Set{
    (\alpha x + y, x^d \cdot \cL_0)}^-\,, 
\end{alignat*}
where now $\cL_0$ is the monomial ideal, which defines the $0$-dimensional
cell in $\Hilb^c( \Proj k[x,z,t])$.

\begin{corollary}
    \label{cor:C.2}
 $A_1(\cG)$ is freely generated (over $\Z$) by $[Z_1], \dots, [Z_4]$.
 \hfill $\qed$
\end{corollary}

\section{Geometry of $H_m$}
\label{sec:C.7}

We write $S = k[x,y,z,t]$ and we let $\G_m$ operate on $S$ as in
Section~\ref{sec:C.6}. Now it is not difficult to see:

\begin{proposition}
  \label{prop:C.2}
The $1$-dimensional cells of the B-B-decomposition of $H_m$ are:
\begin{alignat*}{2}
  Z_0 & = \Set{ (x, y^d(\alpha y +z, z^c))}^-\,, &\qquad Z_1 & = \Set{
    (x, y^d(y, z^{c-1}(\alpha z +t))) }^-\,, \\
  Z_2 & = \Set{ (x, y^{d-1}(\alpha y +z)(y, z^c))}^-\,, & \qquad Z_3 & = \Set{
    (\alpha x + y, x^d(x, z^c))}^-\,.
\end{alignat*}  
\hfill $\qed$
\end{proposition}

\begin{remark*}
 $Z_i$ is equal to the tautological cycle $C_i$, $1 \leq i \leq 3$.  Besides
 $(\cM_n \cdot Z_0) = \rho$ (cf.~equations~\eqref{eq:C.4},~\eqref{eq:C.5})
 below). Finally one has $[C_3] = \beta[C_1] + \gamma[C_0]$ in $A_1(\HH)$,
 where $\beta= \tbinom{a-1}{2}$ and $\gamma = (b-a)\tbinom{a}{2} +
 \tbinom{a+1}{3}$ (see~\cite[Hilfssatz 1, p. 50]{T3}).
\end{remark*}

The projection $p: H_m \to X = \P(S_1)$ is defined by $(\ell, f(g,h))
\mapsto \langle \ell \rangle$. Let be $\cL_3:=p^*(\cO_X(1))$. As $p(Z_i)$ is
one single point, $0 \leq i \leq 2$, one has $(\cL_3\cdot Z_i)=0$, $0 \leq i
\leq 2$. As $p|Z_3$ is injective and $p(Z_3)\simeq \P^1 \subset X$, one has 
$(\cL_3\cdot Z_3)=1$ and one obtains:

\begin{lemma}
  \label{lem:C.3}
   If one puts $\cF_0 = \cL_0 \otimes \cL^{-\gamma}_3$, $\cF_1 = \cL_1
   \otimes \cL^{-\beta}_3$, $\cF_2 = \cL_2$, $\cF_3 = \cL_3$, where $\cL_0,
   \cL_1, \cL_2$ are the line bundles  as in Section~\ref{sec:1.3.3}, then
   one gets the following intersection numbers:
\[
   \bordermatrix{
        & Z_0  & Z_1 & Z_2 & Z_3 \cr 
  \cF_0 & \rho & 0   & 0   &  0  \cr 
  \cF_1 &  0   & 1   & 0   &  0  \cr 
  \cF_2 &  0   & 0   & 1   &  0  \cr 
  \cF_3 &  0   & 0   & 0   &  1  \cr 
}
\]
\hfill $\qed$
\end{lemma}

\begin{proposition}
  \label{prop:C.3}
  Numerical equivalence $=$ rational equivalence on $H_m$.
\end{proposition}
\begin{proof}
  This follows from Proposition~\ref{prop:C.2} and Lemma~\ref{lem:C.3}.
\end{proof}

\begin{lemma}
  \label{lem:C.4}
 Suppose that $d\geq 3$ and $g \leq g(d)$. Then $\Pic(\HH)/ \Pic^\tau(\HH)$
 is generated by $\cM_n$, $\cM_{n+1}$, $\cM_{n+2}$, if $n\geq b$ is any
 natural number.
\end{lemma}
\begin{proof}
  Let be $\cL \in  \Pic(\HH)$ and $\cN = \cL \otimes  M^u_n \otimes
  M^v_{n+1} \otimes M^w_{n+2}$. One has to solve the equations $(\cN \cdot
  C_i)=0$, $0 \leq i \leq 2$, i.e.
  \begin{align*}
     u +v +w & = - (\cL \cdot C_0)\\ 
  u(n-b+1) +v(n+1-b+1) +w(n+2-b +1) &= - (\cL\cdot C_1) \\
u \bigl[ \tbinom{n-a+2}{2} +(n-b+1) \bigr] + v \bigl[ \tbinom{n+1-a+2}{2}
+(n+1-b+1) \bigr]  & \\
+ w \bigl[ \tbinom{n+2-a+2}{2} +(n+2-b+1) \bigr]  & = - (\cL\cdot C_2)\,.
  \end{align*}
As the determinant of the matrix formed by the coefficients is equal to $1$,
there is a solution with $u,v,w \in \Z$. Now from $(\cN \cdot C_i)=0$ and
Theorem~\ref{thm:1.2} it follows that $(\cN \cdot C)=0$ for all curves $C
\subset \HH$, hence $\cN \in \Pic^\tau(\HH)$.
\end{proof}

\begin{corollary}
  \label{cor:C.3}
  Let be $Z \in A_1(H_m)$. If $(\cL_3 \cdot Z)=0$ and $(\cM_n \cdot Z)=0$
  for all $n \gg 0$, then $Z =0$. 
\end{corollary}
\begin{proof}
  Write $\cL_i = M^u_n \otimes M^v_{n+1} \otimes M^w_{n+2} \otimes \cN$, $\cN
  \in \Pic^\tau(\HH)$. Then $(\cL_i \cdot Z)=0$, $0 \leq i \leq 2$, and
  $(\cL_3 \cdot Z)=0$ by assumption. It follows that $(\cF_i \cdot Z) =0$,
  $0 \leq i \leq 3$.  Writing $[Z] = \sum q_i[Z_i]$
  (cf.~Proposition~\ref{prop:C.2}), then from Lemma~\ref{lem:C.3} it follows
  that $q_i=0$. 
\end{proof}

\subsection*{Computation of $A^+_1(H_m)$}
 
 It is easy to see that $H_m$ has only one fixed point under $U(4;k)$,
 namely the lexicographic point. It follows that $A^+_1(H_m)$ is generated
 by combinatorial cycles of type $i$, i.e.~by cycles of the form
 $\overline{\G_a \cdot \xi_i}$, where $\xi_i \in H_m(k)$ is invariant under
 $T(4;k)$ and the subgroup $G_i \subset U(4;k)$ (see
 Appendix~\ref{cha:H}). If $\xi \leftrightarrow (\ell, f(h,g)) \in H_m(k)$
 is fixed by $T(4;k)$, then all forms are monomials. 
, if $\xi$ is
 fixed by $G_i$, then $\ell =x$, if $i=1,2$ and $\ell =x$ or $\ell =y$, if
 $i=3$.

\fbox{$i=1$}
$\xi \leftrightarrow \cI = (x,f(h,g)$ monomial and invariant under $G_1 =
\left\{
\left(\begin{smallmatrix} 1&*&*& *\\ 0&1&*&*\\ 0&0&1&0\\ 0&0&0&1
  \end{smallmatrix}\right)\right\}$ $\Rightarrow f \in k[y,z,t]_d$
$G_1$-invariant modulo $x$ $\Rightarrow f= y^d$. $h$ monomial and
$G_1$-invariant modulo $x$ $\Rightarrow h = y \Rightarrow g \in k[z,t]_c$
monomial and $G_1$-invariant modulo $(x,y)$ $\Rightarrow g = z^\nu t^\mu$,
$\nu + \mu =c$.

If $i=1$, $\G_a$ operates by $\psi^1_\alpha:x\mapsto x$, $y\mapsto y$,
$z\mapsto z$, $t \mapsto \alpha z + t$.

Let be $\cI_\alpha = \psi^1_\alpha(\cI)$. Then
$
   H^0(\cI_\alpha(n)) = x S_{n-1} \oplus y^a k[y,z,t]_{n-a} \oplus y^{a-1}
   z^\nu(\alpha z +t)^\mu k[z,t]_{n-b}.
$
$\Rightarrow \alphadeg\dot\bigwedge H^0(\cI_\alpha(n)) = \mu(n-b+1)
\Rightarrow$
\begin{equation}
  \label{eq:C.1}
  ( \cM_n \cdot C) = \mu (\cM_n \cdot C_1) \,.
\end{equation}

\fbox{$i=3$}

\textsc{Subcase 1}:
$\ell =x$. Then $f$ is a monomial, which is invariant under $G_3 = 
\left\{
\left(\begin{smallmatrix} 1&0&*&*\\ 0&1&*&*\\ 0&0&1&*\\ 0&0&0&1
  \end{smallmatrix}\right)\right\}$ modulo $x \Rightarrow f = y^d$. $h$ is a
monomial and $G_3$-invariant modulo $x \Rightarrow h=y$ and $g =
z^\nu t^\mu$, $\nu+\mu=c$, $G_3$-invariant modulo $(x,y) \Rightarrow g = z^c
\Rightarrow \xi$ is the lexicographical point and does not generate a cycle,
at all.

\textsc{Subcase 2}:
$\ell =y$. Then $f \in k[x,z,t]_d$ is a monomial and $G_3$-invariant modulo
$y \Rightarrow f=x^d$. $h$ is a monomial in $k[x,z,t]_1$ and $G_3$-invariant
modulo $y \Rightarrow h = x$ and $g=z^\nu t^\mu$, $\nu + \mu =c$,
$G_3$-invariant modulo $(x,y) \Rightarrow g = z^c \Rightarrow \xi
\leftrightarrow (y,x^d(x,z^c)) \Rightarrow C = \overline{\G_a \cdot \xi} =
\Set{ (\alpha x+y, x^d(x,z^c))}^- = C_3$.

\fbox{$i=2$} $\xi \leftrightarrow (x,f(h,g))$ monomial and invariant under
$G_2 = \left\{
\left(\begin{smallmatrix} 1&*&*& *\\ 0&1&0&*\\ 0&0&1&*\\ 0&0&0&1
  \end{smallmatrix}\right)\right\}$ $\Rightarrow f = y^\nu z^\mu$, $\nu+\mu
=d$. $h$ monomial and $G_2$-invariant modulo $x$. There are two
possibilities:

\textsc{Subcase 1}:
$h=y$. Then $g = z^\nu t^\mu$, $\nu+\mu=c$, $g$ $G_2$-invariant modulo
$(x,y) \Rightarrow g = z^c \Rightarrow \xi \leftrightarrow \cI =(x,y^\nu z^\mu(y,z^c))$. 

\textsc{Subcase 2}: $h=z$. Then $g = y^\nu t^\mu$, $\nu+\mu=c$, $g$
invariant under $G_2$ modulo $(x,z) \Rightarrow g = y^c$
and $\xi \leftrightarrow \cI =(x,y^\nu z^\mu(z,y^c))$. 

If $i=2$, then $\G_a$ operates by $\psi^2_\alpha:x\mapsto x$, $y\mapsto y$,
$z\mapsto \alpha y +z$, $t \mapsto t$.

Let be $\cI_\alpha = \psi^2_\alpha(\cI)$ and $\cI'_\alpha$ the restriction
with respect to $t$.

\textsc{Subcase 1}: $H^0(\cI'_\alpha(n)) = x k[x,y,z]_{n-1}\oplus
y^\nu(\alpha y +z)^\mu (y,z^c)_{n-d} = x k[x,y,z]_{n-1}\oplus y^\nu(\alpha y
+z)^\mu k[y,z]_{n-d}$, if $n\geq b$, because $(y,z^c)_{n-d} = k[y,z]_{n-d}$,
if $n-d \geq c$\\
 $\Rightarrow \alphadeg\dot\bigwedge H^0(\cI'_\alpha(n)) =
\mu(n-d+1)$ if $n\geq b$.  The sequence
\begin{equation}
  \label{eq:C.2}
 0 \longrightarrow H^0(\cI_\alpha(n-1)) \longrightarrow H^0(\cI_\alpha(n)) \longrightarrow H^0(\cI'_\alpha(n)) \longrightarrow 0
\end{equation}
is exact, if $n\geq b$, hence
\[
\alphadeg\dot\bigwedge H^0(\cI_\alpha(n)) = 
\alphadeg\dot\bigwedge H^0(\cI_\alpha(b-1)) + \sum^n_{i=b}\mu(i-d+1)\,.
\]
Now $H^0(\cI(b-1)) = xS_{b-2} \oplus y^{\nu+1}z^\mu k[y,z,t]_{b-a-1}$, hence
$\alphadeg\dot\bigwedge H^0(\cI_\alpha(b-1)) = \mu \cdot \rho$. Now
$\sum^n_{i=b} (i-d+1) = \tbinom{n-a+3}{2}- \tbinom{b-a+2}{2}$ and one
checks that 
\[
\tbinom{n-a+3}{2}- \tbinom{b-a+2}{2} + \tbinom{b-a+1}{2} =
\tbinom{n-a+2}{2}+(n-b+1)\,.
\]  
It follows that
\begin{equation}
  \label{eq:C.3}
  (\cM_n \cdot C) = \mu (\cM_n \cdot C_2)\,.
\end{equation}

\textsc{Subcase 2}:
$H^0(\cI'_\alpha(n)) = x k[x,y,z]_n\oplus y^\nu z^\mu k[y,z]_{n-d}$, if
$n\geq b$. From the sequence~\eqref{eq:C.2} it follows that 
\[
\alphadeg\dot\bigwedge H^0(\cI_\alpha(n)) = 
\alphadeg\dot\bigwedge H^0(\cI_\alpha(b-1)) + \sum^n_{i=b}\mu(i-d+1)\,.
\]
Now $H^0(\cI(b-1)) = xS_{b-2} \oplus y^\nu (\alpha y + z)^{\mu+1} k[y,z,t]_{b-a-1}$, hence
$\alphadeg\dot\bigwedge H^0(\cI_\alpha(b-1)) = (\mu +1 )\cdot \rho$ and
\begin{equation}
  \label{eq:C.4}
  (\cM_n \cdot C) = \mu (\cM_n \cdot C_2) + \rho\,.
\end{equation}

If $\mu =0$ one gets 
  $C = \Set{(x,y^d(\alpha y +z,y^c))}^- = \Set{(x,y^d(\alpha y +z,z^c))}^-$,
  i.e.
\begin{equation}
  \label{eq:C.5}
  C = Z_0\,.
\end{equation}
If $i=1$ or $i=2$, then $p(C) = 1$ point, hence $(\cL_3\cdot C)=0$ in
these cases. From Corollary~\ref{cor:C.3} and the equations~\eqref{eq:C.1},
\eqref{eq:C.3} and \eqref{eq:C.4} it follows that that $[C] = \mu [C_1]$, if
$i=1$, and $[C] = \mu [C_2]$ or $[C] = \mu [C_2] + [Z_0]$, if $i=2$.
We have proved

\begin{proposition}
  \label{prop:C.4}
 $A^+_1(H_m)$ is freely generated by the cycle classes of $Z_0 =                                                                                                                  \Set{
   (x,y^d(\alpha y + z, y^c)}^-$ and $C_i$, $1\leq i \leq 3$.
  \hfill $\qed$
\end{proposition}

\chapter{The Hilbert scheme of points in $\P^2$}
\label{cha:D}

 \section{Tautological line bundles}
 \label{sec:D.1}

 The ground field is $k=\C$. The group $T=T(3;k)$ of diagonal matrices, the
 upper unitriangular group $U = U(3;k)$ and the Borel group $B = T \cdot U$
 all act on $S = k[x,y,z]$ and therefore on the Hilbert scheme $H^d =
 \Hilb^d(\P^2)$, which parametrizes subschemes  of $\P^2$ of length $d$. (We
 always assume $d\geq 3$.) 

 If $A$ is a $k$-algebra, an element of $H^d(A)$ is a closed subscheme $Z$ of
 $\P^2 \otimes A$, flat over $A$, such that $Z \otimes k(p)$ has the Hilbert
 polynomial $P(n) =d$, for all points $p \in \Spec A$. If $Z = V(\cI)$,
 i.e.~if $Z$ is defined by the ideal $\cI \subset \cO_{\P^2 \otimes A}$,
 then $\cI \otimes k(p)$ has the Hilbert polynomial $Q(n) = \binom{n-1+2}{2}
 + \binom{n-d+1}{1}$, and therefore $\reg(\cI \otimes k(p)) \leq d$~\cite[p.
 65]{G78}. From standard results on the cohomology of coherent
 sheaves~\cite[lecture 7]{M2} it follows that that $H^0(\P^2 \otimes A,
 \cI(n)) \subset S_n \otimes A$ is a subbundle of rank $Q(n)$, and the
 formation of $H^0(-)$ commutes with base extensions $A \to A'$, if $n \geq
 d-1$\,.  Thus $\cI \mapsto H^0(\cI (n))$ defines a morphism $H^d \to
 \Grass_{Q(n)}(S_n)$, if $n\geq d-1$\,. From $d$-regularity it follows that
 $S_{n-d}H^0(\cI (d)) = H^0(\cI (n))$, $n\geq d$, which implies that this
 ``Hilbert-Grassmann'' morphism is a closed immersion, if $n\geq
 d$~\cite[Lectures 7 and 14]{M2}.

In the following diagram, $\cZ$ is the universal subscheme of length $d$.
\[
 \xymatrix{\cZ \ar@{^{(}->}[rr]  & & H^d\times_k \P^2_k \ar[dl]_\pi 
                        \ar[dr]^\kappa &  \\
                                         & H^d &  & \P^2_k}
\]
Set $\cF (n) = \cO_\cZ \otimes \kappa^*\cO_{\P^2}(n)$. Then $\cF_n:= \pi_*
\cF(n)$ is locally free of rank $d$, for all $n$, and $\cF_n$ is globally
generated for $n \geq d-1$\,. Therefore the tautological line bundles
$\cM_n := \bigwedge^d \cF_n$ are globally generated for $n\geq d-1$\,. For
$n\geq d$, the line bundle $\cM_n$ is very ample, because it defines the
``Hilbert--Pl\"ucker'' embedding $H^d \to \P^N$, the composition of the 
Hilbert--Grassmann embedding with the ``Grassmann--Pl\"ucker'' embedding
$\Grass_{Q(n)}(S_n) \to \P^N$.

\section{Combinatorial and algebraic cycles on $H^d$}
\label{sec:D.2}

 A weak version of a theorem of Hirschowitz gives: \\
The cone of effective $1$-cycles 
\[
  A^+_1(H^d) = \Set{ \sum q_i [C_i] | q_i \in \Q, q_i \geq 0, C_i \subset
    H^d\; 1\text{-prime cycle}} 
\]
is generated by $B$-invariant $1$-prime cycles ($=$ closed, reduced,
irreducible curves in $H^d$).

This is true for $\Z$-coefficients, too, but in the following we will always
take $\Q$-coefficients.

Now, if $C$ is such a $B$-invariant $1$-prime cycle, the following cases can
occur: \\

\textbf{Either:} $C$ is a so called combinatorial cycle, i.e.~$C=
\overline{\G_a \cdot\xi}$, $\xi \in H^d(k)$ is
fixed by $T$ and by $G_1 = \Set{\left(\begin{smallmatrix} 1&*&* \\ 0&1&0\\
      0&0&1 \end{smallmatrix}\right) }$ or $G_2 =
\Set{\left(\begin{smallmatrix} 1&0&* \\ 0&1&*\\ 0&0&1
    \end{smallmatrix}\right) }$; $\G_a$ operates on $S$ via the
automorphisms $\psi^1_\alpha:x\mapsto x$, $y\mapsto y$, $z\mapsto \alpha y +
z$ and $\psi^2_\alpha:x\mapsto x$, $y\mapsto \alpha x + y$, $z\mapsto z$ and
respectively. \\
\textbf{Or:} $C$ is an ``algebraic cycle'', i.e.~ $C = \overline{\G_m \cdot
  \xi}$, $\xi \in H^d(k)$ is fixed by $U$, and $\G_m$ operates on $S$ via the
automorphism $\sigma(\lambda): x\mapsto x$, $y\mapsto y$, $z\mapsto \lambda
z$.
 
\begin{remark}
  \label{rem:D.1}
  A somewhat more detailed description of $B$-invariant $1$-prime cycles in
  the case of $n \geq 3$ variables is given in~\cite{T1}.
\end{remark}

\section{Intersection numbers and basis cycles}
\label{sec:D.3}

Let $C = \overline{\G_a \cdot \xi}$, $\xi \leftrightarrow \cI$, be a
combinatorial cycle. As $\psi: \A^1 \to H^d$, defined by $\alpha \mapsto
\psi_\alpha(\xi)$, is injective, the intersection number can be computed by
the following formula (cf.~\cite[1.3]{T1};~\cite[4.1]{T2};~\cite[Anhang
2]{T3}):
\[
   (\cM_n \cdot C) = \alphadeg \bigwedge^{Q(n)}\psi_\alpha(H^0(\cI (n))), \quad
   n\geq d\,.
\]
Here $\alphadeg(-)$ denotes the highest power with which the parameter
$\alpha$ appears in the bracket. (Take a monomial basis $m_1,\dots,m_q$ of
$H^0(\cI(n))$, replace $z$  by $\alpha y + z$, and express
$\psi_\alpha(m_1), \dots,\psi_\alpha(m_q)$ as a linear combination in a
monomial basis of $S_n$. The coefficients are polynomials in $k[\alpha]$.)

In the case of an algebraic cycle, as $\sigma: \A^1-\{0\} \to H^d$ defined
by $\xi \mapsto \sigma(\lambda)\xi$ need not to be injective, in order to
compute $(\cM_n \cdot C)$, one has to take a ``reduced-$\lambda$-degree'' as
defined in~\cite[equation (2) on p. 9]{T2}.

We start the computation with the cycle $E = \Set{ (x^2, xy, y^{d-1} +
  \alpha xz^{d-2}) | \alpha \in k}^-$.\\
If $\xi \leftrightarrow \cI = (x^2, xy, y^{d-1} + xz^{d-2})$, one
sees that
\[
  H^0(\sigma(\lambda)\cI (n)) =
  x^2 S_{n-2} \oplus xyk[y,z]_{n-2} \oplus y^d k[y,z]_{n-d} \oplus \langle
(y^{d-1} + \lambda^{d-2} xz^{d-2}) z^{n-d+1} \rangle. 
\]
Therefore, the reduced-$\lambda$-degree of $\bigwedge^{Q(n)}
H^0(\sigma(\lambda)\cI (n))$ is equal to $1$\,.

Now we consider $F = \Set{ (x, y^{d-1}(\alpha y +z)) | \alpha \in k}^-$.  
If $\xi \leftrightarrow \cI = (x,y^{d-1}z)$, then 
\[
H^0(\psi^1_\alpha\cI (n)) = xS_{n-1} \oplus y^{d-1}(\alpha y +z)
k[y,z]_{n-d}\,, 
\]
and therefore
\[
 \bigwedge^{Q(n)}H^0(\psi^1_\alpha\cI (n)) = \bigwedge^p x S_{n-1} \otimes 
 \bigwedge^q y^{d-1}(\alpha y +z) k[y,z]_{n-d}\,,
\]
where $p:=\binom{n-1+2}{2}$  and $q:=n-d+1$. We get:
\begin{equation}
  \label{eq:D.1}
  (\cM_n \cdot E) = 1, \quad (\cM_n \cdot F) = n-d +1, \quad
   n\geq d-1\,.
\end{equation}

\section{Intersection numbers of combinatorial  cycles}
\label{sec:D.4}

Let $C$ be a combinatorial cycle of type $1$, i.e.~$C = \overline{\G_a \cdot
\xi}$, and $\xi \in H^d(k)$ invariant under $T$ and $G_1$. Writing $S =
k[y,z,t]$, an analogous argumentation as in the proof of
Conclusion~\ref{concl:1.1} in Chapter~\ref{cha:1} (see Fig.~\ref{fig:1.1})
gives:
\begin{equation}
  \label{eq:D.2}
  (\cM_n \cdot C) = a(n-d+1) +b\,,
\end{equation}
where $a,b \in \N$ are independent of $n \geq d$. 

In the case of a combinatorial cycle of type $2$, i.e.~$C = \overline{\G_a
  \cdot \xi}$, and $\xi \in H^d(k)$ invariant under $T$ and $G_2$, one can
argue as in the proof of Conclusion~\ref{concl:1.3} in Chapter~\ref{cha:1}
(see Fig.~\ref{fig:1.6}) and one obtains the formula
\begin{equation}
  \label{eq:D.3}
  (\cM_n \cdot C) = c\,,
\end{equation}
where $c\in \N$ is independent of $n \geq d$. 

\section{Intersection numbers of algebraic  cycles}
\label{sec:D.5}

We start with a general lemma, which is interesting for itself, possibly.

Set $S= k[x_1,\dots,x_r,t]$, $R = k[x_1,\dots,x_r]$. $\G_m$ operates on $S$
via $\sigma(\lambda): x_i \mapsto x_i$, $1 \leq i \leq r$, and $t \mapsto
\lambda t$, $\lambda \in k^*$. Let $\HH$ be the Hilbert scheme of ideals
$\cI \subset \cO_{\P^r}$ with Hilbert polynomial $Q$, i.e.~ $\HH =
\Hilb^P(\P^r)$, $P(n) = \binom{n+r}{r} - Q(n)$ the complementary Hilbert
polynomial of the subscheme $V(\cI) \subset \P^r$. We suppose $\HH \neq
\emptyset$. Then the ideals $\cI \subset \cO_{\P^r}$ with Hilbert polynomial
$Q$, such that $t$ is a non-zero divisor of $\cO_{\P^r}/\cI$, form an open
non-empty subset $U(t)$ of $\HH$.

If $K/k$ is a field extension and $\cI \in \HH(K)$, then the limit ideals
$\cI_{0/\infty} = \lim_{\lambda \to 0/\infty} \sigma(\lambda)\cI$ are in
$\HH(K)$, and if $\cI \in U(t)$, then $\cI_0$ is in $U(t)$
again~\cite[Lemma 4]{G88}. If $Q'(n):= Q(n) -Q(n-1)$ and $\cI':=\cI + t
\cO_{\P^r}(-1)/t \cO_{\P^r}(-1)$, then $\cI'$ can be considered as a sheaf
of ideals on $\P^{r-1}$, whose Hilbert polynomial is equal to $Q'$.

\begin{lemma}
  \label{lem:D.1}
  Let $\cI \in U(t)$ and suppose that $\cI_\infty$ is in $U(t)$, too
  (this condition is fulfilled, e.g., if $\cI$ is invariant under
  $U(r+1,k)$). Then for all integers $d \geq \max(
  \reg(\cI_0),\reg(\cI_\infty))$, one has $\dim H^0(\cI (d)) \cap R_d =
  Q'(d)$.
\end{lemma}
\begin{proof}
There is a basis of $M:= H^0(\cI (d))$ of the form $g_i = t^{e_i}g^0_i +
t^{e_i-1}g^1_i + \cdots$ with $0 \leq e_1 \leq \cdots \leq e_m$, $m:= Q(d)$,
$g^j_i \in R$, and $g^0_i \in R_{d-e_i}$, $1 \leq i \leq m$, linear
independent. Then $M_\infty := \lim_{\lambda \to \infty}\sigma(\lambda)M =
\langle \Set{ t^{e_i}g^0_i | 1\leq i \leq m}\rangle$ (limit in
$\Grass_M(S_d)$) has dimension $m$.  As, by assumption, $d\geq
\reg(\cI_\infty)$, one has $Q(d) = h^0(\cI_\infty(d))$, and therefore
$M_\infty = H^0(\cI_\infty(d))$. As $t$ is a non-zero divisor of $S/
\bigoplus_{n \geq 0} H^0(\cI(n))$ by assumption, this implies
\[
H^0(\cI_\infty(n)) = \langle \Set{ t^{e_i-(d-n)}g^0_i | e_i \geq
  d-n}\rangle\,.
\]
Especially, for $n = d-1$, one obtains
\[
H^0(\cI_\infty(d-1)) = \langle \Set{ t^{e_i-1}g^0_i | e_i \geq
  1}\rangle\,,
\]
therefore $Q(d-1) = h^0(\cI_\infty(d-1)) = \# \Set{ i | e_i \geq 1}$. It
follows that $Q'(d) = Q(d) -Q(d-1) = \# \Set{ i | e_i =0}$. Thus $M \cap R_d
 \supset  \langle \Set{ g^0_i | e_i = 0 }\rangle$ has a dimension $\geq
 Q'(d)$. It cannot be greater than $Q'(d)$, as the canonical restriction
 mapping ``reduction modulo $t$'' $ M = H^0(\cI(d)) \mapsto
 H^0(\cI'(d))$ is injective on $M \cap R_d$ and, because of $\reg(\cI
') \leq \reg(\cI)$, one has $ h^0(\cI'(d)) = Q'(d)$.
\end{proof}

\begin{remark}
  \label{rem:D.2}
    There is a partial inverse of the lemma. Suppose $\cI \in U(t)$ and
    $\cI_\infty \in U(t)$, and let $d \geq \reg(\cI')$ be any integer. If
    $\dim( H^0(\cI (d)) \cap R_d) \geq Q'(d)$, then $d \geq \max(
  \reg(\cI_0),\reg(\cI_\infty))$.
\end{remark}

Now, let $C = \overline{\G_m \cdot \xi}$ be an algebraic cycle, $\xi
\leftrightarrow \cI$ $U(3;k)$-invariant, $\sigma: \G_m \to \Aut(S)$ defined
by $\sigma(\lambda): x\mapsto x$, $y\mapsto y$, $z\mapsto \lambda z$.  As
$\cI \in U(z)$ and $U(3;k)$ is normalized by $\G_m$, $\cI_\infty$ is
fixed by $U(3;k)$, thus $\cI_\infty \in U(z)$. Obviously, for $n \geq d$,
there is an inclusion
\[
 z^{n-d} H^0(\cI (d)) \bigoplus^n_{\nu = d+1} [ H^0(\cI (\nu)) \cap R_\nu]
 \cdot z^{n -\nu} \subseteq H^0(\cI (n))\,.
\]
As $\reg(\cI) \leq d$ for all $\cI \in H^d$, the lemma gives equality, at
once. As $\G_m$ operates trivially on $H^0(\cI (\nu)) \cap R_\nu$, all
polynomials, which appear in $\bigwedge^{Q(n)} \sigma(\lambda) H^0(\cI (n))$
have a constant $\lambda\text{-degree} \leq Q(d)\cdot d$, essentially:
\begin{equation}
  \label{eq:D.4}
  (\cM_n \cdot C) = c \in \N, \quad \text{independent of } n\geq d\,.
\end{equation}

\section{The cone of effective $1$-cycles of $H^d$}
\label{sec:D.6}

 We need the following results:

 \begin{itemize}
 \item $H^d$ is smooth of dimension $2d$, and $\Pic(H^d) = \Z^2$ \quad
   (Fogarty).
 \item $A_n(H^d)\simeq H^\mathrm{BM}_{2n}(H^d)$ \quad (Ellingsrud--Str\o{}mme).
 \end{itemize}

Using Poincar\'e duality, one gets $A_1(H^d) \simeq \Q^2$.

Now $[E]$ and $[F]$ are linear independent: If $q_1[E] + q_2[F] =0$ in
$A_1(H^d)$, formula~\eqref{eq:D.1} gives $q_1+q_2(n-d+1) =0$, $n\geq d$,
thus $q_1 = q_2 =0$\,. It follows that $[E]$ and $[F]$ generate $A_1(H^d)$,
and the same argumentation shows $A^\tau_1(H^d) =(0)$. One deduces the following

\begin{lemma}
  \label{lem:D.2}
  \begin{enumerate}
  \item The cycles $[E]$ and $[F]$ form a basis of $A_1(H^d)$.
  \item Numerical and rational equivalence coincide on $H^d$.
  \end{enumerate}
\end{lemma}

Using this lemma and the formulae~\eqref{eq:D.1}--\eqref{eq:D.4}, one
immediately obtains: 

\begin{proposition}
  \label{prop:D.1}
 The cone of effective $1$-cycles $A^+_1(H^d)$ is spanned by $[E]$ and $[F]$.  \hfill \qed
\end{proposition}

\section{The ample cone of $H$}
\label{sec:D.7}

 If $\cL \in \Pic(H^d)$ and $(\cL\cdot E) = \nu$, $(\cL \cdot F) = \mu$, we
 set $\cM :=\cM^{\nu -\mu}_{d-1} \otimes \cM^\mu_d$. Then $(\cL \otimes
 \cM^{-1} \cdot E) = (\cL \otimes  \cM^{-1} \cdot F) = 0$, thus $\cL \otimes
 \cM^{-1} \in \Pic^\tau(H^d)$. But $\Pic(H^d) \simeq \Z^2$, therefore
 $\Pic^\tau(H^d)= \Pic^0(H^d) =(0)$, and we have:

 \begin{lemma}
   \label{lem:D.3}
  $\Pic(H^d)$ is generated by $\cM_{d-1}$ and $\cM_d$.
 \hfill \qed
 \end{lemma}

 We first determine the pseudoample cone of $H^d$. Let $\cL \in \Pic(H^d)$
 be such that $(\cL \cdot z) \geq 0$, for all $z \in A^+_1(H^d)$. Writing
 $\cL = \cM^{\nu_1}_{d-1} \otimes \cM^{\nu_2}_d$ and using
 Proposition~\ref{prop:D.1}, we see that this is equivalent to $\nu_1 +
 \nu_2 \geq 0$ and $\nu_2 \geq 0$\,. By Kleiman's theorem, the ample cone is
 the interior of the pseudoample cone, hence we get:

 \begin{theorem}
   \label{thm:D.1}
  The ample cone of $\Hilb^d(\P^2)$ is generated by $\cL_1 = \cM_{d-1}$ and
  $\cL_2 = \cM^{-1}_{d-1} \otimes \cM_d$.  \hfill \qed
 \end{theorem}

\section{Globally generated line bundles on $H^d$}
\label{sec:D.8}

We have already noted that $\cM_{d-1}$ is globally generated and the same is
true for $\cM^{-1}_{d-1} \otimes \cM_d$ (see Section~\ref{sec:1.5.2},
Lemma~\ref{lem:1.2}).

\begin{proposition}
  \label{prop:D.2}
 $\cL_1$ and $\cL_2$ are globally generated.
\end{proposition}

\begin{remark}
    \label{rem:D.3}
  If $\cL$ is any line bundle on $H^d$, we can write $\cL = \cL^{\nu_1}_1
  \otimes \cL^{\nu_2}_2$. Now, if $\cL$ is globally generated, then $\nu_1,
  \nu_2\in \N$.
\end{remark}

As the vertex of the ample cone $\cL_1 \otimes \cL_2 = \cM_d$ is very
ample, by Proposition~\ref{prop:D.2} this implies:

\begin{proposition}
  \label{prop:D.3}
  Every ample line bundle on $\Hilb^d(\P^2)$ is very ample.
\end{proposition}

\begin{remark}
   \label{rem:D.4}
  If one computes the $\alpha$-degree of the ideal $\cI_\alpha$
  corresponding to a ``general'' point of the cycle $\beta_\ell$
  (resp.~$\beta_d$) defined in \cite[(1.1) and (1.2)]{LQZ}, then one obtains
  $(\cM_n \cdot \beta_\ell) =n$, $(\cM_n \cdot \beta_d) =1$, therefore
  $[\beta_d] = [E]$, $[\beta_\ell -(d-1)\beta_d] = [F]$.
\end{remark}

\begin{remark}
   \label{rem:D.5}
   The basic result $A_1(H^d)\simeq \Q^2$ was deduced from Fogarty's result
   $\Pic(H^d) \simeq \Z^2$. Of course, one could have used the method
   of~\cite{E-S} by counting the number of $1$-dimensional cells in a
   Bialynicki--Birula decomposition of $H^d$. Finally, one could have used
   Iarrobino's result $\Pic(H^d)\otimes \Q \simeq \Q^2$, too, which is
   earlier than Fogarty's result (cf.~\cite[p.~821]{I} and~\cite[p.~660]{F3}).
\end{remark}

\section{The action of $\Aut(H^d)$ on $A_1(H^d)$}
\label{sec:D.9}

\subsection{}
\label{sec:D.9.1}

 Let $\cH \subset H^d$ be the closed subscheme parametrizing the ideals with
 maximal regularity, Then $\cH(k)  = \Set{ (\ell, f) | \ell \in S_1 - (0),
   \, f \in [S/\ell S(-1)]_d - (0)}$, and we show that $\cH$ has a natural
 scheme structure: Let be $A$ a $k$-algebra and $\ell \in S_1 \otimes A$ a
 form, which generates a $1$-subbundle, and $f \in [ S \otimes  A /\ell
 S(-1)\otimes A]_d$ a form, which generates a $1$-subbundle. Let $\Spec A$
 be so small that, without restriction, $\ell = ax +by +z$, $a,b \in
 A$. Then $S \otimes  A /\ell  S(-1)\otimes A \simeq R:= A[x,y]$. 

Let be $I = (\ell,f) \subset S \otimes A$ the ideal, which is generated by
$\ell$ and $f$.  We show that $I_n  = \ell S_{n-1} \otimes A \oplus f \cdot
R_{n-d}$ is a subbundle of $S_n \otimes A$ of rank $Q(n) = \binom{n-1+2}{2}
+ \binom{n-d+1}{1}$. As one can suppose $f \in R_d$, it suffices to show
that $f R_{n-d} \subset R_n$ is a subbundle of rank $n-d+1$\,. If $y \in
\Spec A$, one has to show that the canonical homomorphism $f \cdot R_{n-d}
\otimes k(y) \to R_n \otimes k(y)$ is injective. Let be $g \in R_{n-d}$ such
that $\bar{f} \cdot \bar{g} =0$ in $R \otimes k(y)$. But as $f \cdot A
\subset R_d$ is a $1$-subbundle, $\bar{f} \neq 0$ and hence
$\bar{g}=0$\,. It follows that $(\ell, f) \mapsto \langle \ell \rangle$
makes $\cH \to \P^2$ a projective bundle.

\subsection{}
\label{sec:D.9.2}

Let $f: H^d \to \P$ be the morphism, which is defined by the globally
generated line bundle $\cM_{d-1}$. Then $(\cM_{d-1} \cdot F) =0$\,. Let $C
\subset H^d$ be a curve such that $C \sim \nu F$. Then $0 = (\cM_{d-1} \cdot
C) = \deg(f| C) \cdot \bigl(\cO_\P(1)\cdot f(C)\bigr)$, hence $f(C)$ is a
point. As the restriction of $f$ to $H^d - \cH$ is an isomorphism, one
obtains:

 \begin{corollary}
   \label{cor:D.1}
  Let $C \subset H^d$ be a curve such that $[C] = \nu [F]$, where $\nu \in
  \N -\{0\}$. Then $C \subset \cH$.
 \hfill $\qed$
 \end{corollary}
\subsection{}
\label{sec:D.9.3}

 From Proposition~\ref{prop:D.1} it follows that $\phi \in \Aut(H^d)$
 permutes the cycles classes $[E]$ and $[F]$. We show that $\phi$ leaves
 invariant these classes, and we suppose that $[\phi{(E)}] = [F]$.

 Let $M$ be the set of subscheme of $\P^2$ of the shape $\Set{
   P_0,\dots,P_{d-2}}$, where $P_0$ is a point of multiplicity $2$, and the
 points $P_1,\dots,P_{d-2}$ are different points of multiplicity $1$\,. If
 one fixes  $P_1,\dots,P_{d-2}$, then $M$ contains the set
 $D:=\Quot^2(\cO_{\P^2}/ P^2_0)$, where $P_0 =(x,y)$, without
 restriction. If one puts $\cL_\alpha:= (\alpha x +y, x^2)$ and $\cI_\alpha
 := P_1 \cap \cdots \cap P_{d-2} \cap \cL_\alpha$, then $D \simeq \Set{
   \cI_\alpha | \alpha \in k}^-$ and one has (see~\cite[p. 49]{T3}):
\[
  \alphadeg \dot\bigwedge H^0(\cI_\alpha(n)) = \alphadeg \dot\bigwedge
  H^0(\cL_\alpha(n)) = 1\,,
\]
for all $n \geq d$, hence $[D] = [E]$ by Lemma~\ref{lem:D.2}. It follows
that $[\phi(D)] = [\phi(E)] = [F]$, and by Corollary~\ref{cor:D.1} one
obtains $\phi(D) \subset \cH$. Now $\dim M = 2 \cdot (d-2) +1$ and $\dim \cH
= d+2$\,. Thus from $\phi(M) \subset \cH$ it follows that $d \leq 5$ and
hence we get:

\begin{proposition}
  \label{prop:D.4}
  If $d\geq 6$, then $\Aut(H^d)$ trivially acts on $A_1(H^d)$.
 \hfill $\qed$
\end{proposition}

\begin{corollary}
  \label{cor:D.2}
 Each $\phi \in \Aut(H^d)$ leaves $\cH$ invariant.
\end{corollary}
\begin{proof}
   Let be $(\ell, f) \leftrightarrow \xi \in \cH(k)$.  If $g \in S_d /\ell
   S_{d-1}$ is linearly independent of $f$, then $C := \Set{ (\ell, \alpha f
     + \beta g)}^- \subset \cH$ and a similar argumentation as in the proof
   of Proposition~\ref{prop:2.2} in Chapter~\ref{cha:2} shows that $[C] =
   [F]$. It follows that  $ [\phi(C)] = [C] = [F]$, hence $C \subset \cH$ by
   Corollary~\ref{cor:D.1}. 
\end{proof}

\subsection{}
\label{sec:D.9.4}

If $\G_m$ operates by $\sigma(\lambda): x \mapsto \lambda^{g^2}x$, $y
\mapsto \lambda^g y$, $z \mapsto z$, then the two cycles $F= \Set{ (x,y^{d-1}(\alpha y
  +z))}^-$ and $G:=\Set{(\alpha x +y,x^d)}^-$ are the $1$-dimensional cells
of the B-B-decomposition of $\cH$. There is only $1$ fixed point under the
action of $U(3,k)$, namely the point corresponding to $(x,y^d)$. Hence there
are no algebraic cycles. The same argumentation as in Section~\ref{sec:C.7}
shows that $A^+_1(\cH)$ is freely generated by $[F]$ and $[G]$. If $\phi \in
\Aut(H^d)$, then by Corollary~\ref{cor:D.1} $\phi_*$ permutes these cycle
classes. If $[\phi(F)] = [G]$ in $A_1(\cH)$ then from
Proposition~\ref{prop:D.4} it follows that $[F] = [G]$ in $A_1(H^d)$. In
order to show that this is not possible, we compute the intersection numbers
with $\cM_n$:

Let be $\cI_\alpha = (\alpha x +y, x^d)$. From the exact sequence 
\[
0 \longrightarrow \cI_\alpha(n-1) \stackrel{\cdot z}{\longrightarrow}
\cI_\alpha(n-1) \longrightarrow \cI'_\alpha(n-1) \longrightarrow 0\,,
\]
it follows that 
\[
 0 \longrightarrow H^0(\cI_\alpha(n-1)) \longrightarrow H^0(\cI_\alpha(n)) \longrightarrow k[x,y]_n \longrightarrow 0
\]
is exact for all $n\geq d$. It follows that 
\[
  \alphadeg \dot\bigwedge H^0(\cI_\alpha(n)) = \alphadeg \dot\bigwedge
   (\alpha x +y) S_{d-2} = \tbinom{d}{2}\,,
\]
hence $(\cM_n \cdot C) = \binom{d}{2}$. Comparing with~\eqref{eq:D.1}  shows
that $[F] = [G]$ is not possible.

\begin{corollary}
  \label{cor:D.3}
  Each $\phi \in \Aut(H^d)$ operates as the identity on $A_1(\cH)$.
  \hfill $\qed$
\end{corollary}

As in the general situation of Section~\ref{sec:5.1} and
Section~\ref{sec:5.2}, one deduces:

\begin{corollary}
  \label{cor:D.4}
   If $\phi \in \Aut(H^d)$, then $\phi | \cH$ is induced by a linear
   transformation $\gamma \in \Aut_k(\P^2) = \PGL(2;k)$.
  \hfill $\qed$
\end{corollary}

Now as in Section~\ref{sec:5.4}, if one replaces $\phi$ by $\gamma^{-1}
\circ \phi$, one obtains a \emph{normed} automorphism of $H^d$, i.e.~an
automorphism, whose restriction to $\cH$ is the identity.

\section{Computation of $\Aut(H^d)$}
\label{sec:D.10}

\subsection{Preliminary remarks}
\label{sec:D.10.1}

Let $\cZ  \subset H^d \times \P^2$  be the universal subscheme and $\pi: \cZ
\to H^d$ the projection. If $\xi \in H^d(k)$, then the fiber $\pi^{-1}(\xi)
\subset \P^2$ is denoted by $Z_\xi$. $H^{(d)}$ is the open subscheme of
$H^d$ such that $H^{(d)} = \Set{ \xi \in H^d(k) | \# Z_\xi = d}$. If $\phi
\in \Aut(H^d)$, then one has a cartesian diagram 
\[
    \xymatrix{
      H^d \times \P^2 \ar[r]^{\phi'} \ar[d]_\pi & H^d \times \P^2 \ar[d]^\pi \\
      H^d \ar[r]^\phi  & H^d }
\]
and $\phi'$ induces an isomorphism of $\cZ$, which we denote by $\psi$. 

If $\xi\in H^{(d)}(k)$, then $Z_\xi$ consists of $d$ different closed points
$P_1,\dots,P_d \in \P^2$. Conversely, $d$ different, closed points of $\P^2$
define a point $\xi \in H^{(d)}(k)$ and we express this connection by
writing $\xi \leftrightarrow \cZ_\xi = P_1 \dcup \cdots \dcup P_d$. 

If $\xi \in H^{(d)}(k)$ and $\zeta = \phi(\xi)$, then $\psi$ induces an
isomorphism $\cZ_\xi \to \cZ_\zeta$, which we denote by 
\[
   P \mapsto \phi_\xi(P) \quad \text{or} \quad \phi(P_1 \dcup \cdots \dcup
   P_d) = \phi_\xi(P_1) \dcup \cdots \dcup \phi_\xi(P_d)\,.
\]
If $\xi \in H^{(d)}(k)$ and $\phi(\xi) = \xi$, then $\psi$ induces a
permutation of the points $P_1,\dots,P_d$.

\subsection{}
\label{sec:D.10.2} In the following $\phi \in Aut(H^d)$ is normed, i.e.~the
 restriction $\phi | \cH $ is the identity.

\begin{theorem}
 \label{thm:D.2}
  $\Aut_k(\Hilb^d(\P^2)) = \PGL(2;k)$, if $d\geq 6$\,.
\end{theorem}
\begin{proof}
  Let be $\xi \leftrightarrow P_1 \dcup \cdots \dcup P_d \in H^{(d)}(k)$ and
  let $g$ be a line through $P_d$ such that the projection of
  $P_1,\dots,P_d$ onto $g$ gives different points $P'_1,\dots,P'_d=P_d$. The
  projection is defined by a suitable $\G_m$-action $\tau(\lambda)$ such
  that $\tau(1)\xi = \xi$ and $\xi_\infty :=\lim\limits_{\lambda \to
    \infty}\tau(\lambda)\xi = P'_1 \cup \cdots \cup P'_d$ (see
  Appendix~\ref{cha:A}).

 Let be $T = \P^1_k$, $U = T- \{ 0 ,\infty \}$ and $\alpha: U  \to H^{(d)}$
 the morphism defined by $\lambda \mapsto \tau(\lambda)(\xi)
 =:\xi(\lambda)$. Then $\alpha$ has a uniquely determined extension to a
 morphism $T \to H^d$, which still is denoted by $\alpha$. Put $\beta:= \phi
 \circ \alpha$. Then $\deg(\alpha | T) = \deg(\beta |T) =1$, as $\alpha | U$
 is injective if $g$ is chosen general enough.  The image of $\alpha$ is
 a curve $\cC \subset H^d$ such that $\cC_\lambda \leftrightarrow
 \xi(\lambda)$ for all $\lambda \in T$. Then $\cD = \phi(\cC)$ is the image
 of $\beta$ and $\cD_\lambda \leftrightarrow \phi\xi(\lambda)$.

Let $\psi$ be the induced automorphism of the universal subscheme $\cZ
\hookrightarrow X \times H^d$, $X:= \P^2_k$. $\cC$ can be taken as a
closed subscheme of $X \times T$, which is flat over $T$.  Let $\cF$ be the
structure sheaf of $\cC$. Now $U = \Spec A$, $A = k[\lambda,1/\lambda] =
k[\lambda]_\lambda$ and $\cF \otimes_T \cO_U = \bigoplus^d_1 L_i$, $L_i$
flat over $U$ with Hilbert polynomial $1$, and $L_i \otimes k(\lambda) =
\cO_X / P_i(\lambda)$, where $P_i(\lambda):= \tau(\lambda)P_i$ corresponds
to a closed point in $X$.

Let $p:X \times T \to T$ be the projection. Then $p_* \cF(n) \otimes \cO_U 
= \bigoplus^d_1 p_* L_i(n) \otimes \cO_U$, hence 
\begin{equation}
  \label{eq:D.5}
  \dot\bigwedge  p_* \cF(n) \otimes \cO_U = \bigotimes^d_1 p_* L_i(n) 
  \otimes \cO_U\,.
\end{equation}
As $p_* \cF(n)$ and $p_* L_i(n)$ are globally generated by the monomials in
$S_n$ if $n \gg 0$, all the line bundles, which occur in eq.~\eqref{eq:D.5} have
uniquely determined extensions all over $T$, which are denoted by the same
letters, i.e.~\eqref{eq:D.5} holds true if $U$ is replaced by $T$. It
follows that 
\begin{align*}
  \bigl(\dot\bigwedge p_* \cF(n)\cdot T \bigr) & = (\alpha^* \cM_n \cdot T)
  = \deg(\alpha)
  (\cM_n \cdot \cC) \\
  & = \sum^d_1 (p_* L_i(n) \cdot T)\,.
\end{align*}
As $L_i$ is a line in $X$, if $1\leq i \leq d-1$, one has $(p_* L_i(n) \cdot
T)=n $, if $1\leq i \leq d-1$, hence
\begin{equation}
  \label{eq:D.6}
  (\cM_n \cdot \cC) = (d-1) \cdot n\,.
\end{equation}
If $\cG$ is the structure sheaf of the subscheme $\cD \subset X \times T$,
then one again has 
\[
   \cG  \otimes_T \cO_U = \bigoplus^d_1 \cL_i \,.
\]
Here $\cL_i$ is flat over $U$ with Hilbert polynomial $1$ and has the form
$\cL_i = \cO_{X\times U}/ \cP_i$ and $\cP_i(\lambda):= \cP_i \otimes k(\lambda)
\leftrightarrow \phi_{\xi(\lambda)}(\tau(\lambda)P_i)$, for all $\lambda \in
U$. This again implies
\[
  \dot\bigwedge  p_* \cG(n) \otimes \cO_U = \bigotimes^d_1 p_*\cL_i(n) 
  \otimes \cO_U\,.
\]
If one again denotes the extension of $\cL_i \otimes \cO_U$ to a module,
which is flat over $T$, with Hilbert polynomial $1$, by the letter $\cL_i$,
then one obtains
\[
\dot\bigwedge  p_* \cG(n)  = \bigotimes^d_1 p_* \cL_i(n) 
\]
and one deduces:
\begin{equation}
  \label{eq:D.7}
    (\cM_n \cdot \cD) = \sum^d_1 (p_* \cL_i(n) \cdot T)\,.
\end{equation}
If $\cP_i(\lambda) \in X$  does not depend on $\lambda$, then $\cL_i$ is a
constant sheaf, hence $(p_* \cL_i(n) \cdot T) =0$. If $\cP_i(\lambda)$
depends on $\lambda$, then $\lambda \mapsto \cP_i(\lambda)$ defines a
morphism $U \to X$, which has a unique extension $T \to X$, and its image is
a curve of degree $d_i \geq 1$\,. It follows that either $(p_* \cL_i(n)
\cdot T) =0$ or $(p_* \cL_i(n) \cdot T) = d_i n + c_i$. As $[\cC] = [\cD]$
by Proposition~\ref{prop:D.4}, one has $(\cM_n \cdot \cC) = (\cM_n \cdot
\cD)$, i.e.~$(d-1)\cdot n = \sum^d_1 d_i n +c_i$. It follows that there is
at least one index $i$ such that $\cP_i(\lambda)$ is independent of
$i$. Hence there is an index $i$ such that
$\phi_{\xi(\lambda)}(\tau(\lambda)P_i)$ is independent of $\lambda \in
U$. It follows that $\phi_\xi(P_i) = \phi_{\xi(\lambda)}(\tau(\lambda)P_i)$
for all $\lambda \in U$, hence for all $\lambda \in T$. Now $\xi_\infty =
\lim_{\lambda \to \infty} \tau(\lambda)\xi \leftrightarrow
\Set{P'_1,\dots,P'_d }$ is a closed point in $\cH(k)$ and $\phi | \cH =
\id$, as $\phi$ is normed, hence $\phi_\xi(P_i) \in \Set{P'_1,\dots,P'_d }
\subset g$. If one substitutes the line $g$ by a line $h$, such that $P_d
\in h$ and the projections of $P_1,\dots,P_d$ onto $h$ again give distinct
points, the same argumentation shows $\phi_\xi(P_j) \in h$, for an index
$j$. From this it follows that there is an index $1\leq i \leq d$ such that
$\phi_\xi(P_i)$ is in the intersection of infinitely many such lines. It
follows that $\phi_\xi(P_i) = P_d$. The same argumentation with $P_{d-1}$
instead of $P_d$ shows that $\phi_\xi(P_j) = P_{d-1}$, etc. 

We conclude that $\phi_\xi(P_1) \dcup \cdots \dcup \phi_\xi(P_d)$ is a
permutation of $P_1 \dcup \cdots \dcup P_d$, i.e.~we have $\phi(\xi)= \xi$.
But as the closure of $H^{(d)}$ is equal to $H^d$, the theorem follows.
 \end{proof}

\chapter{Filtration of the structure sheaf of a curve}
\label{cha:E}
\begin{auxlemma}
   \label{auxlem:E.1}
	Let be $k$ an algebraically closed field, $S = k[x_0,\cdots, x_r]$, $Y/k$ an integral scheme and
	$\cM $ a coherent module on $\P^r\times_k Y$, which is flat over $Y$ with constant 
	Hilbert polynomial $s\geq 1$. Then for each sufficiently small open set $ U =\Spec A\subset Y$
  there is a filtration $(0) = M^0\subset\cdots\subset M^s $ of $ \cM\otimes\cO_U$ such that 
	$ M^i/M^{i-1}\simeq(S/\fp_i)(-d_i)$, where  $\fp_i\in \Proj(S\otimes A) $ is a prime ideal, which is
	generated by a subbundle $L_i\subset S_1\otimes A $ of rank $r$, and the isomorphism is defined by
	multiplication with $ f_i\in S\otimes A$ of degree $d_i$.
	\end {auxlemma}
 \begin{proof}
This is a simple variant of ~\cite[Prop.7.4, p.50]{H}. We replace $ A$ by a suitable localization $A_f$, which
is denoted $A$ again and writing $\fp$ instead of $\fp_i$, one obtains $ M^i/M^{i-1}\simeq S\otimes A/\fp $ is
flat over $A$ with constant Hilbert polynomial $c$. Let  $ K:= A_0$ be the quotient field of $A$. Then $(S\otimes A/\fp)\otimes K)
\simeq S\otimes_k K /\fp\otimes_A K$  has the Hilbert polynomial c, hence the dimension of the support of $\cO_X \otimes K $, 
$X := \Proj(S\otimes_k A/\fp)$, has the dimension $ 0$. But then $X\otimes_A K$ is an artinian scheme, which is connected,
as $ X$ is connected. It follows that $X\otimes_A K$ consists of one single closed point $\fp\otimes_A K \in X\otimes_A K$.
After tensorizing with an algebraic closure $ K^{-}$ of $K$, one obtains that $X\otimes_A K^{-}$ consists of the closed point
$\fp\otimes_A K^{-}$, and $ X\otimes_A K^{-} =\Proj(S\otimes_A K^{-}/\fp\otimes_A K^{-})$ has the Hilbert polynomial $c$. 
As $ \fp\otimes_A K^{-}$ is maximal in $ S\otimes_k K^{-}$, it follows that $c = 1$. As the Hilbert polynomials of
$X\otimes_A K^{-}$, $X\otimes_A K$ and $X$ are equal, it follows that $c = 1$ and the Hilbert polynomial of $\fp$ is equal
to $\binom{n-1+r}{r}+\cdots +\binom{n-1+1}{1}$, hence $\fp$ is 1-regular and $\fp_1$ is generated by a subbundle $L$ of 
$S_1\otimes_k A$ of rank $r$.
\hfill $\qedhere$
\end{proof}
\begin{lemma}
  \label{lem:E.1}
  Let $Y/k$ be an integral scheme,
$C \subset \P^3 \times_k Y$ a curve, which is flat over $Y$ with Hilbert
polynomial $P(n) = dn -g +1$\,. There exists an open set $U = \Spec A
\subset Y$ such that the following conditions are fulfilled: 
\begin{enumerate}[$1^\circ$]
\item 
If $S = A[x,y,z,t]$, there is a finitely generated graded
  $S$-algebra $M$ such that $\tilde{M}$ is the structure sheaf of the
  subscheme $C \times_Y U \subset \P^3 \times U$.
\item 
There is a filtration $(0) = M^0 \subset \cdots \subset M^r =M$
  such that $M^i/M^{i-1} \simeq (S/\fp_i)(-d_i)$ is flat over $A$, $\fp_i
  \subset S$ is a graded prime ideal, and the isomorphism is defined by
  multiplication with a form $f_i \in S_{d_i}$.
\item 
  For each $\fp_i$ two cases can occur: $\fp_i$ is a minimal prime of $M$ and
   $\Proj(S/\fp_i)$ is a curve, flat over $A$. 
   OR: $\fp_i$ is generated by a subbundle $L\subset S_1$ of
  rank $3$\,.
\end{enumerate}
\end{lemma}
\begin{proof}
  The existence of such a filtration is shown in (loc.cit.). Applying the
  Generic-flatness-Lemma, one sees that either $S/\fp_i$ has a Hilbert
  polynomial of the form $an+b$, which is the first case, or $S/\fp_i$ has
  a constant Hilbert polynomial $s$, in which case the assertion follows from
	the auxiliary lemma.
\end{proof}

\chapter{Lower semicontinuity of the complexity}
\label{cha:F}

If $M = x^\alpha y^\beta z^\gamma t^\delta \in S = k[x,y,z,t]$, then
$T(M):=\delta$. $\GG = \Grass_m(S_d)$ parametrizes the $m$-dimensional
subspaces of $S_d$. Let $e_1,\dots,e_n$, $n=\binom{d+3}{3}$, be the
monomials in $S_d$ in any order. If $V \in \GG(k)$ and $f_i = \sum^n_{j=1}
a_{ij}e_j$, $1 \leq j \leq m$, is a basis of $V$, then $f_1 \wedge \cdots
\wedge f_m = \sum P_{(j)}e_{(j)}$, where $e_{(j)} = e_{j_1} \wedge \cdots
\wedge e_{j_m}$ and $P_{(j)} = \det \begin{pmatrix}
  a_{1 j_1}& \cdots & a_{1 j_m} \\
  &        &    \\
  a_{m j_1}& \cdots & a_{m j_m}
      \end{pmatrix}$ is the Pl\"ucker--coordinate belonging to the
multi-index $(j) =(j_1,\dots,j_m)$, where $1 \leq j_1 < \dots < j_m \leq
n$. 

Let $\G_m$ act on $S$ by $\sigma(\lambda): x \mapsto x$, $y \mapsto
y$, $z \mapsto z$, $t \mapsto \lambda t$.  Because of 
\[
   \sigma(\lambda) f_i  = \sum_j a_{ij}\lambda^{T(e_j)}e_j
\]
it follows that
\begin{equation}
  \label{eq:F.1}
    \sigma(\lambda) f_1 \wedge \cdots \wedge \sigma(\lambda) f_m 
   = \sum_{(j)} P_{(j)}\lambda^{T(e_{(j)})}e_{(j)}\,,
\end{equation}
where $T(e_{(j)}) := T(e_{j_1}) + \cdots + T(e_{j_m})$.

Let be $N:= \# \Set{ (j) \text{ multi-index as above}}-1$\,. The
Pl\"ucker-embedding  $p:\GG \to \P^N$ is defined by $V \mapsto \bigwedge^m
V$, that means, it is defined by $V \mapsto \Set{
  \text{Pl\"ucker--coordinates of } V}/\sim$, and $\sim$ is defined by
multiplication with elements in $k^*$. It follows that $\G_m$ acts in an
equivariant way on $\GG$ and $\P^N$ with respect to $p$.

Let $V \leftrightarrow \xi \in \GG(k)$ and $C(\xi):=
\Set{\sigma(\lambda)p(\xi) | \lambda \in k^*} = \Set{p(\sigma(\lambda)\xi) |
\lambda \in k^*}$. From~\eqref{eq:F.1} it follows that 
\[
   C(\xi) = \Set{ P_{(j)}\lambda^{T(e_{(j)})} | \lambda \in k^*}/\sim\,.
\]
\textsc{Case 1}: $\lambda \mapsto \sigma(\lambda)V$ is injective, $\lambda\in
k^*$. \\
Then from the argumentation in the proofs of~\cite[Bemerkung 2 and 3, p. 11]{T1} follows that for the closure
$\overline{C(\xi)}\subset \P^N$ one has 
\begin{equation}
  \label{eq:F.2}
  \deg \overline{C(\xi)} = \max_{(j)} T(e_{(j)}) - \min_{(j)} T(e_{(j)})\;,
\end{equation}
where the maximum and the minimum refers to such multi-indices with
$P_{(j)}\neq 0$, and the $P_{(j)}$ are the Pl\"ucker-coordinates of $V$.\\
\textsc{Case 2}: $\lambda \mapsto \sigma(\lambda)V$ is not injective. \\
In the proof of~\cite[Hilfssatz 5, pp. 8]{T2} it had been shown that this is
equivalent with the following  statements $1^\circ$--$3^\circ$: \\
There is an integer $\ell > 0$ and a basis $f_1,\dots,f_m$ of $V$, such that
\begin{enumerate}[$1^\circ$]
\item $f_i = t^{d_i} \cdot \sum^{n_i}_{\nu=0} f^\nu_i t^{\ell \nu}$, $d_i$
  chosen maximal,  $0 \leq d_1 \leq \dots \leq d_m$, $f^\nu_i \in k[x,y,z]$ of
  degree $d-(d_i+\ell \nu)$, for all $i$ and $0 \leq \nu \leq n_i$.
\item The map $\G_m/\mu_\ell \to \P^N$ defined by $V \to \bigwedge^m
  \sigma(\lambda)V$ is injective, $\mu_\ell = \Set{ \varepsilon \in\C |
    \varepsilon^\ell =1}$\,.
\item 
  \begin{equation}
    \label{eq:F.3}
    \deg C(\xi) = \tfrac{1}{\ell} \cdot \Set{ \text{right-hand side
        of}~\eqref{eq:F.2}}\,.
  \end{equation}
\end{enumerate}
Note that in~\eqref{eq:F.2} $T(e_{(j)}) = D + \ell \cdot n(j)$, with $D =
d_1 + \cdots + d_m$ and $n(j)\in \N$ depending on $(j)$. If conversely it is
supposed that $T(e_{(j)})$ has this form, then one gets 
\[
  \sigma(\lambda)V \xmapsto{p} \left( \Set{ P_{(j)}\lambda^{D + \ell \cdot n(j)} | P_{(j)} \neq 0} \right)/\sim \;  = \;  
\left( \Set{ P_{(j)} | P_{(j)} \neq 0 }\right)/\sim
\]
for all $\lambda \in \mu_\ell$.  As $p$ is a closed immersion it follows
that $\sigma(\lambda)V = V$ if $\lambda \in \mu_\ell$. 

The Pl\"ucker--coordinates of $\xi$ and the number $\ell$ depend on $\xi$,
we therefore write $P_{(j)}(\xi)$ and $\ell (\xi)$. It is clear that there
is an open neighborhood $U = U(\xi)$ of $\xi$ in $\GG$, such that 
\begin{equation}
  \label{eq:F.4}
   P_{(j)}(\xi) \neq 0 \Rightarrow P_{(j)}(\zeta) \neq 0 \text{ if } \zeta
   \in U(\xi)\,.
\end{equation}
Suppose that $\zeta \in U(\xi)$.
Then we conclude:
\begin{align*}
  & \lambda \in \mu_{\ell(\zeta)} \Rightarrow \sigma(\lambda)\zeta = \zeta \\
\Rightarrow {} & \left( \Set{ \lambda^{T(e_{(j)})} P_{(j)}(\zeta) | P_{(j)}(\zeta) \neq 0 }\right)/\sim \; = \; \left( \Set{ P_{(j)}(\zeta) | P_{(j)}(\zeta) \neq 0 }\right)/\sim \\
\Rightarrow {} & \lambda^{T(e_{(j)})} = c \in k^* \text{ for all $(j)$ such that } P_{(j)}(\zeta) \neq 0\,.
\end{align*}
Because of~\eqref{eq:F.4} it follows that $\lambda^{T(e_{(j)})} = c$ for all $(j)$ such that $P_{(j)}(\xi)\neq 0$\,. 
\begin{align*}
  \Rightarrow {} &
  \left( \Set{ \lambda^{T(e_{(j)})} P_{(j)}(\xi) | P_{(j)}(\xi) \neq 0 }\right)/\sim \; = \; \left( \Set{ P_{(j)}(\xi) | P_{(j)}(\xi) \neq 0 }\right)/\sim \\
  \Rightarrow {} & \sigma(\lambda)\xi = \xi \Rightarrow \lambda \in
  \mu_{\ell(\xi)} \Rightarrow \mu_{\ell(\zeta)} \subset \mu_{\ell(\xi)}
  \Rightarrow \ell(\zeta) \text{ divides } \ell(\xi)\,.
\end{align*}
One gets:
\begin{equation}
  \label{eq:F.5}
   \ell(\zeta) \leq \ell(\xi) \quad \text{ for all }
 \zeta \in U(\xi)\,.
\end{equation}
Because of~\eqref{eq:F.4} one has for all $\zeta \in U(\xi)$:
\[
   \max_{(j)} \Set{ T(e_{(j)}) | P_{(j)}(\xi) \neq 0 }
 \leq \max_{(j)} \Set{ T(e_{(j)}) | P_{(j)}(\zeta) \neq 0 }
\]
and
\[
   \min_{(j)} \Set{ T(e_{(j)}) | P_{(j)}(\zeta) \neq 0 }
 \leq \min_{(j)} \Set{ T(e_{(j)}) | P_{(j)}(\xi) \neq 0 }\,.
\]
Then from~\eqref{eq:F.3}, \eqref{eq:F.4} and~\eqref{eq:F.5} we get: 
\begin{conclusion}
  \label{conc:F.1}
  For each $\xi \in \GG(k)$ there is an open neighborhood $U \neq
  \emptyset$ in $\GG$ such that $\deg C(\xi) \leq \deg C(\zeta)$ for all
  closed points $\zeta \in U$.  \hfill $\qed$
\end{conclusion}

We now embed $\HH =\HH_Q$ into $\Grass^{P(n)}(S_n)$ by means of $\cF_n =
\pi_* \cF(n)$ and then by means of $\cM_n = \dot\bigwedge \cF_n$ into a
projective space $\P^{N(n)}$. We recall that $P(n) = \binom{n+3}{3} -Q(n)$,
$n$ is a sufficiently large number, e.g.~$n\geq b$, and that $g \leq g(d)
=(d-2)^2/4$ is supposed.

If $\xi \in \HH(k)$, then by Theorem~\ref{thm:1.2} in Chapter~\ref{cha:1} we
have the rational equivalence
\[
  \Set{ \sigma(\lambda)\xi | \lambda \in k^* }^- =: \bar{C}(\xi) \sim
  q_2(\xi)C_2 + q_1(\xi)C_1  +q_0(\xi)C_0\,.
\]
Here the natural numbers $q_2(\xi)$ and $q_1(\xi)$ are called the
\emph{complexity} of $\xi$ with regard to $C_2$ respectively to $C_1$. 

Now $\deg \bar{C}(\xi) = (\cM_n \cdot \bar{C}(\xi))$ and
Conclusion~\ref{conc:F.1} shows that  
\[
(\cM_n \cdot \bar{C}(\zeta)) \geq
(\cM_n \cdot \bar{C}(\xi)) \qquad \text{for all } \zeta \in U(\xi)\,,
\]
hence
\begin{align*}
  &  q_2(\zeta)\bigl[ \tbinom{n-a+2}{2} +(n-b+1)\bigr] + q_1(\zeta)(n-b+1) +
  q_0(\zeta) \\
\geq {} & q_2(\xi)\bigl[ \tbinom{n-a+2}{2} +(n-b+1)\bigr] +
  q_1(\xi)(n-b+1) + q_0(\xi)
\end{align*}
for all $n\gg 0$. We get:
\begin{conclusion}
  \label{conc:F.2}
  For each $\xi \in \HH(k)$ there is an open neighborhood $U(\xi)$ of $\xi$ in $\HH$ such that for each $\zeta \in U(\xi) \cap \HH(k)$ one has: \\
  Either \quad $q_2(\zeta) > q_2(\xi)$ \quad or \quad $q_2(\zeta) =
  q_2(\xi)$ and $q_1(\zeta) \geq q_1(\xi)$.  \hfill $\qed$
\end{conclusion}

\chapter{The graded Hilbert scheme}
\label{cha:G}

Let be $S = k[x_1,\dots,x_r,t]$ the polynomial ring in $r+1$ variables, $X =
\Proj S$, $\HH = \Hilb^P(X)$ the Hilbert scheme, which parametrizes the
quotients $\cO_X/\cI$ with Hilbert polynomial $P(n)$, i.e.~the ideals $\cI
\subset \cO_X $ with Hilbert polynomial $Q(n) = \binom{n+r}{r} -P(n)$.

Let be $\fX = X \times_k \HH$, $\cI \subset \cO_\fX$ the universal ideal
sheaf with Hilbert polynomial $Q(n)$, $\cF = \cO_\fX/\cI$. If $\ell$ is any
linear form, then $U(\ell) = \Set{ y\in \HH | \ell \text{ non-zero divisor
    of } \cF \otimes k(y)}$ is open and non-empty in $\HH$
(see~\cite[Section 1]{G88}).

Let $\G_m$ act on $S$ by $\sigma(\lambda): x_i \mapsto x_i$, $1\leq i \leq
r$, $t \mapsto  \lambda t$.

\section{Limit points}
\label{sec:G.1}

\begin{lemma}
  \label{lem:G.1}
  Let be $\cI \leftrightarrow \xi \in \HH(K)$, and $\cI_0 \leftrightarrow
  \xi_0 := \lim\limits_{\lambda \to 0} \sigma(\lambda)\xi$, where $K/k$ is a
  field extension. Then one has:
  \begin{enumerate}[(i)]
  \item $\xi_0$ is $\G_m$-invariant.
  \item $\xi_0 \in U(t) \iff \xi \in U(t)$.
  \item If $\xi \in U(t)$, then the Hilbert functions of $\cI$ and $\cI_0$
    are equal.
  \end{enumerate}
\end{lemma}
\begin{proof}
  A small modification of the proof of~\cite[Lemma 4]{G88}:\\ Write $F_i =
  t^{d_i} f^0_i + t^{d_i+1} f^1_i + \cdots$, $1\leq i \leq p := Q(d)$, $d_1
  \leq d_2 \leq \cdots$, and $f_i:=t^{-d_i}F_i$. By linearly combining the
  $F_i$, one can achieve that the $f_i$ are linearly independent and the
  proof goes through with $t$ instead of the variable $X_0$.
\end{proof}

\section{The restriction morphism}
\label{sec:G.2}

Let be $R = k[x_1,\dots,x_r]$, $Y = \Proj R$, $P'(n) = P(n)- P(n-1),$ $\HH' =
\Hilb^{P'}(Y)$.

\begin{lemma}
   \label{lem:G.2}
   Let $T/k$ be a scheme, $T \to U(t)$ a morphism and $\cI \in \HH(T)$ the
   corresponding ideal. Then $\cI':=\cI + t\cO_{X\times T}(-1) / t
   \cO_{X\times T}(-1)$ is an element of $\HH'(T)$ and $\cI \mapsto \cI'$
   defines a morphism $r:U(t) \to \HH'$.
\end{lemma}
\begin{proof}
  The same argumentation as in~\cite[Section 1]{G88}.
\end{proof}

\section{The case of space curves}
\label{sec:G.3}

  We now write $S = k[x,y,z,t]$, 
$X =
\Proj S$, $\HH = \Hilb^P(X)$, $P(n) =dn -g +1$, $R = k[x,y,z]$, $Y =\Proj
R$. If $\cI \subset \cO_Y$ is an ideal, then $\cI^* \subset \cO_X$ is the
ideal, which is generated by $\cI$, i.e.
\[
    H^0(X,\cI^*(n)) = \bigoplus^n_{i=0} t^{n-i}H^0(Y,\cI (i )) \quad \text{for all
    } n\,.
\]

Let now be $\cI \leftrightarrow \xi \in H(k) \cap U(t)$, $\cI'
\leftrightarrow r(\xi) \in \Hilb^d(Y)$ and $\cI_0 \leftrightarrow \xi_0 :=
\lim_{\lambda \to 0} \sigma(\lambda)\xi$. Here, and in the following, $\G_m$
operates by $\sigma(\lambda): x \mapsto x$, $y \mapsto y$, $z \mapsto z$,
$t \mapsto \lambda t$.
\begin{lemma}
  \label{lem:G.3}
 $\cI_0 =(\cI')^* \cap \cR$, $(\cI')^*$ is the $\CM$-part of $\cI_0$ and $\cR$ is $(x,y,z)$-primary.
\end{lemma}
\begin{proof}
  $1^\circ$ As $\cI_0$ is $\G_m$-invariant, one has 
\[
H^0(\cI_0(n)) =
  \bigoplus^n_{i=0}t^{n-i}U_i, \quad  H^0(\cI_0 (n+1))
  =\bigoplus^n_{i=0}t^{n+1-i}V_i\,,
\]
$U_i \subset V_i \subset R_i$ vector spaces. As $\cI_0 \in U(t)$ by
Lemma~\ref{lem:G.1}, one has $H^0(\cI_0(n)) = \Set{ f \in S_n | t \cdot f
  \in H^0(\cI_0(n+1))}$, hence $U_i = V_i$, if $0 \leq i \leq n$, and $R_1
V_n \subset R_1 H^0(\cI_0 (n)) \subset H^0(\cI_0(n+1))$, i.e.~$R_1 V_n
\subset V_{n+1}$ for all $n$. It follows that there is a sequence of vector
spaces $U_i \subset R_i$ such that $H^0(\cI_0(n)) = \bigoplus^n_{i=0}
t^{n-i} U_i$ and $R_1 U_n \subset U_{n+1}$ for all $n$.

  $2^\circ$ As $r$ is continuous, from $r(\sigma(\lambda)\xi) = r(\xi)$ for
  all $\lambda \in k^*$ it follows that that $r(\xi_0) =r(\xi)$,
  i.e.~$(\cI_0)' = \cI'$. From the exact sequence
\[
 0 \longrightarrow \cI_0(n-1) \stackrel{t}{\longrightarrow} \cI_0(n)
\longrightarrow \cI'(n) \longrightarrow 0
\]
it follows that the canonical map 
\[
  H^0(\cI_0(n))/tH^0(\cI_0(n-1)) \rightarrowtail H^0(\cI'(n))
\]
is an isomorphism, if $n \gg 0$\,.  Hence one has $U_n \subset H^0(\cI'(n))$
for all $n$ and $U_n = H^0(\cI'(n))$ if $n \gg 0$\,. It follows that $\cI_0
\subset \cJ := (\cI')^*$.

$3^\circ$ Next we want to show that $\cJ$ is $\CM$, and we assume that $\cP$
is an associated prime of $\cO_X/\cJ$, which corresponds to a closed point of
$X$. Then $\cP$ is $\G_m$-invariant.

\textsc{Case 1}: $\cP =(\ell_1,\ell_2,t)$, where $\ell_1$, $\ell_2$ are
linear forms in $R$. But as $t$ is not a zero-divisor of $\cO_X/\cJ$, this
is a contradiction.

\textsc{Case 2}:
  $\cP = (x,y,z)$. Let be $\cP = \Ann(f)$, i.e.~$f\in S$ such that $\cP
  \cdot f \subset \cJ$.  Write $f = t^e \cdot g$. If $e > 0$, it follows
  that that $\cP = \Ann(g)$. Hence one can suppose $e =0$ and $f = f^0 +
  tf^1 + \cdots$, $f^0 \in R_n$. It follows that $\cP \cdot f^0 \in H^0(\cI
  (n+1))$, hence $f^0 \in H^0(\cJ(n))$. By an induction argument on gets $f
  \in \cJ$, contradiction.

  $4^\circ$ Let $\cP$ be an associated prime of $\cO_X/\cI_0$, which
  corresponds to a closed point of $X$. The same argumentation as in
  $3^\circ$ shows that $\cP =(x,y,z)$. Hence one can write $\cI_0 = \cN \cap
  \cR$, where $\cN$ is the $\CM$-part and $\cR$ is $(x,y,z)$-primary. Now
  one has $\cI_0 = \cI_0 \cap \cJ = \cN \cap \cJ \cap \cR$, hence $\cN \cap
  \cJ = \cN \subset \cJ$ and besides $\cJ' = \cI' = \cI'_0 = \cN'$. It
  follows that $\cJ/\cN$ has finite length. If one assumes that this is not
  equal to zero, there is a prime ideal $\cP$, which corresponds to a closed
  point of $X$ and is an associated prime of $\cJ /\cN$, hence an
  associated prime of $\cO_X /\cN$, contradiction. It follows that $\cJ
  =\cN$.
\end{proof}

\section{The graded Hilbert scheme}
\label{sec:G.4}
 At the moment, we go back to the general situation as in~\ref{sec:G.1}
 and~\ref{sec:G.2}. In~\cite[Abschnitt 2]{G93} it is shown:
 \begin{enumerate}[(i)]
 \item If $\HH^{\G_m}$ denotes the fixed point scheme, then
   $G:=\HH^{\G_m}\cap U(t)$ is a closed subscheme of $\HH$\, (!).
 \item Let $\phi$ be a numerical function, i.e.~a map $\phi:\N \to \N$, such
   that $\phi(n)= Q(n)$ if $n\gg 0$\,. Put $\phi'(n) = \phi(n)
   -\phi(n-1)$. If $A$ is a $k$-algebra, let be $G_\phi(A)$ the set of all
   subbundles $V_n \subset R_n \otimes A$ of rank $\phi'(n)$ such that
   $R_1V_n \subset V_{n+1}$ for all $n \in \N$. Then $G_\phi$ is
   (represented by) a closed subscheme of $\HH$, and $G$ is the disjoint
   union of those $G_\phi$, which are non-empty.The closed immersion $G_\phi
   \to \HH$ is defined by $(V_0,V_1,\dots) \mapsto \cJ$, where $\cJ \subset
   \cO_{X \otimes A}$ is the ideal generated by the $V_n$, i.e.~$H^0(X
   \otimes A, \cJ (n)) = \bigsum^n_{i=0} t^{n-i}V_i$ for all $n$. $G_\phi$ is
   called \emph{graded} Hilbert scheme.
 \item Let $\HH_\phi$ be the locally closed subset of $\HH$ of all ideals $\cJ
   \subset \HH(K)$ with Hilbert function $\phi$, for all field extensions
   $K/k$. We take $\HH_\phi$ as a subscheme of $\HH$ with the \emph{reduced}
   scheme structure. Then one has $(G_\phi)_\red \subset \HH_\phi$.
 \item If $\xi \in \HH_\phi(K) \cap U(t)$, then Lemma~\ref{lem:G.1} shows that
   $\xi_0:=\lim\limits_{\lambda \to 0} \sigma(\lambda)\xi \in G_\phi(K)$,
   and $\xi \mapsto \xi_0$ defines a morphism $\HH_\phi\cap U(t) \to
   (G_\phi)_\red$.
 \end{enumerate}

 The statements (i)--(iii), whose proof is easy, are used in
 Section~\ref{sec:2.1} in the cases $r=2$ and $r=3$. The statement (iv)
 requires some work  ~\cite[Prop.2, p.20]{G93}, but is needed only in the proof of
 Proposition~\ref{prop:2.3}, which is not used in later chapters.

\section{Very general linear forms}
\label{sec:G.5}

\index{very general linear forms}
\subsection{}
\label{sec:G.5.1}

\begin{auxlemma}
  \label{auxlem:G.1}
  Let be $S = k[x,y,z]$, $Y = \Proj S$, $P \in Y$ a closed point, $Q$ a
  graded ideal in $S$, which is primary to $P$. Then $\HP(\cO_Y/Q)$ is equal
  to the length $r$ of the localization $(S/Q)_{(P)}$ over $S_{(p)}$.
\end{auxlemma}
\begin{proof}
   By~\cite[Prop. 7.4, p. 50]{H} there is a filtration $0 = M^0 \subset
   \cdots \subset M^r = S/Q$ such that $M^i/M^{i-1} \simeq f_i \cdot
   (S/P)(-d_i)$, and $\HP(S/P) =1$\,.
\end{proof}

\subsection{}
\label{sec:G.5.2}
  Let now be $P = k[x,y,z,t]$, $S = k[x,y,z] = P/tP(-1)$, $\fp \subset P$ a
  graded prime ideal such that $V(\fp) \subset X = \Proj P$ is a curve of
  degree $d$.  Let be $I \subset P$ a graded ideal, which is $\fp$-primary
  of multiplicity $\mu$. By (loc.~ cit.) there is a filtration
  \begin{equation}
    \label{eq:G.1}
    0 = M^0 \subset \cdots \subset M^r = [P/I]^{\sim}
  \end{equation}
  such that
\begin{equation}
    \label{eq:G.2}
  M^i/M^{i-1} \simeq f_i \cdot [(P/\fp)(-d_i)]^{\sim}
\end{equation}
  for $\mu$ indices and for the remaining indices
  \begin{equation}
    \label{eq:G.3}
 M^i/M^{i-1} \simeq g_i \cdot [(P/\cP_i)(-e_i)]^{\sim}\,,
\end{equation}
$\cP_i\in X$ a closed point, which is contained in the support of $P/I =
V(\fp)$.  We choose a linear form $\ell \in P$ such that $V(\fp) \cap V(\ell
) = \{ P_1, \dots, P_d\}$, $P_i$ distinct points and  $\ell \centernot\in
\cup \cP_i$. Applying a suitable linear transformation, one can assume $\ell
= t$. One can write $M^i = I^i/I$, where $I^i$ is a graded ideal, $I^0=I$
and $I^r=P$. We denote the images of the canonical morphism $P \to S$ by ${}'$
and from~\eqref{eq:G.1} we get a filtration
  \begin{equation}
    \label{eq:G.1prime}
  \tag{$\ref{eq:G.1}'$}
  0 = (M^0)' \subset \cdots \subset (M^r)' = [S/I']^{\sim}\,,
\end{equation}
where 
\[
(M^i)' = [(I^i)'/I']^{\sim},\quad (I^i)' = I^i + tP(-1)/tP(-1), \quad I' = I +
tP(-1)/tP(-1)\,.
\]
As $[P/\cP_i + tP(-1)]^{\sim} = 0$, from~\eqref{eq:G.3} it follows that one can
write~\eqref{eq:G.1prime} as
  \begin{equation}
    \label{eq:G.4}
0 = N^0 \subset \cdots \subset N^\mu = [S/I']^{\sim}.
\end{equation}
Because of~\eqref{eq:G.2} one has surjective homomorphisms 
  \begin{equation}
    \label{eq:G.5}
 [(P/\fp + t P(-1))(-d_i)]^{\sim} \to N^i/N^{i-1}\,.
\end{equation}
But $P/\fp + t P(-1) = \bigoplus^d_1 S/P_i$ and the localization
of~\eqref{eq:G.4} and~\eqref{eq:G.5} at the point $P_1$, for example, gives a
filtration
  \begin{equation}
    \label{eq:G.6}
0 = N^0_{(P_1)} \subset \cdots \subset N^\mu_{(P_1)} = S_{(P_1)}/I'_{(P_1)}
\end{equation}
and surjective homomorphisms
  \begin{equation}
    \label{eq:G.7}
 [(S_{(P_1)}/(P_1)_{(P_1)}](-d_i) \to N^i_{(P_1)}/N^{i-1}_{(P_1)}\,.
\end{equation}
Now the left hand side of~\eqref{eq:G.7} is a field, hence either
$N^i_{(P_1)}/N^{i-1}_{(P_1)}$ is equal to zero or has the length $1$ over
$S_{(P_1)}$.

\begin{conclusion}
 \label{conc:G.1}
   The multiplicity $\mu_i$ of $(S/I')_{(P_1)}$ over $S_{(P_1)}$ is $\leq \mu$.
 \hfill $\qed$
\end{conclusion}

Put $\cI = \tilde{I}$, $\cI' = \cI + t\cO_X(-1)/t\cO_X(-1)$, $Y = \Proj S$.
As the Hilbert polynomial of $P/\fp$ has the form $dn +a$,
from~\eqref{eq:G.1} and~\eqref{eq:G.2} follows that $\HP(P/I) = d \mu n +b$.
From the exact sequence
\[
 0 \longrightarrow (\cO_X/\cI)(-1)  \stackrel{t}{\longrightarrow} \cO_X/\cI
 \longrightarrow \cO_Y/ \cI'   \longrightarrow 0
\]
we get:
\begin{conclusion}
 \label{conc:G.2}
 $\HP(\cO_Y /\cI') = d\mu$.
 \hfill $\qed$
\end{conclusion}

From the Aux-Lemma~\ref{auxlem:G.1}  and Conclusion~\ref{conc:G.1} and
Conclusion~\ref{conc:G.2}  it follows that $\mu_i = \mu$ for all $1 \leq i
\leq d$. From this one deduces:
\begin{lemma}
  \label{lem:G.4}
  Let $C \subset X = \P^3_k$ be a curve. Write $C = C_1 \cup \dots \cup C_r
  \cup \{ \text{points} \}$, where the $C_i$ are the different irreducible
  components of dimension $1$, of degree $d_i$ and multiplicity $\mu_i$, and
  $\{ \text{points} \}$ denotes the $0$-dimensional components, embedded  or
  isolated. Then for Zariski-many linear forms $\ell \in k[x,y,z,t]$ one has
  \begin{enumerate}[(a)]
  \item $C_i \cap C_j \cap V(\ell) = \emptyset$, if $i\neq j$.
  \item $\{ \text{points} \} \cap V(\ell) = \emptyset$.
  \item $C_i \cap V(\ell) = \{ P_{i1}, \dots, P_{id_i}\}$, where the points
    $P_{ij}$ are different from each other and the multiplicity of $P_{ij}$
    in $C_i \cap V(\ell)$ is equal to $\mu_i$. \hfill $\qed$
  \end{enumerate}
\end{lemma}

\chapter{Notations and explanations}
 \label{cha:H}

\section{Notations}
\label{sec:H.1}

The ground field is $\C$; all schemes are of finite type over $\C$; $k$
denotes an extension field of $\C$.
 $P=k[x,y,z,t], S=k[x,y,z],
R=k[x,y]$ are the graded polynomial rings.\\
$T=T(4;k)$ group of diagonal matrices\\
$\Delta=U(4;k)$ unitriangular group\\
$B=B(4;k)$ Borel group of upper triangular matrices\\
$T(\rho)$ subgroup of $T(3;k)$ or of $T(4;k)$ defined as follows:
If $\rho =(\rho_0,\rho_1,\rho_2) \in \Z^3$, $\rho_0 +\rho_1 + \rho_2 =0$,
then $T(\rho) = \Set{ (\lambda_0,\lambda_1,\lambda_2) \in (k^*)^3 |
   \lambda_0^{\rho_0}\lambda_1^{\rho_1}\lambda_2^{\rho_2}=1}$, and
 analogously in the case $r=3$\,.  \\
$\Gamma=\left\{ \left(\begin{smallmatrix}
    1&*&*&*\\
    0&1&*&*\\
    0&0&1&*\\
    0&0&0&1
 \end{smallmatrix}\right)
\right\} <U(4;k) $ \\
$G_1,G_2,G_3$ subgroups of $U(4;k)$ (see below).\\
$\NNT$ = abbreviation for non-zero divisor \\
$\sim$ = abbreviation for rational equivalent \\
$\cI\subset\cO_{\P^3}$ is a $ CM $-ideal \index{CM- ideal}, if the curve in $\P^3$,
 which is defined by $\cI$, has no embedded or isolated points, i.e. is a "pure" curve.-
Generally "curve" means a 1-dimensional (mostly closed) subscheme of something.\\
Cohen-Macaulay part, respectively punctual part of an ideal $\cJ$ - see page iii.\\
$\HH=H_{d,g}$ Hilbert scheme of curves in $\P^3$ with degree $d\ge 1$ and 
genus $g$, i.e. $\HH=\Hilb^P(\P^3_k)$, where $P(T)=dT-g+1$. \\
$Q(T)=\binom{T+3}{3}-P(T)$ complementary Hilbert polynomial.\\
$\HH_Q=$ Hilbert scheme of ideals $\cI\subset\cO_{\P^3}$ with
Hilbert polynomial $Q(T)$, i.e.~$\HH=H_{d,g}=\HH_Q$.\\
$\pi$ and $\kappa$ denote the projections from $\HH\times_k\P^3$ to $\HH$ resp. $\P^3$.\\
$\HH_Q\neq \emptyset$ if and only if
$Q(T)=\binom{T-1+3}{3}+\binom{T-a+2}{2}$ or
$Q(T)=\binom{T-1+3}{3}+\binom{T-a+2}{2}+\binom{T-b+1}{1}$, where $a$ and $b$
are natural numbers $1\le a\le b$. The first case is equivalent with $d=a$
and $g=(d-1)(d-2)/2$, i.e., equivalent with the case of plane curves.    \\
If $\xi_1,\xi_2\in\HH(k)$, then we write $\xi_1\equiv\xi_2$ iff $f(\xi_1)= f(\xi_2)$,
where $f$ is a tautological morphism.

We consider only the case $g<(d-1)(d-2)/2$. In this case we have the
relations $d=a-1$ and $g=(a^2-3a+4)/2-b$. 

$\GG=\Grass_m(P_d)$ Grassmann scheme of $m$-dimensional subspaces of
$P_d$.

Let $\varphi:\N\to\N$ be a function with the following
properties: There is an ideal $\cI\subset\cO_{\P^2}$ of finite
colength with Hilbert function $h(n)=h^0(\cI(n))$, such that
$0\le\varphi(n)\le h(n)$ for all $n\in\N$ and $\varphi(n)=h(n)$ for
 $n$ large enough, e.g.~$n\ge d :=\colength(\cI)$.
On the category of $k$-schemes a functor is defined by:
\[
G_\phi(\Spec A)=\Set{(U_0,\dots,U_d)|
  \begin{aligned}
    & U_n\subset S_n\otimes A \text{ subbundle of rank } \varphi(n) \\
    & \text{ such that } S_1 U_{n-1}\subset U_n, 1\le n\le d
  \end{aligned}
}
\]
$G_{\varphi}$ is a closed subscheme of a suitable product of Grassmann
schemes; it is called \emph{graded Hilbert scheme}.

To each ideal $\cJ\subset\cO_{\P^3_k}$ with Hilbert polynomial $Q$
corresponds a point $\xi\in\HH(k)$, which we denote by
$\xi\leftrightarrow\cJ$.

$h(\cJ)$ denotes the Hilbert function of $\cJ$, that means
$h(\cJ)(n)=\dim_kH^0(\cJ(n)), n\in\N$.

If $\varphi$ is the Hilbert function of an ideal in $\cO_{\P_k^2}$ 
such that $\phi(n) = \tbinom{n+2}{2} - d $ for all sufficiently great
 natural numbers $n$, then 
\[
H_\varphi : = \Set{ \cI \subset \cO_{\P_k^2} | h^0(\cI(n)) = \varphi(n),
  n\in \N }
\]
is a locally closed subset of $\Hilb^d(\P^2)$, which we regard to have the
induced reduced scheme structure. $H^{(d)} \subset \Hilb^d(\P^2)$ is the open 
subscheme of points $\xi\leftrightarrow Z\subset \P^2$ such that $Z$ consists
of $d$ points in $\P^2$.

If $G$ is a subgroup of $GL(4;k)$, then $ \HH^G$ denotes the
fixed-point scheme, which is to have the induced reduced scheme structure. 
The same convention is to be valid for all fixed-point subschemes of 
$H^d = \Hilb^d(\P^2)$. 

If $ C\hookrightarrow \HH $ is a curve, then by means of the
Grothendieck-Pl\"ucker embedding $\HH \rightarrow \P^N$ we can regard $C$ as
a curve in a projective space, whose Hilbert polynomial has the form
$\deg(C)\cdot T + c$. Here $\deg(C)$ is defined as follows:

If $\cI$ is the universal sheaf of ideals on $X = \HH \times
\P_k^3$, then $\cF : = \cO_X /\cI$ is the structure
sheaf of the universal curve $\CC$ over $\HH $, and the direct
image $\pi_*\cF(n)$ is locally free on $\HH$ of rank $P(n)$ for
all $n \geq b$. The line bundles $\cM_n : =
\dot\bigwedge\pi_* \cF(n)$ are called the tautological line bundles
on $\HH$, which are very ample and thus define the Grothendieck--Pl\"ucker
embeddings in suitable projective spaces.  Here $\dot\wedge$ is to denote
the highest exterior power.  Then $\deg(C)$ is the intersection number
$\deg(\cM_n|C) : = (\cM_n\cdot C)$. (If $C$ is a so called tautological or
basis cycle one can compute this intersection number directly,
see~\cite[Section 4.1]{T2}.)

After these more or less conventional notations, we introduce some notations
concerning monomial ideals.

If $ \cJ \subset \cO_{\P^3}$ is $T$-invariant, then $H^0(\P_k^3;\cJ(d))
\subset \cO_{\P^3}$ is generated by monomials. To each monomial $
x^{d-(a+b+c)}y^{a}z^{b}t^{c}$ in $ H^0(\cJ(d))$ we associate the cube $
[a,a+1]\times [b,b+1]\times [c,c+1]$ in a $y$-$z$-$t$-coordinate system, and
the union of these cubes gives a so called pyramid, \index{pyramid}, which is
denoted by $E(\cJ(d))$. Usually we assume that $\cJ$ is invariant under
$\Delta$ or $\Gamma$. Then we can write $H^0(\cJ(d)) =
\bigoplus\limits^d_{n=0}t^{d-n}U_n $, where $ U_n \subset S_n $ are
subspaces such that $ S_1\cdot U_n \subset U_{n+1}, 0\leq n\leq d-1 $, which
we call the layers of the pyramid. (In \cite{T1}--\cite{T4} we made
extensive use of this concept, but here it occurs only once in
Section~\ref{sec:1.3}

A graded ideal $I \subset S = k[x_0,\dots,x_r]$ is \emph{Borel normed},
\index{Borel normed} if $\inn(I) = \bigoplus_{n\geq 0} \inn (I_n)$ is
invariant under $B(r+1;k)$.  To each graded ideal $I \subset S$ there is a
non-empty, open set $U \subset \GL(r+1;k)$ such that $g(I)$ is Borel normed
for all $g \in U$.

If $H$ is a Hilbert scheme of ideals in $\cO_{\P^r}$ and if $b \in H(k)$ is
fixed by $B(r+1;k)$, then $W_H(b) \subset H$ is the subscheme of all ideals
$\cI \subset \cO_{\P^r}$ such that the initial ideal $\inn(\cI)
\leftrightarrow b$. (For more details, see~\cite[Section 2]{G88}.)

\pagebreak
\section{Explanations}
\label{sec:H.2}

In \cite{T1}--\cite{T4} it was tried to describe the first Chow group
$A_1(\HH)$, where we always take rational coefficients, and we write
$A_1(\HH)$ instead of $A_1(\HH)\mathop{\otimes}_{\Z}\Q$. The starting point
is the following consideration: If the Borel group $B=B(4;k)$ operates on
$\HH=\HH_Q$ in the obvious way, then one can deform each 1-cycle on $\HH$ in
a 1-cycle, whose prime components are $B$-invariant, irreducible, reduced
and closed curves on $\HH$. It follows that $A_1(\HH)$ is generated by such
$B$-invariant 1-prime cycles on $\HH$. This is a partial statement of a
theorem of Hirschowitz, which can be applied to any projective scheme with
a $B$-action (see ~\cite{Hi}). 

 Now from [T1, Section 1.1] it follows that such a $B$-invariant 
 $1$-prime cycle (i.e.~closed, irreducible and
 reduced curve) $C$ on $\HH$ can be formally described as follows: Either
 each point of $C$ is invariant under $\Delta:=U(4;k)$, or one has
 $C=\overline{\G_a^i\cdot\eta}$, where $\eta$ is a closed point of $\HH$,
 which is invariant under $T=T(4;k)$ and the group $G_i, i\in\{ 1,2,3\}$.
 Here $\G_a^i$ is the group $\G_a$, acting by
  \begin{alignat*}{3}
    \psi_{\alpha}^1 :x\mapsto x,\quad y &\mapsto y, & \quad z &\mapsto z,\quad & 
    t& \mapsto\alpha z+t. \\
    \psi_{\alpha}^2 :x\mapsto x,\quad y & \mapsto y,\quad & z& \mapsto\alpha
    y+z,\quad&  t&\mapsto t, \\
    \psi_{\alpha}^3 :x\mapsto x,\quad y & \mapsto\alpha x+ y,\quad & z& \mapsto z,\quad     & t& \mapsto t,
  \end{alignat*}
 respectively, on $P=k[x,y,z,t]$, and $G_i$ is the subgroup of $\Delta$, which is complementary to $\G_a^i$, that means, one defines
 \[
 G_1:=\left\{ \left( \begin{array}{cccc}
       1&*&*&*\\0&1&*&*\\0&0&1&0\\0&0&0&1\end{array}\right)\right\},\quad
 G_2:=\left\{ \left( \begin{array}{cccc}
       1&*&*&*\\0&1&0&*\\0&0&1&*\\0&0&0&1\end{array}\right)\right\},\quad
 G_3:=\left\{ \left( \begin{array}{cccc}
       1&0&*&*\\0&1&*&*\\0&0&1&*\\0&0&0&1\end{array}\right)\right\}.
 \]
 If $C$ has this form, then $C$ is called a combinatorial cycle of type $i$,
 where $i\in \{ 1,2,3\}$. $\cA(\HH):=\Im(A_1(\HH^{\Delta})\to A_1(\HH))$ is
 called the ``algebraic part'' and $\overline{A}_1(\HH):=A_1(\HH)/\cA(\HH)$
 is called the ``combinatorial part'' of the first Chow group of $\HH$. Here
 $\HH^{\Delta}$ denotes the fixed point scheme, which, just as all other
 fixed point schemes that will occur later on, is supposed to have the
 induced reduced scheme structure.\\
 This convention is valid also for the Hilbert scheme
 $H^d:=\Hilb^d(\mP^2_{\C})$.

 In order to formulate the results of \cite{T1}--\cite{T5}, one has to
 introduce the following ``tautological'' 1-cycles on $\HH$:
\index{tautological cycle}
 \begin{align*}
   C_1&=\Set{(x,y^a,y^{a-1}z^{b-a}(\alpha z+t))|\alpha\in k }^-\\
   C_2&=\Set{(x,y^{a-1}(\alpha y+z),y^{a-2}z^{b-a+1}(\alpha y+z))|\alpha\in k}^-\\
   C_3&=\Set{(x^a,\alpha x+y,x^{a-1}z^{b-a+1})|\alpha\in k}^-\\
   D&=\Set{(x^2,xy,y^{a-1},z^{b-2a+4}(y^{a-2}+\alpha xz^{a-3}))|\alpha\in k}^-\\
   E&=\Set{(x^2,xy,xz,y^a,y^{a-1}z^{b-a+1},xt^{b-2}+\alpha
   y^{a-1}z^{b-a})|\alpha\in k}^-
 \end{align*}

Then the final results are (cf.~\cite[Thm. 15.1 and 16.1]{T5}):\\

Suppose that $d\geq 5$ and $g < \tbinom{d-1}{2}$, i.e.~$g$ is not
maximal. Put $g(d):= (d-2)^2/4$ and $\gamma(d):= \tbinom{d-2}{2}$.

\begin{theorem*}[{\cite[p. 123]{T5}}]
  \begin{enumerate}[(i)]
  \item If $g>\gamma(d)$, then $A_1(H_{d,g})$ is freely generated by
    $E,C_1,C_2,C_3$.
  \item If $g(d)<g\leq\gamma(d)$, then $A_1(H_{d,g})$ is freely generated by
    $E,D,C_2,C_3$.
  \item If $g\leq g(d)$, then $A_1(H_{d,g})$ is freely generated by
    $E,D,C_2$.
   \end{enumerate}
\end{theorem*}

\begin{theorem*}[{\cite[p. 127]{T5}}]
  \begin{enumerate}[(i)]
  \item If $g\le 0$, then $A_1(H_{3,g})$ is freely generated by $[E], [D],
    [C_2]$.
  \item $A_1(H_{4,2})\simeq\Q^4$ and if $g\le 1$, then $A_1(H_{4,g})$ is
    freely generated by $[E],[D],[C_2]$.
  \end{enumerate}
\end{theorem*}

\backmatter

Gerd Gotzmann, Isselstrasse 34, 48431 Rheine, Germany,
g.gotzmann@t-online.de 
\printindex 

\begin{thebibliography}{EGA}
\bibitem[EGA]{EGA} Grothendieck, A.: {\it El\'ements de G\'eometrie
    Alg\'ebrique, Chapter II}, IHES, 1961.

\bibitem[D]{D} Davis, E.: {\it $0$-dimensional subschemes of $P^2$: new
  application of Castelnuovo's function}. Ann. Univ. Ferrara Sez. VII (N.S.)
{\bf 32} (1986), 93--107, (1987).

\bibitem[E-S]{E-S} Ellingsrud, G., Str\o{}mme, S.:
  {\it On the homology of the Hilbert scheme of points in the plane}.
  Invent. Math. {\bf 87} (1987), no. 2, 343--352.
 
\bibitem[F1]{F1} Fogarty, J.: {\it Truncated Hilbert functors}.
J. Reine Angew. Math. {\bf 234},  65--88, (1969).

\bibitem[F2]{F2} \bysame: {\it Algebraic families on an algebraic surface}.  Amer. J.
  Math. {\bf 90}, 511--521, (1968).

\bibitem[F3]{F3} \bysame: {\it Algebraic families on an algebraic surface,
    II.  The Picard scheme of the punctual Hilbert scheme}. Amer. J. Math.
  {\bf 95}, 660--687, (1973).

\bibitem[Fu]{Fu} Fulton, W.: {\it Intersection theory}. Springer-Verlag 1984.

\bibitem[G1]{G78} Gotzmann, G.: {\it Eine Bedingung f\"ur die Flachheit und das
  Hilbertpolynom eines graduierten Ringes}. Math. Z. {\bf 158}, 61--70, (1978).

\bibitem[G2]{G82} \bysame: {\it Einige einfach-zusammenh\"angende
    Hilbertschemata}, Math. Z. {\bf 180}, (1982), 291--305.

\bibitem[G3]{G88} \bysame: {\it A stratification of the Hilbert scheme of
    points in the projective plane}. Math. Z. {\bf 199}, 539--547, (1988).

\bibitem[G4]{G89} \bysame: {\it Some irreducible Hilbert schemes}.  Math. Z.
  {\bf 201} (1989), 13--17.

\bibitem[G5]{G90} \bysame: {\it Einfacher Zusammenhang der Hilbertschemata
     von Kurven im komplex-projectiven Raum}. Invent. Math. {\bf 99},
   655--675 (1990).

\bibitem[G6]{G93} \bysame: {\it Topologische Eigenschaften von
    Hilbertfunktion--Strata}. Habilitationsschrift, Universit\"at M\"unster,
  1993.

\bibitem[Gre]{Green} Green, M.: {\it Generic initial ideals}. in: {\it Six
    lectures on commutative algebra}. Progress in Mathematics {\bf 166}, 119--186,
  Birkh\"auser, Basel (1998).

\bibitem[Har]{Har} Harris, J.: {\it Algebraic geometry}, Springer--Verlag
  1992.

\bibitem[HM]{HM} Harris, J., Morrison, I.: {\it Moduli of curves},
  Springer--Verlag 1998.

\bibitem[H1]{H} Hartshorne, R.: {\it Algebraic geometry}, Springer--Verlag
  1977.
	
\bibitem[H2]{H2} \bysame, {\it Questions of Connectedness of the Hilbert Scheme
   of Curves in $\P^3$}, \href{http://arXiv:math/1004265v1}{arXiv:math/1004265v1}
   [math.AG], Apr 27, 2001.


 \bibitem[Hi]{Hi} Hirschowitz, A.: {\it Le group de Chow \'equivariant}. C.R.
Acad, Sc. Paris, t. {\bf 298}, S\'erie I. Math\'ematique, no. 5, 87--89
  (1984).
 
\bibitem[I]{I} Iarrobino, A.: {\it Punctual Hilbert schemes}. Bull. Amer.
  Math. Soc.  Vol. {\bf 78}, no. 5, (1972).

\bibitem[K]{K} Kleiman, S.: {\it Toward a numerical theory of ampleness}.
Ann. of Math. {\bf 84}, 293--344 (1966)

\bibitem[LQZ]{LQZ} Li, W.; Qin, Z.; Zhang, Q.: 
{\it On the geometry of the Hilbert schemes of points in the projective plane}.
\href{http://arxiv.org/abs/math/0105213v2}{arxiv.org/abs/math/0105213v2}

\bibitem[M1]{M1} Mumford, D.: {\it Geometric invariant
    theory}. Springer--Verlag 1965.

\bibitem[M2]{M2} \bysame:  {\it Lectures on curves on an algebraic surface},
  Princeton, 1966. 

\bibitem[T1]{T1} Gotzmann, G.: {\it Der kombinatorische Teil der ersten
    Chowgruppe eines Hilbertschemas von Raumkurven}. Schriftenreihe des
  Mathematischen Instituts der Universit\"at M\"unster, 3. Serie, Heft {\bf
    13}, September 1994.

  \bibitem[T2]{T2} \bysame: {\it Der algebraische Teil der ersten Chowgruppe
      eines Hilbertschemas von Raumkurven}, ibid., Heft {\bf 19}, Februar 1997.

  \bibitem[T3]{T3} \bysame: {\it Die N\'eron--Severi--Gruppe eines
      Hilbertschemas von Raumkurven und der universellen Kurve}, ibid., Heft
    {\bf 23}, Januar 1999.

\bibitem[T4]{T4} \bysame: {\it Die erste Chowgruppe eines Hilbertschemas von
  Raumkurven}, ibid., Heft {\bf 25}, M\"arz 2000. 

\bibitem[T5]{T5} \bysame: {\it Computation of the first Chow group of a Hilbert
  scheme of space curves}, \href{http://arxiv.org/abs/1103.0122}{arxiv.org/abs/1103.0122}v2.
\end{thebibliography}
\end{document}